\documentclass[12pt]{amsart}
\usepackage{amsfonts}
\usepackage{natbib}
\usepackage{eurosym}
\usepackage{caption}
\usepackage{amssymb, amsmath, url,color,  amsthm, amssymb,graphicx,multirow, appendix}
\usepackage[margin=1in]{geometry}
\usepackage{bigints}
\usepackage{lscape,amssymb, amsmath,verbatim,graphicx}
\usepackage{epsfig,subfigure,float,subfigure}
\usepackage{graphicx}
\usepackage{epsfig}
\usepackage{threeparttable}
\usepackage{ragged2e}
\usepackage{epstopdf}
\usepackage{booktabs}
\usepackage{float}
\usepackage{multirow}
\usepackage{adjustbox}
\usepackage{lscape,lipsum}
\usepackage{url}
\usepackage{natbib}
\usepackage[T1]{fontenc}
\usepackage{a4wide, url}
\usepackage[english]{babel}
\usepackage{graphicx}
\usepackage{array}
\usepackage{multirow}
\usepackage{amsthm}
\usepackage{amsmath}
\usepackage[foot]{amsaddr}
\usepackage{amsfonts}
\usepackage{amssymb}
\usepackage{pgf}
\usepackage{paralist}
\usepackage{pdflscape}
\usepackage{lscape,amssymb, amsmath,verbatim,graphicx}
\usepackage{epsfig,subfigure,float,subfigure}
\usepackage{graphicx}
\usepackage{epsfig}
\usepackage{amssymb}
\usepackage{amsmath}
\usepackage{amsfonts}
\usepackage{geometry}
\usepackage{scalefnt}
\usepackage{graphicx}
\usepackage{setspace}
\usepackage{array}
\usepackage{multirow}
\usepackage{amsthm}
\usepackage{amsmath}
\usepackage{amsfonts}
\usepackage{amssymb}
\usepackage{pgf}
\usepackage{paralist}
\usepackage{pdflscape}
\usepackage{rotating}
\usepackage{scalefnt}

\allowdisplaybreaks
\DeclareFontFamily{OT1}{pzc}{}
\DeclareFontShape{OT1}{pzc}{m}{it}{<-> s * [1.10] pzcmi7t}{}
\DeclareMathAlphabet{\mathpzc}{OT1}{pzc}{m}{it}

\numberwithin{table}{section}
\numberwithin{figure}{section}
\newtheorem{theorem}{Theorem}[section]
\newtheorem{lemma}{Lemma}[section]
\newtheorem{corollary}{Corollary}[section]

\newtheorem{assumption}{Assumption}[section]

\newcommand\eps{\epsilon}

\def\beq{\begin{equation}}
\def\eeq{\end{equation}}
\def\bals{\begin{align*}}
\def\eals{\end{align*}}
\def\bal{\begin{align}}
\def\eal{\end{align}}

\numberwithin{equation}{section}
\numberwithin{theorem}{section}
\numberwithin{corollary}{section}
\allowdisplaybreaks
\begin{document}
\title[Change points in RCA models ]{Change point detection in random
coefficient autoregressive models}
\author{Lajos Horv\'ath}
\address{Lajos Horv\'ath, Department of Mathematics, University of Utah,
Salt Lake City, UT 84112--0090 USA}
\author{Lorenzo Trapani}
\address{Lorenzo Trapani, School of Economics, University of Nottingham,
University Park, Nottingham NG7 2RD U.K.}
\subjclass{Primary 62M10; Secondary 62G20}
\keywords{Changepoint problem, weighted CUSUM process, Random Coefficient
AutoRegression, Darling-Erd\H{o}s theorem.}

\begin{abstract}
We propose a family of CUSUM-based statistics to detect the presence of
changepoints in the deterministic part of the autoregressive parameter in a
Random Coefficient AutoRegressive (RCA) sequence. In order to ensure the
ability to detect breaks at sample endpoints, we thoroughly study \textit{%
weighted} CUSUM statistics, analysing the asymptotics for virtually all
possible weighing schemes, including the standardised CUSUM process (for
which we derive a Darling-Erd\H{o}s theorem) and even heavier weights
(studying the so-called R\'enyi statistics). Our results are valid
irrespective of whether the sequence is stationary or not, and indeed prior
knowledge of stationarity or lack thereof is not required from a practical
point of view. From a technical point of view, our results require the
development of strong approximations which, in the nonstationary case, are
entirely new. Similarly, we allow for heteroskedasticity of unknown form in
both the error term and in the stochastic part of the autoregressive
coefficient, proposing a family of test statistics which are robust to
heteroskedasticity; again, our tests can be readily applied, with no prior
knowledge as to the presence or type of heteroskedasticity. Simulations show
that our procedures work well in finite samples, under all cases considered
(stationarity versus nonstationarity, homoskedasticity versus various forms
of heteroskedasticity). We complement our theory with applications to
financial, economic and epidemiological time series.
\end{abstract}

\maketitle

%\newpage

\section{Introduction\label{intro}}

In this paper we study the stability of the autoregressive parameter of an
RCA(1) sequence: 
\begin{comment}
\begin{align}\label{arcnew1}
y_i=\left\{
\begin{array}{ll}
(\beta_k_1+\eps_{i,1})y_{i-1}+\eps_{i,2},\quad\mbox{if}\;\;\;1\leq i \leq k_1,
\vspace{.3cm}\\
(\beta_A+\eps_{i,1})y_{i-1}+\eps_{i,2},\quad\mbox{if}\;\;\;k_1+1\leq i \leq k_2,
\\
\quad \vdots\quad \vdots
\\
(\beta_{R+1}+\eps_{i,1})y_{i-1}+\eps_{i,2}, \;\;\;k_{R}+1\leq i\leq N,
\end{array}
\right.
\end{align}
where $R$ is the number of changes in the autoregression coefficient. We test for the null hypothesis of is no change versus the alternative of at most one change (AMOC) i.e.
\beq\label{rcanu}
H_0:\;k_1>N
\eeq
tested against
\beq\label{rcaor1}
H_A:\;1<k_1<k_2<\ldots<k_R<N\;\;\mbox{and}\;\;\beta_1\neq \beta_2\neq \ldots\neq \beta_{R+1}.
\eeq
\end{comment}%
\begin{equation}
y_{i}=\left\{ 
\begin{array}{ll}
\left( \beta _{0}+\epsilon _{i,1}\right) y_{i-1}+\epsilon _{i,2},\quad %
\mbox{if}\;\;\;1\leq i\leq k^{\ast }, &  \\ 
\left( \beta _{A}+\epsilon _{i,1}\right) y_{i-1}+\epsilon _{i,2},\quad %
\mbox{if}\;\;\;k^{\ast }+1\leq i\leq N, & 
\end{array}%
\right.  \label{arcnew1}
\end{equation}%
where $y_{0}$ denotes an initial value. We test for the null hypothesis of
no change versus the alternative of at most one change (AMOC) i.e. 
\begin{align}
H_{0}& :\;k^{\ast }>N,  \label{rcanu} \\
H_{A}& :\;1<k^{\ast }<N\;\;\mbox{and}\;\;\beta _{0}\neq \beta _{A}.
\label{racor1}
\end{align}

The RCA model was firstly studied by \citet{andel} and \citet{nichollsquinn}%
. It belongs in the wider class of nonlinear models for time series (see %
\citealp{fanyao}), which have been proposed \textquotedblleft as a reaction
against the supremacy of linear ones - a situation inherited from strong,
though often implicit, Gaussian assumptions\textquotedblright\ (%
\citealp{akharif2003}). Arguably, (\ref{arcnew1}) is very flexible, allowing
for the autoregressive \textquotedblleft root\textquotedblright\ $\beta
_{0}+\epsilon _{i,1}$ to vary over time, and thus for the possibility of
having stationary and nonstationary regimes. This may be a more appropriate
model than a linear specification (see \citealp{lieberman2012}; %
\citealp{leybourne1996}); \citet{gky} argue that a time-varying parameter
model like (\ref{arcnew1}) can be viewed as a competitor for a model with an
abrupt break in the autoregressive root. Furthermore, equation (\ref{arcnew1}%
) also allows for the possibility of (conditional) heteroskedasticity in $%
y_{i}$; \citet{tsay1987} shows that the widely popular ARCH\ model by %
\citet{engle1982} can be cast into (\ref{arcnew1}), which therefore can be
viewed as a second-order equivalent. Finally, a major advantage of (\ref%
{arcnew1}) compared to standard autoregressive models is that estimators of $%
\beta _{0}$ are always asymptotically normal, irrespective of whether $y_{i}$
is stationary or nonstationary, thus avoiding the risk of over-differencing
(see \citealp{leybourne1996}).\newline

Given such generality and flexibility, (\ref{arcnew1}) has been used in many
applied sciences, including biology (\citealp{stenseth}), medicine (%
\citealp{fryz}), and physics (\citealp{slkezak2019random}). The RCA\ model
has also been applied successfully in the analysis of economic and financial
data, and we refer to the recent contribution by \citet{regis} for a
comprehensive review.

The inferential theory for (\ref{arcnew1}) has been studied extensively. %
\citet{schick1996}, \citet{koul1996} and \citet{praskova2004} study Weighted
Least Squares (WLS)\ estimation of $\beta _{0}$; \citet{berkes2009} and %
\citet{aue2011} study Quasi Maximum Likelihood estimation, and %
\citet{hillpeng2016} develop an Empirical Likelihood estimator. Several
tests have also been developed, including tests for stationarity (see e.g. %
\citealp{zhao2012}; and \citealp{trapanistrict}) and for the randomness of
the autoregressive coefficient (\citealp{akharif2003}; \citealp{nagakura2009}%
; and \citealp{HT16}).

In contrast, changepoint detection is still underexplored in the RCA
framework. To the best of our knowledge, the only exceptions are %
\citet{lee1998}, \citet{lee2003cusum} and \citet{aue2004strong}; in these
papers, a CUSUM\ test is proposed, but only for the stationary case and
based on the unweighted CUSUM\ process. The latter is well-known to suffer
from low power, being in particular less able to detect changepoints
occurring at the beginning/end of the sample. As a solution, the literature
has proposed weighted versions of the CUSUM\ process on the interval $\left[
0,1\right] $, where more emphasis is given to observations at the sample
endpoints (see \citealp{csorgo1997}). Weighing functions are typically of
the form $\left[ t\left( 1-t\right) \right] ^{\kappa }$ with $0\leq \kappa
<\infty $, for $t\in \left[ 0,1\right] $, with more weight placed on
observations at the endpoints as $\kappa $ increases. In particular, the
case $\kappa =\frac{1}{2}$ corresponds to the standardised CUSUM\ process
also proposed by \citet{andrews1993}, whereas the more heavily weighted case 
$\kappa >\frac{1}{2}$ corresponds to a family of test statistics known as
\textquotedblleft R\'{e}nyi statistics\textquotedblright\ (see %
\citealp{horvathmiller}). When $\kappa >0$, the asymptotics becomes more
complicated, since the weighted statistics diverge at the endpoints $t=0$
and $t=1$, and one can no longer rely on weak convergence to derive the
limiting distributions. In order to overcome this issue, \citet{andrews1993}
proposes trimming the interval on which the weighted CUSUM\ process is
studied; however, this has the undesirable consequence that tests are unable
to detect breaks when these occurs e.g. at the end of the sample.

\bigskip

\textit{Contribution of this paper}

\bigskip

In this paper, we bridge all the gaps mentioned above by proposing a family
of weighted, untrimmed CUSUM statistics. Our paper makes the following four
contributions.

First, we study virtually all possible weighing schemes, deriving the
asymptotics for all $0\leq \kappa <\infty $. From a practical viewpoint,
this entails that our test statistics are designed to detect breaks even
when these are very close to the sample endpoints. Second, all our results
hold irrespective of whether $y_{i}$ is stationary or not; this robustness
arises from using the WLS estimator, and from the well-known fact that the
RCA\ model does not suffer from the \textquotedblleft knife edge
effect\textquotedblright\ which characterizes linear models (%
\citealp{lumsdaine1996consistency}). From a practical point of view, this
entails that the tests can be applied with no modifications required, and no
prior knowledge of the stationarity of $y_{i}$ or lack thereof. This feature
is particularly desirable e.g. in the context of detecting the beginning (or
end) of bubbles (see \citealp{harvey2016}): with our set-up, it is possible
to detect changes from stationary to nonstationary/explosive behaviour (as
e.g. in \citealp{horvath2020sequential}; and \citealp{horvath2021sequential}%
) which characterize the emergence of a bubble, but it is also possible -
again with no modifications required - to detect changes from explosive to
non-explosive behaviour, as would be the case at the end of a bubble. Being
able to accommodate both cases is a distinctive advantage of the RCA set-up:
whilst tests for changes \textit{towards} an explosive behaviour have been
developed in the literature (see, \textit{inter alia}, \citealp{phillips2011}%
; \citealp{phillips2015testing}; and the review by \citealp{homm2012testing}%
), tests to detect changes \textit{from} an explosive behaviour are more
rare, possibly due to the more complicated asymptotics in this case. Third,
we allow for heteroskedasticity in both $\epsilon _{i,1}$ and $\epsilon
_{i,2}$, which is usually not considered in the RCA context; interestingly,
for the case $\kappa \geq \frac{1}{2}$, we recover the same, nuisance free
distribution as in the homoskedastic case (in particular, when $\kappa =%
\frac{1}{2}$, we obtain a \textquotedblleft classical\textquotedblright\
Darling-Erd\H{o}s limit theorem). Hence, our modified test statistics can be
used from the outset, with no prior knowledge required as to whether $%
\epsilon _{i,1}$, or $\epsilon _{i,2}$, or both, is heteroskedastic. Fourth,
our asymptotics is based on strong approximations for the partial sums of an
RCA\ sequence, which are valid irrespective of the stationarity or lack
thereof of $y_{i}$; the strong approximation for the nonstationary case is
entirely new. 
\begin{comment}
As a fifth and
final point, our approach can be applied to the whole (very broad) class of
decomposable Bernoulli shifts (see \citealp{wu05}; \citealp{aue09}; %
\citealp{berkeshormann}); thus, although our focus is on equation (\ref%
{arcnew1}), all our proof and results can be extended to sequences which can
be represented as Bernoulli shifts: these include, \textit{inter alia},
linear processes, GARCH\ sequences, threshold autoregressive models and
other nonlinear models (see \citealp{linliu}, for a comprehensive set of
examples).

Our simulations show that our procedures have very good finite sample
properties, being able to detect breaks both in the middle and at endpoints.
As predicted by the theory, R\'{e}nyi-type statistics (based on using $%
\kappa >\frac{1}{2}$) prove very effective at detecting breaks at either end
of the sample.
\end{comment}%\bigskip
\newline

The remainder of the paper is organised as follows. We present our test
statistics in Section \ref{tests}, and study their asymptotics in the
homoskedastic case, as a benchmark, in Section \ref{homosk}. The
heteroskedastic case is studied in Section \ref{heterosk}. In Section \ref%
{simulations}, we report a simulation exercise; applications to real data
are in Section \ref{empirics}. Section \ref{conclusions} concludes.
Extensions, technical lemmas and all proofs are relegated to the Supplement.

NOTATION. We use the following notation: \textquotedblleft $\overset{{%
\mathcal{D}}}{\rightarrow }$\textquotedblright\ for weak convergence;
\textquotedblleft $\overset{\mathcal{P}}{\rightarrow }$\textquotedblright\
for convergence in probability; \textquotedblleft \textit{a.s.}%
\textquotedblright\ for \textquotedblleft almost surely\textquotedblright ;
\textquotedblleft $\overset{{\mathcal{D}}}{=}$\textquotedblright\ for
equality in distribution; $\lfloor \cdot \rfloor $ is the integer value
function. Positive, finite constants are denoted as $c_{0}$, $c_{1}$, ...
and their value may change from line to line. Other notation is introduced
further in the paper.

\section{The test statistics\label{tests}}

Our approach is based on comparing the estimates of $\beta _{0}$ before and
after each point in time $k$, by dividing the data into two subsets at $k$
and estimating the autoregressive parameter in both subsamples. As mentioned
above, we use WLS, with weights $1+y_{i-1}^{2}$. This has the advantages of 
\textit{(i)} avoiding restrictions on the moments of the observations, and 
\textit{(ii)} ensuring standard normal asymptotics irrespective of whether $%
y_{i}$ is stationary or not. The WLS estimators are 
\begin{equation}
\widehat{\beta }_{k,1}=\left( \sum_{i=2}^{k}\frac{y_{i-1}^{2}}{1+y_{i-1}^{2}}%
\right) ^{-1}\left( \sum_{i=2}^{k}\frac{y_{i}y_{i-1}}{1+y_{i-1}^{2}}\right)
,\quad 2\leq k\leq N,  \label{rcabet}
\end{equation}%
and 
\begin{equation}
\widehat{\beta }_{k,2}=\left( \sum_{i=k+1}^{N}\frac{y_{i-1}^{2}}{%
1+y_{i-1}^{2}}\right) ^{-1}\left( \sum_{i=k+1}^{N}\frac{y_{i}y_{i-1}}{%
1+y_{i-1}^{2}}\right) ,\quad 1\leq k\leq N-1.  \label{rcabet2}
\end{equation}%
Our test statistics will be functionals of the process 
\begin{equation}
Q_{N}(t)=\left\{ 
\begin{array}{ll}
0,\quad \mbox{if}\;\;\;0\leq t<2/(N+1), &  \\ 
\displaystyle N^{1/2}(t(1-t))(\widehat{\beta }_{\lfloor (N+1)t\rfloor ,1}-%
\widehat{\beta }_{\lfloor (N+1)t\rfloor ,2}),\quad \mbox{if}%
\;\;\;2/(N+1)\leq t<1-2/(N+1), &  \\ 
0,\quad \mbox{if}\;\;\;1-2/(N+1)<t\leq 1. & 
\end{array}%
\right.  \label{qnt}
\end{equation}

A \textquotedblleft natural\textquotedblright\ choice to detect the presence
of a possible change is to use the sup-norm of (\ref{qnt}), viz. $%
\sup_{0<t<1}\left\vert Q_{N}(t)\right\vert $, but, as mentioned above, this
choice may have low power in detecting changes which occur early or late in
the sample. In order to enhance the power at sample endpoints, one can use
weight functions:%
\begin{equation}
\sup_{0<t<1}\frac{\left\vert Q_{N}(t)\right\vert }{w\left( t\right) }.
\label{sup-w-norm}
\end{equation}

\begin{assumption}
\label{as-wc-1} It holds that: (i) $\inf_{\delta \leq t\leq 1-\delta }w(t)>0$
for all $0<\delta <1/2$; (ii) $w(t)$ is non decreasing in a neighborhood of $%
0$; (iii) $w(t)$ is non increasing in a neighborhood of $1$.
\end{assumption}

The functions $w(t)$ satisfying Assumption \ref{as-wc-1} belong in a very
wide class; a possible example is $w(t)=\left( t\left( 1-t\right) \right)
^{\kappa }$ with $\kappa >0$. The existence of the limit of (\ref{sup-w-norm}%
) can be determined based on the finiteness of the integral functional (see %
\citealp{csorgo1993}) 
\begin{equation}
I(w,c)=\int_{0}^{1}\frac{1}{t(1-t)}\exp \left( -\frac{cw^{2}(t)}{t(1-t)}%
\right) dt.  \label{defiwc}
\end{equation}%
As we show below, (\ref{defiwc}) entails that $w(t)=\left( t\left(
1-t\right) \right) ^{\kappa }$ with $0<\kappa <\frac{1}{2}$ can be employed
in this context.

In order to further enhance the power of our testing procedures, functions
which place more weight at the sample endpoints can also be used, i.e.%
\begin{equation}
\sup_{0<t<1}\frac{\left\vert Q_{N}(t)\right\vert }{\left( t\left( 1-t\right)
\right) ^{\kappa }},  \label{sup-R\'{e}nyi}
\end{equation}%
with $\kappa \geq \frac{1}{2}$. As mentioned above, when $\kappa =\frac{1}{2}
$, the corresponding limit theorems will be of the Darling-Erd\H{o}s type (%
\citealp{darling1956limit}); when $\kappa >\frac{1}{2}$, the test statistics
defined in (\ref{sup-R\'{e}nyi}) are known as \textquotedblleft R\'{e}nyi
statistics\textquotedblright\ (\citealp{horvathmiller}).

\section{Testing for changepoint under homoskedasticity\label{homosk}}

We begin by assuming that the errors $\{\epsilon _{i,1},\epsilon
_{i,2},-\infty <i<\infty \}$ have constant variance.

\begin{assumption}
\label{rcaas} It holds that:\ (i) $\{\epsilon _{i,1},-\infty <i<\infty \}$
and $\{\epsilon _{i,2},-\infty <i<\infty \}$ are independent sequences; (ii) 
$\{\epsilon _{i,1},-\infty <i<\infty \}$ are independent and identically
distributed random variables with $E\epsilon _{i,1}=0$, $0<E\epsilon
_{i,1}^{2}=\sigma _{1}^{2}<\infty $ and $E|\epsilon _{i,1}|^{4}<\infty $;
(iii) $\{\epsilon _{i,2},-\infty <i<\infty \}$ are independent and
identically distributed random variables with $E\epsilon _{i,2}=0$, $%
0<E\epsilon _{i,2}^{2}=\sigma _{2}^{2}<\infty $ and $E|\epsilon
_{i,2}|^{4}<\infty $.
\end{assumption}

In (\ref{arcnew1}), the stationarity or lack thereof of $y_{i}$ is
determined by the value of $E\ln \left\vert \beta _{0}+\epsilon
_{0,1}\right\vert $ (see \citealp{aue2006}). In particular, if $-\infty \leq
E\ln \left\vert \beta _{0}+\epsilon _{0,1}\right\vert <0$, then $y_{i}$
converges exponentially fast to a strictly stationary solution for all
initial values $y_{0}$. Conversely, if $E\ln \left\vert \beta _{0}+\epsilon
_{0,1}\right\vert \geq 0$, then $y_{i}$ is nonstationary - specifically, $%
\left\vert y_{i}\right\vert $ diverges exponentially fast a.s. when $E\ln
\left\vert \beta _{0}+\epsilon _{0,1}\right\vert >0$, whereas it diverges in
probability, but at a rate slower than exponential, in the boundary case $%
E\ln \left\vert \beta _{0}+\epsilon _{0,1}\right\vert =0$ (see %
\citealp{HT2016}). \newline

We show that the asymptotic variance of the limiting process depends on
whether $y_{i}$ is stationary or not: we therefore study the two cases
(stationarity versus lack thereof) separately. We show that the variance of
the weak limit of $Q_{N}\left( t\right) $ is 
\begin{comment}
\begin{equation}
\eta ^{2}=\left\{ 
\begin{array}{ll}
\displaystyle\left[ E\left( \frac{\bar{y}_{0}^{2}}{1+\bar{y}_{0}^{2}}\right)
^{2}\sigma _{1}^{2}+E\left( \frac{\bar{y}_{0}}{1+\bar{y}_{0}^{2}}\right)
^{2}\sigma _{2}^{2}\right] \left( E\frac{\bar{y}_{0}^{2}}{1+\bar{y}_{0}^{2}}%
\right) ^{-2}, &  \\ 
\vspace{0.3cm}\quad \quad \quad \quad \quad \quad \mbox{if}\;\;\;-\infty
\leq E\ln |\beta _{0}+\epsilon _{0,1}|<0,\vspace{0.3cm} &  \\ 
\sigma _{1}^{2},\quad \mbox{if}\;\;\;E\ln |\beta _{0}+\epsilon _{0,1}|\geq 0.
& 
\end{array}%
\right.   \label{etasq}
\end{equation}%
\end{comment}%
\begin{equation}
\eta ^{2}=\left\{ 
\begin{array}{l}
\displaystyle\left[ E\left( \frac{\bar{y}_{0}^{2}}{1+\bar{y}_{0}^{2}}\right)
^{2}\sigma _{1}^{2}+E\left( \frac{\bar{y}_{0}}{1+\bar{y}_{0}^{2}}\right)
^{2}\sigma _{2}^{2}\right] \left( E\frac{\bar{y}_{0}^{2}}{1+\bar{y}_{0}^{2}}%
\right) ^{-2}, \\ 
\quad \quad \quad \quad \quad \quad \quad \mbox{if}\;\;\;-\infty \leq E\ln
|\beta _{0}+\epsilon _{0,1}|<0, \\ 
\sigma _{1}^{2},\quad \mbox{if}\;\;\;E\ln |\beta _{0}+\epsilon _{0,1}|\geq 0.%
\end{array}%
\right.  \label{etasq}
\end{equation}%
We require the following notation. In order to study the case $\kappa =\frac{%
1}{2}$, we define%
\begin{equation}
a(x)=(2\ln x)^{1/2}\quad \mbox{and}\quad b(x)=2\ln x+\frac{1}{2}\ln \ln x-%
\frac{1}{2}\ln \pi .  \label{da-1/0}
\end{equation}%
Also, in order to study the case $\kappa >\frac{1}{2}$, let 
\begin{equation}
r_{N}=\min (r_{1}(N),r_{2}(N)),\quad \lim_{N\rightarrow \infty }\frac{r_{N}}{%
r_{1}(N)}=\gamma _{1}\quad \mbox{and}\quad \lim_{N\rightarrow \infty }\frac{%
r_{N}}{r_{2}(N)}=\gamma _{2},  \label{fadef1}
\end{equation}%
\begin{equation}
\mathfrak{a}(\kappa )=\sup_{1\leq t<\infty }\frac{|W(t)|}{t^{\kappa }}.
\label{fadef3}
\end{equation}

\begin{assumption}
\label{thrcaas11}\ It holds that $r_{1}(N)\rightarrow \infty $, $%
r_{1}(N)/N\rightarrow 0$, and $r_{2}(N)\rightarrow \infty $, $%
r_{2}(N)/N\rightarrow 0$.
\end{assumption}

We start with the stationary case $-\infty \leq E\ln |\beta _{0}+\epsilon
_{0,1}|<0$. In this case, the solution of (\ref{arcnew1}) under the null
hypothesis is close to $\overline{y}_{i}$, the unique anticipative
stationary solution of 
\begin{equation}
\overline{y}_{i}=(\beta _{0}+\epsilon _{i,1})\overline{y}_{i-1}+\epsilon
_{i,2},\quad -\infty <i<\infty .  \label{rcasta}
\end{equation}%
We need the following (technical) assumption, to rule out the degenerate
case that, under stationarity, the denominator of $\eta ^{2}$ defined in (%
\ref{etasq}) is zero with probability $1$.

\begin{assumption}
\label{stationary-nonzero} It holds that $P\{\overline{y}_{0}=0\}<1$.
\end{assumption}

\begin{theorem}
\label{thrca1} We assume that $H_{0}$ of (\ref{rcanu}), Assumptions \ref%
{as-wc-1}, \ref{rcaas} and \ref{stationary-nonzero} hold, and $-\infty \leq
E\ln |\beta _{0}+\epsilon _{0,1}|<0$.\newline
(i) If $I(w,c)<\infty $ for some $c>0$, then it holds that 
\begin{equation*}
\sup_{0<t<1}\frac{\left\vert Q_{N}(t)\right\vert }{w(t)}\;\overset{{\mathcal{%
D}}}{\rightarrow }\;\eta \sup_{0<t<1}\frac{|B(t)|}{w(t)},
\end{equation*}%
where $\{B(t),0\leq t\leq 1\}$ is a standard Brownian bridge and $\eta $ is
defined in (\ref{etasq}). \newline
(ii) For all $x$, it holds that%
\begin{equation*}
\lim_{N\rightarrow \infty }P\Biggl\{a(\ln N)\frac{1}{\eta }%
\max_{1<k<N}\left( \frac{k(N-k)}{N}\right) ^{1/2}|\widehat{\beta }_{k,1}-%
\widehat{\beta }_{k,2}|\leq x+b(\ln N)\Biggl\}=\exp (-2e^{-x}).
\end{equation*}

(iii) If Assumption \ref{thrcaas11} is satisfied, then it holds that 
\begin{equation*}
\left( \frac{r_{N}}{N}\right) ^{\kappa -1/2}\frac{1}{\eta }%
\sup_{r_{1}(N)/N\leq t\leq 1-r_{2}(N)/N}\frac{|Q_{N}(t)|}{\left( t\left(
1-t\right) \right) ^{\kappa }}\;\overset{{\mathcal{D}}}{\rightarrow }\;\max
\left( \gamma _{1}^{\kappa -1/2}\mathfrak{a}_{1}(\kappa ),\gamma
_{2}^{\kappa -1/2}\mathfrak{a}_{2}(\kappa )\right) ,
\end{equation*}%
for all $\kappa >1/2$, where and $r_{N}$, $\gamma _{1}$ and $\gamma _{2}$
are defined in (\ref{fadef1}), and $\mathfrak{a}_{1}\left( \kappa \right) $
and $\mathfrak{a}_{2}\left( \kappa \right) $ are independent copies of $%
\mathfrak{a}\left( \kappa \right) $ defined in (\ref{fadef3}).
\end{theorem}

We now turn to the nonstationary case. We need an additional technical
condition:

\begin{assumption}
\label{rcaap11}\ It holds that $\epsilon _{0,2}$ has a bounded density.
\end{assumption}

\begin{theorem}
\label{thrca111} We assume that $H_{0}$ of (\ref{rcanu}), Assumptions \ref%
{as-wc-1}, \ref{rcaas}, \ref{rcaap11} hold, and $0\leq E\ln |\beta
_{0}+\epsilon _{0,1}|<\infty $.\newline
(i) If $I(w,c)<\infty $ for some $c>0$, then it holds that%
\begin{equation*}
\sup_{0<t<1}\frac{\left\vert Q_{N}(t)\right\vert }{w(t)}\;\overset{{\mathcal{%
D}}}{\rightarrow }\;\eta \sup_{0<t<1}\frac{|B(t)|}{w(t)},
\end{equation*}%
where $\{B(t),0\leq t\leq 1\}$ is a standard Brownian bridge and $\eta $ is
defined in (\ref{etasq}). \newline
(ii) For all $x$ 
\begin{equation*}
\lim_{N\rightarrow \infty }P\Biggl\{a(\ln N)\frac{1}{\eta }%
\max_{1<k<N}\left( \frac{k(N-k)}{N}\right) ^{1/2}|\widehat{\beta }_{k,1}-%
\widehat{\beta }_{k,2}|\leq x+b(\ln N)\Biggl\}=\exp (-2e^{-x}),
\end{equation*}%
where $a(x)$ and $b(x)$ are defined in (\ref{da-1/0}).

(iii) If Assumption \ref{thrcaas11} is satisfied, then it holds that 
\begin{equation*}
\left( \frac{r_{N}}{N}\right) ^{\kappa -1/2}\frac{1}{\eta }%
\sup_{r_{1}(N)/N\leq t\leq 1-r_{2}(N)/N}\frac{|Q_{N}(t)|}{\left( t\left(
1-t\right) \right) ^{\kappa }}\;\overset{{\mathcal{D}}}{\rightarrow }\;\max
\left( \gamma _{1}^{\kappa -1/2}\mathfrak{a}_{1}(\kappa ),\gamma
_{2}^{\kappa -1/2}\mathfrak{a}_{2}(\kappa )\right) ,
\end{equation*}%
for all $\kappa >1/2$, where and $r_{N}$, $\gamma _{1}$ and $\gamma _{2}$
are defined in (\ref{fadef1}), and $\mathfrak{a}_{1}\left( \kappa \right) $
and $\mathfrak{a}_{2}\left( \kappa \right) $ are independent copies of $%
\mathfrak{a}\left( \kappa \right) $ defined in (\ref{fadef3}).
\end{theorem}

Theorems \ref{thrca1} and \ref{thrca111} stipulate that the limiting
distributions of the weighted CUSUM statistics are the same irrespective of
whether $y_{i}$ is stationary, explosive or at the boundary: the impact of
nonstationarity is only on $\eta ^{2}$. Hence, it is important to find an
estimator for $\eta ^{2}$ which is consistent for all cases. Let 
\begin{equation*}
\widehat{a}_{N,1}=\frac{1}{N-1}\sum_{i=2}^{N}\frac{(y_{i}-\widehat{\beta }%
_{N,1}y_{i-1})^{2}y_{i-1}^{2}}{(1+y_{i-1}^{2})^{2}}\quad \mbox{and}\quad 
\widehat{a}_{N,2}=\frac{1}{N-1}\sum_{i=2}^{N}\frac{y_{i-1}^{2}}{1+y_{i-1}^{2}%
}
\end{equation*}%
We use the following estimator for $\eta ^{2}$ 
\begin{equation}
\widehat{\eta }_{N}^{2}=\frac{\widehat{a}_{N,1}}{\widehat{a}_{N,2}^{2}}.
\label{eta-est}
\end{equation}

\begin{corollary}
\label{rcaco1} The results of Theorems \ref{thrca1}--\ref{thrca111} remain
true if $\eta $ is replaced with $\widehat{\eta }_{N}$.
\end{corollary}

Corollary \ref{rcaco1} states that the feasible versions of our test
statistics, based on $\widehat{\eta }_{N}$, have the same distribution as
the infeasible ones, based on $\eta $. Practically, this means that the test
statistics developed above can be implemented with no prior knowledge as to
whether $y_{i}$ is stationary or not.

\section{Change point detection with heteroskedastic errors\label{heterosk}}

In the previous section we assumed, as is typical in the RCA\ literature,
that the innovations $\{\epsilon _{i,1},\epsilon _{i,2},1\leq i\leq N\}$ are
homoskedastic, which may be an undesirable restriction. The literature on
the changepoint problem has recently considered this issue, but
contributions are still relatively rare: exceptions include \citet{xu2015}, %
\citet{gorecki2017} \citet{bardsley2017} and \citet{horvath2021} (see also %
\citealp{xuphillips2008}, for adaptive estimation in autoregressive models).
Heteroskedasticity is particularly interesting and challenging in the RCA\
case: if the distribution of $\epsilon _{i,1}$ is allowed to change, the
observations might change from stationarity to non stationarity even if $%
\beta _{0}$ does not undergo any change; however, inference on the RCA\
model will still be asymptotically normal in light of the properties of the
WLS\ estimator discussed above.

In this section, we extend all the results above allowing for
heteroskedasticity in both $\epsilon _{i,1}$ and $\epsilon _{i,2}$. Our
results are valid also in the baseline case of homoskedasticity, and do not
require any explicit knowledge of the form of heteroskedasticity.

\bigskip

Changes in the distribution of $\{\epsilon _{i,1},\epsilon _{i,2},1\leq
i\leq N\}$ at times $1<m_{1}<\ldots <m_{M}<N$ are allowed through the
following assumption.

\begin{assumption}
\label{rcamult} It holds that $m_{\ell }=\lfloor N\tau _{\ell }\rfloor $,
for $1\leq \ell \leq M$, with $0<\tau _{1}<\tau _{2}<\ldots <\tau _{M}<1$.
\end{assumption}

Henceforth, we will use the notation: $m_{0}=0,m_{M+1}=N,\tau _{0}=0$ and $%
\tau _{M+1}=1$.

For each subsequence $\{y_{i},m_{\ell -1}<i\leq m_{\ell }\}$, $1\leq \ell
\leq M+1$, the condition for stationarity can be satisfied; in this case,
the elements of this subsequence can be approximated with stationary
variables $\{\bar{y}_{\ell ,j},-\infty <j<\infty \}$ defined by the
recursion 
\begin{equation*}
\bar{y}_{\ell ,j}=(\beta _{0}+\epsilon _{\ell ,j,1})\bar{y}_{\ell
,j-1}+\epsilon _{\ell ,j,2},\quad -\infty <j<\infty ,
\end{equation*}%
where $\epsilon _{\ell ,j,1}=\epsilon _{j,1},m_{\ell -1}<j\leq m_{\ell }$,
and $\epsilon _{\ell ,j,1},-\infty <j<\infty ,j\not\in (m_{\ell -1},m_{\ell
-1}+1,\ldots ,m_{\ell }]$ are independent and identically distributed copies
of $\epsilon _{m_{\ell },1}$. The random variables $\epsilon _{\ell ,j,2}$
are defined in the same way.

To allow for changes in the distributions of the errors, we replace
Assumptions \ref{rcaas}-\ref{rcaap11} with

\begin{assumption}
\label{rcaas1} It holds that: (i) $\{\epsilon _{i,1},-\infty <i<\infty \}$
and $\{\epsilon _{i,2},-\infty <i<\infty \}$ are independent sequences of
independent random variables; (ii) $P\{y_{0}=0\}<1$; (iii) for each $1\leq
\ell \leq M+1$, $\{\epsilon _{i,1},\;m_{\ell -1}<i\leq m_{\ell }\}$ are
identically distributed with $E\epsilon _{m_{\ell },1}=0,E\epsilon _{m_{\ell
},1}^{2}=\sigma _{m_{\ell },1}^{2}$ and $E|\epsilon _{m_{\ell
},1}|^{4}<\infty $; (iv) for each $1\leq \ell \leq M+1$, $\{\epsilon
_{i,2},m_{\ell -1}<i\leq m_{\ell }\}$ are identically distributed with $%
E\epsilon _{m_{\ell },2}=0,E\epsilon _{m_{\ell },2}^{2}=\sigma _{m_{\ell
},2}^{2}$ and $E|\sigma _{m_{\ell },2}|^{4}<\infty $; (v) if $E\ln |\beta
_{0}+\epsilon _{m_{\ell },1}|<0$, $P\left( \bar{y}_{\ell ,0}=0\right) <1$, $%
1\leq \ell \leq M+1$; (vi) if $E\ln |\beta _{0}+\epsilon _{m_{\ell },1}|\geq
0$, then $\epsilon _{m_{\ell },2}$ has a bounded density, $1\leq \ell \leq
M+1$.
\end{assumption}

By Assumption \ref{rcaas1}, the WLS estimator may have different variances
in the various regimes. In order to study the limit theory, consider the
following notation:

\begin{equation}
\eta _{\ell }^{2}=\left\{ 
\begin{array}{ll}
\begin{array}{l}
\displaystyle E\left( \frac{\bar{y}_{\ell ,0}^{2}}{1+\bar{y}_{\ell ,0}^{2}}%
\right) ^{2}\sigma _{\ell ,1}^{2}+E\left( \frac{\bar{y}_{\ell ,0}}{1+\bar{y}%
_{\ell ,0}^{2}}\right) ^{2}\sigma _{\ell ,2}^{2}, \\ 
\quad \quad \quad \quad \quad \quad \quad \mbox{if}\;\;\;-\infty \leq E\ln
|\beta _{0}+\epsilon _{m_{\ell },1}|<0,%
\end{array}
&  \\ 
\sigma _{\ell ,1}^{2},\quad \mbox{if}\;\;\;E\ln |\beta _{0}+\epsilon
_{m_{\ell },1}|\geq 0 & 
\end{array}%
\right.  \label{eta-l}
\end{equation}%
and 
\begin{equation}
a_{\ell }=\left\{ 
\begin{array}{ll}
\displaystyle E\left( \frac{\bar{y}_{\ell ,0}^{2}}{1+\bar{y}_{\ell ,0}^{2}}%
\right) ,\quad \mbox{if}\;\;\;-\infty \leq E\ln |\beta _{0}+\epsilon
_{m_{\ell },1}|<0, &  \\ 
1,\quad \mbox{if}\;\;\;E\ln |\beta _{0}+\epsilon _{m_{\ell },1}|\geq 0, & 
\end{array}%
\right.  \label{a-l}
\end{equation}%
$1\leq \ell \leq M+1$. Also, let 
\begin{equation}
\bar{\eta}(t)=\sum_{j=1}^{\ell -1}\frac{\eta _{j}^{2}}{a_{j}^{2}}(\tau
_{j}-\tau _{j-1})+\frac{\eta _{\ell }^{2}}{a_{\ell }^{2}}(t-\tau _{\ell
-1}),\quad \tau _{\ell -1}<t\leq \tau _{\ell },1\leq \ell \leq M+1,
\label{eta-bar}
\end{equation}%
and define the zero mean Gaussian process $\left\{ \Gamma \left( t\right)
,0\leq t\leq 1\right\} $, with $E\left[ \Gamma \left( t\right) \Gamma \left(
s\right) \right] =\bar{\eta}\left( \min \left( t,s\right) \right) $.

We begin by investigating how the limits in Theorems \ref{thrca1} and \ref%
{thrca111} behave under heteroskedasticity.

\begin{theorem}
\label{thrca4} We assume that $H_{0}$ of (\ref{rcanu}), and Assumptions \ref%
{as-wc-1}, \ref{rcamult} and \ref{rcaas1} hold.\newline
(i) If $I(w,c)<\infty $ for some $c>0$, then it holds that 
\begin{equation*}
\sup_{0<t<1}\frac{\left\vert Q_{N}(t)\right\vert }{w(t)}\;\overset{{\mathcal{%
D}}}{\rightarrow }\;\sup_{0<t<1}\frac{\left\vert \Gamma (t)-t\Gamma
(1)\right\vert }{w(t)}.
\end{equation*}%
\newline
(ii) For all $x$, it holds that 
\begin{equation*}
\lim_{N\rightarrow \infty }P\Biggl\{a(\ln N)\sup_{0<t<1}\frac{\left\vert
Q_{N}(t)\right\vert }{\eta _{0}^{1/2}(t,t)}\leq x+b(\ln N)\Biggl\}=\exp
(-2e^{-x}).
\end{equation*}%
\newline
(iii) If Assumption \ref{thrcaas11} is also satisfied, then it holds that,
for all $\kappa >1/2$ 
\begin{equation*}
\left( \frac{r_{N}}{N}\right) ^{\kappa -1/2}\sup_{r_{1}(N)/N\leq t\leq
1-r_{2}(N)/N}\frac{(t(1-t))^{-\kappa +1/2}}{\eta _{0}^{1/2}(t,t)}|{Q}%
_{N}(t)|\;\overset{{\mathcal{D}}}{\rightarrow }\;\max \left( \gamma
_{1}^{\kappa -1/2}\mathfrak{a}_{1}(\kappa ),\gamma _{2}^{\kappa -1/2}%
\mathfrak{a}_{2}(\kappa )\right) .
\end{equation*}
\end{theorem}

Theorem \ref{thrca4} is only of theoretical interest, but we point out that
heteroskedasticity impacts only on part \textit{(i)}. In that case, the
limiting distribution of the weighted $Q_{N}(t)$ is given by a Gaussian
process with covariance kernel 
\begin{equation}
\eta _{0}(t,s)=E\left[ \left( \Gamma (t)-t\Gamma (1)\right) \left( \Gamma
(s)-s\Gamma (1)\right) \right] =\bar{\eta}(\min (t,s))-t\bar{\eta}(s)-s\bar{%
\eta}(t)+st\bar{\eta}(1).  \label{rcabarga}
\end{equation}

Parts \textit{(ii)-(iii)} of the theorem are the same as in the case of
homoskedasticity. 
\begin{comment}
Based on (\ref{rcabarga}), it may seem harder to use the Darling-Erd\H{o}s
and the R\'{e}nyi-type statistics, since the covariance kernel $\eta
_{0}(t,t)$ is not asymptotically proportional to $t(1-t)$. However, 
\end{comment}Upon inspecting the proof, in these cases the asymptotic
distribution is driven only by the observations which are as close to sample
endpoints as $o\left( N\right) $. On these intervals, (\ref{eta-bar})
ensures that the asymptotic variance $\eta _{0}(t,t)$ is proportional to $%
t(1-t)$.

Finally, note that, in light of the definitions of $\eta _{0}(t,t)$ and $%
\eta _{\ell }^{2}$ and $a_{\ell }$, heteroskedasticity in $\epsilon _{i,2}$
does not play a role in the nonstationary case.

\subsection{Feasible tests under heteroskedasticity\label{feasible}}

By Theorem \ref{thrca4}, the implementation of tests based on $Q_{N}(t)$
requires an estimate of $\eta _{0}\left( t,t\right) $. However, this is
fraught with difficulties, since it requires knowledge of the different
regime dates, $m_{\ell }$. Thus, we consider a modification of $Q_{N}(t)$ to
reflect the possible changes in the variances of the errors.

Let 
\begin{equation}
\widehat{\mathfrak{c}}_{N,1}(t)=\frac{1}{N}\sum_{i=2}^{\lfloor (N+1)t\rfloor
}\frac{y_{i-1}^{2}}{1+y_{i-1}^{2}}\quad \mbox{and}\quad \widehat{\mathfrak{c}%
}_{N,2}(t)=\frac{1}{N}\sum_{i=\lfloor (N+1)t\rfloor +1}^{N}\frac{y_{i-1}^{2}%
}{1+y_{i-1}^{2}},\;\;\;0\leq t\leq 1;  \label{c-hat}
\end{equation}%
clearly, $\widehat{\mathfrak{c}}_{N,2}(t)=\widehat{\mathfrak{c}}_{N,1}(1)-%
\widehat{\mathfrak{c}}_{N,1}(t)$. We then define the modified test statistic 
\begin{equation*}
\overline{Q}_{N}(t)=\left\{ 
\begin{array}{ll}
0,\quad \mbox{if}\;\;\;0<t<2/(N+1),\vspace{0.3cm} &  \\ 
N^{1/2}\widehat{\mathfrak{c}}_{N,1}(t)\widehat{\mathfrak{c}}_{N,2}(t)\left( 
\widehat{\beta }_{\lfloor (N+1)t\rfloor ,1}-\widehat{\beta }_{\lfloor
(N+1)t\rfloor ,2}\right) ,\quad \mbox{if}\;\;\;2/(N+1)\leq t<1-2/(N+1),%
\vspace{0.3cm} &  \\ 
0,\quad \mbox{if}\;\;\;1-2/(N+1)\leq t<1. & 
\end{array}%
\right.
\end{equation*}%
Under the null of no change, the same arguments as in the proof of Corollary %
\ref{rcaco1} guarantee that $\widehat{\mathfrak{c}}_{N,1}(t)$ and $\widehat{%
\mathfrak{c}}_{N,2}(t)$ converge to the functions 
\begin{equation}
\mathfrak{c}_{1}(t)=\sum_{j=1}^{\ell -1}(\tau _{\ell }-\tau _{\ell
-1})a_{\ell }+(t-\tau _{\ell -1})a_{\ell },\quad \tau _{\ell -1}<t\leq \tau
_{\ell },1\leq \ell \leq M+1,  \label{c1-t}
\end{equation}%
and $\mathfrak{c}_{2}(t)=\mathfrak{c}_{1}(1)-\mathfrak{c}_{1}(t)$, for $%
0\leq t\leq 1$, where $\tau _{\ell }$ is defined in Assumption \ref{rcamult}%
, and $a_{\ell },1\leq \ell \leq M+1$ is defined in (\ref{a-l}). In order to
present our main results, we define the zero mean Gaussian process%
\begin{equation}
\Theta (t)=\mathfrak{c}_{2}(t)\Delta (t)-\mathfrak{c}_{1}(t)(\Delta
(1)-\Delta (t)),\quad 0\leq t\leq 1;  \label{rcadeTh}
\end{equation}%
$\{\Delta (t),0\leq t\leq 1\}$ is also a zero mean Gaussian process with $%
E\Delta (t)\Delta (s)=\mathfrak{b}(\min (t,s))$, where 
\begin{subequations}
\begin{equation}
{\mathfrak{b}}(t)=\sum_{j=1}^{\ell -1}(\tau _{j}-\tau _{j-1})\eta
_{j}^{2}+(t-\tau _{\ell -1})\eta _{\ell }^{2},\quad \tau _{\ell -1}<t\leq
\tau _{\ell },1\leq \ell \leq M+1.  \label{b-t}
\end{equation}%
Let ${\mathfrak{g}}(t,s)=E\left( \Theta (t)\Theta (s)\right) $; elementary
calculations yield 
\end{subequations}
\begin{equation}
{\mathfrak{g}}(t,s)=\mathfrak{c}_{1}(1)\mathfrak{c}_{1}(1)\mathfrak{b}\left(
\min (t,s)\right) -\mathfrak{c}_{1}(1)\mathfrak{c}_{1}(t)\mathfrak{b}(s)-%
\mathfrak{c}_{1}(1)\mathfrak{c}_{1}(s)\mathfrak{b}(t)+\mathfrak{c}_{1}(t)%
\mathfrak{c}_{1}(s)\mathfrak{b}(1).  \label{rcadefg}
\end{equation}

\begin{theorem}
\label{thrca5}We assume that $H_{0}$ of (\ref{rcanu}), and Assumptions \ref%
{as-wc-1}, \ref{rcamult} and \ref{rcaas1} hold.\newline
(i) If $I(w,c)<\infty $ for some $c>0$, then 
\begin{equation*}
\sup_{0<t<1}\frac{\left\vert \overline{Q}_{N}(t)\right\vert }{w(t)}\;\overset%
{{\mathcal{D}}}{\rightarrow }\;\sup_{0<t<1}\frac{\left\vert \Theta
(t)\right\vert }{w(t)},
\end{equation*}%
where $\{\Theta (t),0\leq t\leq 1\}$ is the Gaussian process defined in %
\eqref{rcadeTh}.\newline
(ii) For all $x$ 
\begin{equation*}
\lim_{N\rightarrow \infty }P\Biggl\{a(\ln N)\sup_{0<t<1}\frac{\left\vert 
\overline{Q}_{N}(t)\right\vert }{\mathfrak{g}^{1/2}\left( t,t\right) }\leq
x+b(\ln N)\Biggl\}=\exp (-2e^{-x}),
\end{equation*}%
where $a(x)$, $b(x)$ are defined in (\ref{da-1/0}) and $\mathfrak{g}(t,s)$
is given in (\ref{rcadefg}).\newline
(iii) If Assumption \ref{thrcaas11} also holds and $\kappa >1/2$, then 
\begin{equation*}
\left( \frac{r_{N}}{N}\right) ^{\kappa -1/2}\sup_{t_{1}<t<t_{2}}\frac{%
(t(1-t))^{-\kappa +1/2}}{\mathfrak{g}^{1/2}\left( t,t\right) }\left\vert 
\overline{Q}_{N}(t)\right\vert \;\overset{{\mathcal{D}}}{\rightarrow }\;\max
\left( \gamma _{1}^{\kappa -1/2}\mathfrak{a}_{1}(\kappa ),\gamma
_{2}^{\kappa -1/2}\mathfrak{a}_{2}(\kappa )\right) ,
\end{equation*}%
where $t_{1}=r_{N}/N$, $t_{2}=1-t_{1}$, and $r_{N}$, $\gamma _{1}$, $\gamma
_{2}$ are defined in (\ref{fadef1}), $\mathfrak{a}_{1}\left( \kappa \right) $
and $\mathfrak{a}_{2}\left( \kappa \right) $ are independent copies of $%
\mathfrak{a}\left( \kappa \right) $ defined in (\ref{fadef3}), and $%
\mathfrak{g}(t,s)$ is given in (\ref{rcadefg}).
\end{theorem}

Some comments on the practical implementation of the results in Theorem \ref%
{thrca5} are in order. Parts \textit{(ii)} and \textit{(iii)} require an
estimate of $\mathfrak{g}(t,t)$; to this end, we use $\widehat{\mathfrak{c}}%
_{N,1}(t)$ defined in (\ref{c-hat}) instead of $\mathfrak{c}_{1}(t)$, and we
estimate ${\mathfrak{b}}(t,s)$ as 
\begin{equation*}
\widehat{\mathfrak{b}}_{N}(t)=\frac{1}{N}\sum_{i=2}^{\lfloor Nt\rfloor
}\left( \frac{(y_{i}-\widehat{\beta }_{N,1}y_{i-1})y_{i-1}}{1+y_{i-1}^{2}}%
\right) ^{2},\quad 0\leq t\leq 1.
\end{equation*}%
Then we can define 
\begin{equation}
\widehat{{\mathfrak{g}}}(t,s)=\widehat{\mathfrak{c}}_{N,1}^{2}\left(
1\right) \widehat{\mathfrak{b}}_{N}\left( \min \left( t,s\right) \right) -%
\widehat{\mathfrak{c}}_{N,1}\left( 1\right) \left( \widehat{\mathfrak{c}}%
_{N,1}\left( t\right) \widehat{\mathfrak{b}}_{N}\left( s\right) +\widehat{%
\mathfrak{c}}_{N,1}\left( s\right) \widehat{\mathfrak{b}}_{N}\left( t\right)
\right) +\widehat{\mathfrak{c}}_{N,1}\left( t\right) \widehat{\mathfrak{c}}%
_{N,1}\left( s\right) \widehat{\mathfrak{b}}_{N}\left( 1\right) .
\label{g-hat}
\end{equation}%
The implementation of part \textit{(i)} of Theorem \ref{thrca5} is more
complicated, since the presence of nuisance parameters is not relegated to a
multiplicative function. We reject the null hypothesis in (\ref{rcanu}) if 
\begin{equation*}
\sup_{0<t<1}\frac{\left\vert \overline{Q}_{N}(t)\right\vert }{w(t)}\geq
c(\alpha ),
\end{equation*}%
with $c(\alpha )$ defined as $P\left\{ \sup_{0<t<1}\frac{\left\vert \Theta
\left( t\right) \right\vert }{w(t)}\geq c(\alpha )\right\} =\alpha $.
Computing the covariance functions, one can verify that $\{\Delta (t),0\leq
t\leq 1\}\;\overset{{\mathcal{D}}}{=}\;\{W(\mathfrak{b}(t)),0\leq t\leq 1\}$%
, where $\{W(x),0\leq x<\infty \}$ is a Wiener process. In order to
approximate the critical values, one can simulate independent Wiener
processes $W_{i}(x)$, $1\leq i\leq L$, and compute the empirical
distribution function

\begin{equation}
{\mathcal{F}}_{N,L}(x)=\frac{1}{L}\sum_{i=1}^{L}I\left\{ \sup_{0<t<1}\frac{%
\left\vert \widehat{\Theta }_{i}(t)\right\vert }{w(t)}\leq x\right\} ,
\label{cr-val-method}
\end{equation}%
where $\widehat{\Theta }_{i}(t)=\widehat{\mathfrak{c}}_{N,2}(t)W_{i}(%
\widehat{\mathfrak{b}}_{N}(t))-\widehat{\mathfrak{c}}_{N,1}(t)(W_{i}(%
\widehat{\mathfrak{b}}_{N}(1))-W_{i}(\widehat{\mathfrak{b}}_{N}(t)))$.

Let $c_{N,L}(\alpha )$ be defined as $c_{N,L}(\alpha )=\inf \{x:\;{\mathcal{F%
}}_{N,L}(x)\geq 1-\alpha \}$.

\begin{corollary}
\label{cnlalpha}We assume that $H_{0}$ of (\ref{rcanu}) holds. Under the
same assumptions as Theorem \ref{thrca5}(i), it holds that%
\begin{equation*}
\lim_{\min \left( N,L\right) \rightarrow \infty }P\left\{ \sup_{0<t<1}\frac{%
\left\vert \overline{Q}_{N}(t)\right\vert }{w(t)}\geq c_{N,L}(\alpha
)\right\} =\alpha .
\end{equation*}%
The results of Theorem \ref{thrca5}(ii)-(iii) remain true if ${\mathfrak{g}}%
(t,s)$ is replaced with $\widehat{{\mathfrak{g}}}(t,s)$.
\end{corollary}

\begin{comment}
We point out that we need to estimate the \textquotedblleft long-run
variances\textquotedblright\ $\widehat{{\mathfrak{g}}}(t,s)$; however, this
does not require the specification of tuning parameters such as e.g. a
kernel function or a bandwidth.
\end{comment}

\subsection{Consistency versus alternatives\label{consistency}}

We study the consistency of our tests versus the AMOC alternative\footnote{%
In Section \ref{mbreaks} in the Supplement, we also discuss the power of our
tests versus the alternative of multiple breaks.}%
\begin{equation*}
y_{i}=\left\{ 
\begin{array}{ll}
\left( \beta _{0}+\epsilon _{i,1}\right) y_{i-1}+\epsilon _{i,2},\quad %
\mbox{if}\;\;\;1\leq i\leq k^{\ast } &  \\ 
\left( \beta _{A}+\epsilon _{i,1}\right) y_{i-1}+\epsilon _{i,2},\quad %
\mbox{if}\;\;\;k^{\ast }+1\leq i\leq N & 
\end{array}%
\right. .
\end{equation*}%
Let $\Delta _{N}=\left\vert \beta _{0}-\beta _{A}\right\vert $, and define $%
t^{\ast }$ as $\left\lfloor Nt^{\ast }\right\rfloor =k^{\ast }$.

\begin{theorem}
\label{amoc}We assume that $H_{0}$ of (\ref{rcanu}) holds. Under the same
assumptions as Theorem \ref{thrca5}, if, as $N\rightarrow \infty $%
\begin{equation}
N^{1/2}\Delta _{N}\frac{\displaystyle \left( \frac{k^{\ast }}{N}\left( \frac{%
N-k^{\ast }}{N}\right) \right) }{\displaystyle w\left( t^{\ast }\right) }%
\rightarrow \infty ,  \label{restr-1}
\end{equation}%
then it holds that, as $\min \left( N,L\right) \rightarrow \infty $%
\begin{equation}
\sup_{0<t<1}\frac{\left\vert \overline{Q}_{N}(t)\right\vert }{w\left(
t\right) }\overset{P}{\rightarrow }\infty .  \label{part3}
\end{equation}%
Further, if, as $N\rightarrow \infty $ 
\begin{equation}
\frac{N^{1/2}}{\left( \ln \ln N\right) ^{1/2}}\Delta _{N}\left( \frac{%
k^{\ast }}{N}\left( \frac{N-k^{\ast }}{N}\right) \right) ^{1/2}\rightarrow
\infty ,  \label{restr-2}
\end{equation}%
then it holds that%
\begin{equation}
a(\ln N)\sup_{0<t<1}\frac{1}{\widehat{\mathfrak{g}}^{1/2}\left( t,t\right) }%
\left\vert \overline{Q}_{N}(t)\right\vert -b(\ln N)\overset{P}{\rightarrow }%
\infty .  \label{part2}
\end{equation}%
Finally, if 
\begin{equation}
\left( \frac{r_{N}}{N}\right) ^{\kappa -1/2}N^{1/2}\Delta _{N}\left( \frac{%
k^{\ast }}{N}\left( \frac{N-k^{\ast }}{N}\right) \right) ^{1-\kappa
}\rightarrow \infty ,  \label{restr-3}
\end{equation}%
then it holds that, for all $\kappa >\frac{1}{2}$%
\begin{equation}
\left( \frac{r_{N}}{N}\right) ^{\kappa -1/2}\sup_{t_{1}<t<t_{2}}\frac{%
(t(1-t))^{-\kappa +1/2}}{\widehat{\mathfrak{g}}^{1/2}\left( t,t\right) }%
\left\vert \overline{Q}_{N}(t)\right\vert \overset{P}{\rightarrow }\infty .
\label{part33}
\end{equation}
\end{theorem}

The theorem ensures that, as long as (\ref{restr-1}), (\ref{restr-2}) and (%
\ref{restr-3}) hold, our tests reject the null with probability
(asymptotically) $1$. Conditions (\ref{restr-1}), (\ref{restr-2}) and (\ref%
{restr-3}) essentially state that breaks will be detected as long as they
are \textquotedblleft not too small\textquotedblright , and
\textquotedblleft not too close\textquotedblright\ to the endpoints of the
sample.

Consider (\ref{restr-1}). This condition can be understood by considering
two cases. First, when $\frac{k^{\ast }}{N}\rightarrow c\in \left(
0,1\right) $, it is required that $N^{1/2}\Delta _{N}\rightarrow \infty $:
this entails that $\beta _{A}$ may depend on the sample size $N$, so that
even small changes in the regression parameter are allowed. When $\Delta
_{N}>0$, (\ref{restr-1}) holds as long as $k^{\ast }N^{\frac{2\kappa -1}{%
2\left( 1-\kappa \right) }}\rightarrow \infty $: tests based on weight
functions $w\left( t\right) =\left( t\left( 1-t\right) \right) ^{\kappa }$
can detect breaks almost as close to the sample endpoints as $O\left( N^{%
\frac{1-2\kappa }{2\left( 1-\kappa \right) }}\right) $.

Turning to (\ref{restr-2}), when $\frac{k^{\ast }}{N}\rightarrow c>0$, the
test is powerful as long as $\left( \frac{N}{\ln \ln N}\right) ^{1/2}\Delta
_{N}\rightarrow \infty $: again small changes are allowed for, but these are
now \textquotedblleft less small\textquotedblright\ by a $O\left( \ln \ln
N\right) $ factor. Conversely, when $\Delta _{N}>0$, (\ref{restr-1}) holds
as long as $k^{\ast }\left( \ln \ln N\right) ^{-1/2}\rightarrow \infty $:
breaks that are as close as $O\left( \sqrt{\ln \ln N}\right) $ periods to
the sample endpoints can be detected. This effect is reinforced in the case
of R\'{e}nyi statistics, where, on account of (\ref{restr-3}), the only
requirement is that $k^{\ast }>r_{N}$.

\section{Simulations\label{simulations}}

We provide some Monte Carlo evidence on the performance of the test
statistics proposed in Section \ref{feasible}.\footnote{%
In Section \ref{crval} in the Supplement, we also evaluate, as a benchmark,
the performance of our tests under homoskedasticity, based on the theory in
Section \ref{homosk}.}

Data are generated using (\ref{arcnew1}). In all experiments, we use $\beta
_{0}\in \left\{ 0.5,0.75,1,1.05\right\} $ to consider both the cases of
stationary and nonstationary $y_{i}$. We have experimented also with
different values of $\beta _{0}$, but results are essentially the same.
Under the alternative, we consider both a mid-sample and an end-of-sample
break%
\begin{align}
y_{i} =&\left( \beta _{0}+\Delta I\left( i\geq 0.5N\right) +\epsilon
_{i,1}\right) y_{i-1}+\epsilon _{i,2},  \label{po1} \\
y_{i} =&\left( \beta _{0}+\Delta I\left( i\geq 0.9N\right) +\epsilon
_{i,1}\right) y_{i-1}+\epsilon _{i,2}.  \label{po2}
\end{align}

The shocks $\epsilon _{i,1}$ and $\epsilon _{i,2}$ are simulated as
independent of one another and \textit{i.i.d.} with distributions $N\left(
0,\sigma _{1}^{2}\right) $ and $N\left( 0,\sigma _{2}^{2}\right) $
respectively. We report results for $\sigma _{1}^{2}=0.01$ and $\sigma
_{2}^{2}=0.5$ - the value of $\sigma _{1}^{2}$ is based on \textquotedblleft
typical\textquotedblright\ values as found e.g. in the empirical
applications in \citet{HT16}. We note however that, in unreported
simulations using different values of $\sigma _{1}^{2}$ and $\sigma _{2}^{2}$%
, the main results do not change, except for the (expected) fact that tests
have better properties (in terms of size and power) for smaller values of $%
\sigma _{2}^{2}$. Similarly, the test performs better (with empirical
rejection frequencies closer to their nominal value) when $\sigma _{1}^{2}$
is larger, and tends to be undersized for smaller values of $\sigma _{1}^{2}$%
. Both effects (of $\sigma _{1}^{2}$ and $\sigma _{2}^{2}$) vanish as $N$
increases. When allowing for heteroskedasticity, we generate $\epsilon
_{i,1} $ and $\epsilon _{i,2}$ as \textit{i.i.d.}$N\left( 0,\sigma
_{1}^{2}\right) $ and \textit{i.i.d.}$N\left( 0,\sigma _{2}^{2}\right) $ for 
$1\leq i\leq N/2$, and \textit{i.i.d.}$N\left( 0,1.5\sigma _{1}^{2}\right) $
and \textit{i.i.d.}$N\left( 0,1.5\sigma _{2}^{2}\right) $ for $N/2+1\leq
i\leq N$.

Finally, we generate $N+1,000$ values of $y_{i}$ from (\ref{arcnew1}) - with 
$y_{0}=0$ - and discard the first $1,000$ values. All our routines are based
on $2,000$ replications, and we use critical values corresponding to a
nominal level equal to $5\%\footnote{%
When using (\ref{cr-val-method}), we use $L=200$. Results are however not
particulary sensitive to this specification.}$ - hence, empirical rejection
frequencies under the null have a $95\%$ confidence interval $\left[
0.04,0.06\right] $. \newline

We consider four different cases: \textit{(i)} homoskedasticity in both $%
\epsilon _{i,1}$ and $\epsilon _{i,2}$; \textit{(ii)} homoskedasticity in $%
\epsilon _{i,1}$ and heteroskedasticity in $\epsilon _{i,2}$; \textit{(iii) }%
homoskedasticity in $\epsilon _{i,2}$ and heteroskedasticity in $\epsilon
_{i,1}$; and, finally, \textit{(iv)} heteroskedasticity in both $\epsilon
_{i,1}$ and $\epsilon _{i,2}$.

\begin{table*}[h]
\caption{Empirical rejection frequencies under the null - part 1}
\label{tab:SizeHet1}\centering
\resizebox{\textwidth}{!}{

\begin{tabular}{llllllllllllllllllllll}
\hline\hline
&  &  &  &  &  &  &  &  &  &  &  &  &  &  &  &  &  &  &  &  &  \\ 
&  & $\beta $ & \multicolumn{4}{c}{$0.5$} & \multicolumn{1}{c}{} & 
\multicolumn{4}{c}{$0.75$} & \multicolumn{1}{c}{} & \multicolumn{4}{c}{$1$}
& \multicolumn{1}{c}{} & \multicolumn{4}{c}{$1.05$} \\ 
&  &  &  &  &  &  &  &  &  &  &  &  &  &  &  &  &  &  &  &  &  \\ 
& $N$ &  & $200$ & $400$ & $800$ & $1600$ &  & $200$ & $400$ & $800$ & $1600$
&  & $200$ & $400$ & $800$ & $1600$ &  & $200$ & $400$ & $800$ & $1600$ \\ 
$\kappa $ &  &  &  &  &  &  &  &  &  &  &  &  &  &  &  &  &  &  &  &  &  \\ 
&  &  &  &  &  &  &  &  &  &  &  &  &  &  &  &  &  &  &  &  &  \\ 
$0$ &  &  & $0.056$ & $0.066$ & $0.048$ & $0.060$ &  & $0.058$ & $0.049$ & $0.054$ & $0.058$ &  & $0.058$ & $0.059$ & $0.056$ & $0.060$ &  & $0.057$ & $0.057$ & $0.063$ & $0.059$ \\ 
$0.25$ &  &  & $0.048$ & $0.053$ & $0.046$ & $0.053$ &  & $0.046$ & $0.050$
& $0.050$ & $0.058$ &  & $0.052$ & $0.060$ & $0.056$ & $0.056$ &  & $0.057$
& $0.056$ & $0.062$ & $0.060$ \\ 
$0.45$ &  &  & $0.033$ & $0.034$ & $0.036$ & $0.038$ &  & $0.035$ & $0.027$
& $0.031$ & $0.040$ &  & $0.033$ & $0.042$ & $0.038$ & $0.036$ &  & $0.029$
& $0.034$ & $0.038$ & $0.041$ \\ 
$0.5$ &  &  & $0.033$ & $0.026$ & $0.033$ & $0.050$ &  & $0.035$ & $0.030$ & 
$0.034$ & $0.035$ &  & $0.037$ & $0.038$ & $0.042$ & $0.041$ &  & $0.040$ & $0.039$ & $0.042$ & $0.045$ \\ 
$0.51$ &  &  & $0.015$ & $0.020$ & $0.019$ & $0.037$ &  & $0.016$ & $0.019$
& $0.028$ & $0.038$ &  & $0.016$ & $0.017$ & $0.025$ & $0.038$ &  & $0.015$
& $0.025$ & $0.028$ & $0.036$ \\ 
$0.75$ &  &  & $0.034$ & $0.047$ & $0.042$ & $0.040$ &  & $0.042$ & $0.043$
& $0.047$ & $0.047$ &  & $0.039$ & $0.043$ & $0.041$ & $0.051$ &  & $0.035$
& $0.042$ & $0.045$ & $0.034$ \\ 
$0.85$ &  &  & $0.034$ & $0.050$ & $0.045$ & $0.046$ &  & $0.047$ & $0.050$
& $0.050$ & $0.047$ &  & $0.043$ & $0.049$ & $0.049$ & $0.053$ &  & $0.043$
& $0.060$ & $0.048$ & $0.040$ \\ 
$1$ &  &  & $0.038$ & $0.050$ & $0.047$ & $0.045$ &  & $0.050$ & $0.048$ & $0.049$ & $0.043$ &  & $0.043$ & $0.051$ & $0.046$ & $0.052$ &  & $0.043$ & $0.061$ & $0.046$ & $0.042$ \\ 
&  &  &  &  &  &  &  &  &  &  &  &  &  &  &  &  &  &  &  &  &  \\ 
\hline\hline
&  &  &  &  &  &  &  &  &  &  &  &  &  &  &  &  &  &  &  &  &  \\ 
&  & $\beta $ & \multicolumn{4}{c}{$0.5$} & \multicolumn{1}{c}{} & 
\multicolumn{4}{c}{$0.75$} & \multicolumn{1}{c}{} & \multicolumn{4}{c}{$1$}
& \multicolumn{1}{c}{} & \multicolumn{4}{c}{$1.05$} \\ 
&  &  &  &  &  &  &  &  &  &  &  &  &  &  &  &  &  &  &  &  &  \\ 
& $N$ &  & $200$ & $400$ & $800$ & $1600$ &  & $200$ & $400$ & $800$ & $1600$
&  & $200$ & $400$ & $800$ & $1600$ &  & $200$ & $400$ & $800$ & $1600$ \\ 
$\kappa $ &  &  &  &  &  &  &  &  &  &  &  &  &  &  &  &  &  &  &  &  &  \\ 
&  &  &  &  &  &  &  &  &  &  &  &  &  &  &  &  &  &  &  &  &  \\ 
$0$ &  &  & $0.055$ & $0.058$ & $0.051$ & $0.053$ &  & $0.054$ & $0.050$ & $0.055$ & $0.057$ &  & $0.058$ & $0.052$ & $0.052$ & $0.058$ &  & $0.054$ & $0.058$ & $0.058$ & $0.060$ \\ 
$0.25$ &  &  & $0.054$ & $0.053$ & $0.053$ & $0.054$ &  & $0.057$ & $0.043$
& $0.053$ & $0.055$ &  & $0.052$ & $0.056$ & $0.054$ & $0.054$ &  & $0.050$
& $0.059$ & $0.058$ & $0.060$ \\ 
$0.45$ &  &  & $0.033$ & $0.031$ & $0.033$ & $0.043$ &  & $0.034$ & $0.029$
& $0.033$ & $0.037$ &  & $0.030$ & $0.038$ & $0.040$ & $0.034$ &  & $0.027$
& $0.034$ & $0.034$ & $0.043$ \\ 
$0.5$ &  &  & $0.036$ & $0.026$ & $0.043$ & $0.044$ &  & $0.038$ & $0.030$ & 
$0.036$ & $0.040$ &  & $0.039$ & $0.037$ & $0.040$ & $0.040$ &  & $0.037$ & $0.040$ & $0.040$ & $0.048$ \\ 
$0.51$ &  &  & $0.018$ & $0.019$ & $0.025$ & $0.041$ &  & $0.019$ & $0.023$
& $0.032$ & $0.034$ &  & $0.020$ & $0.018$ & $0.027$ & $0.040$ &  & $0.017$
& $0.029$ & $0.026$ & $0.045$ \\ 
$0.75$ &  &  & $0.031$ & $0.045$ & $0.043$ & $0.043$ &  & $0.045$ & $0.046$
& $0.052$ & $0.050$ &  & $0.046$ & $0.043$ & $0.045$ & $0.054$ &  & $0.039$
& $0.050$ & $0.040$ & $0.042$ \\ 
$0.85$ &  &  & $0.036$ & $0.050$ & $0.049$ & $0.050$ &  & $0.050$ & $0.047$
& $0.053$ & $0.048$ &  & $0.046$ & $0.047$ & $0.049$ & $0.051$ &  & $0.044$
& $0.055$ & $0.042$ & $0.046$ \\ 
$1$ &  &  & $0.037$ & $0.055$ & $0.048$ & $0.051$ &  & $0.051$ & $0.050$ & $0.051$ & $0.047$ &  & $0.045$ & $0.050$ & $0.046$ & $0.051$ &  & $0.044$ & $0.053$ & $0.045$ & $0.047$ \\ 
&  &  &  &  &  &  &  &  &  &  &  &  &  &  &  &  &  &  &  &  &  \\ 
\hline\hline
\end{tabular}
}
\par
\begin{tablenotes}
      \tiny
            \item Empirical rejection frequencies under the null of no change - the first panel refers to homoskedastic innovations; in the second panel we consider heteroskedasticity in $\epsilon_{i,2}$ only. 
            
\end{tablenotes}
\end{table*}

\begin{table*}[h]
\caption{Empirical rejection frequencies under the null - part 2}
\label{tab:SizeHet2}\centering
\resizebox{\textwidth}{!}{

\begin{tabular}{llllllllllllllllllllll}
\hline\hline
&  &  &  &  &  &  &  &  &  &  &  &  &  &  &  &  &  &  &  &  &  \\ 
&  & $\beta $ & \multicolumn{4}{c}{$0.5$} & \multicolumn{1}{c}{} & 
\multicolumn{4}{c}{$0.75$} & \multicolumn{1}{c}{} & \multicolumn{4}{c}{$1$}
& \multicolumn{1}{c}{} & \multicolumn{4}{c}{$1.05$} \\ 
&  &  &  &  &  &  &  &  &  &  &  &  &  &  &  &  &  &  &  &  &  \\ 
& $N$ &  & $200$ & $400$ & $800$ & $1600$ &  & $200$ & $400$ & $800$ & $1600$
&  & $200$ & $400$ & $800$ & $1600$ &  & $200$ & $400$ & $800$ & $1600$ \\ 
$\kappa $ &  &  &  &  &  &  &  &  &  &  &  &  &  &  &  &  &  &  &  &  &  \\ 
&  &  &  &  &  &  &  &  &  &  &  &  &  &  &  &  &  &  &  &  &  \\ 
$0$ &  &  & $0.046$ & $0.061$ & $0.045$ & $0.057$ &  & $0.065$ & $0.045$ & $0.056$ & $0.052$ &  & $0.066$ & $0.063$ & $0.056$ & $0.058$ &  & $0.058$ & $0.054$ & $0.062$ & $0.059$ \\ 
$0.25$ &  &  & $0.044$ & $0.055$ & $0.050$ & $0.056$ &  & $0.062$ & $0.047$
& $0.056$ & $0.054$ &  & $0.059$ & $0.058$ & $0.054$ & $0.060$ &  & $0.051$
& $0.050$ & $0.062$ & $0.059$ \\ 
$0.45$ &  &  & $0.023$ & $0.031$ & $0.029$ & $0.036$ &  & $0.035$ & $0.030$
& $0.040$ & $0.034$ &  & $0.036$ & $0.033$ & $0.041$ & $0.043$ &  & $0.027$
& $0.036$ & $0.038$ & $0.045$ \\ 
$0.5$ &  &  & $0.027$ & $0.033$ & $0.026$ & $0.040$ &  & $0.040$ & $0.026$ & 
$0.042$ & $0.040$ &  & $0.040$ & $0.029$ & $0.043$ & $0.039$ &  & $0.036$ & $0.035$ & $0.039$ & $0.050$ \\ 
$0.51$ &  &  & $0.015$ & $0.006$ & $0.017$ & $0.021$ &  & $0.008$ & $0.014$
& $0.023$ & $0.020$ &  & $0.009$ & $0.013$ & $0.021$ & $0.021$ &  & $0.013$
& $0.018$ & $0.020$ & $0.025$ \\ 
$0.75$ &  &  & $0.035$ & $0.040$ & $0.041$ & $0.043$ &  & $0.031$ & $0.050$
& $0.043$ & $0.045$ &  & $0.035$ & $0.046$ & $0.047$ & $0.043$ &  & $0.036$
& $0.044$ & $0.040$ & $0.040$ \\ 
$0.85$ &  &  & $0.040$ & $0.046$ & $0.043$ & $0.043$ &  & $0.035$ & $0.055$
& $0.046$ & $0.048$ &  & $0.040$ & $0.058$ & $0.053$ & $0.046$ &  & $0.043$
& $0.050$ & $0.042$ & $0.041$ \\ 
$1$ &  &  & $0.044$ & $0.045$ & $0.045$ & $0.044$ &  & $0.038$ & $0.058$ & $0.045$ & $0.047$ &  & $0.042$ & $0.063$ & $0.051$ & $0.047$ &  & $0.047$ & $0.054$ & $0.045$ & $0.041$ \\ 
&  &  &  &  &  &  &  &  &  &  &  &  &  &  &  &  &  &  &  &  &  \\ 
\hline\hline
&  &  &  &  &  &  &  &  &  &  &  &  &  &  &  &  &  &  &  &  &  \\ 
&  & $\beta $ & \multicolumn{4}{c}{$0.5$} & \multicolumn{1}{c}{} & 
\multicolumn{4}{c}{$0.75$} & \multicolumn{1}{c}{} & \multicolumn{4}{c}{$1$}
& \multicolumn{1}{c}{} & \multicolumn{4}{c}{$1.05$} \\ 
&  &  &  &  &  &  &  &  &  &  &  &  &  &  &  &  &  &  &  &  &  \\ 
& $N$ &  & $200$ & $400$ & $800$ & $1600$ &  & $200$ & $400$ & $800$ & $1600$
&  & $200$ & $400$ & $800$ & $1600$ &  & $200$ & $400$ & $800$ & $1600$ \\ 
$\kappa $ &  &  &  &  &  &  &  &  &  &  &  &  &  &  &  &  &  &  &  &  &  \\ 
&  &  &  &  &  &  &  &  &  &  &  &  &  &  &  &  &  &  &  &  &  \\ 
$0$ &  &  & $0.053$ & $0.060$ & $0.046$ & $0.059$ &  & $0.062$ & $0.043$ & $0.057$ & $0.052$ &  & $0.057$ & $0.060$ & $0.060$ & $0.057$ &  & $0.058$ & $0.054$ & $0.062$ & $0.059$ \\ 
$0.25$ &  &  & $0.047$ & $0.059$ & $0.044$ & $0.055$ &  & $0.060$ & $0.044$
& $0.058$ & $0.046$ &  & $0.054$ & $0.053$ & $0.060$ & $0.060$ &  & $0.051$
& $0.050$ & $0.062$ & $0.059$ \\ 
$0.45$ &  &  & $0.026$ & $0.029$ & $0.030$ & $0.035$ &  & $0.036$ & $0.029$
& $0.038$ & $0.033$ &  & $0.040$ & $0.040$ & $0.047$ & $0.042$ &  & $0.027$
& $0.036$ & $0.038$ & $0.045$ \\ 
$0.5$ &  &  & $0.030$ & $0.033$ & $0.027$ & $0.041$ &  & $0.040$ & $0.029$ & 
$0.039$ & $0.031$ &  & $0.034$ & $0.027$ & $0.038$ & $0.038$ &  & $0.036$ & $0.035$ & $0.039$ & $0.050$ \\ 
$0.51$ &  &  & $0.014$ & $0.007$ & $0.020$ & $0.024$ &  & $0.007$ & $0.017$
& $0.023$ & $0.022$ &  & $0.010$ & $0.013$ & $0.018$ & $0.021$ &  & $0.013$
& $0.018$ & $0.020$ & $0.025$ \\ 
$0.75$ &  &  & $0.035$ & $0.040$ & $0.040$ & $0.045$ &  & $0.029$ & $0.049$
& $0.046$ & $0.043$ &  & $0.032$ & $0.041$ & $0.041$ & $0.040$ &  & $0.036$
& $0.044$ & $0.040$ & $0.040$ \\ 
$0.85$ &  &  & $0.041$ & $0.044$ & $0.045$ & $0.046$ &  & $0.035$ & $0.056$
& $0.050$ & $0.044$ &  & $0.036$ & $0.047$ & $0.045$ & $0.042$ &  & $0.043$
& $0.050$ & $0.042$ & $0.041$ \\ 
$1$ &  &  & $0.045$ & $0.050$ & $0.048$ & $0.044$ &  & $0.035$ & $0.060$ & $0.052$ & $0.049$ &  & $0.039$ & $0.049$ & $0.047$ & $0.042$ &  & $0.047$ & $0.054$ & $0.045$ & $0.041$ \\ 
&  &  &  &  &  &  &  &  &  &  &  &  &  &  &  &  &  &  &  &  &  \\ 
\hline\hline
\end{tabular}
}
\par
\begin{tablenotes}
      \tiny
            \item Empirical rejection frequencies under the null of no change - in the first panel we consider heteroskedasticity in $\epsilon_{i,1}$ only; in the second panel we consider heteroskedasticity in both $\epsilon_{i,1}$ and $\epsilon_{i,2}$ only.
            
\end{tablenotes}
\end{table*}

Empirical rejection frequencies under the null are in Tables \ref%
{tab:SizeHet1} and \ref{tab:SizeHet2}. We have used asymptotic critical
values for R\'{e}nyi statistics, and the method described in Section \ref%
{feasible} for the cases where $\kappa <0.5$. When using the Darling-Erd\H{o}%
s statistic ($\kappa =0.5$), asymptotic critical values yield hugely
undersized tests; we have therefore used the critical values in Table I in %
\citet{gombay}. From Tables \ref{tab:SizeHet1} and \ref{tab:SizeHet2}, all
tests work very well in all cases considered, possibly being slightly worse
in the fully homoskedastic case. Tests never over-reject, not even in small
samples - conversely, there are some cases of (severe) under-rejection in
small samples, especially when $\kappa $ is around $0.5$. As $N$ increases,
however, this vanishes and the empirical rejection frequencies all lie
within their $95\%$ confidence interval. The only exception is the R\'{e}nyi
statistic with $\kappa =0.51$, which is severely undersized even in large
samples.

\bigskip

The empirical power of the tests is reported in Figures \ref{fig:FigHo}-\ref%
{fig:FigHeEB}, where we only consider a sample size of $N=400$ to save
space. The figures illustrate the robustness of the approach proposed in
Section \ref{feasible}, showing, essentially, the same pattern: tests work
well in all cases considered, with the power increasing monotonically in $%
\Delta $. Test statistics with lower $\kappa $ exhibit more power versus
alternatives with \textquotedblleft small\textquotedblright\ values of $%
\Delta $: in this case, the power is monotonically decreasing in $\kappa $,
with virtually no exceptions. Figures \ref{fig:FigHo}-\ref{fig:FigHeE},
compared with Figures \ref{fig:FigHeB}-\ref{fig:FigHeEB}, show an
interesting feature: the power of the test is virtually unaffected by the
presence or absence in heteroskedasticity in $\epsilon _{i,2}$; conversely,
heteroskedasticity in $\epsilon _{i,1}$ does have an impact. This is
particularly apparent in the cases where $\beta _{0}\geq 1$, which could be
explained by noting that, in the nonstationary case the asymptotics of the
WLS\ estimator is driven only by $\epsilon _{i,1}$. This can be read in
conjunction with the fact that, as shown in Figures \ref{fig:FigHo}-\ref%
{fig:FigHeE}, the value of $\beta _{0}$ has vitually no impact on the power
of our tests when $\epsilon _{i,1}$ is constant.

We also consider the case of end-of-sample breaks (\ref{po2}). Results are
in Figures \ref{fig:FigHo2}-\ref{fig:FigHeEB2}. The results show,
essentially, the same pattern as above: all test statistics have monotonic
power in $\Delta $, and whilst heteroskedasticity in $\epsilon _{i,2}$ does
not affect the whole picture, heteroskedasticity in $\epsilon _{i,1}$ gives
very different results, with its presence increasing power especially for $%
\beta _{0}\geq 1$. However, the impact of $\kappa $ here is, as expected,
completely reversed: the power versus breaks that occur at the end of the
sample increases monotonically, \textit{ceteris paribus}, with $\kappa $.
This makes a difference particularly in the case of medium-sized changes -
e.g., when $\Delta =0.35$, increases in power from $\kappa =0$ to $\kappa =1$
are in the region of $10-15\%$.

\bigskip

Finally, in Figures \ref{fig:bubb1}-\ref{fig:bubb2} we report a small scale
exercise where we evaluate the empirical rejection frequencies when $\beta
_{0}$ is close to unity. We only consider heteroskedasticity in $\epsilon
_{i,2}$: results for other cases are available upon request and, in general,
no major differences are noted compared to the other results. These
\textquotedblleft boundary\textquotedblright\ cases should be helpful to
shed more light on the performance of our procedure when detecting changes
from stationarity to nonstationarity (when $\beta _{0}<1$ and changes are
positive), and vice versa (when $\beta _{0}>1$ and changes are negative).
The main message of Figures \ref{fig:bubb1}-\ref{fig:bubb2} is that our
tests work very well in these boundary cases. In particular, the tests are
very effective in detecting changes from stationarity to explosive
behaviour, and vice versa. The power is especially high when $\beta _{0}>1$
- i.e. when the RCA\ process changes from an explosive to a stationary
behaviour. This suggests a possible, effective test to detect e.g. the
collapse of a bubble in financial econometrics applications.

\section{Empirical applications\label{empirics}}

We illustrate our approach through three applications to real data. In
Sections \ref{inflation}-\ref{ibm}, we use economic and financial time
series; in Section \ref{covid}, we use Covid-19 data.

\subsection{Application I: changes in the persistence of US CPI\label%
{inflation}}

In his landmark paper, \citet{engle1982} applies an ARCH(1) specification to
monthly inflation data, showing that these exhibit conditional
heteroskedasticity. Inspired by this, and by the fact that the RCA model is
a second-order equivalent to the ARCH\ model, in this section we test for
the presence of changepoints in the dynamics of US CPI over the last
century. There is an increasing literature on testing for changes in the
persistence of inflation: in particular, \citet{benati} carry out a
systematic study, applying tests for structural breaks to an AR$(p)$ model
for inflation for several countries. Their analysis shows that not only the
average level of inflation (as is well-documented), but also its serial
correlation, may be subject to numerous changes.

We use monthly CPI data taken from the FRED dataset over a period spanning
from January 1913 until January 2021, with $N=1297$. We use monthly
inflation rates, calculated as the month-on-month log differences of the
series. Given that the series is quite long, we expect to see more than one
break; hence, we use binary segmentation (as suggested in %
\citealp{vostrikova}), reporting the point in time at which the relevant
test statistic is maximised as the breakdate estimate.

\begin{table*}[h]
\caption{Changepoint detection with US CPI data.}
\label{tab:TableCPI}%
\resizebox{\textwidth}{!}{

\begin{tabular}{llllllllllllllllllll}
\hline\hline
&  &  &  &  &  &  &  &  &  &  &  &  &  &  &  &  &  &  &  \\ 
\textbf{Breakdate} &  & \multicolumn{18}{l}{\textbf{Notes}} \\ 
&  &  &  &  &  &  &  &  &  &  &  &  &  &  &  &  &  &  &  \\ 
Jan 1918 &  & \multicolumn{18}{l}{Date found with $\kappa =0.85$, $1$ (found
at Apr 1918 with other $\kappa >0.51$; not found with $\kappa =0.51$). Found
at $10\%$ level with $\kappa =0.25$ and $0.5$} \\ 
&  &  &  &  &  &  &  &  &  &  &  &  &  &  &  &  &   &  &  \\ 
Sep 1921 &  & \multicolumn{18}{l}{Date found with $\kappa =0.85,1$ (found
Jan 1922 with $\kappa >0.65$; not found with other values of $\kappa $)} \\ 
&  &  &  &  &  &  &  &  &  &  &  &  &  &  &  &  &  &  &  \\ 
May 1929 &  & \multicolumn{18}{l}{Break found at $10\%$ nominal level only,
with $\kappa <0.5$ only} \\ 
&  &  &  &  &  &  &  &  &  &  &  &  &  &  &  &  &  &  &  \\ 
Jul 1957 &  & \multicolumn{18}{l}{Same date found by all R\'{e}nyi-type stats
(also found at Jan 1957 with $\kappa =0.5$). Break not found with $\kappa
<0.5$} \\ 
&  &  &  &  &  &  &  &  &  &  &  &  &  &  &  &  &  &  &  \\ 
Nov 1966 &  & \multicolumn{18}{l}{Same date found by all tests - except $\kappa =0.5$ (found at Feb 1966), and $\kappa =0.25$ (at $10\%$ level)} \\ 
&  &  &  &  &  &  &  &  &  &  &  &  &  &  &  &  &  &  &  \\ 
Jun 1982 &  & \multicolumn{18}{l}{Same date found by all tests (found at Dec
1981 with $\kappa =0.5$) - R\'{e}nyi-type stats at $5\%$, nominal level, all
others at $10\%$} \\ 
&  &  &  &  &  &  &  &  &  &  &  &  &  &  &  &  &  &  &  \\ 
Nov 1989 &  & \multicolumn{18}{l}{Same date found by all R\'{e}nyi-type stats
(found at Mar 1989 with $\kappa =1$). Break not found with $\kappa \leq 0.5$}
\\ 
&  &  &  &  &  &  &  &  &  &  &  &  &  &  &  &  &  &  &  \\ \hline\hline
\end{tabular}
}
\par
\begin{tablenotes}
      \tiny
            \item We use the monthly log differences of CPI. 
            \item Detected changepoints, and their estimated date, are presented in \textit{chronological} order; as mentioned in the text, breakdates have been estimated as the earliest detection date across different values of $\kappa$. Whilst details are available upon request, we note that breaks were detected with this order (from the first to be detected to the last one): November 1966; July 1957 and January 1918; June 1982 and September 1921; March 1989 and May 1929.

\end{tablenotes}
\end{table*}

Results in Table XXX differ across tests only marginally, and suggest the
presence of several changepoints in the autoregressive coefficient - see
also Figure \ref{fig:FigCPI} in the Supplement. Some of the estimated
breakdates have a clear economic interpretation. In chronological order, the
first break is found (only by R\'{e}nyi statistics) around January 1918,
which should reflect not only the war effort, but also the increasingly more
comprehensive data collection from the Bureau of Labor Statistics; evidence
in favour of the changepoint increases as $\kappa $ increases, as the theory
would suggest given that this is an early change (occurring circa at $5\%$
of the sample). Indeed, when $\kappa \leq 0.5$, the break is found at $10\%$%
, but not at $5\%$ level. Similar considerations, on the detection ability
and timing corresponding to different values of $\kappa $, can be made for
the break found in 1921; in this case, a possible cause is the impact of the
severe recession at the beginning of the decade, as well as the very rapid
deflation which had occurred in 1920. Conversely, there is limited evidence
for a break in the autoregressive parameter of inflation around the Great
Depression - looking at Figure \ref{fig:FigCPI}, this may be explained as a
shift which occurred only in the mean as opposed to the persistence. The
break occurring in 1957 may be viewed as related to inflation reemerging,
albeit modestly, in the spring of 1956 after a long period of price
stability, with the All-Items CPI increasing by 3.6 percent from April 1956
to April 1957 (comparing with the previous period, the All-Items CPI had
risen by $0.2$ percent annualized rate from July 1952 to April 1956). The
changepoint in 1966 (which is the first one to be found, by all tests) can
be explained by noting that food prices had started accelerating early at
the end of 1965; and, by October 1966, the change in the All-Items CPI
reached its highest since 1957. The change in 1981 is documented also in
other studies (\citealp{eo}), and it corresponds to the beginning of an
aggressive FED policy to rein in inflation after the 1970s. Finally, the
evidence for the break in 1989 (which is not picked up by all tests) is less
clear, but it may be the outcome of the cooling off in FED policy towards
the end of the 1980s.

\subsection{Application II: monthly IBM\ returns\label{ibm}}

We use IBM\ monthly returns, over a period spanning January 1962 till March
2021 (corresponding to $T=710$). The data used in this application has also
been employed, in the context of testing for changepoints, by \citet{yauzhao}%
, who find evidence of two breaks: one around June 1987 (with confidence
interval between June 1986 and June 1988), and another around October 2002
(with confidence interval between April 2001 and April 2004). As in the
previous section, we have applied the tests using binary segmentation, using
all the test statistics developed above. None of them rejected the null of
no break at $5\%$ level; when applying the tests at $10\%$ level, though,
several breaks were found.

\bigskip

\begin{table}[h!]
\centering
\begin{threeparttable}
{\tiny

\caption{Changepoint detection with IBM data.}
\label{tab:TableIBM}\centering
\par

\begin{tabular}{lllllll}
\hline\hline
&  &  &  &  &  &  \\ 
&  & \multicolumn{1}{c}{\textbf{Changepoint 1}} & \multicolumn{1}{c}{} & 
\multicolumn{1}{c}{\textbf{Changepoint 2}} & \multicolumn{1}{c}{} & 
\multicolumn{1}{c}{\textbf{Changepoint 3}} \\ 
&  &  &  &  &  &  \\ 
\textit{Date} &  & \multicolumn{1}{c}{Sep 1973} & \multicolumn{1}{c}{} & 
\multicolumn{1}{c}{Nov 1987} & \multicolumn{1}{c}{} & \multicolumn{1}{c}{Oct
1999} \\ 
&  & \multicolumn{1}{c}{} & \multicolumn{1}{c}{} & \multicolumn{1}{c}{} & 
\multicolumn{1}{c}{} & \multicolumn{1}{c}{} \\ 
\textit{Notes} &  & \multicolumn{1}{c}{R\'{e}nyi stats ($10\%$ nominal level)} & 
\multicolumn{1}{c}{} & \multicolumn{1}{c}{R\'{e}nyi stats ($10\%$ nominal level)}
& \multicolumn{1}{c}{} & \multicolumn{1}{c}{Weighted stats ($10\%$ nominal
level)} \\ 
&  &  &  &  &  &  \\ \hline\hline
\end{tabular}

\begin{tablenotes}
      \tiny
            \item We use the logs of the original data, with no further transformations.
            \item For each estimated changepoint, we indicate which statistic has detected it, and at which (nominal) level. When more than one statistic finds a break, the estimated date is computed as the majority vote across statistics, using the point in time at which the relevant statistic finds a break as estimator. Whilst details are available upon request, we note that breaks were detected with this order (from the first to be detected to the last one): break in 1973; break in 1987; break in 1999.
            
\end{tablenotes}
}  
\end{threeparttable}
\end{table}

We found three changepoints in the whole series (see Table \ref{tab:TableIBM}%
). The first one, whose date corresponds to the well-known 1973-74 market
crash (due to the collapse of the Bretton-Woods system, and compounded by
the oil shock), is relatively close to the beginning of the sample, and
indeed it has been identified by the R\'{e}nyi statistics (the other tests
do not identify such break). The second changepoint can also be related to a
specific event, i.e. the Black Monday (the break is found in November 1987,
i.e. one month later the actual event). Finally, R\'{e}nyi statistics do not
find the third changepoint, which occurs mid-sample, confirming the idea
that mid-sample breaks are better detected using milder weight functions
(indeed, not even the Darling-Erd\H{o}s test finds evidence of such a
break); the break is found before the collapse of the dot-com bubble
(traditionally dated around March 2000), reflecting the trouble brewing in
the months leading to the event.

\subsection{Application III: Covid-19 UK hospitalisation data\label{covid}}

In this section, we consider UK data on Covid-19 - in particular, we use
data on hospitalisations rather than cases, as the latter may be less
reliable due to the change in number of tests administered. \citet{shtatland}
\textit{inter alia }advocate using a low-order autoregression as an
approximation of the popular SIR\ model, especially as a methodology for the
early detection of outbreaks. In this context, the autoregressive root is of
crucial importance since, as the authors put it, if \textquotedblleft the
parameter is greater than one, we have an explosive case (an outbreak of
epidemic)\textquotedblright . It is therefore important to check whether the
observations change from an explosive to a stationary regime (meaning that
the epidemic is slowing down), or vice versa whether the change occurs from
a stationary to an explosive regime (i.e., the epidemic undergoes a surge,
or \textquotedblleft wave\textquotedblright ). In this respect, the
empirical exercise in this section should be read in conjunction with
Figures \ref{fig:bubb1}-\ref{fig:bubb2}.

We use (logs of) UK daily data, for the four UK\ nations, and for the
various regions of England\footnote{%
The data are available from
\par
https://ourworldindata.org/grapher/uk-daily-covid-admissions?tab=chart%
\&stackMode=absolute\&time=2020-03-29..latest\&region=World}, again using
binary segmentation to detect multiple breaks. We only report results
obtained using R\'{e}nyi statistics (with $\kappa =0.51$, $0.55$, $0.65$, $%
0.75$, $0.85$ and $1$); the other tests give very similar results, available
upon request. As far as breakdates are concerned, we pick the ones
corresponding to the \textquotedblleft majority vote\textquotedblright\
across $\kappa $, although discrepancies are, when present, in the region of
few days ($2-5$ at most).

\begin{table*}[h!]
\caption{Changepoint analysis for Covid-19 daily hospitalisation UK data.}
\label{tab:TabCovid}%\centering
%\begin{threeparttable}
%{\tiny
\par
\centering
\resizebox{\textwidth}{!}{
\begin{tabular}{lllllllllllllll}
\hline\hline
&  &  &  &  &  &  &  &  &  &  &  &  &  &  \\ 
\textbf{Region} &  & \multicolumn{1}{c}{\textit{Start Date}} & \textit{Sample size} &  &  & \textbf{Changepoint 1} &  & \textbf{Changepoint 2} &  & 
\textbf{Changepoint 3} &  & \textbf{Changepoint 4} &  & \textbf{Changepoint 5} \\ 
&  & \multicolumn{1}{c}{} &  &  &  &  &  &  &  &  &  &  &  &  \\ 
East of England &  & \multicolumn{1}{c}{19/3/20} & \multicolumn{1}{c}{$317 $} &  &  & \multicolumn{1}{c}{$\underset{\left[ 1.015;0.990\right] }{12\text{ Apr}}$} & \multicolumn{1}{c}{} & \multicolumn{1}{c}{$\underset{\left[
0.993;1.008\right] }{25\text{ Aug}}$} & \multicolumn{1}{c}{} & 
\multicolumn{1}{c}{} &  & \multicolumn{1}{c}{} &  & \multicolumn{1}{c}{$\underset{\left[ 1.011;1.000\right] }{08\text{ Jan}}$} \\ 
London &  & \multicolumn{1}{c}{19/3/20} & \multicolumn{1}{c}{$317$} &  & 
& \multicolumn{1}{c}{$\underset{\left[ 1.014;0.991\right] }{10\text{ Apr}}$}
& \multicolumn{1}{c}{} & \multicolumn{1}{c}{$\underset{\left[ 0.991;1.006\right] }{08\text{ Aug}}$} & \multicolumn{1}{c}{} & \multicolumn{1}{c}{} & 
& \multicolumn{1}{c}{} &  & \multicolumn{1}{c}{$\underset{\left[ 1.008;0.999\right] }{12\text{ Jan}}$} \\ 
Midlands &  & \multicolumn{1}{c}{19/3/20} & \multicolumn{1}{c}{$317$} & 
&  & \multicolumn{1}{c}{$\underset{\left[ 1.013;0.992\right] }{05\text{ Apr}}
$} & \multicolumn{1}{c}{} & \multicolumn{1}{c}{$\underset{\left[ 0.994;1.005\right] }{12\text{ Aug}}$} & \multicolumn{1}{c}{} & \multicolumn{1}{c}{$\underset{\left[ 1.011;1.000\right] }{14\text{ Nov}}$} &  & 
\multicolumn{1}{c}{$\underset{\left[ 0.997;1.001\right] }{12\text{ Dec}}$} & 
& \multicolumn{1}{c}{$\underset{\left[ 1.03;1.000\right] }{13\text{ Jan}}$}
\\ 
North East &  & \multicolumn{1}{c}{19/3/20} & \multicolumn{1}{c}{$317$} & 
&  & \multicolumn{1}{c}{$\underset{\left[ 1.018;0.992\right] }{07\text{ Apr}}
$} & \multicolumn{1}{c}{} & \multicolumn{1}{c}{$\underset{\left[ 0.995;1.005\right] }{19\text{ Aug}}$} & \multicolumn{1}{c}{} & \multicolumn{1}{c}{$\underset{\left[ 1.013;1.000\right] }{13\text{ Nov}}$} &  & 
\multicolumn{1}{c}{$\underset{\left[ 0.996;1.000\right] }{13\text{ Dec}}$} & 
& \multicolumn{1}{c}{$\underset{\left[ 1.002;1.000\right] }{12\text{ Jan}}$}
\\ 
North West &  & \multicolumn{1}{c}{19/3/20} & \multicolumn{1}{c}{$317$} & 
&  & \multicolumn{1}{c}{$\underset{\left[ 1.023;0.993\right] }{09\text{ Apr}}
$} & \multicolumn{1}{c}{} & \multicolumn{1}{c}{$\underset{\left[ 1.013;1.000\right] }{18\text{ Aug}}$} & \multicolumn{1}{c}{} & \multicolumn{1}{c}{$\underset{\left[ 1.015,1.000\right] }{29\text{ Oct}}$} &  & 
\multicolumn{1}{c}{$\underset{\left[ 0.997;1.001\right] }{13\text{ Dec}}$} & 
& \multicolumn{1}{c}{$\underset{\left[ 1.003;1.001\right] }{12\text{ Jan}}$}
\\ 
South East &  & \multicolumn{1}{c}{19/3/20} & \multicolumn{1}{c}{$317$} & 
&  & \multicolumn{1}{c}{$\underset{\left[ 1.016;0.992\right] }{22\text{ Apr}}
$} & \multicolumn{1}{c}{} & \multicolumn{1}{c}{$\underset{\left[ 0.993;1.007\right] }{21\text{ Aug}}$} & \multicolumn{1}{c}{} & \multicolumn{1}{c}{} & 
& \multicolumn{1}{c}{} &  & \multicolumn{1}{c}{$\underset{\left[ 1.009;1.000\right] }{07\text{ Jan}}$} \\ 
South West &  & \multicolumn{1}{c}{19/3/20} & \multicolumn{1}{c}{$317$} & 
&  & \multicolumn{1}{c}{$\underset{\left[ 1.013;0.989\right] }{12\text{ Apr}}
$} & \multicolumn{1}{c}{} & \multicolumn{1}{c}{$\underset{\left[ 0.992;1.010\right] }{01\text{ Sep}}$} & \multicolumn{1}{c}{} & \multicolumn{1}{c}{$\underset{\left[ 1.027;1.002\right] }{13\text{ Nov}}$} &  & 
\multicolumn{1}{c}{} &  & \multicolumn{1}{c}{$\underset{\left[ 1.002;1.000\right] }{21\text{ Jan}}$} \\ 
&  & \multicolumn{1}{c}{} & \multicolumn{1}{c}{} &  &  & \multicolumn{1}{c}{}
& \multicolumn{1}{c}{} & \multicolumn{1}{c}{} & \multicolumn{1}{c}{} & 
\multicolumn{1}{c}{} &  & \multicolumn{1}{c}{} &  & \multicolumn{1}{c}{} \\ 
\hline
&  & \multicolumn{1}{c}{} & \multicolumn{1}{c}{} &  &  &  &  &  &  & 
\multicolumn{1}{c}{} &  & \multicolumn{1}{c}{} &  &  \\ 
England &  & \multicolumn{1}{c}{19/3/20} & \multicolumn{1}{c}{$317$} &  & 
& \multicolumn{1}{c}{$\underset{\left[ 1.009;0.995\right] }{10\text{ Apr}}$}
& \multicolumn{1}{c}{} & \multicolumn{1}{c}{$\underset{\left[ 0.996;1.004\right] }{26\text{ Aug}}$} & \multicolumn{1}{c}{} & \multicolumn{1}{c}{$\underset{\left[ 1.010;1.001\right] }{29\text{ Oct}}$} &  & 
\multicolumn{1}{c}{} &  & \multicolumn{1}{c}{$\underset{\left[ 1.002;1.000\right] }{12\text{ Jan}}$} \\ 
Northern Ireland &  & \multicolumn{1}{c}{02/3/20} & \multicolumn{1}{c}{$325$} &  &  & \multicolumn{1}{c}{$\underset{\left[ 1.071;1.000\right] }{04\text{ Apr}}$} & \multicolumn{1}{c}{} & \multicolumn{1}{c}{$\underset{\left[
0.984;1.008\right] }{23\text{ Jul}}$} & \multicolumn{1}{c}{} & 
\multicolumn{1}{c}{} &  & \multicolumn{1}{c}{} &  & \multicolumn{1}{c}{$\underset{\left[ 1.009;1.000\right] }{12\text{ Jan}}$} \\ 
Scotland &  & \multicolumn{1}{c}{01/3/20} & \multicolumn{1}{c}{$326$} & 
&  & \multicolumn{1}{c}{$\underset{\left[ 1.026;0.999\right] }{04\text{ Apr}}
$} & \multicolumn{1}{c}{} & \multicolumn{1}{c}{$\underset{\left[ 0.984;1.007\right] }{12\text{ Aug}}$} & \multicolumn{1}{c}{} & \multicolumn{1}{c}{} & 
& \multicolumn{1}{c}{} &  & \multicolumn{1}{c}{$\underset{\left[ 1.009;1.001\right] }{12\text{ Jan}}$} \\ 
Wales &  & \multicolumn{1}{c}{23/3/20} & \multicolumn{1}{c}{$313$} &  & 
& \multicolumn{1}{c}{} & \multicolumn{1}{c}{} & \multicolumn{1}{c}{$\underset{\left[ 0.998;1.000\right] }{21\text{ Aug}}$} & \multicolumn{1}{c}{} &  &  & 
&  & \multicolumn{1}{c}{} \\ 
&  &  &  &  &  &  &  &  &  &  &  &  &  &  \\ \hline\hline
\end{tabular}
}
\par
\begin{tablenotes}
      \tiny
            \item All series end at 30 January 2021. We use the logs of the original data (plus one, given that, in some days, hospitalisations are equal to zero): no further transformations are used.
            \item All changepoints have been detected by all R\'{e}nyi-type tests - no discrepancies were noted. Detected changepoints, and their estimated date, are presented in \textit{chronological} order; breakdates have been estimated as the points in time where the majority of tests identifies a changepoint. Whilst details are available upon request, we note that breaks were detected with this order (from the first to be detected to the last one): breaks in August; breaks in April and January; breaks in October-November; breaks in December.
            \item For each changepoint, we report in square brackets, for reference, the left and right WLS estimates of $\beta_0$.
            
\end{tablenotes}
%}  
%\end{threeparttable}
\end{table*}

The results in Table \ref{tab:TabCovid} suggest that, with the exception of
Wales, there were multiple breaks in all series considered\footnote{%
Figure \ref{fig:FigCovid} contains the same information, albeit limited to
the four UK nations only to save space}; we note that Wales is an outlier as
regards hospital admissions, because these are counted in a different way
than the rest of the UK\footnote{%
Specifically, Wales reports also \textit{suspected} Covid-19 cases, whereas
all the other nations only report \textit{confirmed} cases; see
https://www.cebm.net/covid-19/the-flaw-in-the-reporting-of-welsh-data-on-covid-hospital-admissions/%
}.

\begin{figure}[!b]
\caption{Daily Covid-19 hospitalisations for the four UK nations }
\label{fig:FigCovid}\centering
\hspace{-2.5cm} 
\begin{minipage}{0.4\textwidth}
\centering
    \includegraphics[scale=0.4]{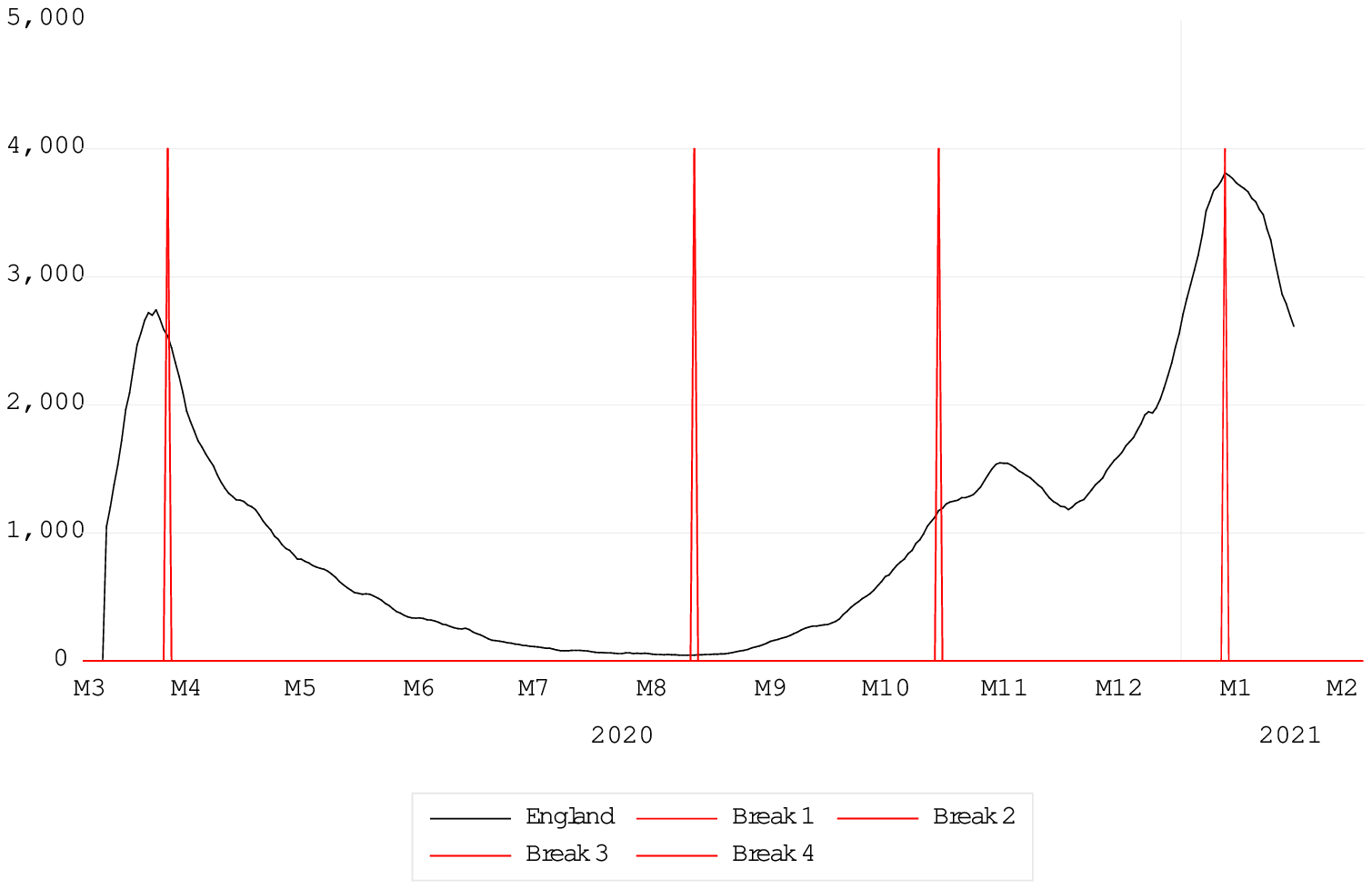}
%    \captionof{subfigure}{All series}
    \label{fig:t1}
\end{minipage}%
\begin{minipage}{0.4\textwidth}
\centering
   \includegraphics[scale=0.4]{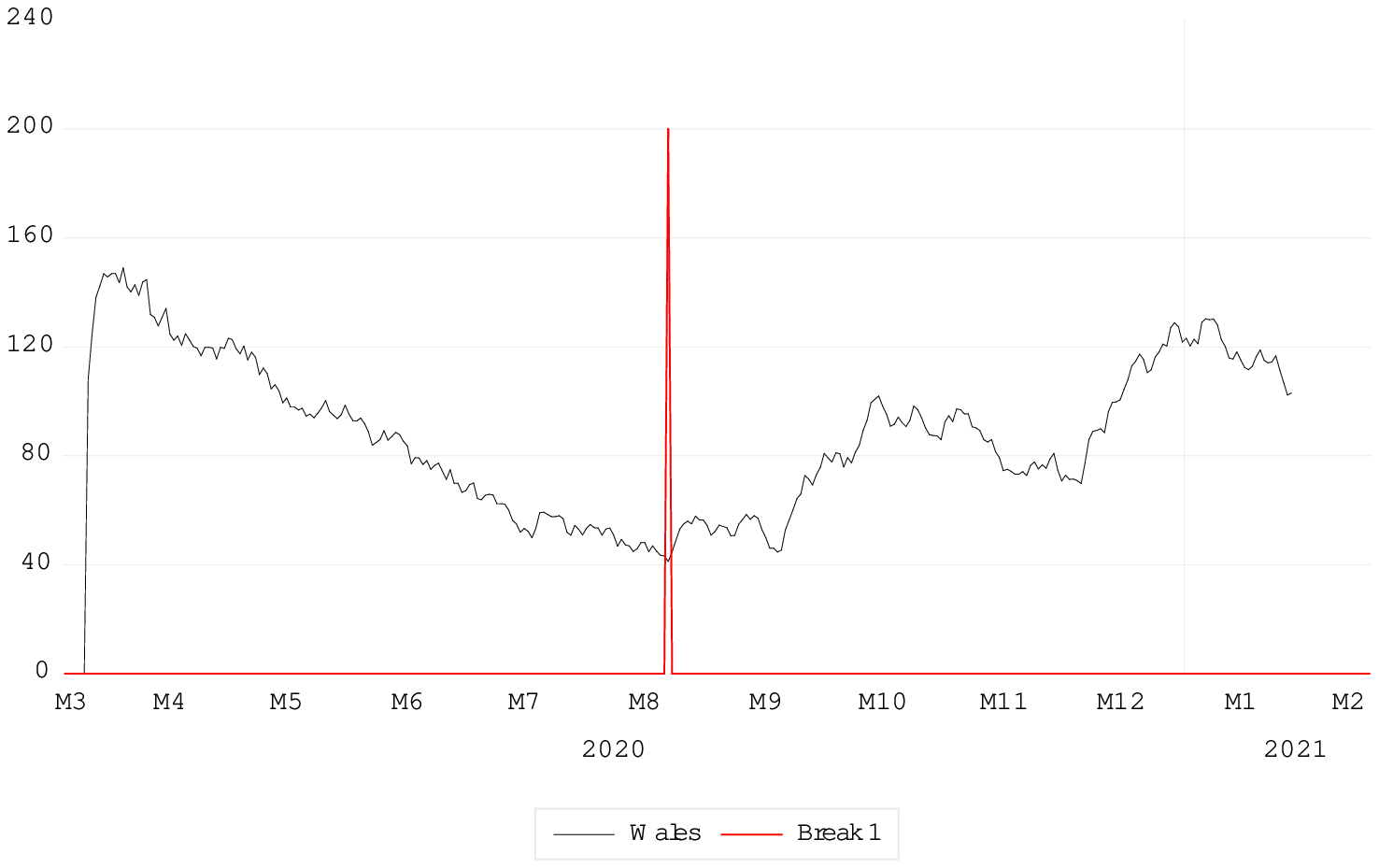}
%    \captionof{subfigure}{Metals}
    \label{fig:t2}
\end{minipage} \\[0.25cm]
\par
\hspace{-2.5cm} 
\begin{minipage}{0.4\textwidth}
\centering
    \includegraphics[scale=0.4]{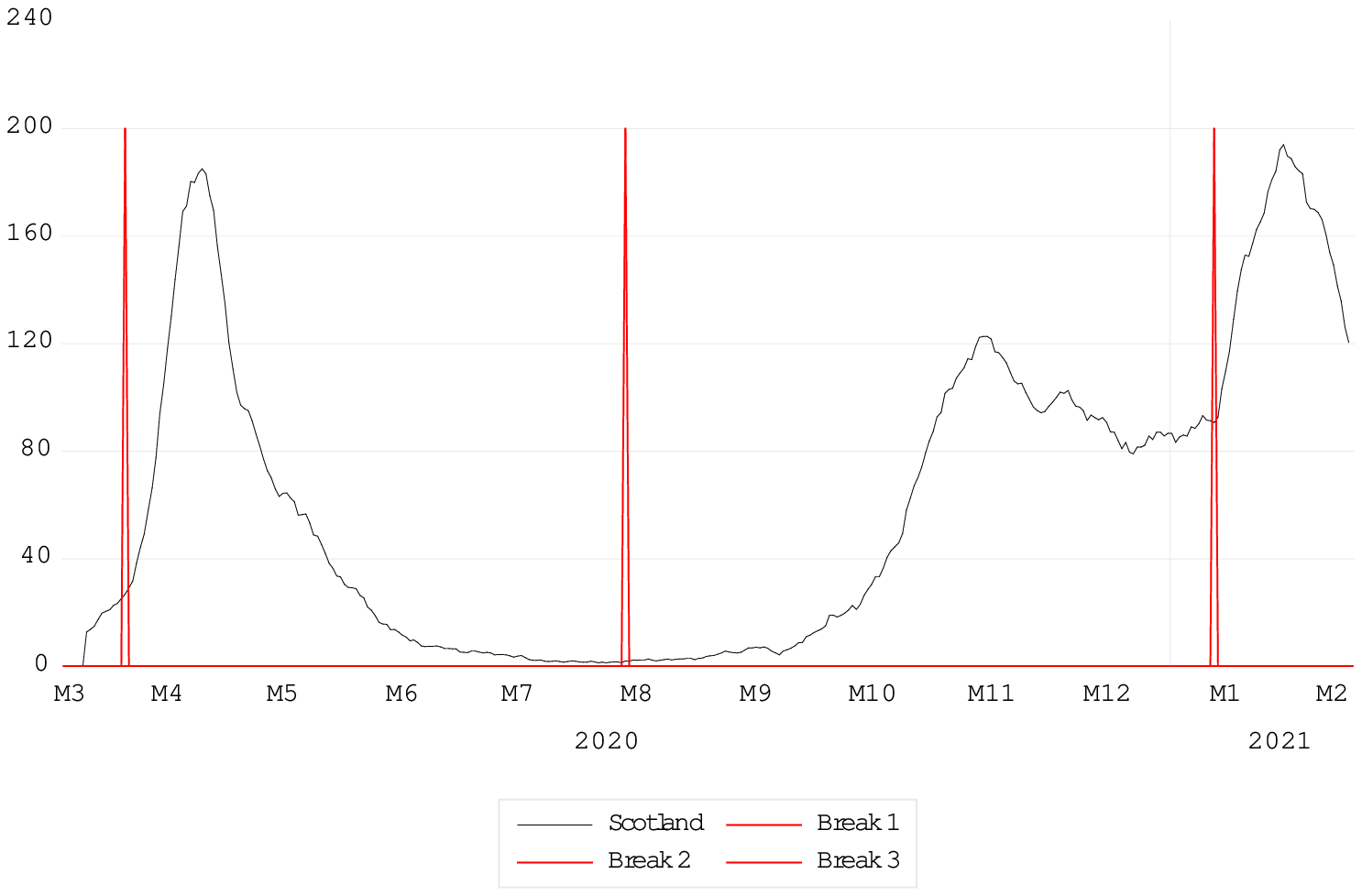}
%    \captionof{subfigure}{Crude}
    \label{fig:t3}
\end{minipage}%
\begin{minipage}{0.4\textwidth}
\centering
   \includegraphics[scale=0.4]{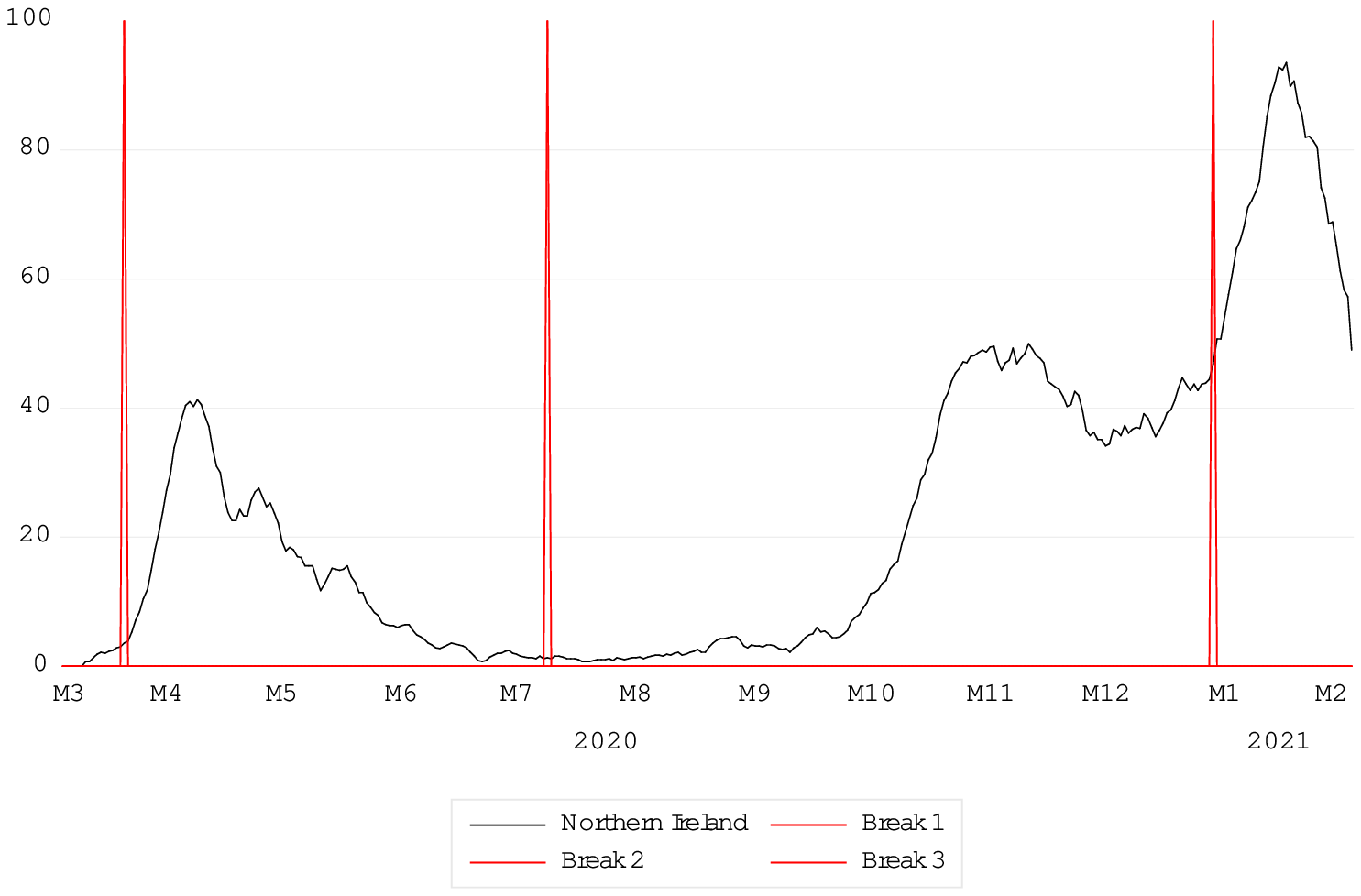}
%    \captionof{subfigure}{Metals}
    \label{fig:t4}
\end{minipage} \\[0.25cm]
\par
%\caption*{.}
\end{figure}

Some breaks occur closely to the sample endpoints, highlighting the
importance of using R\'{e}nyi statistics. Also, all changepoints indicate a
transition of the autoregressive coefficient $\beta _{0}$ around unity.
Differences between pre- and post-break values of $\beta _{0}$ are small,
but sufficient to trigger, or quench, an outbreak - on account of the Monte
Carlo evidence contained in Figures \ref{fig:bubb1}-\ref{fig:bubb2}, we
would not expect spurious detection of breaks when these are absent.

Considering first regions of England, all of these experience a break in
early April as a consequence of the first national lockdown, which started
on March 23rd, 2020, but was preceded by growing concerns, and closures in
the education and hospitality sectors, the week before. Similarly, all
series have a subsequent change (with $\beta _{0}$ exceeding unity after the
breaks) in late August - one exception is London, where the change occurred
in early August. These breaks indicate the beginning of the
\textquotedblleft second wave\textquotedblright\ in the UK, which has been
ascribed (also) to an increase in travelling during the holiday season and
which was officially acknowledge by the PM\ on September 18th, 2020. The
breaks in autumn, where present, can be explained as the effect of the local
and national lockdowns which were implemented at the end of October, and of
the easing of restrictions in early December. Finally, all series have a
change towards stationarity around mid-January, which again can be explained
as the effect of the national lockdown announced on January 4th, 2021, and
of the growing concerns about a third wave voiced before and during the
Christmas holidays.

The same picture applies to England as a whole. Conversely, the other UK\
nations experienced slightly different patterns, likely as a consequence of
different policies implemented by local governments. With the exception of
Wales, which seems to have only one break (but note the caveat about Welsh
data mentioned above), Scotland and Northern Ireland are essentially aligned
with the results for England in terms of the effects of the first lockdown,
the summer holiday, and the third lockdown.

\section{Discussion and conclusions\label{conclusions}}

In this paper, we study changepoint detection in the deterministic part of
the autoregression coefficient of a Random Coefficient AutoRegressive model.
We use the CUSUM\ process based on comparing the left and right WLS
estimators. In order to be able to detect changepoints close to the sample
endpoints, we study \textit{weighted} statistics, where more weight is
placed at the sample endpoints. We consider a very wide class of weighing
functions, studying: \textit{(i)} weighing schemes based on the functions $%
w\left( t\right) $, which drift to zero, at sample endpoints, more slowly
than $\left( t\left( 1-t\right) \right) ^{1/2}$; \textit{(ii)} standardised
statistics, with weighting $\left( t\left( 1-t\right) \right) ^{1/2}$; and 
\textit{(iii)} R\'{e}nyi statistics, where heavier weights are used. The
last class of statistics is still not fully studied (with the notable
exception of \citealp{horvathmiller}), and looks extremely promising in the
detection of early or late breaks.

From a practical point of view, our tests can be applied in the presence of
heteroskedasticity (requiring no knowledge as to the actual presence, or the
form, thereof), and simulations show that our procedures work very well in
practice - indeed, they work even better than procedures based on asymptotic
critical values in the baseline case of homoskedasticity. Technically, all
our results are based on a (strong) approximation of the weighted maximum of
partial sums. We have developed these both in the stationary and in the
nonstationary case: in the latter case (nonstationary data), our
approximations are entirely novel, and yield the same results as in the
stationary case. This, too, has important practical implications: our tests
can be applied with no prior knowledge as to the stationarity or lack
thereof of the data. This robustness reinforces the case made by %
\citet{aue2011} for RCA models, where the authors advocate the use of these
models as an alternative to the AR(1) model, which does not possess the same
property and may require differencing, with the well-known problems attached
to this transformation (\citealp{leybourne1996}). Hence, our procedures lend
themselves to several interesting applications and extensions. As a leading
example, our theory could be used as the building block to develop
procedures for the sequential detection of bubbles starting or collapsing.
This, and other extensions, are under investigation by the authors.

{\footnotesize {\ 
\bibliographystyle{chicago}
\bibliography{LTbiblio}

\begin{thebibliography}{}

\bibitem[\protect\citeauthoryear{Akharif and Hallin}{Akharif and
  Hallin}{2003}]{akharif2003}
Akharif, A. and M.~Hallin (2003).
\newblock Efficient detection of random coefficients in autoregressive models.
\newblock {\em The Annals of Statistics\/}~{\em 31\/}(2), 675--704.

\bibitem[\protect\citeauthoryear{And{\v{e}}l}{And{\v{e}}l}{1976}]{andel}
And{\v{e}}l, J. (1976).
\newblock Autoregressive series with random parameters.
\newblock {\em Mathematische Operationsforschung und Statistik\/}~{\em 7\/}(5),
  735--741.

\bibitem[\protect\citeauthoryear{Andrews}{Andrews}{1993}]{andrews1993}
Andrews, D.~W. (1993).
\newblock Tests for parameter instability and structural change with unknown
  change point.
\newblock {\em Econometrica\/}, 821--856.

\bibitem[\protect\citeauthoryear{Aue}{Aue}{2004}]{aue2004strong}
Aue, A. (2004).
\newblock Strong approximation for {RCA}(1) time series with applications.
\newblock {\em Statistics \& Probability Letters\/}~{\em 68\/}(4), 369--382.

\bibitem[\protect\citeauthoryear{Aue, H{\"o}rmann, Horv{\'a}th, and
  Hu{\v{s}}kov{\'a}}{Aue et~al.}{2014}]{aue2014}
Aue, A., S.~H{\"o}rmann, L.~Horv{\'a}th, and M.~Hu{\v{s}}kov{\'a} (2014).
\newblock Dependent functional linear models with applications to monitoring
  structural change.
\newblock {\em Statistica Sinica\/}, 1043--1073.

\bibitem[\protect\citeauthoryear{Aue and Horv{\'a}th}{Aue and
  Horv{\'a}th}{2011}]{aue2011}
Aue, A. and L.~Horv{\'a}th (2011).
\newblock Quasi-likelihood estimation in stationary and nonstationary
  autoregressive models with random coefficients.
\newblock {\em Statistica Sinica\/}, 973--999.

\bibitem[\protect\citeauthoryear{Aue, Horv{\'a}th, and Steinebach}{Aue
  et~al.}{2006}]{aue2006}
Aue, A., L.~Horv{\'a}th, and J.~Steinebach (2006).
\newblock Estimation in random coefficient autoregressive models.
\newblock {\em Journal of Time Series Analysis\/}~{\em 27\/}(1), 61--76.

\bibitem[\protect\citeauthoryear{Bardsley, Horv{\'a}th, Kokoszka, and
  Young}{Bardsley et~al.}{2017}]{bardsley2017}
Bardsley, P., L.~Horv{\'a}th, P.~Kokoszka, and G.~Young (2017).
\newblock Change point tests in functional factor models with application to
  yield curves.
\newblock {\em The Econometrics Journal\/}~{\em 20\/}(1), 86--117.

\bibitem[\protect\citeauthoryear{Benati and Kapetanios}{Benati and
  Kapetanios}{2003}]{benati}
Benati, L. and G.~Kapetanios (2003).
\newblock Structural breaks in inflation dynamics.
\newblock In {\em Computing in Economics and Finance}, Volume 169, pp.\
  563--587. Society for Computational Economics.

\bibitem[\protect\citeauthoryear{Berkes, Horv{\'a}th, and Ling}{Berkes
  et~al.}{2009}]{berkes2009}
Berkes, I., L.~Horv{\'a}th, and S.~Ling (2009).
\newblock Estimation in nonstationary random coefficient autoregressive models.
\newblock {\em Journal of Time Series Analysis\/}~{\em 30\/}(4), 395--416.

\bibitem[\protect\citeauthoryear{Berkes, Liu, and Wu}{Berkes
  et~al.}{2014}]{berkesliuwu}
Berkes, I., W.~Liu, and W.~B. Wu (2014).
\newblock Koml{\'o}s--{Major}--{Tusn{\'a}dy} approximation under dependence.
\newblock {\em The Annals of Probability\/}~{\em 42\/}(2), 794--817.

\bibitem[\protect\citeauthoryear{Cs{\"o}rg{\H{o}}, Cs{\"o}rg{\H{o}},
  Horv{\'a}th, and Mason}{Cs{\"o}rg{\H{o}} et~al.}{1986}]{csorgo1986}
Cs{\"o}rg{\H{o}}, M., S.~Cs{\"o}rg{\H{o}}, L.~Horv{\'a}th, and D.~M. Mason
  (1986).
\newblock Weighted empirical and quantile processes.
\newblock {\em The Annals of Probability\/}~{\em 14}, 31--85.

\bibitem[\protect\citeauthoryear{Cs{\"o}rg{\H{o}} and
  Horv{\'a}th}{Cs{\"o}rg{\H{o}} and Horv{\'a}th}{1993}]{csorgo1993}
Cs{\"o}rg{\H{o}}, M. and L.~Horv{\'a}th (1993).
\newblock {\em Weighted approximations in probability and statistics}.
\newblock J. Wiley \& Sons.

\bibitem[\protect\citeauthoryear{Cs{\"o}rg{\H{o}} and
  Horv{\'a}th}{Cs{\"o}rg{\H{o}} and Horv{\'a}th}{1997}]{csorgo1997}
Cs{\"o}rg{\H{o}}, M. and L.~Horv{\'a}th (1997).
\newblock {\em Limit theorems in change-point analysis}, Volume~18.
\newblock John Wiley \& Sons.

\bibitem[\protect\citeauthoryear{Darling and Erd{\H{o}}s}{Darling and
  Erd{\H{o}}s}{1956}]{darling1956limit}
Darling, D.~A. and P.~Erd{\H{o}}s (1956).
\newblock A limit theorem for the maximum of normalized sums of independent
  random variables.
\newblock {\em Duke Math. J\/}~{\em 23\/}(1), 143--155.

\bibitem[\protect\citeauthoryear{Engle}{Engle}{1982}]{engle1982}
Engle, R.~F. (1982).
\newblock Autoregressive conditional heteroscedasticity with estimates of the
  variance of united kingdom inflation.
\newblock {\em Econometrica\/}, 987--1007.

\bibitem[\protect\citeauthoryear{Eo}{Eo}{2016}]{eo}
Eo, Y. (2016).
\newblock Structural changes in inflation dynamics: multiple breaks at
  different dates for different parameters.
\newblock {\em Studies in Nonlinear Dynamics \& Econometrics\/}~{\em 20\/}(3),
  211--231.

\bibitem[\protect\citeauthoryear{Fan and Yao}{Fan and Yao}{2008}]{fanyao}
Fan, J. and Q.~Yao (2008).
\newblock {\em Nonlinear time series: nonparametric and parametric methods}.
\newblock Springer Science \& Business Media.

\bibitem[\protect\citeauthoryear{Franke, Hefter, Herzwurm, Ritter, and
  Schwaar}{Franke et~al.}{2020}]{schwaar}
Franke, J., M.~Hefter, A.~Herzwurm, K.~Ritter, and S.~Schwaar (2020).
\newblock Adaptive quantile computation for brownian bridge in change-point
  analysis.
\newblock {\em arXiv preprint arXiv:2101.00064\/}.

\bibitem[\protect\citeauthoryear{Fryz}{Fryz}{2017}]{fryz}
Fryz, M. (2017).
\newblock Conditional linear random process and random coefficient
  autoregressive model for {EEG} analysis.
\newblock In {\em 2017 IEEE First Ukraine Conference on Electrical and Computer
  Engineering (UKRCON)}, pp.\  305--309. IEEE.

\bibitem[\protect\citeauthoryear{Giraitis, Kapetanios, and Yates}{Giraitis
  et~al.}{2014}]{gky}
Giraitis, L., G.~Kapetanios, and T.~Yates (2014).
\newblock Inference on stochastic time-varying coefficient models.
\newblock {\em Journal of Econometrics\/}~{\em 179\/}(1), 46--65.

\bibitem[\protect\citeauthoryear{Gombay and Horv\'ath}{Gombay and
  Horv\'ath}{1996}]{gombay}
Gombay, E. and L.~Horv\'ath (1996).
\newblock On the rate of approximations for maximum likelihood tests in
  change-point models.
\newblock {\em Journal of Multivariate Analysis\/}~{\em 56\/}(1), 120--152.

\bibitem[\protect\citeauthoryear{G{\'o}recki, Horv{\'a}th, and
  Kokoszka}{G{\'o}recki et~al.}{2018}]{gorecki2017}
G{\'o}recki, T., L.~Horv{\'a}th, and P.~Kokoszka (2018).
\newblock Change point detection in heteroscedastic time series.
\newblock {\em Econometrics and Statistics\/}~{\em 7}, 63--88.

\bibitem[\protect\citeauthoryear{Harvey, Leybourne, Sollis, and Taylor}{Harvey
  et~al.}{2016}]{harvey2016}
Harvey, D.~I., S.~J. Leybourne, R.~Sollis, and A.~R. Taylor (2016).
\newblock Tests for explosive financial bubbles in the presence of
  non-stationary volatility.
\newblock {\em Journal of Empirical Finance\/}~{\em 38}, 548--574.

\bibitem[\protect\citeauthoryear{Hill, Li, and Peng}{Hill
  et~al.}{2016}]{hillpeng2016}
Hill, J., D.~Li, and L.~Peng (2016).
\newblock Uniform interval estimation for an {AR(1)} process with {AR} errors.
\newblock {\em Statistica Sinica\/}~{\em 26\/}(1), 119--136.

\bibitem[\protect\citeauthoryear{Homm and Breitung}{Homm and
  Breitung}{2012}]{homm2012testing}
Homm, U. and J.~Breitung (2012).
\newblock Testing for speculative bubbles in stock markets: a comparison of
  alternative methods.
\newblock {\em Journal of Financial Econometrics\/}~{\em 10\/}(1), 198--231.

\bibitem[\protect\citeauthoryear{Horv\'{a}th, Hu\v{s}kov\'{a}, Kokoszka, and
  Steinebach}{Horv\'{a}th et~al.}{2004}]{lajos04}
Horv\'{a}th, L., M.~Hu\v{s}kov\'{a}, P.~Kokoszka, and J.~Steinebach (2004).
\newblock Monitoring changes in linear models.
\newblock {\em Journal of Statistical Planning and Inference\/}~{\em 126},
  225--251.

\bibitem[\protect\citeauthoryear{Horv{\'a}th, Liu, and Lu}{Horv{\'a}th
  et~al.}{2021}]{horvath2021sequential}
Horv{\'a}th, L., Z.~Liu, and S.~Lu (2021).
\newblock Sequential monitoring of changes in dynamic linear models, applied to
  the {US} housing market.
\newblock {\em Econometric Theory\/}, 1--64.

\bibitem[\protect\citeauthoryear{Horv{\'a}th, Liu, Rice, and Wang}{Horv{\'a}th
  et~al.}{2020}]{horvath2020sequential}
Horv{\'a}th, L., Z.~Liu, G.~Rice, and S.~Wang (2020).
\newblock Sequential monitoring for changes from stationarity to mild
  non-stationarity.
\newblock {\em Journal of Econometrics\/}~{\em 215\/}(1), 209--238.

\bibitem[\protect\citeauthoryear{Horv{\'a}th, Miller, and Rice}{Horv{\'a}th
  et~al.}{2020a}]{horvath2021}
Horv{\'a}th, L., C.~Miller, and G.~Rice (2020a).
\newblock Detecting early or late changes in linear models with heteroscedastic
  errors.
\newblock {\em Scandinavian Journal of Statistics\/}.

\bibitem[\protect\citeauthoryear{Horv{\'a}th, Miller, and Rice}{Horv{\'a}th
  et~al.}{2020b}]{horvathmiller}
Horv{\'a}th, L., C.~Miller, and G.~Rice (2020b).
\newblock A new class of change point test statistics of r{\'e}nyi type.
\newblock {\em Journal of Business \& Economic Statistics\/}~{\em 38\/}(3),
  570--579.

\bibitem[\protect\citeauthoryear{Horv{\'a}th and Trapani}{Horv{\'a}th and
  Trapani}{2016}]{HT2016}
Horv{\'a}th, L. and L.~Trapani (2016).
\newblock Statistical inference in a random coefficient panel model.
\newblock {\em Journal of Econometrics\/}~{\em 193\/}(1), 54--75.

\bibitem[\protect\citeauthoryear{Horv\'{a}th and Trapani}{Horv\'{a}th and
  Trapani}{2019}]{HT16}
Horv\'{a}th, L. and L.~Trapani (2019).
\newblock Testing for randomness in a random coefficient autoregression.
\newblock {\em Journal of Econometrics\/}~{\em 209}, 338--352.

\bibitem[\protect\citeauthoryear{Jane{\v{c}}kov{\'a} and
  Pr{\'a}{\v{s}}kov{\'a}}{Jane{\v{c}}kov{\'a} and
  Pr{\'a}{\v{s}}kov{\'a}}{2004}]{praskova2004}
Jane{\v{c}}kov{\'a}, H. and Z.~Pr{\'a}{\v{s}}kov{\'a} (2004).
\newblock {CWLS} and {ML} estimates in a heteroscedastic {RCA (1)} model.
\newblock {\em Statistics \& Decisions/International mathematical journal for
  stochastic methods and models\/}~{\em 22\/}(3/2004), 245--259.

\bibitem[\protect\citeauthoryear{Koml\'{o}s, Major, and Tusn\'{a}dy}{Koml\'{o}s
  et~al.}{1975}]{KMT1}
Koml\'{o}s, J., P.~Major, and G.~Tusn\'{a}dy (1975).
\newblock An approximation of partial sums of independent {R.V.'s} and the
  sample {DF.I}.
\newblock {\em Z. Wahrscheinlichkeitstheorie und verwandte Gebiete\/}~{\em 32},
  111--131.

\bibitem[\protect\citeauthoryear{Koml\'{o}s, Major, and Tusn\'{a}dy}{Koml\'{o}s
  et~al.}{1976}]{KMT2}
Koml\'{o}s, J., P.~Major, and G.~Tusn\'{a}dy (1976).
\newblock An approximation of partial sums of independent {R.V.'s} and the
  sample {DF.II}.
\newblock {\em Z. Wahrscheinlichkeitstheorie und verwandte Gebiete\/}~{\em 34},
  33--58.

\bibitem[\protect\citeauthoryear{Koul and Schick}{Koul and
  Schick}{1996}]{koul1996}
Koul, H.~L. and A.~Schick (1996).
\newblock Adaptive estimation in a random coefficient autoregressive model.
\newblock {\em The Annals of Statistics\/}~{\em 24\/}(3), 1025--1052.

\bibitem[\protect\citeauthoryear{Lee}{Lee}{1998}]{lee1998}
Lee, S. (1998).
\newblock Coefficient constancy test in a random coefficient autoregressive
  model.
\newblock {\em Journal of Statistical Planning and Inference\/}~{\em 74\/}(1),
  93--101.

\bibitem[\protect\citeauthoryear{Lee, Ha, Na, and Na}{Lee
  et~al.}{2003}]{lee2003cusum}
Lee, S., J.~Ha, O.~Na, and S.~Na (2003).
\newblock The {CUSUM} test for parameter change in time series models.
\newblock {\em Scandinavian Journal of Statistics\/}~{\em 30\/}(4), 781--796.

\bibitem[\protect\citeauthoryear{Leybourne, McCabe, and Tremayne}{Leybourne
  et~al.}{1996}]{leybourne1996}
Leybourne, S.~J., B.~P. McCabe, and A.~R. Tremayne (1996).
\newblock Can economic time series be differenced to stationarity?
\newblock {\em Journal of Business \& Economic Statistics\/}~{\em 14\/}(4),
  435--446.

\bibitem[\protect\citeauthoryear{Lieberman}{Lieberman}{2012}]{lieberman2012}
Lieberman, O. (2012).
\newblock A similarity-based approach to time-varying coefficient
  non-stationary autoregression.
\newblock {\em Journal of Time Series Analysis\/}~{\em 33\/}(3), 484--502.

\bibitem[\protect\citeauthoryear{Lumsdaine}{Lumsdaine}{1996}]{lumsdaine1996consistency}
Lumsdaine, R.~L. (1996).
\newblock Consistency and asymptotic normality of the quasi-maximum likelihood
  estimator in {IGARCH}(1, 1) and covariance stationary {GARCH}(1, 1) models.
\newblock {\em Econometrica\/}, 575--596.

\bibitem[\protect\citeauthoryear{M{\'o}ricz, Serfling, and Stout}{M{\'o}ricz
  et~al.}{1982}]{moricz1982}
M{\'o}ricz, F., R.~Serfling, and W.~Stout (1982).
\newblock Moment and probability bounds with quasi-superadditive structure for
  the maximum partial sum.
\newblock {\em The Annals of Probability\/}~{\em 10}, 1032--1040.

\bibitem[\protect\citeauthoryear{Nagakura}{Nagakura}{2009}]{nagakura2009}
Nagakura, D. (2009).
\newblock Asymptotic theory for explosive random coefficient autoregressive
  models and inconsistency of a unit root test against a stochastic unit root
  process.
\newblock {\em Statistics \& Probability Letters\/}~{\em 79\/}(24), 2476--2483.

\bibitem[\protect\citeauthoryear{Nicholls and Quinn}{Nicholls and
  Quinn}{2012}]{nichollsquinn}
Nicholls, D.~F. and B.~G. Quinn (2012).
\newblock {\em Random Coefficient Autoregressive Models: An Introduction: An
  Introduction}, Volume~11.
\newblock Springer Science \& Business Media.

\bibitem[\protect\citeauthoryear{Phillips, Shi, and Yu}{Phillips
  et~al.}{2015}]{phillips2015testing}
Phillips, P.~C., S.~Shi, and J.~Yu (2015).
\newblock Testing for multiple bubbles: Historical episodes of exuberance and
  collapse in the {S}\&{P} 500.
\newblock {\em International Economic Review\/}~{\em 56\/}(4), 1043--1078.

\bibitem[\protect\citeauthoryear{Phillips, Wu, and Yu}{Phillips
  et~al.}{2011}]{phillips2011}
Phillips, P.~C., Y.~Wu, and J.~Yu (2011).
\newblock Explosive behavior in the 1990s {Nasdaq}: When did exuberance
  escalate asset values?
\newblock {\em International Economic Review\/}~{\em 52\/}(1), 201--226.

\bibitem[\protect\citeauthoryear{Regis, Serra, and Heuvel}{Regis
  et~al.}{2021}]{regis}
Regis, M., P.~Serra, and E.~R. Heuvel (2021).
\newblock Random autoregressive models: A structured overview.
\newblock {\em Econometric Reviews\/}.

\bibitem[\protect\citeauthoryear{Schick}{Schick}{1996}]{schick1996}
Schick, A. (1996).
\newblock $\sqrt{n}$-consistent estimation in a random coefficient
  autoregressive model.
\newblock {\em Australian \& New Zealand Journal of Statistics\/}~{\em
  38\/}(2), 155--160.

\bibitem[\protect\citeauthoryear{Shtatland and Shtatland}{Shtatland and
  Shtatland}{2008}]{shtatland}
Shtatland, E.~S. and T.~Shtatland (2008).
\newblock Another look at low-order autoregressive models in early detection of
  epidemic outbreaks and explosive behaviors in economic and financial time
  series.
\newblock {\em SGF Proceedings\/}~{\em 363}.

\bibitem[\protect\citeauthoryear{{\'S}l{\k{e}}zak, Burnecki, and
  Metzler}{{\'S}l{\k{e}}zak et~al.}{2019}]{slkezak2019random}
{\'S}l{\k{e}}zak, J., K.~Burnecki, and R.~Metzler (2019).
\newblock Random coefficient autoregressive processes describe brownian yet
  non-gaussian diffusion in heterogeneous systems.
\newblock {\em New Journal of Physics\/}~{\em 21\/}(7), 073056.

\bibitem[\protect\citeauthoryear{Stenseth, Falck, Chan, Bj{\o}rnstad,
  O'Donoghue, Tong, Boonstra, Boutin, Krebs, and Yoccoz}{Stenseth
  et~al.}{1998}]{stenseth}
Stenseth, N.~C., W.~Falck, K.-S. Chan, O.~N. Bj{\o}rnstad, M.~O'Donoghue,
  H.~Tong, R.~Boonstra, S.~Boutin, C.~J. Krebs, and N.~G. Yoccoz (1998).
\newblock From patterns to processes: phase and density dependencies in the
  {C}anadian lynx cycle.
\newblock {\em Proceedings of the National Academy of Sciences\/}~{\em
  95\/}(26), 15430--15435.

\bibitem[\protect\citeauthoryear{Trapani}{Trapani}{2021}]{trapanistrict}
Trapani, L. (2021).
\newblock Testing for strict stationarity in a random coefficient
  autoregressive model.
\newblock {\em Econometric Reviews\/}.

\bibitem[\protect\citeauthoryear{Tsay}{Tsay}{1987}]{tsay1987}
Tsay, R.~S. (1987).
\newblock Conditional heteroscedastic time series models.
\newblock {\em Journal of the American Statistical Association\/}~{\em
  82\/}(398), 590--604.

\bibitem[\protect\citeauthoryear{Vostrikova}{Vostrikova}{1981}]{vostrikova}
Vostrikova, L.~Y. (1981).
\newblock Detecting disorder in multidimensional random processes.
\newblock In {\em Doklady Akademii Nauk}, Volume 259, pp.\  270--274. Russian
  Academy of Sciences.

\bibitem[\protect\citeauthoryear{Xu}{Xu}{2015}]{xu2015}
Xu, K.-L. (2015).
\newblock Testing for structural change under non-stationary variances.
\newblock {\em The Econometrics Journal\/}~{\em 18\/}(2), 274--305.

\bibitem[\protect\citeauthoryear{Xu and Phillips}{Xu and
  Phillips}{2008}]{xuphillips2008}
Xu, K.-L. and P.~C. Phillips (2008).
\newblock Adaptive estimation of autoregressive models with time-varying
  variances.
\newblock {\em Journal of Econometrics\/}~{\em 142\/}(1), 265--280.

\bibitem[\protect\citeauthoryear{Yau and Zhao}{Yau and Zhao}{2016}]{yauzhao}
Yau, C.~Y. and Z.~Zhao (2016).
\newblock Inference for multiple change points in time series via likelihood
  ratio scan statistics.
\newblock {\em Journal of the Royal Statistical Society: Series B\/}~{\em 78},
  895--916.

\bibitem[\protect\citeauthoryear{Zhao and Wang}{Zhao and Wang}{2012}]{zhao2012}
Zhao, Z.-W. and D.-H. Wang (2012).
\newblock Statistical inference for generalized random coefficient
  autoregressive model.
\newblock {\em Mathematical and Computer Modelling\/}~{\em 56\/}(7), 152--166.

\end{thebibliography}
} }\newpage

%\clearpage
\appendix\setcounter{section}{0} \setcounter{subsection}{-1} %
\setcounter{equation}{0} \setcounter{lemma}{0} \renewcommand{\thelemma}{A.%
\arabic{lemma}} \renewcommand{\theequation}{A.\arabic{equation}}

\section{Further results\label{extend}}

We discuss the computation of asymptotic critical values, and report some of
them (Section \ref{crval}). Further, in Section \ref{mbreaks}, we study the
power versus more general alternatives that the AMOC one considered in
Section \ref{consistency}.

\subsection{Computation of critical values and further simulations under
homoskedasticity\label{crval}}

The asymptotic critical values fo the homoskedastic case are in Table \ref%
{tab:TableCrVal}.

\begin{table}[h!]
\caption{Asymptotic critical values}
\label{tab:TableCrVal}\centering
%\begin{table}
{\tiny \centering
}
\par
{\tiny 
\begin{tabular}{llllll}
\hline\hline
&  &  &  &  &  \\ 
&  & \multicolumn{4}{c}{Asymptotic critical values} \\ 
&  &  &  &  &  \\ 
&  & \multicolumn{1}{c}{$5\%$} & \multicolumn{1}{c}{} & \multicolumn{1}{c}{$%
10\%$} &  \\ 
$\kappa $ &  & \multicolumn{1}{c}{} & \multicolumn{1}{c}{} & 
\multicolumn{1}{c}{} &  \\ 
&  & \multicolumn{1}{c}{} & \multicolumn{1}{c}{} & \multicolumn{1}{c}{} & 
\\ 
$0.0$ &  & \multicolumn{1}{c}{$1.3700$} & \multicolumn{1}{c}{} & 
\multicolumn{1}{c}{$1.2238$} &  \\ 
$0.25$ &  & \multicolumn{1}{c}{$2.0142$} & \multicolumn{1}{c}{} & 
\multicolumn{1}{c}{$1.8106$} &  \\ 
$0.45$ &  & \multicolumn{1}{c}{$3.0320$} & \multicolumn{1}{c}{} & 
\multicolumn{1}{c}{$2.8988$} &  \\ 
$0.51$ &  & \multicolumn{1}{c}{$3.2944$} & \multicolumn{1}{c}{} & 
\multicolumn{1}{c}{$3.0722$} &  \\ 
$0.55$ &  & \multicolumn{1}{c}{$3.0144$} & \multicolumn{1}{c}{} & 
\multicolumn{1}{c}{$2.7992$} &  \\ 
$0.65$ &  & \multicolumn{1}{c}{$2.7394$} & \multicolumn{1}{c}{} & 
\multicolumn{1}{c}{$2.5050$} &  \\ 
$0.75$ &  & \multicolumn{1}{c}{$2.6396$} & \multicolumn{1}{c}{} & 
\multicolumn{1}{c}{$2.3860$} &  \\ 
$0.85$ &  & \multicolumn{1}{c}{$2.5475$} & \multicolumn{1}{c}{} & 
\multicolumn{1}{c}{$2.2996$} &  \\ 
$1.0$ &  & \multicolumn{1}{c}{$2.4948$} & \multicolumn{1}{c}{} & 
\multicolumn{1}{c}{$2.2365$} &  \\ 
&  &  &  &  &  \\ \hline\hline
\end{tabular}
}
\par
{\tiny 
\begin{tablenotes}
      \tiny
            \item Critical values at $5\%$ and $10\%$ nominal levels, for the limiting distributions in Theorems \ref{thrca1}.  
            
\end{tablenotes}
}
\end{table}

When $\kappa <\frac{1}{2}$, we have simulated the critical values using the
algorithm proposed in \citet{schwaar}; our results differ marginally from
the values reported in the original paper, but unreported experiments show
that our critical values yield less undersizement than the original ones, at
least in small samples.

In the case $\kappa >\frac{1}{2}$, we know that critical values are the same
for both the homoskedastic and the heteroskedastic case. In all experiments
(and in the computation of critical values), we use symmetric trimming -
i.e., $r_{1}(N)=r_{2}(N)$; in this case, it is easy to see that%
\begin{equation*}
P\left[ \max \left( \mathfrak{a}_{1}(\kappa ),\mathfrak{a}_{2}(\kappa
)\right) \leq c_{\alpha }\right] =\left( P\left[ \mathfrak{a}(\kappa )\leq
c_{\alpha }\right] \right) ^{2},
\end{equation*}%
and our critical values are based on Table 1 in \citet{lajos04}.

The most critical case is the case $\kappa =\frac{1}{2}$. For a given
nominal level $\alpha $, asymptotic critical values are given by $c_{\alpha
}=-\ln \left( -0.5\ln \left( 1-\alpha \right) \right) $, and Theorems \ref%
{thrca1}-\ref{thrca5} state that the limiting distribution of the max-type
statistics is the same in both the homoskedastic and the heteroskedastic
cases. Interestingly (and contrary to the heteroskedastic case), our
simulations show that in the homoskedastic case, asymptotic critical values
work well, with no under-rejection and good power.

\bigskip

%\bigskip
To complement Section \ref{simulations}, we also report, as a benchmark,
some evidence on the size of our tests under homoskedasticity, using the
theory in Section \ref{homosk}. Empirical rejection frequencies under the
null are reported in Table \ref{tab:SizeHomo1}.\footnote{%
We do not discuss the power for brevity, also on account of the fact that,
in practice, one would apply the tests discussed in Section \ref{feasible}.
Results are anyway in line with the rest of the simulations, and available
upon request.}

\bigskip

\begin{table*}[h]
\caption{Empirical rejection frequencies, homoskedastic case using
asymptotic critical values}
\label{tab:SizeHomo1}\centering
\resizebox{\textwidth}{!}{

\begin{tabular}{llllllllllllllllllllll}
\hline\hline
&  &  &  &  &  &  &  &  &  &  &  &  &  &  &  &  &  &  &  &  &  \\ 
&  & $\beta $ & \multicolumn{4}{c}{$0.5$} & \multicolumn{1}{c}{} & 
\multicolumn{4}{c}{$0.75$} & \multicolumn{1}{c}{} & \multicolumn{4}{c}{$1$}
& \multicolumn{1}{c}{} & \multicolumn{4}{c}{$1.05$} \\ 
&  &  &  &  &  &  &  &  &  &  &  &  &  &  &  &  &  &  &  &  &  \\ 
& $N$ &  & $200$ & $400$ & $800$ & $1600$ &  & $200$ & $400$ & $800$ & $1600$
&  & $200$ & $400$ & $800$ & $1600$ &  & $200$ & $400$ & $800$ & $1600$ \\ 
$\kappa $ &  &  &  &  &  &  &  &  &  &  &  &  &  &  &  &  &  &  &  &  &  \\ 
&  &  &  &  &  &  &  &  &  &  &  &  &  &  &  &  &  &  &  &  &  \\ 
$0$ &  &  & $0.048$ & $0.044$ & $0.040$ & $0.037$ &  & $0.044$ & $0.053$ & $0.037$ & $0.039$ &  & $0.048$ & $0.052$ & $0.054$ & $0.045$ &  & $0.057$ & $0.045$ & $0.038$ & $0.037$ \\ 
$0.25$ &  &  & $0.060$ & $0.052$ & $0.045$ & $0.043$ &  & $0.061$ & $0.062$
& $0.046$ & $0.048$ &  & $0.069$ & $0.069$ & $0.058$ & $0.049$ &  & $0.070$
& $0.059$ & $0.042$ & $0.045$ \\ 
$0.45$ &  &  & $0.073$ & $0.059$ & $0.034$ & $0.030$ &  & $0.062$ & $0.071$
& $0.054$ & $0.039$ &  & $0.086$ & $0.090$ & $0.079$ & $0.067$ &  & $0.079$
& $0.081$ & $0.062$ & $0.060$ \\ 
$0.5$ &  &  & $0.060$ & $0.044$ & $0.023$ & $0.026$ &  & $0.047$ & $0.056$ & 
$0.035$ & $0.030$ &  & $0.062$ & $0.077$ & $0.057$ & $0.061$ &  & $0.058$ & $0.065$ & $0.051$ & $0.047$ \\ 
$0.51$ &  &  & $0.018$ & $0.019$ & $0.020$ & $0.028$ &  & $0.029$ & $0.026$
& $0.034$ & $0.034$ &  & $0.034$ & $0.028$ & $0.037$ & $0.042$ &  & $0.028$
& $0.041$ & $0.037$ & $0.037$ \\ 
$0.75$ &  &  & $0.045$ & $0.037$ & $0.038$ & $0.039$ &  & $0.058$ & $0.049$
& $0.051$ & $0.051$ &  & $0.070$ & $0.060$ & $0.050$ & $0.060$ &  & $0.067$
& $0.072$ & $0.058$ & $0.053$ \\ 
$0.85$ &  &  & $0.051$ & $0.035$ & $0.043$ & $0.041$ &  & $0.066$ & $0.052$
& $0.055$ & $0.051$ &  & $0.076$ & $0.061$ & $0.057$ & $0.052$ &  & $0.074$
& $0.076$ & $0.059$ & $0.057$ \\ 
$1$ &  &  & $0.053$ & $0.036$ & $0.044$ & $0.042$ &  & $0.066$ & $0.051$ & $0.053$ & $0.050$ &  & $0.077$ & $0.061$ & $0.053$ & $0.053$ &  & $0.075$ & $0.076$ & $0.059$ & $0.055$ \\ 
&  &  &  &  &  &  &  &  &  &  &  &  &  &  &  &  &  &  &  &  &  \\ 
\hline\hline
\end{tabular}
}
\par
\begin{tablenotes}
      \tiny
            \item The table contains the empirical rejection frequencies under the null of no changepoint for different sample sizes and different values of $\kappa$, in the homoskedastic case. 
            \item Asymptotic critical values have been used, based on the limiting distributions in Theorems \ref{thrca1} (see Table \ref{tab:TableCrVal}).
            
\end{tablenotes}
\end{table*}

\bigskip

Broadly speaking, all test statistics have the correct size for large
samples; as mentioned above, this also includes the Darling-Erd\H{o}s
statistic, based on asymptotic critical values, despite the notoriously slow
convergence to the extreme value distribution. When $N=200$ (i.e., in small
samples), tests are, occasionally, mildly oversized. This also happens when $%
\kappa =0.45$ (and $\kappa >0.5$ with $\beta _{0}=1.05$); conversely, the
test is grossly undersized almost under any circumstances when $\kappa =0.51$%
, although this seems to impove as $N$ increases. We note that, with the few
exceptions mentioned above, the value of $\beta _{0}$ does not affect the
empirical rejection frequencies in any obvious way (the case $\beta _{0}=1$
is marginally worse than the other ones for small $N$, but this vanishes as $%
N$ increases). \newline

%\bigskip

\subsection{Consistency versus multiple breaks\label{mbreaks}}

As a complement to the results in Section \ref{consistency}, we briefly
discuss the power of our tests against the alternative of $R$ changes: 
\begin{equation*}
y_{i}=\left\{ 
\begin{array}{ll}
(\beta _{1}+\epsilon _{i,1})y_{i-1}+\epsilon _{i,2},\quad \mbox{if}%
\;\;\;1\leq k\leq k_{1},\vspace{0.3cm} &  \\ 
(\beta _{2}+\epsilon _{i,1})y_{i-1}+\epsilon _{i,2},\quad \mbox{if}%
\;\;\;k_{1}<k\leq k_{2},\vspace{0.3cm} &  \\ 
\vdots \vspace{0.3cm} &  \\ 
(\beta _{R+1}+\epsilon _{i,1})y_{i-1}+\epsilon _{i,2},\quad \mbox{if}%
\;\;\;k_{R}<k\leq k_{R+1}, & 
\end{array}%
\right. ,
\end{equation*}%
with $k_{0}=0$ and $k_{R+1}=N$. For the sake of simplicity we require

\begin{assumption}
\label{rcaaltco} It holds that $k_{\ell }=\lfloor N\tau _{\ell }\rfloor \;$%
where$\;0<\tau _{1}<\tau _{2}<\ldots <\tau _{R}<1$, with $\tau _{0}=0$ and $%
\tau _{R+1}=1$.
\end{assumption}

Extending the theory developed above, it can be shown that, for all $1\leq
\ell \leq R$ and $1\leq i\leq N$ whereby $-\infty \leq E\ln |\epsilon _{\ell
}+\epsilon _{i,1}|<\infty $, the following limits exist: 
\begin{equation}
\frac{1}{N}\sum_{i=k_{\ell -1}+1}^{k_{\ell }}\frac{y_{i-1}^{2}}{1+y_{i-1}^{2}%
}\;\overset{P}{\rightarrow }\;\mathfrak{h}_{\ell }>0,\quad 1\leq \ell \leq
R+1.  \label{rcaaltco1}
\end{equation}%
Elementary arguments give that \eqref{rcaaltco1} implies 
\begin{equation*}
\widehat{\beta }_{k_{\ell },1}\;\overset{P}{\rightarrow }\;\beta _{\ell
,1}=\sum_{i=1}^{\ell }\beta _{i}\mathfrak{h}_{i}\left( \sum_{i=1}^{\ell }%
\mathfrak{h}_{i}\right) ^{-1},1\leq \ell \leq R+1
\end{equation*}%
and 
\begin{equation*}
\widehat{\beta }_{k_{\ell },2}\;\overset{P}{\rightarrow }\;\beta _{\ell
,2}=\sum_{i=\ell +1}^{R+1}\beta _{i}\mathfrak{h}_{i}\left( \sum_{i=\ell
+1}^{R+1}\mathfrak{h}_{i}\right) ^{-1},1\leq \ell \leq R+1.
\end{equation*}%
Let 
\begin{equation*}
\Delta _{N}\left( R\right) =\max_{1\leq \ell \leq R}|{\beta }_{\ell ,1}-{%
\beta }_{\ell ,2}|.
\end{equation*}

\begin{corollary}
\label{power}We assume that $H_{0}$ of (\ref{rcanu}) holds. Under the same
Assumptions as Theorem \ref{amoc} and Assumption \ref{rcaaltco}, if, as $%
N\rightarrow \infty $%
\begin{equation*}
N^{1/2}\Delta _{N}\left( R\right) \rightarrow \infty ,
\end{equation*}%
then it holds that%
\begin{equation*}
\sup_{0<t<1}\frac{\left\vert \overline{Q}_{N}(t)\right\vert }{w(t)}\;\overset%
{P}{\rightarrow }\;\infty .
\end{equation*}%
If 
\begin{equation*}
\left( \frac{N}{\ln \ln N}\right) ^{1/2}\Delta _{N}\left( R\right)
\rightarrow \;\infty ,
\end{equation*}%
then 
\begin{equation*}
a(\ln N)\sup_{0<t<1}\frac{\left\vert \overline{Q}_{N}(t)\right\vert }{%
\widehat{\mathfrak{g}}^{1/2}\left( t,t\right) }-b(\ln N)\overset{P}{%
\rightarrow }\infty .
\end{equation*}%
If%
\begin{equation*}
\left( \frac{r_{N}}{N}\right) ^{\kappa -1/2}N^{1/2}\Delta _{N}\left(
R\right) \left( \frac{k^{\ast }}{N}\left( \frac{N-k^{\ast }}{N}\right)
\right) ^{1-\kappa }\rightarrow \infty ,
\end{equation*}%
then%
\begin{equation*}
\left( \frac{r_{N}}{N}\right) ^{\kappa -1/2}\sup_{t_{1}<t<t_{2}}\frac{%
(t(1-t))^{-\kappa +1/2}}{\widehat{\mathfrak{g}}^{1/2}\left( t,t\right) }%
\left\vert \overline{Q}_{N}(t)\right\vert \overset{P}{\rightarrow }\infty .
\end{equation*}
\end{corollary}

Corollary \ref{power} states that our test has power even in the presence of
many breaks, and that the consistency (or lack thereof) depends on the
largest break only, irrespective of the magnitude and number of all the
other ones. The proof follows from the same arguments as that of Theorem \ref%
{amoc}.

\clearpage\appendix\setcounter{section}{1} \setcounter{subsection}{-1} %
\setcounter{equation}{0} \setcounter{lemma}{0} \renewcommand{\thelemma}{B.%
\arabic{lemma}} \renewcommand{\theequation}{B.\arabic{equation}}

\section{Technical lemmas\label{lemmas}}

We begin by stating some preliminary facts.

If $H_{0}$ holds, then 
\begin{align}
\widehat{\beta }_{k,1}-\widehat{\beta }_{k,2} =&\left( \sum_{i=2}^{k}\frac{%
y_{i-1}^{2}}{1+y_{i-1}^{2}}\right) ^{-1}\left( \sum_{i=2}^{k}\frac{%
y_{i-1}^{2}\epsilon _{i,1}+y_{i-1}\epsilon _{i,2}}{1+y_{i-1}^{2}}\right)
\label{rcadeco1} \\
&-\left( \sum_{i=k+1}^{N}\frac{y_{i-1}^{2}}{1+y_{i-1}^{2}}\right)
^{-1}\left( \sum_{i=k+1}^{N}\frac{y_{i-1}^{2}\epsilon _{i,1}+y_{i-1}\epsilon
_{i,2}}{1+y_{i-1}^{2}}\right) .  \notag
\end{align}%
Under the null hypothesis the recursion is 
\begin{equation}
y_{i}=\rho _{i}y_{i-1}+\epsilon _{i,2},\quad 1\leq i<\infty ,  \label{rcadef}
\end{equation}%
where $\rho _{i}=\beta _{0}+\epsilon _{i,1}.$ We can solve the recursion in (%
\ref{rcadef}) explicitly 
\begin{equation*}
y_{i}=\sum_{\ell =1}^{i}\epsilon _{\ell ,2}\prod_{j=\ell +1}^{i}\rho
_{j}+y_{0}\prod_{j=1}^{i}\rho _{j}=\sum_{\ell =0}^{i-1}\epsilon _{i-\ell
,2}\prod_{j=1}^{\ell }\rho _{i-j+1}+y_{0}\prod_{j=1}^{i}\rho _{j}.
\end{equation*}%
If it holds that 
\begin{equation}
-\infty \leq E\ln |\rho _{0}|=-\overline{\kappa },\quad \mbox{where}\;\;%
\overline{\kappa }>0,  \label{rca-1}
\end{equation}%
then the unique anticipative solution of (\ref{rcasta}) is (see %
\citealp{aue2006})

\begin{equation}
\overline{y}_{i}=\sum_{\ell =0}^{\infty }\epsilon _{i-\ell
,2}\prod_{j=1}^{\ell }\rho _{i-j+1}.  \label{rcastsol}
\end{equation}%
We note that 
\begin{equation}
\overline{y}_{i}=\sum_{\ell =0}^{i-1}\epsilon _{i-\ell ,2}\prod_{j=1}^{\ell
}\rho _{i-j+1}+\overline{y}_{0}\prod_{j=1}^{i}\rho _{j}.  \label{rcastso2}
\end{equation}

We are now ready to state our technical lemmas. The first one states that we
can replace $y_{i}$ with $\bar{y}_{i}$ in the sums in (\ref{rcadeco1}), and
the difference will be small.

\begin{lemma}
\label{rcale1} If $H_{0}$ of \eqref{rcanu} and Assumption \ref{rcaas} are
satisfied, and (\ref{rca-1}) holds, then 
\begin{equation}
\sum_{i=2}^{\infty }\left\vert \frac{y_{i-1}^{2}\epsilon _{i,1}}{%
1+y_{i-1}^{2}}-\frac{\overline{y}_{i-1}^{2}\epsilon _{i,1}}{1+\overline{y}%
_{i-1}^{2}}\right\vert <\infty \quad \mbox{a.s.},  \label{rcale1-1}
\end{equation}%
\begin{equation}
\sum_{i=2}^{\infty }\left\vert \frac{y_{i-1}\epsilon _{i,2}}{1+y_{i-1}^{2}}-%
\frac{\overline{y}_{i-1}\epsilon _{i,2}}{1+\overline{y}_{i-1}^{2}}%
\right\vert <\infty \quad \mbox{a.s.},  \label{rcale1-2}
\end{equation}%
and 
\begin{equation}
\sum_{i=2}^{\infty }\left\vert \frac{y_{i-1}^{2}}{1+y_{i-1}^{2}}-\frac{%
\overline{y}_{i-1}^{2}}{1+\overline{y}_{i-1}^{2}}\right\vert <\infty \quad %
\mbox{a.s.}  \label{rcale1-3}
\end{equation}
\end{lemma}

\begin{proof}
Using (\ref{rca-1}), $E(\ln |\rho _{j}|+\bar{\kappa}/2)=-\bar{\kappa}/2<0$;
hence, by Lemma 2 in \citet{aue2006}, there are $\nu _{1}>0$ and $c_{1}<1$
such that $E|e^{\bar{\kappa}/2}\rho _{0}|^{\nu _{1}}=c_{1}<1$. Lemma 2 in %
\citet{aue2006} also entails that 
\begin{equation}
E|\overline{y}_{0}^{2}|^{\nu _{2}}<\infty \quad \mbox{and}\quad E|\overline{y%
}_{0}^{2}|^{\nu _{2}}\leq c_{2}  \label{lemma2aue}
\end{equation}%
with some $\nu _{2}>0$. Hence 
\begin{equation}
\prod_{j=1}^{i}|\rho _{j}|=\exp \left( \sum_{j=1}^{i}\ln |\rho _{j}|\right)
=e^{-i\overline{\kappa }/2}\exp \left( \sum_{j=1}^{i}(\ln |\rho _{j}|+%
\overline{\kappa }/2)\right)  \label{rcam1}
\end{equation}%
and 
\begin{equation}
E\left( \exp \left( \sum_{j=1}^{i}(\ln |\rho _{j}|+\overline{\kappa }%
/2)\right) \right) ^{\nu _{1}}\leq c_{2}\quad \mbox{for all}\;\;1\leq
i<\infty .  \label{rcam2}
\end{equation}%
Given that 
\begin{equation*}
\left\vert \frac{y_{i}^{2}}{1+y_{i}^{2}}-\frac{\overline{y}_{i}^{2}}{1+%
\overline{y}_{i}^{2}}\right\vert \leq 2|y_{i}-\overline{y}_{i}||y_{i}+%
\overline{y}_{i}|,
\end{equation*}%
we have 
\begin{equation*}
\sum_{i=2}^{\infty }|\epsilon _{i}|\left\vert \frac{y_{i}^{2}}{1+y_{i-1}^{2}}%
-\frac{\overline{y}_{i}^{2}}{1+\overline{y}_{i}^{2}}\right\vert \leq
2\sum_{i=2}^{\infty }|\epsilon _{i}||y_{i-1}+\overline{y}_{i-1}||y_{i-1}-%
\overline{y}_{i-1}|.
\end{equation*}%
We can assume that $0<\nu _{3}=\min (\nu _{1},\nu _{2})/3<1$ and we conclude 
\begin{align*}
E\left( \sum_{i=2}^{\infty }|\epsilon _{i}|\left\vert \frac{y_{i}^{2}}{%
1+y_{i-1}^{2}}-\frac{\overline{y}_{i}^{2}}{1+\overline{y}_{i-1}^{2}}%
\right\vert \right) ^{\nu _{3}}& \leq 2^{\nu _{3}}\sum_{i=2}^{\infty
}E\left( |\epsilon _{i}||y_{i-1}+\overline{y}_{i-1}||y_{i-1}-\overline{y}%
_{i-1}|\right) ^{\nu _{3}} \\
& \leq 2^{\nu _{3}}\sum_{i=2}^{\infty }\left( E|\epsilon _{i}|^{3\nu
_{3}}E|y_{i-1}+\overline{y}_{i-1}|^{3\nu _{3}}E|y_{i-1}-\overline{y}%
_{i-1}|^{3\nu _{3}}\right) ^{1/3} \\
& \leq c_{3}\sum_{i=2}^{\infty }e^{-3i\overline{\kappa }\nu _{3}/2},
\end{align*}%
completing the proof of (\ref{rcale1-1}). The same arguments give (\ref%
{rcale1-2}) and (\ref{rcale1-3}).
\end{proof}

Lemma \ref{rcale1} entails that we only need to work with the stationary
solution. Equation (\ref{rcastsol}) entails that there is a function $%
g\left( \cdot \right) :R^{\infty \times \infty }\rightarrow R$ such that 
\begin{equation}
\overline{y}_{i}=g(\mbox{\boldmath${\epsilon}$}_{i},\mbox{\boldmath${%
\epsilon}$}_{i-1},\ldots ),\quad \mbox{\boldmath${\epsilon}$}_{i}=(\epsilon
_{i,1},\epsilon _{i,2})^{\prime }.  \label{rcaberd1}
\end{equation}%
The representation in (\ref{rcaberd1}) means that $\{y_{i},-\infty <i<\infty
\}$ is a Bernoulli shift. Let $\{\mbox{\boldmath${\epsilon}$}_{i}^{\ast
},-\infty <i<\infty \}$ be independent copies of $\mbox{\boldmath${%
\epsilon}$}_{0}$ with $\mbox{\boldmath${\epsilon}$}_{i}^{\ast }\overset{%
\emph{D}}{=}\mbox{\boldmath${\epsilon}$}_{0}$, independent of $\{%
\mbox{\boldmath${\epsilon}$}_{i},-\infty <i<\infty \}$ and define the
construction 
\begin{equation}
y_{i,\ell }=g(\mbox{\boldmath${\epsilon}$}_{i},\mbox{\boldmath${\epsilon}$}%
_{i-1},\ldots ,\mbox{\boldmath${\epsilon}$}_{i-\ell -1},\mbox{\boldmath${%
\epsilon}$}_{i-\ell }^{\ast },\mbox{\boldmath${\epsilon}$}_{i-\ell -1}^{\ast
},\ldots ),\quad \ell \geq 1.  \label{shift}
\end{equation}

\begin{lemma}
\label{rcale2} If $H_{0}$ of \eqref{rcanu} and Assumption \ref{rcaas} are
satisfied, and (\ref{rca-1}) holds, then 
\begin{equation}
E\left\vert \frac{\epsilon _{i+1,1}\overline{y}_{i}^{2}}{1+\overline{y}%
_{i}^{2}}-\frac{\epsilon _{i+1,1}{y}_{i,\ell }^{2}}{1+{y}_{i,\ell }^{2}}%
\right\vert ^{4}\leq c\ell ^{-\alpha },  \label{rca-dr1}
\end{equation}%
\begin{equation}
E\left\vert \frac{\epsilon _{i+1,2}\overline{y}_{i}}{1+\overline{y}_{i}^{2}}-%
\frac{\epsilon _{i+1,2}{y}_{i,\ell }}{1+{y}_{i,\ell }^{2}}\right\vert
^{4}\leq c\ell ^{-\alpha },  \label{rca-dr2}
\end{equation}%
and 
\begin{equation}
E\left\vert \frac{\overline{y}_{i}^{2}}{1+\overline{y}_{i}^{2}}-\frac{{y}%
_{i,\ell }^{2}}{1+{y}_{i,\ell }^{2}}\right\vert ^{4}\leq c\ell ^{-\alpha },
\label{rca-dr3}
\end{equation}%
for all $\alpha >0$ with some $c>0$.
\end{lemma}

\begin{proof}
We begin by showing (\ref{rca-dr3}). Using (\ref{rcastso2}), we have $\bar{y}%
_{0}=z_{\ell }+u_{\ell }$, and $y_{0,\ell }=z_{\ell }+u_{\ell }^{\ast }$,
where 
\begin{align*}
z_{\ell } =&\sum_{j=0}^{\ell -1}\epsilon _{-j,2}\prod_{k=1}^{j}\rho _{-k+1},
\\
u_{\ell } =&\overline{y}_{-\ell }\prod_{j=0}^{\ell -1}\rho _{-j}, \\
u_{\ell }^{\ast } =&\overline{y}_{-\ell }^{\ast }\prod_{j=0}^{\ell -1}\rho
_{-j}^{\ast },
\end{align*}%
with 
\begin{equation*}
\overline{y}_{-\ell }^{\ast }=g(\mbox{\boldmath${\epsilon}$}_{-\ell }^{\ast
},\mbox{\boldmath${\epsilon}$}_{-\ell -1}^{\ast },\cdots )\;\;\;\mbox{and}%
\;\;\;\rho _{-j}^{\ast }=\beta _{0}+\epsilon _{-j,1}^{\ast }.
\end{equation*}%
Consider now the set $U_{\ell }=\left\{ u_{\ell }:|u_{\ell }|\leq \ell
^{-\beta },u_{\ell }^{\ast }:|u_{\ell }^{\ast }|\leq \ell ^{-\beta }\right\} 
$, with $\beta >0$. It holds that%
\begin{align*}
& E\left\vert \frac{\overline{y}_{i}^{2}}{1+\overline{y}_{i}^{2}}-\frac{{y}%
_{i,\ell }^{2}}{1+{y}_{i,\ell }^{2}}\right\vert ^{4} \\
=& E\left( \left\vert \frac{\overline{y}_{i}^{2}}{1+\overline{y}_{i}^{2}}-%
\frac{{y}_{i,\ell }^{2}}{1+{y}_{i,\ell }^{2}}\right\vert ^{4}|\left( u_{\ell
},u_{\ell }^{\ast }\right) \in U_{\ell }\right) P\left( \left( u_{\ell
},u_{\ell }^{\ast }\right) \in U_{\ell }\right) \\
& +E\left( \left\vert \frac{\overline{y}_{i}^{2}}{1+\overline{y}_{i}^{2}}-%
\frac{{y}_{i,\ell }^{2}}{1+{y}_{i,\ell }^{2}}\right\vert ^{4}|\left( u_{\ell
},u_{\ell }^{\ast }\right) \notin U_{\ell }\right) P\left( \left( u_{\ell
},u_{\ell }^{\ast }\right) \notin U_{\ell }\right) \\
\leq & E\left( \left\vert \frac{\overline{y}_{i}^{2}}{1+\overline{y}_{i}^{2}}%
-\frac{{y}_{i,\ell }^{2}}{1+{y}_{i,\ell }^{2}}\right\vert ^{4}|\left(
u_{\ell },u_{\ell }^{\ast }\right) \in U_{\ell }\right) +P\left( \left(
u_{\ell },u_{\ell }^{\ast }\right) \notin U_{\ell }\right) =I+II.
\end{align*}%
Consider $II$ first. Using (\ref{lemma2aue}), (\ref{rcam1}) and (\ref{rcam2}%
) we get via Markov's inequality, 
\begin{equation}
P\{|u_{\ell }|>\ell ^{-\beta }\}\leq c_{1}\ell ^{\beta \nu _{1}}E|u_{\ell
}|^{\nu _{1}}e^{-\ell \nu _{2}}\leq c_{2}e^{-\ell \nu _{3}},  \label{markov1}
\end{equation}%
with some constants $\nu _{1}>0,\nu _{2}>0,\nu _{3}>0$ and $c_{1}$, $%
c_{2}=c_{2}(\beta )$, and similarly it can be shown that 
\begin{equation}
P\{|u_{\ell }^{\ast }|>\ell ^{-\beta }\}\leq c_{2}e^{-\ell \nu _{3}}.
\label{markov11}
\end{equation}%
Combining (\ref{markov1}) and (\ref{markov11}), it follows that 
\begin{equation}
II\leq c_{2}e^{-\ell \nu _{3}}.  \label{markov12}
\end{equation}%
Turning to $I$, elementary calculations give 
\begin{equation*}
\left\vert \frac{\bar{y}_{0}^{2}}{1+\bar{y}_{0}^{2}}-\frac{{y}_{0,\ell }^{2}%
}{1+{y}_{0,\ell }^{2}}\right\vert \leq \frac{4|z_{\ell }|(|u_{\ell
}|+|u_{\ell }^{\ast }|)}{1+\bar{y}_{0}^{2}}+\frac{2(u_{\ell }^{2}+(u_{\ell
}^{\ast })^{2})}{1+\bar{y}_{0}^{2}}.
\end{equation*}%
In (\ref{shift}), we can assume $\ell $ is large enough that $\ell ^{-\beta
}<1/8$ and therefore 
\begin{equation*}
\frac{|z_{\ell }|}{1+(z_{\ell }+u_{\ell })^{2}}\leq c_{3}.
\end{equation*}%
Therefore, on the set $U_{\ell }$, we have%
\begin{equation}
\left\vert \frac{\bar{y}_{i}^{2}}{1+\bar{y}_{i}^{2}}-\frac{{y}_{i,\ell }^{2}%
}{1+{y}_{i,\ell }^{2}}\right\vert \leq c_{4}\ell ^{-\beta }.  \label{markov2}
\end{equation}%
Putting together (\ref{markov1}), (\ref{markov11}) and (\ref{markov2}), (\ref%
{rca-dr3}) follows immediately. Also, noting that $\epsilon _{1}$ and $(\bar{%
y}_{0},y_{0,\ell })$ are independent, (\ref{rca-dr1}) also obtains (note
that the equation does not depend on $i$). Similar arguments give (\ref%
{rca-dr2}).
\end{proof}

%\clearpage
\appendix\setcounter{section}{2} \setcounter{subsection}{2} %
\setcounter{equation}{0} \setcounter{lemma}{0} \renewcommand{\thelemma}{C.%
\arabic{lemma}} \renewcommand{\theequation}{C.\arabic{equation}}

\section{Proofs\label{proofs}}

\textit{Proof of Theorem \ref{thrca1}.} Let 
\begin{equation*}
a_{0}=E\frac{\overline{y}_{0}^{2}}{1+\overline{y}_{0}^{2}}.
\end{equation*}%
Combining equation (\ref{rca-dr3}) with the approximations in \citet{aue2014}
and \citet{berkesliuwu}, we can define a sequence of Wiener processes $%
\{W_{N,1}(x),x\geq 0\}$ such that 
\begin{equation}
\sup_{1\leq k\leq N}\frac{1}{k^{\zeta _{1}}}\left\vert \sum_{i=1}^{k}\left( 
\frac{\overline{y}_{i}^{2}}{1+\overline{y}_{i}^{2}}-a_{0}\right)
-c_{1}W_{N,1}\left( k\right) \right\vert =O_{P}(1),  \label{rca-ap1}
\end{equation}%
with some $c_{1}>0$ and $\zeta _{1}<1/2$. It follows from (\ref{rca-ap1})
and the Law of the Iterated Logarithm (LIL henceforth) that 
\begin{equation}
\max_{1\leq k\leq N}\frac{1}{k^{\zeta _{2}}}\left\vert \sum_{i=1}^{k}\left( 
\frac{\overline{y}_{i}^{2}}{1+\overline{y}_{i}^{2}}-a_{0}\right) \right\vert
=O_{P}(1),  \label{raclil}
\end{equation}%
for all $\zeta _{2}<1/2$. Therefore, using Taylor's expansion, we obtain
from (\ref{raclil}) that

\begin{equation*}
\max_{1\leq k\leq N}{k^{\zeta _{2}}}\left\vert \left( \frac{1}{k}%
\sum_{i=1}^{k}\frac{\overline{y}_{i}^{2}}{1+\overline{y}_{i}^{2}}\right)
^{-1}-\frac{1}{a_{0}}\right\vert =O_{P}(1).
\end{equation*}%
Since the proofs of the approximations in \citet{aue2014} and %
\citet{berkesliuwu} are based on the blocking technique, we can define
Wiener processes $\{W_{N,2}\left( x\right) ,x\geq 0\}$, independent of $%
W_{N,1}\left( x\right) $, such that 
\begin{equation}
\sup_{1\leq k<N}\frac{1}{(N-k)^{\zeta _{1}}}\left\vert
\sum_{i=k+1}^{N}\left( \frac{\overline{y}_{i}^{2}}{1+\overline{y}_{i}^{2}}%
-a_{0}\right) -c_{1}W_{N,2}\left( N-k\right) \right\vert =O_{P}(1)
\label{rca-ap2}
\end{equation}%
which implies 
\begin{equation*}
\max_{1\leq k<N}{(N-k)^{\zeta _{2}}}\left\vert \left( \frac{1}{N-k}%
\sum_{i=k+1}^{N}\frac{\overline{y}_{i}^{2}}{1+\overline{y}_{i}^{2}}\right)
^{-1}-\frac{1}{a_{0}}\right\vert =O_{P}(1).
\end{equation*}%
Using the decomposability of the Bernoulli shifts in \eqref{rca-dr2} and %
\eqref{rca-dr3} and the approximations of \citet{aue2014} and %
\citet{berkesliuwu}, we can define two independent Wiener processes $%
\{W_{N,3}(x),0\leq x\leq N/2\}$ and $\{W_{N,4}(x),0\leq x\leq N/2\}$ such
that 
\begin{equation}
\max_{1\leq x\leq N/2}\frac{1}{x^{\zeta _{3}}}\Biggl|\sum_{i=2}^{\lfloor
x\rfloor }\frac{\overline{y}_{i-1}^{2}\epsilon _{i,1}+\overline{y}%
_{i-1}\epsilon _{i,2}}{1+\overline{y}_{i-1}^{2}}-a_{0}\eta W_{N,3}(x)\Biggl|%
=O_{P}(1),  \label{kmt-1}
\end{equation}%
and 
\begin{equation}
\max_{N/2\leq x\leq N-1}\frac{1}{(N-x)^{\zeta _{3}}}\Biggl|\sum_{i=\lfloor
x\rfloor +1}^{N}\frac{\overline{y}_{i-1}^{2}\epsilon _{i,1}+\overline{y}%
_{i-1}\epsilon _{i,2}}{1+\overline{y}_{i-1}^{2}}-a_{0}\eta W_{N,4}(N-x)%
\Biggl|=O_{P}(1),  \label{kmt-2}
\end{equation}%
with some $\zeta _{3}<1/2$. Putting together we get 
\begin{equation*}
\max_{1\leq k\leq N}\frac{1}{k^{\zeta _{4}}}\left\vert \left( \frac{1}{k}%
\sum_{i=2}^{k}\frac{\overline{y}_{i-1}^{2}}{1+\overline{y}_{i-1}^{2}}\right)
^{-1}-\frac{1}{a_{0}}\right\vert \left\vert \sum_{i=2}^{k}\frac{\overline{y}%
_{i-1}^{2}\epsilon _{i,1}+\overline{y}_{i-1}\epsilon _{i,2}}{1+\overline{y}%
_{i-1}^{2}}\right\vert =O_{P}(1),
\end{equation*}%
and 
\begin{equation*}
\max_{1\leq k\leq N-1}\frac{1}{(N-k)^{\zeta _{4}}}\left\vert \left( \frac{1}{%
N-k}\sum_{i=k+1}^{N}\frac{\overline{y}_{i-1}^{2}}{1+\overline{y}_{i-1}^{2}}%
\right) ^{-1}-\frac{1}{a_{0}}\right\vert \left\vert \sum_{i=k+1}^{N}\frac{%
\overline{y}_{i-1}^{2}\epsilon _{i,1}+\overline{y}_{i-1}\epsilon _{i,2}}{1+%
\overline{y}_{i-1}^{2}}\right\vert =O_{P}(1),
\end{equation*}%
for all $\zeta _{4}>0$. It follows from (\ref{rcadeco1}), (\ref{kmt-1}) and (%
\ref{kmt-2}) 
\begin{align}
\max_{2\leq x\leq N/2}\frac{1}{x^{\zeta _{5}}}& \Biggl|\frac{x(N-x)}{N}%
\left( \widehat{\beta }_{\lfloor x\rfloor ,1}-\widehat{\beta }_{\lfloor
x\rfloor ,2}\right)  \label{rcabo1} \\
& -\eta \left( W_{N,3}(x)-\frac{x}{N}\left[ W_{N,4}(N/2)+W_{N,3}(N/2)\right]
\right) \Biggl|=O_{P}(1),  \notag
\end{align}%
and 
\begin{align}
\max_{N/2\leq x\leq N-1}& \frac{1}{(N-x)^{\zeta _{5}}}\Biggl|\frac{x(N-x)}{N}%
\left( \widehat{\beta }_{\lfloor x\rfloor ,1}-\widehat{\beta }_{\lfloor
x\rfloor ,2}\right)  \label{rcabo2} \\
& -\eta \left( -W_{N,4}(N-x)+\frac{N-x}{N}\left[ W_{N,4}(N/2)+W_{N,3}(N/2)%
\right] \right) \Biggl|=O_{P}(1),  \notag
\end{align}%
with some $\zeta _{5}<1/2$. The computation of the covariance function shows
that 
\begin{equation*}
B_{N}(t)=\left\{ 
\begin{array}{l}
\displaystyle N^{-1/2}\left( W_{N,3}(Nt)-t\left[ W_{N,4}(N/2)+W_{N,3}(N/2)%
\right] \right) ,\;\;\;0\leq t\leq 1/2, \\ 
\displaystyle N^{-1/2}\left( -W_{N,4}(N(1-t))+(1-t)\left[
W_{N,4}(N/2)+W_{N,3}(N/2)\right] \right) ,\text{ \ \ }1/2\leq t\leq 1,%
\end{array}%
\right.
\end{equation*}%
is a Brownian bridge. \newline
Let $0<\delta <1/2$. Recalling Assumption \ref{as-wc-1}\textit{(i)}, it
follows from (\ref{rcabo1}) and (\ref{rcabo2}) that 
\begin{equation*}
\sup_{\delta \leq t\leq 1-\delta }\frac{\left\vert Q_{N}(t)-\eta
B_{N}(t)\right\vert }{w(t)}\leq c_{0}\sup_{\delta \leq t\leq 1-\delta
}\left\vert Q_{N}(t)-\eta B_{N}(t)\right\vert =O_{P}\left( N^{-1/2+\zeta
_{5}}\right) =o_{P}\left( 1\right) .
\end{equation*}%
It follows from (\ref{rcabo1}) that 
\begin{align*}
\sup_{2/(N+1)\leq t\leq \delta }\frac{\left\vert Q_{N}(t)-\eta
B_{N}(t)\right\vert }{w(t)}& \leq \sup_{0<t\leq \delta }\frac{t^{1/2}}{w(t)}%
\sup_{2/(N+1)\leq t\leq \delta }\frac{\left\vert Q_{N}(t)-\eta
B_{N}(t)\right\vert }{t^{1/2}} \\
& =\sup_{0<t\leq \delta }\frac{t^{1/2}}{w(t)}\sup_{2/(N+1)\leq t\leq \delta }%
\frac{(Nt)^{\zeta _{5}}}{t^{1/2}}\frac{N^{-1/2}}{(Nt)^{\zeta _{5}}}%
|Q_{N}(t)-B_{N}(t)| \\
& =O_{P}(1)\sup_{0<t\leq \delta }\frac{t^{1/2}}{w(t)},
\end{align*}%
and (\ref{rcabo2}) implies 
\begin{equation*}
\sup_{1-\delta \leq t\leq 1-2/(N+1)}\frac{\left\vert
Q_{N}(t)-B_{N}(t)\right\vert }{w(t)}=O_{P}(1)\sup_{1-\delta \leq t<1}\frac{%
(1-t)^{1/2}}{w(t)}.
\end{equation*}%
\citet{csorgo1986} proved that, if $I(w,c)<\infty $ for some $c>0$, then 
\begin{equation}
\lim_{t\rightarrow 0^{+}}\frac{t^{1/2}}{w(t)}=0\quad \mbox{and}\quad
\lim_{t\rightarrow 1^{-}}\frac{(1-t)^{1/2}}{w(t)}=0.  \label{CCHM}
\end{equation}%
Hence for every $x>0$ we have 
\begin{equation*}
\lim_{\delta \rightarrow 0^{+}}\limsup_{N\rightarrow \infty }P\left\{
\sup_{2/(N+1)\leq t\leq \delta }\frac{\left\vert
Q_{N}(t)-B_{N}(t)\right\vert }{w(t)}\geq x\right\} =0,
\end{equation*}%
and 
\begin{equation*}
\lim_{\delta \rightarrow 0^{+}}\limsup_{N\rightarrow \infty }P\left\{
\sup_{1-\delta \leq t\leq 1-2/(N+1)}\frac{\left\vert
Q_{N}(t)-B_{N}(t)\right\vert }{w(t)}\geq x\right\} =0.
\end{equation*}%
Using again \citet{csorgo1986} we obtain for all $N$ 
\begin{equation*}
\sup_{\delta \leq t\leq 1-\delta }\frac{|B_{N}(t)|}{w(t)}\overset{{\mathcal{D%
}}}{=}\sup_{\delta \leq t\leq 1-\delta }\frac{|B(t)|}{w(t)}\overset{a.s.}{%
\rightarrow }\sup_{0<t<1}\frac{|B(t)|}{w(t)},
\end{equation*}%
as $\delta \rightarrow 0$, where $\{B(t),0\leq t\leq 1\}$ is a Brownian
bridge. The first part of Theorem \ref{thrca1} now follows from putting
everything together.\newline
We now turn to proving part \textit{(ii)} of the theorem. Let $c(N)=(\ln
N)^{4}$. First we observe that according to \eqref{rcabo1} 
\begin{align*}
\max_{2\leq k\leq N/2}& \frac{N^{1/2}}{(k(N-k))^{1/2}}\Biggl|\frac{k(N-k)}{N}%
\left( \widehat{\beta }_{k,1}-\widehat{\beta }_{k,2}\right) -\eta \left(
W_{N,3}(k)-\frac{k}{N}\left[ W_{N,4}(N/2)+W_{N,3}(N/2)\right] \right) \Biggl|
\\
& \leq 2\max_{2\leq k\leq N/2}k^{\zeta _{5}-1/2}\frac{1}{k^{\zeta _{5}}}%
\Biggl|\frac{k(N-k)}{N}(\hat{\beta}_{k,1}-\hat{\beta}_{k,2})-\eta \left(
W_{N,3}(k)-\frac{k}{N}\left[ W_{N,4}(N/2)+W_{N,3}(N/2)\right] \right) \Biggl|
\\
& =O_{P}(1),
\end{align*}%
and by (\ref{rcabo2}) 
\begin{align*}
&\max_{N/2\leq k\leq N-1}\frac{N^{1/2}}{(k(N-k))^{1/2}}\left\vert \frac{%
k(N-k)}{N}\left( \widehat{\beta }_{k,1}-\widehat{\beta }_{k,2}\right) \right.
\\
&-\left. \eta \left( -W_{N,4}(N-k)+\frac{N-k}{N}\left[
W_{N,4}(N/2)+W_{N,3}(N/2)\right] \right) \right\vert \\
=&O_{P}(1).
\end{align*}%
Hence the LIL for the Wiener process implies 
\begin{equation*}
\frac{1}{(2\ln \ln N)^{1/2}}\max_{2\leq k\leq N/2}\frac{N^{1/2}}{%
(k(N-k))^{1/2}}\Biggl|\frac{k(N-k)}{N}\left( \widehat{\beta }_{k,1}-\widehat{%
\beta }_{k,2}\right) \Biggl|\overset{P}{\rightarrow }\eta ,
\end{equation*}

\begin{equation*}
\frac{1}{(2\ln \ln N)^{1/2}}\max_{N/2\leq k\leq N-1}\frac{N^{1/2}}{%
(k(N-k))^{1/2}}\Biggl|\frac{k(N-k)}{N}\left( \widehat{\beta }_{k,1}-\widehat{%
\beta }_{k,2}\right) \Biggl|\overset{P}{\rightarrow }\eta ,
\end{equation*}

\begin{equation*}
\max_{2\leq k\leq c(N)}\frac{N^{1/2}}{(k(N-k))^{1/2}}\Biggl|\frac{k(N-k)}{N}%
\left( \widehat{\beta }_{k,1}-\widehat{\beta }_{k,2}\right) \Biggl|%
=O_{P}((\ln \ln \ln N)^{1/2}),
\end{equation*}%
and 
\begin{equation*}
\max_{N-c(N)\leq k\leq N-1}\frac{N^{1/2}}{(k(N-k))^{1/2}}\Biggl|\frac{k(N-k)%
}{N}\left( \widehat{\beta }_{k,1}-\widehat{\beta }_{k,2}\right) \Biggl|%
=O_{P}((\ln \ln \ln N)^{1/2}).
\end{equation*}%
Putting all these results together, we conclude that 
\begin{align}
\lim_{N\rightarrow \infty }P\Biggl\{\max_{2\leq k\leq N-1}& \frac{N^{1/2}}{%
(k(N-k))^{1/2}}\Biggl|\frac{k(N-k)}{N}\left( \widehat{\beta }_{k,1}-\widehat{%
\beta }_{k,2}\right) \Biggl|  \label{rcabb} \\
& =\max_{c(N)\leq k\leq N-c(N)}\frac{N^{1/2}}{(k(N-k))^{1/2}}\Biggl|\frac{%
k(N-k)}{N}\left( \widehat{\beta }_{k,1}-\widehat{\beta }_{k,2}\right) \Biggl|%
\Biggl\}=1.  \notag
\end{align}%
Futher, the approximations in (\ref{rcabo1}) and (\ref{rcabo2}) yield 
\begin{align}
\max_{c(N)\leq k\leq N/2}& \frac{N^{1/2}}{(k(N-k))^{1/2}}\Biggl|\frac{k(N-k)%
}{N}\left( \widehat{\beta }_{k,1}-\widehat{\beta }_{k,2}\right)
\label{rcabb1} \\
& \hspace{1cm}-\eta \left( W_{N,3}(k)-\frac{k}{N}\left[
W_{N,4}(N/2)+W_{N,3}(N/2)\right] \right) \Biggl|  \notag \\
& =O_{P}\left( (\ln N)^{4(\zeta _{5}-1/2)}\right) =o_{P}\left( 1\right) , 
\notag
\end{align}%
and 
\begin{align}
\max_{N/2\leq k\leq 1-c(N)}& \frac{N^{1/2}}{(k(N-k))^{1/2}}\Biggl|\frac{%
k(N-k)}{N}\left( \widehat{\beta }_{k,1}-\widehat{\beta }_{k,2}\right)
\label{rcabb2} \\
& \hspace{1cm}-\eta \left( -W_{N,4}(N-k)+\frac{N-k}{N}\left[
W_{N,4}(N/2)+W_{N,3}(N/2)\right] \right) \Biggl|  \notag \\
& =O_{P}\left( (\ln N)^{4(\zeta _{5}-1/2)}\right) =o_{P}\left( 1\right) , 
\notag
\end{align}%
since $\zeta _{5}<1/2$. Finally, Theorem A.4.2 in Cs\"{o}rg\H{o} and Horv%
\'{a}th (1997) states that 
\begin{equation*}
\lim_{N\rightarrow \infty }P\left\{ a\left( \ln N\right) \max_{c(N)\leq
k\leq N-c(N)}\left( \frac{k}{N}\left( 1-\frac{k}{N}\right) \right)
^{-1/2}|B(k/N)|\leq x+b\left( \ln N\right) \right\} =\exp (-2e^{-x}),
\end{equation*}%
for all $x$, which implies the second part of the theorem. Finally, the
proof of part \textit{(iii)} follows automatically from the approximations
to Wiener processes in \eqref{rcabo1}, \eqref{rcabo2} and \eqref{kmt1} and %
\eqref{kmt2} - see \citet{horvathmiller}. \qed

\medskip \textit{Proof of Theorem \ref{thrca111}.} We begin by showing that
the approximations in (\ref{rcabo1}) and (\ref{rcabo2}) hold in the non
stationary case too. These approximations imply immediately the limit
results in the present theorem, repeating exactly the same passages as in
the proof of Theorem \ref{thrca1}\newline
We begin by noting that Lemma A.4 of \citet{HT2016} implies that there are
two constants, $0<\delta <1$ and $c_{1}>0$, such that 
\begin{equation}
P\{\left\vert y_{i}\right\vert \leq i^{\delta }\}\leq c_{1}i^{-\delta }.
\label{ht1}
\end{equation}%
Equation (\ref{ht1}), in turn, implies that 
\begin{align}
E\left( \frac{1}{1+y_{i}^{2}}\right) =&E\left( \frac{1}{1+y_{i}^{2}}%
I\{\left\vert y_{i}\right\vert \leq i^{\delta }\}\right) +E\left( \frac{1}{%
1+y_{i}^{2}}I\{\left\vert y_{i}\right\vert \geq i^{\delta }\}\right)
\label{ht2} \\
\leq &P\{\left\vert y_{i}\right\vert \leq i^{\delta }\}+(1+i^{2\delta
})^{-1}\leq c_{2}i^{-\delta },  \notag
\end{align}%
with some constant $c_{2}$. Using (\ref{ht2}) and Markov's inequality, we
have for all $x>0$ and $\zeta _{1}>1-\delta $ 
\begin{align*}
P\left\{ \max_{M\leq k<\infty }\frac{1}{k^{\zeta _{1}}}\sum_{i=1}^{k}\frac{1%
}{1+y_{i}^{2}}>x\right\} & \leq P\left\{ \max_{\ln M\leq \ell <\infty
}\max_{e^{\ell }\leq k<e^{\ell +1}}\frac{1}{k^{\zeta _{1}}}\sum_{i=1}^{k}%
\frac{1}{1+y_{i}^{2}}>x\right\} \\
& \leq \sum_{\ell =\ln M}^{\infty }P\left\{ \max_{e^{\ell }\leq k<e^{\ell
+1}}\frac{1}{k^{\zeta _{1}}}\sum_{i=1}^{k}\frac{1}{1+y_{i}^{2}}>x\right\} \\
& \leq \sum_{\ell =\ln M}^{\infty }P\left\{ \max_{e^{\ell }\leq k<e^{\ell
+1}}\frac{1}{k^{\zeta _{1}}}\sum_{i=1}^{k}\frac{1}{1+y_{i}^{2}}>x\right\} \\
& \leq \sum_{\ell =\ln M}^{\infty }P\left\{ \max_{e^{\ell }\leq k<e^{\ell
+1}}\sum_{i=1}^{k}\frac{1}{1+y_{i}^{2}}>xe^{{\zeta _{1}}\ell }\right\} \\
& =\sum_{\ell =\ln M}^{\infty }P\left\{ \sum_{i=1}^{\exp (\ell +1)}\frac{1}{%
1+y_{i}^{2}}>xe^{{\zeta _{1}}\ell }\right\} \\
& \leq \frac{1}{x}\sum_{\ell =\ln M}^{\infty }e^{-{\zeta _{1}}\ell
}\sum_{i=1}^{\exp (\ell +1)}E\frac{1}{1+y_{i}^{2}} \\
& \leq \frac{c_{2}}{x}\sum_{\ell =\ln M}^{\infty }e^{-{\zeta _{1}}\ell
}\sum_{i=1}^{\exp (\ell +1)}i^{-\delta } \\
& \leq \frac{c_{3}}{x}M^{-({\zeta _{1}}-(1-\delta ))}.
\end{align*}%
Hence there is ${\zeta _{2}}<1$ such that for all $x>0$ 
\begin{equation}
\lim_{M\rightarrow \infty }P\left\{ \max_{M\leq k<\infty }\frac{1}{k^{\zeta
_{2}}}\sum_{i=1}^{k}\frac{1}{1+y_{i}^{2}}>x\right\} =0.  \label{ht3}
\end{equation}%
Similar arguments give 
\begin{equation}
\lim_{M\rightarrow \infty }\limsup_{N\rightarrow \infty }P\left\{
\max_{1\leq k\leq N-M}\frac{1}{(N-k)^{\zeta _{2}}}\sum_{i=k+1}^{N}\frac{1}{%
1+y_{i}^{2}}>x\right\} =0.  \label{ht4}
\end{equation}%
We obtain immediately from (\ref{ht3}) and (\ref{ht4}) 
\begin{equation}
\max_{1\leq k\leq N}k^{1-{\zeta _{2}}}\left\vert \left( \frac{1}{k}%
\sum_{i=2}^{k}\frac{y_{i-1}^{2}}{1+y_{i-1}^{2}}\right) ^{-1}-1\right\vert
=O_{P}(1),  \label{ht-a}
\end{equation}%
and 
\begin{equation}
\max_{1\leq k\leq N-1}(N-k)^{1-{\zeta _{2}}}\left\vert \left( \frac{1}{N-k}%
\sum_{i=k+1}^{N}\frac{y_{i-1}^{2}}{1+y_{i-1}^{2}}\right) ^{-1}-1\right\vert
=O_{P}(1).  \label{ht-b}
\end{equation}%
Further, using Assumption \ref{rcaas} we get via the Cauchy--Schwartz
inequality 
\begin{align*}
E& \left( \sum_{\ell =i}^{j}\frac{\epsilon _{\ell ,1}}{1+y_{\ell -1}^{2}}%
\right) ^{4} \\
& =\sum_{\ell =i}^{j}E\epsilon _{\ell ,1}^{4}E\left( \frac{1}{1+y_{\ell
-1}^{2}}\right) ^{4}+\sum_{i\leq \ell \neq \ell ^{\prime }\leq j}E\left[
\left( \frac{\epsilon _{\ell ,1}}{1+y_{\ell -1}^{2}}\right) ^{2}\left( \frac{%
\epsilon _{\ell ^{\prime },1}}{1+y_{\ell ^{\prime }-1}^{2}}\right) ^{2}%
\right] \\
& \leq \sum_{\ell =i}^{j}E\epsilon _{0,1}^{4}E\left( \frac{1}{1+y_{\ell
-1}^{2}}\right) ^{4}+\sum_{i\leq \ell \neq \ell ^{\prime }\leq j}E\epsilon
_{0,1}^{4}\left[ E\left( \frac{1}{1+y_{\ell -1}}\right) ^{4}\right] ^{1/2}%
\left[ E\left( \frac{\epsilon _{\ell ^{\prime },1}}{1+y_{\ell ^{\prime
}-1}^{2}}\right) ^{4}\right] ^{1/2} \\
& \leq E\epsilon _{0,1}^{4}\left[ \left( \sum_{\ell =i}^{j}c_{\ell
-1}^{1/2}\right) ^{2}+\left( \sum_{\ell =i}^{j}c_{\ell -1}\right) ^{2}\right]
\\
& \leq 2E\epsilon _{0,1}^{4}\left( \sum_{\ell =i}^{j}c_{\ell -1}\right) ^{2},
\end{align*}%
where 
\begin{equation*}
c_{\ell }=E\left( \frac{1}{(1+y_{\ell }^{2})^{4}}\right) .
\end{equation*}%
Note that we can assume - without loss of generality - that $c_{\ell }\leq 1$
since, along the lines of the proof of (\ref{ht2}), $c_{\ell }=O(\ell
^{-\delta })$ as $\ell \rightarrow \infty $. Theorem 3.1 of %
\citet{moricz1982} implies 
\begin{equation*}
E\left( \max_{2\leq k\leq j}\left\vert \sum_{i=2}^{k}\epsilon _{i,1}\left( 
\frac{y_{i-1}^{2}}{1+y_{i-1}^{2}}-1\right) \right\vert ^{4}\right) \leq
c_{4}\left( \sum_{\ell =1}^{j}c_{\ell -1}\right) ^{2}.
\end{equation*}%
Hence by Markov's inequality for all $x>0$ 
\begin{align*}
P\left\{ \max_{2\leq k\leq N}\frac{1}{k^{\zeta _{3}}}\left\vert
\sum_{i=2}^{k}\frac{\epsilon _{i,1}y_{i-1}^{2}}{1+y_{i-1}^{2}}%
-\sum_{i=2}^{k}\epsilon _{i,1}\right\vert >x\right\} & \leq P\left\{
\max_{0\leq \ell \leq \ln N}\max_{e^{\ell }\leq k\leq e^{\ell +1}}\frac{1}{%
k^{\zeta _{3}}}\left\vert \sum_{i=2}^{k}\frac{\epsilon _{i,1}}{1+y_{i-1}^{2}}%
\right\vert >x\right\} \\
& \leq \sum_{\ell =0}^{\ln N}P\left\{ \max_{e^{\ell }\leq k\leq e^{\ell +1}}%
\frac{1}{k^{\zeta _{3}}}\left\vert \sum_{i=2}^{k}\frac{\epsilon _{i,1}}{%
1+y_{i-1}^{2}}\right\vert >x\right\} \\
& \leq \sum_{\ell =0}^{\ln N}P\left\{ \max_{e^{\ell }\leq k\leq e^{\ell
+1}}\left\vert \sum_{i=2}^{k}\frac{\epsilon _{i,1}}{1+y_{i-1}^{2}}%
\right\vert >xe^{{\zeta _{3}}\ell }\right\} \\
& \leq \frac{1}{x^{4}}\sum_{\ell =0}^{\ln N}e^{-4{\zeta _{3}}\ell
}E\max_{1\leq k\leq e^{\ell +1}}\left\vert \sum_{i=2}^{k}\frac{\epsilon
_{i,1}}{1+y_{i-1}^{2}}\right\vert ^{4} \\
& \leq \frac{c_{5}}{x^{4}}\sum_{\ell =0}^{\ln N}e^{-4{\zeta _{3}}\ell
}e^{2\ell (1-\delta )} \\
& \leq \frac{c_{6}}{x^{4}},
\end{align*}%
with the choice of ${\zeta _{3}>}\left( 1-\delta \right) /2$. Hence, there
exists a ${\zeta _{4}}<1/2$ such that 
\begin{equation}
\max_{2\leq k\leq N}\frac{1}{k^{\zeta _{4}}}\left\vert \sum_{i=2}^{k}\frac{%
\epsilon _{i,1}y_{i-1}^{2}}{1+y_{i-1}^{2}}-\sum_{i=2}^{k}\epsilon
_{i,1}\right\vert =O_{P}(1);  \label{ht17}
\end{equation}%
similar arguments give 
\begin{equation}
\max_{2\leq k<N}\frac{1}{(N-k)^{\zeta _{4}}}\left\vert \sum_{i=k+1}^{N}\frac{%
\epsilon _{i,1}y_{i-1}^{2}}{1+y_{i-1}^{2}}-\sum_{i=2}^{k}\epsilon
_{i,1}\right\vert =O_{P}(1).  \label{ht18}
\end{equation}%
Following the proof of (\ref{ht17}) and (\ref{ht18}) one also can verify
that 
\begin{equation}
\max_{2\leq k\leq N}\frac{1}{k^{\zeta _{5}}}\left\vert \sum_{i=2}^{k}\frac{%
\epsilon _{i,1}y_{i-1}}{1+y_{i-1}^{2}}\right\vert =O_{P}(1),  \label{htb1}
\end{equation}%
and 
\begin{equation}
\max_{1\leq k<N}\frac{1}{(N-k)^{\zeta _{5}}}\left\vert \sum_{i=k+1}^{N}\frac{%
\epsilon _{i,1}y_{i-1}}{1+y_{i-1}^{2}}\right\vert =O_{P}(1),  \label{htb2}
\end{equation}%
with some ${\zeta _{5}}<1/2$.

Combining (\ref{ht-a}), (\ref{ht-b}), (\ref{ht17}), (\ref{ht18}), (\ref{htb1}%
) and (\ref{htb2}), it is easy to see that it is possible to find a $\zeta
_{6}<1/2$ such that%
\begin{equation}
\max_{1\leq k\leq N/2}\frac{1}{k^{\zeta _{6}}}\left\vert \frac{k\left(
N-k\right) }{N}\left( \widehat{\beta }_{k,1}-\widehat{\beta }_{k,2}\right)
-\left( \sum_{i=1}^{k}\epsilon _{i,1}-\frac{k}{N}\sum_{i=1}^{N}\epsilon
_{i,1}\right) \right\vert =O_{P}(1),  \label{app1}
\end{equation}%
and 
\begin{equation}
\max_{N/2\leq k\leq N-1}\frac{1}{(N-k)^{\zeta _{6}}}\left\vert \frac{k\left(
N-k\right) }{N}\left( \widehat{\beta }_{k,1}-\widehat{\beta }_{k,2}\right)
-\left( -\sum_{i=k+1}^{N}\epsilon _{i,1}+\frac{N-k}{N}\sum_{i=1}^{N}\epsilon
_{i,1}\right) \right\vert =O_{P}(1).  \label{app2}
\end{equation}%
Finally, by the Koml\'{o}s--Major--Tusn\'{a}dy approximation (see %
\citealp{KMT1}, and \citealp{KMT2}), we can define two independent Wiener
processes $\{W_{N,1}(x),0\leq x\leq N/2\}$ and $\{W_{N,2}(x),0\leq x\leq
N/2\}$ such that 
\begin{equation}
\max_{1\leq k\leq N/2}\frac{1}{k^{1/4}}\left\vert \sum_{i=1}^{k}\epsilon
_{i,1}-\sigma _{1}W_{N,1}(k)\right\vert =O_{P}(1),  \label{kmt1}
\end{equation}%
and 
\begin{equation}
\max_{1\leq k\leq N/2}\frac{1}{(N-k)^{1/4}}\left\vert
\sum_{i=k+1}^{N}\epsilon _{i,1}-\sigma _{1}W_{N,2}(N-k)\right\vert =O_{P}(1).
\label{kmt2}
\end{equation}%
Putting together (\ref{app1}), (\ref{app2}), (\ref{kmt1}) and (\ref{kmt2}),
it finally follows that, for some $\zeta _{7}<1/2$%
\begin{align}
& \max_{1\leq k\leq N/2}\frac{1}{k^{\zeta _{7}}}\left\vert \frac{k\left(
N-k\right) }{N}\left( \widehat{\beta }_{k,1}-\widehat{\beta }_{k,2}\right)
\right.  \label{app-a} \\
& \left. -\sigma _{1}\left( W_{N,1}(k)-\frac{k}{N}\left(
W_{N,2}(N/2)+W_{N,1}(N/2)\right) \right) \right\vert  \notag \\
& =O_{P}(1),  \notag
\end{align}%
and 
\begin{align}
& \max_{N/2\leq k\leq N}\frac{1}{k^{\zeta _{7}}}\left\vert \frac{k\left(
N-k\right) }{N}\left( \widehat{\beta }_{k,1}-\widehat{\beta }_{k,2}\right)
\right.  \label{app-b} \\
& \left. -\sigma _{1}\left( -W_{N,2}(k)+\frac{N-k}{N}\left(
W_{N,2}(N/2)+W_{N,1}(N/2)\right) \right) \right\vert  \notag \\
& =O_{P}(1).  \notag
\end{align}%
The desired results now follow from repeating exactly the same passages as
in the proof of Theorem \ref{thrca1}.

\textit{Proof of Corollary \ref{rcaco1}.} We note that 
\begin{align}
(y_{i}-\widehat{\beta }_{N,1}y_{i-1})^{2}& =\left( \left( \left( \beta _{0}-%
\widehat{\beta }_{N,1}\right) y_{i-1}+\epsilon _{i,1}\right)
y_{i-1}+\epsilon _{i,2}\right) ^{2}  \label{res-sq} \\
& =\epsilon _{i,1}^{2}y_{i-1}^{2}+\epsilon _{i,2}^{2}+\left( \beta _{0}-%
\widehat{\beta }_{N,1}\right) ^{2}y_{i-1}^{2}+2\left( \beta _{0}-\widehat{%
\beta }_{N,1}\right) y_{i-1}^{2}\epsilon _{i,1}  \notag \\
& \hspace{1cm}+2\left( \beta _{0}-\widehat{\beta }_{N,1}\right)
y_{i-1}\epsilon _{i,2}+2y_{i-1}\epsilon _{i,1}\epsilon _{i,2}.  \notag
\end{align}%
Hence, we have%
\begin{align*}
\widehat{a}_{N,1}& =\frac{1}{N-1}\sum_{i=2}^{N}\frac{\epsilon
_{i,1}^{2}y_{i-1}^{4}}{\left( 1+y_{i-1}^{2}\right) ^{2}}+\frac{1}{N-1}%
\sum_{i=2}^{N}\frac{\epsilon _{i,2}^{2}y_{i-1}^{2}}{\left(
1+y_{i-1}^{2}\right) ^{2}}+\left( \beta _{0}-\widehat{\beta }_{N,1}\right)
^{2}\frac{1}{N-1}\sum_{i=2}^{N}\frac{y_{i-1}^{4}}{\left(
1+y_{i-1}^{2}\right) ^{2}} \\
& +2\frac{1}{N-1}\left( \beta _{0}-\widehat{\beta }_{N,1}\right)
\sum_{i=2}^{N}\frac{\epsilon _{i,1}y_{i-1}^{3}}{\left( 1+y_{i-1}^{2}\right)
^{2}}+2\frac{1}{N-1}\left( \beta _{0}-\widehat{\beta }_{N,1}\right)
\sum_{i=2}^{N}\frac{\epsilon _{i,2}y_{i-1}^{3}}{\left( 1+y_{i-1}^{2}\right)
^{2}} \\
& +2\frac{1}{N-1}\sum_{i=2}^{N}\frac{y_{i-1}\epsilon _{i,1}\epsilon _{i,2}}{%
\left( 1+y_{i-1}^{2}\right) ^{2}}.
\end{align*}%
The proofs of Theorems \ref{thrca1} and \ref{thrca111} show that,
irrespective of $y_{i}$ being stationary or not, 
\begin{align*}
|\widehat{\beta }_{N}-\beta _{0}| =&O_{P}(N^{-1/2}), \\
\frac{1}{N}\sum_{i=2}^{N}\frac{y_{i-1}^{4}}{(1+y_{i-1}^{2})^{2}}
=&O_{P}(1),\quad \quad \frac{1}{N}\sum_{i=2}^{N}\frac{y_{i-1}^{4}\epsilon
_{i,1}}{(1+y_{i-1}^{2})^{2}}=O_{P}(N^{-1/2}), \\
\frac{1}{N}\sum_{i=2}^{N}\frac{y_{i-1}^{3}\epsilon _{i-2}}{%
(1+y_{i-1}^{2})^{2}} =&O_{P}(N^{-1/2}),\quad \quad \frac{1}{N}\sum_{i=2}^{N}%
\frac{y_{i-1}^{3}\epsilon _{i,1}\epsilon _{i-2}}{(1+y_{i-1}^{2})^{2}}%
=O_{P}(N^{-1/2}).
\end{align*}%
Therefore, it holds that%
\begin{equation}
\widehat{a}_{N,1}=\frac{1}{N-1}\sum_{i=2}^{N}\frac{\epsilon
_{i,1}^{2}y_{i-1}^{4}}{\left( 1+y_{i-1}^{2}\right) ^{2}}+\frac{1}{N-1}%
\sum_{i=2}^{N}\frac{\epsilon _{i,2}^{2}y_{i-1}^{2}}{\left(
1+y_{i-1}^{2}\right) ^{2}}+O_{P}(N^{-1/2}).  \label{an1c}
\end{equation}%
We study the leading term in (\ref{an1c}), showing that, when $E\ln |\beta
_{0}+\epsilon _{0,1}|<0$ 
\begin{equation}
\frac{1}{N-1}\sum_{i=2}^{N}\frac{\epsilon
_{i,1}^{2}y_{i-1}^{4}+y_{i-1}^{2}\epsilon _{i,2}^{2}}{(1+y_{i-1}^{2})^{2}}%
=E\left( \frac{\overline{y}_{0}^{2}}{1+\overline{y}_{0}^{2}}\right)
^{2}\sigma _{1}^{2}+E\left( \frac{\overline{y}_{0}^{2}}{1+\overline{y}%
_{0}^{2}}\right) ^{2}\sigma _{2}^{2}+O_{P}(N^{-\zeta }),  \label{an1c2}
\end{equation}%
for some $\zeta >0$. We begin by showing 
\begin{equation}
\frac{1}{N-1}\sum_{i=2}^{N}\frac{\epsilon _{i,1}^{2}y_{i-1}^{4}}{%
(1+y_{i-1}^{2})^{2}}=E\left( \frac{\overline{y}_{0}^{2}}{1+\overline{y}%
_{0}^{2}}\right) ^{2}\sigma _{1}^{2}+O_{P}(N^{-\zeta }).  \label{an1c1}
\end{equation}%
Note first that%
\begin{align*}
& \frac{1}{N-1}\sum_{i=2}^{N}E\left\vert \frac{\epsilon _{i,1}^{2}y_{i-1}^{4}%
}{(1+y_{i-1}^{2})^{2}}-\frac{\epsilon _{i,1}^{2}\overline{y}_{i-1}^{4}}{(1+%
\overline{y}_{i-1}^{2})^{2}}\right\vert \\
\leq & \frac{1}{N-1}\sum_{i=2}^{N}E\left\vert \epsilon _{i,1}^{2}\right\vert
\left( E\left\vert y_{i-1}-\overline{y}_{i-1}\right\vert ^{2}\right)
^{1/2}\left( E\left\vert y_{i-1}+\overline{y}_{i-1}\right\vert ^{2}\right)
^{1/2}=O\left( \frac{1}{N}\right) ,
\end{align*}%
which follows from the definition of $\overline{y}_{i}$ and Assumption \ref%
{rcaas}\textit{(ii)}. Recall that \citet{aue2006} show that $\theta
_{1}=E\left\vert \beta _{0}+\epsilon _{0,1}\right\vert ^{k}<1$, for some $%
0<k<1$, and consider the construction%
\begin{equation*}
\widehat{y}_{i}\left( N\right) =\sum_{j=i-\left\lfloor \theta \ln
N\right\rfloor }^{i}\epsilon _{j,2}\prod\limits_{z=j}^{i-1}\left( \beta
_{0}+\epsilon _{z,1}\right) +\sum_{j=-\infty }^{i-\left\lfloor \theta \ln
N\right\rfloor -1}\epsilon _{j,2}^{\ast }\prod\limits_{z=j}^{i-1}\left(
\beta _{0}+\epsilon _{z,1}^{\ast }\right) ,
\end{equation*}%
where $\epsilon _{j,1}^{\ast }$ and $\epsilon _{j,2}^{\ast }$ are completely
independent, and independent of $\epsilon _{j,1}$ and $\epsilon _{j,2}$,
with $\epsilon _{j,1}^{\ast }\overset{D}{=}\epsilon _{j,1}$, $\epsilon
_{j,2}^{\ast }\overset{D}{=}\epsilon _{j,2}$, and $\theta =2/\left( 1-\theta
_{1}\right) $. Following the proof of Lemma A.3 in \citet{HT2016}, it
follows that%
\begin{equation*}
E\left\vert \frac{\overline{y}_{i-1}^{4}}{(1+\overline{y}_{i-1}^{2})^{2}}-%
\frac{\widehat{y}_{i-1}^{4}\left( N\right) }{(1+\widehat{y}_{i-1}^{2}\left(
N\right) )^{2}}\right\vert \leq c_{0}\theta _{1}^{\left\lfloor \theta \ln
N\right\rfloor /\left( 1+k\right) },
\end{equation*}%
so that it is easy to see that%
\begin{equation*}
\frac{1}{N-1}\sum_{i=2}^{N}E\left\vert \frac{\epsilon _{i,1}^{2}\overline{y}%
_{i-1}^{4}}{(1+\overline{y}_{i-1}^{2})^{2}}-\frac{\epsilon _{i,1}^{2}%
\widehat{y}_{i-1}^{4}\left( N\right) }{(1+\widehat{y}_{i-1}^{2}\left(
N\right) )^{2}}\right\vert =o\left( \frac{1}{N}\right) .
\end{equation*}%
Finally, consider%
\begin{align*}
E&\left( \frac{1}{N-1}\sum_{i=2}^{N}\epsilon _{i,1}^{2}\left( \frac{\widehat{%
y}_{i-1}^{4}\left( N\right) }{(1+\widehat{y}_{i-1}^{2}\left( N\right) )^{2}}%
-E\left( \frac{\overline{y}_{0}^{2}}{1+\overline{y}_{0}^{2}}\right)
^{2}\right) \right) ^{2} \\
=& \frac{1}{\left( N-1\right) ^{2}}\sum_{i=2}^{N}\sum_{j=2}^{N}E\left[
\epsilon _{i,1}^{2}\epsilon _{j,1}^{2}\left( \frac{\widehat{y}%
_{i-1}^{4}\left( N\right) }{(1+\widehat{y}_{i-1}^{2}\left( N\right) )^{2}}%
-E\left( \frac{\overline{y}_{0}^{2}}{1+\overline{y}_{0}^{2}}\right)
^{2}\right) \left( \frac{\widehat{y}_{j-1}^{4}\left( N\right) }{(1+\widehat{y%
}_{j-1}^{2}\left( N\right) )^{2}}-E\left( \frac{\overline{y}_{0}^{2}}{1+%
\overline{y}_{0}^{2}}\right) ^{2}\right) \right] \\
\leq & \frac{1}{\left( N-1\right) ^{2}}\sum_{i=2}^{N}\sum_{j=2}^{N}E\left[
\epsilon _{i,1}^{2}\left( \frac{\widehat{y}_{i-1}^{4}\left( N\right) }{(1+%
\widehat{y}_{i-1}^{2}\left( N\right) )^{2}}-E\left( \frac{\overline{y}%
_{0}^{2}}{1+\overline{y}_{0}^{2}}\right) ^{2}\right) \epsilon
_{j,1}^{2}\left( \frac{\widehat{y}_{j-1}^{4}\left( N\right) }{(1+\widehat{y}%
_{j-1}^{2}\left( N\right) )^{2}}-E\left( \frac{\overline{y}_{0}^{2}}{1+%
\overline{y}_{0}^{2}}\right) ^{2}\right) \right] \\
& \times I\left( \left\vert i-j\right\vert \leq 2\theta \left\lfloor \ln
N\right\rfloor \right) \\
=& O\left( \frac{\ln N}{N}\right) .
\end{align*}%
having used the fact that $\epsilon _{i,1}^{2}\widehat{y}_{i-1}^{4}\left(
N\right) $ and $\epsilon _{j,1}^{2}\widehat{y}_{j-1}^{4}\left( N\right) $
are independent for $\left\vert i-j\right\vert >2\theta \left\lfloor \ln
N\right\rfloor $. Putting all together, it follows that 
\begin{equation*}
\frac{1}{N-1}\sum_{i=2}^{N}\epsilon _{i,1}^{2}\left( \frac{y_{i-1}^{4}}{%
(1+y_{i-1}^{2})^{2}}-E\left( \frac{\bar{y}_{0}^{2}}{1+\bar{y}_{0}^{2}}%
\right) ^{2}\right) =O_{P}(N^{-\zeta }),
\end{equation*}%
for some $\zeta >0$, whence (\ref{an1c1}) follows immediately. Using exactly
the same logic, it is also possible to show that 
\begin{equation}
\frac{1}{N-1}\sum_{i=2}^{N}\frac{\epsilon _{i,2}^{2}y_{i-1}^{2}}{%
(1+y_{i-1}^{2})^{2}}=E\left( \frac{\bar{y}_{0}}{1+\bar{y}_{0}^{2}}\right)
^{2}\sigma _{2}^{2}+O_{P}(N^{-\zeta }),  \label{b19}
\end{equation}%
which, with (\ref{an1c1}), finally implies (\ref{an1c2}). The same logic
also yields 
\begin{equation*}
\widehat{a}_{N,2}=E\left( \frac{\bar{y}_{0}^{2}}{1+\bar{y}_{0}^{2}}\right)
+O_{P}(N^{-\zeta }),
\end{equation*}%
whence we have finally shown that $\widehat{\eta }_{N}=\eta +O_{P}(N^{-\zeta
})$ when $y_{i}$ is stationary.

It remains to show that the corollary also holds in the nonstationary case $%
E\ln |\beta _{0}+\epsilon _{0,1}|\geq 0$. Note%
\begin{equation*}
\frac{1}{N-1}\sum_{i=2}^{N}\frac{\epsilon _{i,1}^{2}y_{i-1}^{4}}{%
(1+y_{i-1}^{2})^{2}}=\frac{1}{N-1}\sigma _{1}^{2}\sum_{i=2}^{N}\frac{%
y_{i-1}^{4}}{(1+y_{i-1}^{2})^{2}}+\frac{1}{N-1}\sum_{i=2}^{N}\left( \epsilon
_{i,1}^{2}-\sigma _{1}^{2}\right) \frac{y_{i-1}^{4}}{(1+y_{i-1}^{2})^{2}};
\end{equation*}%
following the proof of (\ref{rca-dr3}), it is easy to see that%
\begin{equation*}
\frac{1}{N-1}\sigma _{1}^{2}\sum_{i=2}^{N}\frac{y_{i-1}^{4}}{%
(1+y_{i-1}^{2})^{2}}=\sigma _{1}^{2}+O_{P}\left( \frac{1}{N}\right) ,
\end{equation*}%
and%
\begin{align*}
E&\left( \frac{1}{N-1}\sum_{i=2}^{N}\left( \epsilon _{i,1}^{2}-\sigma
_{1}^{2}\right) \frac{y_{i-1}^{4}}{(1+y_{i-1}^{2})^{2}}\right) ^{2} \\
=&\frac{1}{\left( N-1\right) ^{2}}\sum_{i=2}^{N}\left( E\left( \epsilon
_{i,1}^{2}-\sigma _{1}^{2}\right) ^{4}\right) \left( E\left( \frac{%
y_{i-1}^{4}}{(1+y_{i-1}^{2})^{2}}\right) ^{2}\right) \\
=&O\left( \frac{1}{N}\right) ,
\end{align*}%
so that%
\begin{equation*}
\frac{1}{N-1}\sum_{i=2}^{N}\frac{\epsilon _{i,1}^{2}y_{i-1}^{4}}{%
(1+y_{i-1}^{2})^{2}}=\sigma _{1}^{2}+O_{P}(N^{-1}).
\end{equation*}%
The same logic yields%
\begin{equation*}
\frac{1}{N-1}\sum_{i=2}^{N}\frac{\epsilon _{i,2}^{2}y_{i-1}^{2}}{%
(1+y_{i-1}^{2})^{2}}=O_{P}\left( \frac{1}{N}\right) .
\end{equation*}%
Hence%
\begin{equation*}
\widehat{a}_{N,1}=\sigma _{1}^{2}+O_{P}(N^{-1}),
\end{equation*}%
and, by the same token%
\begin{equation*}
\widehat{a}_{N,2}=1+O_{P}(N^{-1}),
\end{equation*}%
so that, in this case, it holds that $\widehat{\eta }_{N}=\eta
+O_{P}(N^{-1}) $. \qed

\textit{Proof of Theorem \ref{thrca4}.} We assume, without loss of
generality, that $M\geq 1$ - the $M=0$ case is already covered in Theorem %
\ref{thrca1}.

We begin by showing two preliminary sets of results: (i) that the
approximations developed in the proofs of Theorems \ref{thrca1} and \ref%
{thrca111} are valid on each segment $(m_{\ell -1},m_{\ell }],1\leq \ell
\leq M+1$, with only the variances of the approximating Gaussian processes
depending on $\ell $; and (ii) that the approximating process on a segment
is independent of the approximating processes on the other segments. \newline
Let 
\begin{equation*}
z_{i}=\left\{ 
\begin{array}{ll}
\displaystyle\frac{\bar{y}_{\ell ,i-1}^{2}\epsilon _{i,1}+\bar{y}_{\ell
,i-1}\epsilon _{i,2}}{1+\bar{y}_{\ell ,i-1}^{2}},\quad \mbox{if}%
\;\;\;-\infty \leq E\ln |\beta _{0}+\epsilon _{m_{\ell },1}|<0,\vspace{0.3cm}
&  \\ 
\epsilon _{i,1},\quad \mbox{if}\;\;\;E\ln |\beta _{0}+\epsilon _{m_{\ell
},1}|\geq 0, & 
\end{array}%
\right.
\end{equation*}%
if $m_{\ell -1}<i\leq m_{\ell },1\leq \ell \leq M+1$. We define the sums 
\begin{equation*}
S_{\ell }(j)=\sum_{i=m_{\ell -1}+1}^{m_{\ell -1}+j}z_{i},\quad \mbox{if}%
\;\;\;1\leq j\leq m_{\ell }-m_{\ell -1},\;1\leq \ell \leq M+1.
\end{equation*}%
Following the proofs of Lemma \ref{rcale1} (in the case $-\infty \leq \ln
|\beta _{0}+\epsilon _{m_{\ell },1}|<0$), and of Theorem \ref{thrca111} (in
the case $0\leq \ln |\beta _{0}+\epsilon _{m_{\ell },1}|<\infty $), it can
be shown that 
\begin{equation*}
\sum_{i=m_{\ell -1}+1}^{m_{\ell -1}+j}\left\vert \frac{y_{i-1}^{2}\epsilon
_{i,1}+y_{i-1}\epsilon _{i,2}}{1+y_{i-1}^{2}}-S_{\ell }(j)\right\vert
=O_{P}(1),\quad \mbox{if}\;\;\;1\leq j\leq m_{\ell }-m_{\ell -1},\text{ }%
1\leq \ell \leq M+1.
\end{equation*}%
This entails that we can replace the partial sums of $(y_{i-1}^{2}\epsilon
_{i,1}+y_{i-1}\epsilon _{i,2})/(1+y_{i-1}^{2})$ with the partial sums of the 
$z_{i}$'s. \newline
Using the approximations in \citet{aue2014} or \citet{berkesliuwu}, we can
define independent Wiener processes $\{W_{N,\ell ,1}(x),0\leq x\leq (m_{\ell
}-m_{\ell -1})/2\},\{W_{N,\ell ,2}(x),0\leq x\leq (m_{\ell }-m_{\ell
-1})/2\},1\leq \ell \leq M+1$ such that 
\begin{equation*}
\max_{1\leq k\leq (m_{\ell }-m_{\ell -1})/2}\frac{1}{k^{\zeta }}\left\vert
S_{\ell }(k)-\eta _{\ell }W_{N,\ell ,1}(k)\right\vert =O_{P}(1),
\end{equation*}%
and 
\begin{equation*}
\max_{(m_{\ell }-m_{\ell -1})/2<k<m_{\ell }-m_{\ell -1}}\frac{1}{(m_{\ell
}-k)^{\zeta }}\left\vert S_{\ell }(m_{\ell })-S_{\ell }(k)-\eta _{\ell
}W_{N,\ell ,2}(k)\right\vert =O_{P}(1),
\end{equation*}%
with some $\zeta <1/2$ for all $1\leq \ell \leq M+1$. Recall that the
results in \citet{aue2014} and \citet{berkesliuwu} are based on blocking
arguments: thus, $W_{N,\ell ,1}(k)$ and $W_{N,\ell ,2}(k)$ are - as well as
independent of each other - independent across $\ell $.

Let 
\begin{equation*}
W_{N,\ell }(x)=\left\{ 
\begin{array}{ll}
W_{N,\ell ,1}(x),\quad \mbox{if}\;\;\;0\leq x\leq (m_{\ell }-m_{\ell -1})/2,%
\vspace{0.3cm} &  \\ 
W_{N,\ell ,1}((m_{\ell }-m_{\ell -1})/2)+W_{N,\ell ,2}((m_{\ell }-m_{\ell
-1})/2)-W_{N,\ell ,2}((m_{\ell }-m_{\ell -1})-x),\vspace{0.3cm} &  \\ 
\hspace{0.3cm}\mbox{if}\;\;\;(m_{\ell }-m_{\ell -1})/2\leq x\leq m_{\ell
}-m_{\ell -1}, & 
\end{array}%
\right.
\end{equation*}%
$1\leq \ell \leq M+1$. Now we define 
\begin{equation*}
\Delta _{N}(k)=\sum_{j=1}^{\ell -1}\frac{\eta _{j}}{a_{j}}%
W_{N,j}(m_{j}-m_{j-1})+\frac{\eta _{\ell }}{a_{\ell }}W_{N,\ell }(x-m_{\ell
-1}),\quad m_{\ell -1}<x\leq m_{\ell },1\leq \ell \leq M+1.
\end{equation*}%
Thus we obtain the following approximations 
\begin{equation}
\max_{1\leq k\leq N/2}\frac{1}{k^{\zeta }}\left\vert \frac{k(N-k)}{N}(%
\widehat{\beta }_{k,1}-\widehat{\beta }_{k,2})-\left( \Gamma _{N}(k)-\frac{k%
}{N}\Gamma _{N}(N)\right) \right\vert =O_{P}(1),  \label{rcavol1}
\end{equation}%
and 
\begin{align}
\max_{N/2\leq k\leq N-1}\frac{1}{(N-k)^{\zeta }}& \left\vert \frac{k(N-k)}{N}%
(\widehat{\beta }_{k,1}-\widehat{\beta }_{k,2})-\left( -(\Gamma
_{N}(N)-\Gamma _{N}(k))+\frac{N-k}{N}\Gamma _{N}(N)\right) \right\vert
\label{rcavol2} \\
& =O_{P}(1),  \notag
\end{align}%
for some $0<\zeta <1/2$. We further note that 
\begin{equation*}
\left\{ N^{-1/2}\Gamma _{N}(Nt),0\leq t\leq 1\right\} \;\overset{{\mathcal{D}%
}}{=}\;\left\{ \Gamma (t),0\leq t\leq 1\right\} ,
\end{equation*}%
and that $\eta _{0}(t,t)$ is proportional to $t(1-t)$, if $0<t\leq \tau _{1}$
or $\tau _{M}<t<1$.

We now prove part \textit{(i)} of the theorem; the proof is based on the
same arguments used in the proof of Theorem \ref{thrca1}. Let $0<\delta
<\min \left\{ \tau _{1},\tau _{M}\right\} $; then%
\begin{align*}
\sup_{\delta <t<1-\delta }&\frac{\left\vert Q_{N}\left( t\right) -\left(
N^{-1/2}\Gamma _{N}(Nt)-\frac{\left\lfloor Nt\right\rfloor }{N}%
N^{-1/2}\Gamma _{N}(N)\right) \right\vert }{w\left( t\right) } \\
\leq &c_{0}N^{-1/2+\zeta }\sup_{N\delta \leq k<N/2}\frac{1}{k^{\zeta }}%
\left\vert \frac{k(N-k)}{N}(\widehat{\beta }_{k,1}-\widehat{\beta }%
_{k,2})-\left( \Gamma _{N}(k)-\frac{k}{N}\Gamma _{N}(N)\right) \right\vert \\
&+c_{0}N^{-1/2+\zeta }\sup_{N/2\leq k\leq N\left( 1-\delta \right) }\frac{1}{%
\left( N-k\right) ^{\zeta }}\left\vert \frac{k(N-k)}{N}(\widehat{\beta }%
_{k,1}-\widehat{\beta }_{k,2})-\left( \Gamma _{N}(k)-\frac{k}{N}\Gamma
_{N}(N)\right) \right\vert \\
=&o_{P}\left( 1\right) ,
\end{align*}%
having used Assumption \ref{as-wc-1} in the second passage, and (\ref%
{rcavol1})-(\ref{rcavol2}) in the last one. Also note%
\begin{align*}
\sup_{1/\left( N+1\right) <t<\delta }&\frac{\left\vert Q_{N}\left( t\right)
-\left( N^{-1/2}\Gamma _{N}(Nt)-\frac{\left\lfloor Nt\right\rfloor }{N}%
N^{-1/2}\Gamma _{N}(N)\right) \right\vert }{w\left( t\right) } \\
=&\sup_{1/\left( N+1\right) <t<\delta }\frac{\left\vert Q_{N}\left( t\right)
-\left( N^{-1/2}\Gamma _{N}(Nt)-\frac{\left\lfloor Nt\right\rfloor }{N}%
N^{-1/2}\Gamma _{N}(N)\right) \right\vert }{\left[ t\left( 1-t\right) \right]
^{1/2}}\frac{\left[ t\left( 1-t\right) \right] ^{1/2}}{w\left( t\right) } \\
\leq &\sqrt{2}N^{-1/2+\zeta }\frac{\delta ^{1/2}}{w\left( \delta \right) }%
\sup_{1\leq k\leq N\delta }\frac{\left\vert \frac{k(N-k)}{N}(\widehat{\beta }%
_{k,1}-\widehat{\beta }_{k,2})-\left( \Gamma _{N}(k)-\frac{k}{N}\Gamma
_{N}(N)\right) \right\vert }{k^{\zeta }}\sup_{1/\left( N+1\right) <t<\delta
}t^{\zeta -1/2} \\
\leq &c_{0}\frac{\delta ^{1/2}}{w\left( \delta \right) }O_{P}\left( 1\right)
,
\end{align*}%
having used (\ref{rcavol1}). Using (\ref{CCHM}), it follows that%
\begin{equation*}
\lim_{\delta \rightarrow 0^{+}}\lim \sup_{N\rightarrow \infty }P\left\{
\sup_{1/\left( N+1\right) <t<\delta }\frac{\left\vert Q_{N}\left( t\right)
-\left( N^{-1/2}\Gamma _{N}(Nt)-\frac{\left\lfloor Nt\right\rfloor }{N}%
N^{-1/2}\Gamma _{N}(N)\right) \right\vert }{w\left( t\right) }>x\right\} =0,
\end{equation*}%
for all $x>0$. Similar arguments yield%
\begin{equation*}
\lim_{\delta \rightarrow 0^{+}}\lim \sup_{N\rightarrow \infty }P\left\{
\sup_{1-\delta <t<1-1/\left( N+1\right) }\frac{\left\vert Q_{N}\left(
t\right) -\left( N^{-1/2}\Gamma _{N}(Nt)-\frac{\left\lfloor Nt\right\rfloor 
}{N}N^{-1/2}\Gamma _{N}(N)\right) \right\vert }{w\left( t\right) }>x\right\}
=0,
\end{equation*}%
for all $x>0$. Putting all together, it holds that%
\begin{equation*}
\sup_{1/\left( N+1\right) \leq t\leq 1-1/\left( N+1\right) }\frac{\left\vert
Q_{N}\left( t\right) -\left( N^{-1/2}\Gamma _{N}(Nt)-\frac{\left\lfloor
Nt\right\rfloor }{N}N^{-1/2}\Gamma _{N}(N)\right) \right\vert }{w\left(
t\right) }=o_{P}\left( 1\right) ,
\end{equation*}%
which proves part \textit{(i)} of the theorem.\newline
The proof of part \textit{(ii)} is based again on the approximations in (\ref%
{rcavol1}) and (\ref{rcavol2}). First we show that 
\begin{equation}
\lim_{N\rightarrow \infty }P\Biggl\{\sup_{0<t<1}\frac{\left\vert
Q_{N}(t)\right\vert }{\eta _{0}^{1/2}(t,t)}=\max \Biggl(\sup_{t_{1}\leq
t\leq t_{2}}\frac{\left\vert Q_{N}(t)\right\vert }{\eta _{0}^{1/2}(t,t)}%
,\sup_{1-t_{2}\leq t\leq 1-t_{1}}\frac{\left\vert Q_{N}(t)\right\vert }{\eta
_{0}^{1/2}(t,t)}\Biggl)\Biggl\}=1,  \label{erdos-intervals}
\end{equation}%
where $t_{1}=(\ln N)^{4}/N$ and $t_{2}=1/\ln N$. Indeed, recalling that $%
\eta _{0}(t,t)=c_{0}t\left( 1-t\right) $ for $t\leq \tau _{1}$ and $t\geq
\tau _{M}$%
\begin{equation}
\sup_{0<t<1/\left( N+1\right) }\frac{\left\vert Q_{N}(t)\right\vert }{\eta
_{0}^{1/2}(t,t)}=N^{-1/2}|\widehat{\beta }_{N}|\sup_{0<t<1/\left( N+1\right)
}\eta _{0}^{-1/2}(t,t)\leq c_{0}|\widehat{\beta }_{N}|=O_{P}\left( 1\right) .
\label{erdos-1}
\end{equation}%
Also, it holds that%
\begin{align*}
N^{-1/2}&\sup_{1/\left( N+1\right) \leq t<t_{1}}\frac{\left\vert \frac{k(N-k)%
}{N}(\widehat{\beta }_{k,1}-\widehat{\beta }_{k,2})-\left( W_{N,1,1}\left(
Nt\right) -\frac{\left\lfloor Nt\right\rfloor }{N}\Gamma _{N}(N)\right)
\right\vert }{\left( t\left( 1-t\right) \right) ^{1/2}} \\
\leq &N^{-1/2+\zeta }\sup_{1/\left( N+1\right) \leq t<t_{1}}\frac{\left\vert 
\frac{k(N-k)}{N}(\widehat{\beta }_{k,1}-\widehat{\beta }_{k,2})-\left(
W_{N,1,1}\left( Nt\right) -\frac{\left\lfloor Nt\right\rfloor }{N}\Gamma
_{N}(N)\right) \right\vert }{\left( Nt\right) ^{\zeta }}\sup_{1/\left(
N+1\right) \leq t<t_{1}}t^{\zeta -1/2} \\
=&O_{P}\left( 1\right) ,
\end{align*}%
and%
\begin{align*}
N^{-1/2} & \sup_{1/\left( N+1\right) \leq t<t_{1}}\frac{\left\vert \frac{%
\left\lfloor Nt\right\rfloor }{N}\Gamma _{N}(N)\right\vert }{\left( t\left(
1-t\right) \right) ^{1/2}} \leq c_{0}N^{-1/2}\left\vert \Gamma
_{N}(N)\right\vert \sup_{1/\left( N+1\right) \leq t<t_{1}}t^{1/2} \\
&=O_{P}\left( 1\right) t_{1}^{1/2}=O_{P}\left( N^{-1/2}(\ln N)^{2}\right) .
\end{align*}%
Hence%
\begin{equation*}
\sup_{1/\left( N+1\right) \leq t<t_{1}}\frac{\left\vert
Q_{N}(t)-N^{-1/2}W_{N,1,1}\left( Nt\right) \right\vert }{\left( t\left(
1-t\right) \right) ^{1/2}}=O_{P}\left( 1\right) .
\end{equation*}%
Also%
\begin{align*}
\sup_{1/\left( N+1\right) \leq t<t_{1}}\frac{\left\vert
N^{-1/2}W_{N,1,1}\left( Nt\right) \right\vert }{\eta _{0}^{1/2}(t,t)} \leq
&\sup_{1/\left( N+1\right) \leq t<t_{1}}\frac{\left\vert
N^{-1/2}W_{N,1,1}\left( Nt\right) \right\vert }{\left( t\left( 1-t\right)
\right) ^{1/2}}\sup_{1/\left( N+1\right) \leq t<t_{1}}\frac{\left( t\left(
1-t\right) \right) ^{1/2}}{\eta _{0}^{1/2}(t,t)} \\
=&O_{P}\left( \left( \ln \ln \ln N\right) ^{1/2}\right) O\left( 1\right) ,
\end{align*}%
having used the Darling-Erd\H{o}s theorem (see Theorem A.4.2 in %
\citealp{csorgo1997}). The equations above entail that%
\begin{equation}
\sup_{1/\left( N+1\right) \leq t<t_{1}}\frac{|Q_{N}(t)|}{\eta _{0}^{1/2}(t,t)%
}=O_{P}\left( \left( \ln \ln \ln N\right) ^{1/2}\right) .  \label{erdos-2}
\end{equation}%
Finally, note that, by similar passages as above%
\begin{align*}
\sup_{t_{2}\leq t<1/2}&\frac{\left\vert Q_{N}(t)-\left( N^{-1/2}\Gamma
_{N}(Nt)-\frac{\left\lfloor Nt\right\rfloor }{N}N^{-1/2}\Gamma
_{N}(N)\right) \right\vert }{\eta _{0}^{1/2}(t,t)} \\
\leq &\sup_{t_{2}\leq t<1/2}\frac{|t\Gamma (1)|}{\eta _{0}^{1/2}(t,t)}%
+\sup_{t_{2}\leq t\leq \tau _{1}}\frac{|Q_{N}(t)-N^{-1/2}W_{N,1,1}\left(
Nt\right) |}{\eta _{0}^{1/2}(t,t)}+\sup_{\tau _{1}<t\leq \tau _{2}}\frac{%
|Q_{N}(t)-N^{-1/2}W_{N,2,1}\left( Nt\right) |}{\eta _{0}^{1/2}(t,t)}+... \\
=&O_{P}\left( 1\right) ,
\end{align*}%
with $W_{2,1}\left( t\right) =N^{-1/2}W_{N,2,1}\left( Nt\right) $, etc...
Since%
\begin{align*}
&\sup_{t_{2}\leq t<1/2}\frac{|N^{-1/2}W_{N,1,1}\left( Nt\right) |}{\eta
_{0}^{1/2}(t,t)} \leq \sup_{t_{2}\leq t<1/2}\frac{|N^{-1/2}W_{N,1,1}\left(
Nt\right) |}{\left( t\left( 1-t\right) \right) ^{1/2}}\sup_{t_{2}\leq t<1/2}%
\frac{\left( t\left( 1-t\right) \right) ^{1/2}}{\eta _{0}^{1/2}(t,t)}%
=O_{P}\left( \left( \ln \ln \ln N\right) ^{1/2}\right) , \\
&\sup_{\tau _{1}<t\leq \tau _{2}}\frac{|N^{-1/2}W_{N,2,1}\left( Nt\right) |}{%
\eta _{0}^{1/2}(t,t)} \leq \sup_{\tau _{1}<t\leq \tau _{2}}\frac{%
|N^{-1/2}W_{N,2,1}\left( Nt\right) |}{\left( t\left( 1-t\right) \right)
^{1/2}}\sup_{\tau _{1}<t\leq \tau _{2}}\frac{\left( t\left( 1-t\right)
\right) ^{1/2}}{\eta _{0}^{1/2}(t,t)}=O_{P}\left( 1\right) ,
\end{align*}%
and the same for all $W_{N,\ell ,1}\left( Nt\right) $, $1<\ell $, we
conclude that%
\begin{equation}
\sup_{t_{2}\leq t<1/2}\frac{|Q_{N}(t)|}{\eta _{0}^{1/2}(t,t)}=O_{P}\left(
\left( \ln \ln \ln N\right) ^{1/2}\right) .  \label{erdos-3}
\end{equation}%
The same results can be shown, with the same logic, over the intervals $%
1-t_{1}<t<1$ and $\frac{1}{2}<t<1-t_{2}$. Next we note that 
\begin{equation}
\sup_{t_{1}\leq t\leq t_{2}}\left\vert \frac{\eta _{0}(t,t)}{t}-\frac{\eta
_{1}^{2}}{a_{1}^{2}}\right\vert =O(1/\ln N),  \label{de1}
\end{equation}%
and 
\begin{equation}
\sup_{1-t_{1}\leq t\leq 1-t_{2}}\left\vert \frac{\eta _{0}(t,t)}{1-t}-\frac{%
\eta _{M+1}^{2}}{a_{M+1}^{2}}\right\vert =O(1/\ln N).  \label{de2}
\end{equation}%
It thus follows from \eqref{rcavol1} that 
\begin{align*}
\sup_{t_{1}\leq t\leq t_{2}}&\frac{\left\vert Q_{N}\left( t\right)
-N^{-1/2}W_{N,1,1}\left( Nt\right) \right\vert }{\eta _{0}^{1/2}(t,t)} \\
\leq &c_{0}N^{-1/2+\zeta }\sup_{t_{1}\leq t\leq t_{2}}\frac{\displaystyle %
\left\vert \frac{k(N-k)}{N}(\widehat{\beta }_{k,1}-\widehat{\beta }%
_{k,2})-\left( W_{N,1,1}\left( Nt\right) -\frac{\left\lfloor Nt\right\rfloor 
}{N}\Gamma _{N}\left( N\right) \right) \right\vert }{\displaystyle \left(
Nt\right) ^{\zeta }}\sup_{t_{1}\leq t\leq t_{2}}t^{\zeta -1/2} \\
&+c_{0}N^{-1/2}\left\vert \Gamma _{N}\left( N\right) \right\vert
\sup_{t_{1}\leq t\leq t_{2}}t^{1/2} \\
=&O_{P}\left( \left( \ln N\right) ^{4(\zeta -1/2)}\right) .
\end{align*}%
The same result can be shown, with exactly the same logic, in the interval $%
1-t_{2}<t<1-t_{1}$, viz.%
\begin{equation*}
\sup_{1-t_{2}\leq t\leq 1-t_{1}}\frac{\left\vert Q_{N}\left( t\right)
-N^{-1/2}W_{N,M+1,1}\left( Nt\right) \right\vert }{\eta _{0}^{1/2}(t,t)}%
=O_{P}\left( \left( \ln N\right) ^{4(\zeta -1/2)}\right) .
\end{equation*}%
Further, using (\ref{de1})%
\begin{align*}
\sup_{t_{1}\leq t\leq t_{2}}&\left\vert \frac{\left\vert
N^{-1/2}W_{N,1,1}\left( Nt\right) \right\vert }{\eta _{0}^{1/2}(t,t)}-\left( 
\frac{\eta _{1}^{2}}{a_{1}^{2}}\right) ^{-1/2}\frac{\left\vert
N^{-1/2}W_{N,1,1}\left( Nt\right) \right\vert }{\left( t\left( 1-t\right)
\right) ^{1/2})}\right\vert \\
\leq &c_{0}\sup_{t_{1}\leq t\leq t_{2}}\left( \frac{\eta _{1}^{2}}{a_{1}^{2}}%
\right) ^{-1/2}\frac{\left\vert N^{-1/2}W_{N,1,1}\left( Nt\right)
\right\vert }{\left( t\left( 1-t\right) \right) ^{1/2})}\sup_{t_{1}\leq
t\leq t_{2}}\left\vert \frac{t^{1/2}}{\eta _{0}^{1/2}(t,t)}-\left( \frac{%
a_{1}^{2}}{\eta _{1}^{2}}\right) ^{1/2}\right\vert \\
=&O_{P}\left( \frac{\sqrt{2\ln \ln N}}{\ln N}\right) ,
\end{align*}%
and, by the Darling-Erd\H{o}s theorem 
\begin{equation*}
\frac{1}{\sqrt{2\ln \ln N}}\sup_{t_{1}\leq t<t_{2}}\frac{|N^{-1/2}W_{N,1,1}%
\left( Nt\right) |}{\left( t\left( 1-t\right) \right) ^{1/2})}\overset{P}{%
\rightarrow }1.
\end{equation*}%
This proves (\ref{erdos-intervals}). Since $\{W_{N,1,1}(x),Nt_{1}\leq x\leq
Nt_{2}\}$ and $\{W_{N,M+1,2}(x),N(1-t_{2})\leq x\leq N(1-t_{1})\}$ are
independent Wiener processes, Theorem A.4.2 in Cs\"{o}rg\H{o} and Horv\'{a}%
th (1997) implies 
\begin{align*}
& \lim_{N\rightarrow \infty }P\Biggl\{a(\ln N)\max \Biggl(\sup_{(\ln
N)^{4}\leq x\leq N/\ln N}x^{-1/2}|W_{N,1,1}(x)|, \\
& \hspace{2cm}\sup_{N-N/\ln N\leq x\leq N-(\ln
N)^{4}}(N-x)^{-1/2}|W_{N,M+1,2}(N-x)|\Biggl)\leq x+b(\ln N)\Biggl\} \\
& =\lim_{N\rightarrow \infty }P\Biggl\{a(\ln N)\max \Biggl(\sup_{(\ln
N)^{4}\leq x\leq N/\ln N}x^{-1/2}|W_{N,1,1}(x)|\leq x+b(\ln N)\Biggl\} \\
& \hspace{2cm}\times \lim_{N\rightarrow \infty }P\Biggl\{\sup_{(\ln
N)^{4}\leq x\leq N/\ln N}x^{-1/2}|W_{N,M+1,2}(x)|\leq x+b(\ln N)\Biggl\} \\
& =\left[ \exp (-e^{-x})\right] ^{2}.
\end{align*}%
The proof of Theorem \ref{thrca4}\textit{(ii)} is now complete.

Part \textit{(iii)} follows by repeating the proof of part \textit{(ii)}
with minor modifications, and therefore we only report some passages. Let $%
r_{1,N}=r_{2,N}=r_{N}$ for simplicity; it is easy to see that%
\begin{align}
\sup_{r_{N}/N\leq t\leq 1-r_{N}/N}\frac{\left( t\left( 1-t\right) \right)
^{1/2}}{\eta _{0}^{1/2}(t,t)}=& O\left( 1\right) ,  \label{re-1} \\
\sup_{r_{N}/N\leq t\leq 1-r_{N}/N}\frac{\left( t\left( 1-t\right) \right)
^{1/2}}{\left( t\left( 1-t\right) \right) ^{\kappa }}=& O\left( \left( \frac{%
r_{N}}{N}\right) ^{1/2-\kappa }\right) .  \label{re-2}
\end{align}%
Note now that%
\begin{align*}
& r_{N}^{\kappa -1/2}\sup_{r_{N}/N\leq t\leq 1-r_{N}/N}\frac{\left\vert
Q_{N}(t)-\left( N^{-1/2}\Gamma _{N}(Nt)-\frac{\left\lfloor Nt\right\rfloor }{%
N}N^{-1/2}\Gamma _{N}(N)\right) \right\vert }{\left( t\left( 1-t\right)
\right) ^{\kappa }} \\
\leq & r_{N}^{\kappa -1/2}\sup_{r_{N}/N\leq t\leq 1-r_{N}/N}\frac{\left\vert
Q_{N}(t)-\left( N^{-1/2}\Gamma _{N}(Nt)-\frac{\left\lfloor Nt\right\rfloor }{%
N}N^{-1/2}\Gamma _{N}(N)\right) \right\vert }{\left( t\left( 1-t\right)
\right) ^{1/2}}\sup_{r_{N}/N\leq t\leq 1-r_{N}/N}\frac{\left( t\left(
1-t\right) \right) ^{1/2}}{\left( t\left( 1-t\right) \right) ^{\kappa }} \\
\leq & N^{-1/2+\zeta }\left( \frac{r_{N}}{N}\right) ^{\kappa
-1/2}\sup_{r_{N}\leq k\leq N-r_{N}}\frac{\left\vert \frac{k(N-k)}{N}(%
\widehat{\beta }_{k,1}-\widehat{\beta }_{k,2})-\left( \Gamma _{N}(k)-\frac{k%
}{N}\Gamma _{N}(N)\right) \right\vert }{\left( Nt\right) ^{\zeta }} \\
& \times \sup_{r_{N}/N\leq t\leq 1-r_{N}/N}\frac{\left( t\left( 1-t\right)
\right) ^{\zeta }}{\left( t\left( 1-t\right) \right) ^{\kappa }} \\
\leq & N^{\zeta -1/2}\left( \frac{r_{N}}{N}\right) ^{\kappa -1/2}O_{P}\left(
1\right) \left( \frac{r_{N}}{N}\right) ^{\zeta -\kappa }=O_{P}\left(
r_{N}^{\zeta -\kappa }\right) =o_{P}\left( 1\right) .
\end{align*}%
Using (\ref{re-1}), this also entails that%
\begin{eqnarray}
&&\left( \frac{r_{N}}{N}\right) ^{\kappa -1/2}\sup_{r_{N}\leq t\leq 1-r_{N}}%
\frac{\left( t\left( 1-t\right) \right) ^{1/2-\kappa }}{\eta _{0}^{1/2}(t,t)}%
\left\vert Q_{N}(t)-\left( N^{-1/2}\Gamma _{N}(Nt)-\frac{\left\lfloor
Nt\right\rfloor }{N}N^{-1/2}\Gamma _{N}(N)\right) \right\vert  \label{re-3}
\\
&=&o_{P}\left( 1\right) .  \notag
\end{eqnarray}%
We have 
\begin{align*}
& \left( \frac{r_{N}}{N}\right) ^{\kappa -1/2}\sup_{\left( r_{N}/N\right)
^{1/2}\leq t\leq 1/2}\frac{\left\vert N^{-1/2}\Gamma _{N}(Nt)-\frac{%
\left\lfloor Nt\right\rfloor }{N}N^{-1/2}\Gamma _{N}(N)\right\vert }{\left(
t\left( 1-t\right) \right) ^{\kappa }} \\
\leq & 2^{\kappa }\left( \frac{r_{N}}{N}\right) ^{\kappa -1/2}\sup_{\left(
r_{N}/N\right) ^{1/2}\leq t\leq 1/2}\frac{\left\vert N^{-1/2}\Gamma _{N}(Nt)-%
\frac{\left\lfloor Nt\right\rfloor }{N}N^{-1/2}\Gamma _{N}(N)\right\vert }{%
t^{1/2}}t^{1/2-\kappa } \\
\leq & 2^{\kappa }\left( \frac{r_{N}}{N}\right) ^{\kappa -1/2}\left( \left( 
\frac{r_{N}}{N}\right) ^{1/2}\right) ^{1/2-\kappa }\sup_{\left(
r_{N}/N\right) ^{1/2}\leq t\leq 1/2}\frac{\left\vert N^{-1/2}\Gamma _{N}(Nt)-%
\frac{\left\lfloor Nt\right\rfloor }{N}N^{-1/2}\Gamma _{N}(N)\right\vert }{%
t^{1/2}}=o_{P}\left( 1\right) .
\end{align*}%
By using the same logic as in the proof of (\ref{erdos-3}), it can be shown
that%
\begin{equation*}
\sup_{\left( r_{N}/N\right) ^{1/2}\leq t\leq 1/2}\frac{\left\vert
N^{-1/2}\Gamma _{N}(Nt)-\frac{\left\lfloor Nt\right\rfloor }{N}%
N^{-1/2}\Gamma _{N}(N)\right\vert }{t^{1/2}}=O_{P}\left( \sqrt{\ln \ln
\left( \frac{r_{N}}{N}\right) }\right) ,
\end{equation*}%
so that ultimately 
\begin{equation}
\left( \frac{r_{N}}{N}\right) ^{\kappa -1/2}\sup_{\left( r_{N}/N\right)
^{1/2}\leq t\leq 1/2}\frac{\left\vert N^{-1/2}\Gamma _{N}(Nt)-\frac{%
\left\lfloor Nt\right\rfloor }{N}N^{-1/2}\Gamma _{N}(N)\right\vert }{\left(
t\left( 1-t\right) \right) ^{\kappa }}=o_{P}\left( 1\right) .  \label{re-4}
\end{equation}%
By the same token, it follows that%
\begin{equation}
\left( \frac{r_{N}}{N}\right) ^{\kappa -1/2}\sup_{1/2\leq t\leq 1-\left(
r_{N}/N\right) ^{1/2}}\frac{\left\vert N^{-1/2}\Gamma _{N}(Nt)-\frac{%
\left\lfloor Nt\right\rfloor }{N}N^{-1/2}\Gamma _{N}(N)\right\vert }{\left(
t\left( 1-t\right) \right) ^{\kappa }}=o_{P}\left( 1\right) .  \label{re-5}
\end{equation}%
Further, note that%
\begin{align*}
\left( \frac{r_{N}}{N}\right) ^{\kappa -1/2}\sup_{r_{N}/N\leq t<\left(
r_{N}/N\right) ^{1/2}}& \frac{\left\vert \frac{\left\lfloor Nt\right\rfloor 
}{N}N^{-1/2}\Gamma _{N}(N)\right\vert }{\left( t\left( 1-t\right) \right)
^{\kappa }} \\
& \leq N^{-1/2}\Gamma _{N}(N)\left( \frac{r_{N}}{N}\right) ^{\kappa
-1/2}\sup_{r_{N}/N\leq t<\left( r_{N}/N\right) ^{1/2}}t^{1-\kappa } \\
=& O_{P}\left( \left( \frac{r_{N}}{N}\right) ^{\frac{1}{2}\min \left(
1,\kappa \right) }\right) =o_{P}\left( 1\right) ,
\end{align*}%
and the same holds on the interval $1-\left( r_{N}/N\right) ^{1/2}\leq
t<1-r_{N}/N$. Thus, we only need to focus on finding the limiting
distributions of 
\begin{align*}
& \left( \frac{r_{N}}{N}\right) ^{\kappa -1/2}\sup_{r_{N}/N\leq t<\left(
r_{N}/N\right) ^{1/2}}\frac{\left\vert N^{-1/2}\Gamma _{N}(Nt)\right\vert }{%
\left( t\left( 1-t\right) \right) ^{\kappa }}, \\
& \left( \frac{r_{N}}{N}\right) ^{\kappa -1/2}\sup_{1-\left( r_{N}/N\right)
^{1/2}\leq t<1-r_{N}/N}\frac{\left\vert N^{-1/2}\Gamma _{N}(Nt)\right\vert }{%
\left( t\left( 1-t\right) \right) ^{\kappa }}.
\end{align*}%
On the intervals $r_{N}/N\leq t<\left( r_{N}/N\right) ^{1/2}$ and $1-\left(
r_{N}/N\right) ^{1/2}\leq t<1-r_{N}/N$, however, $\Gamma _{N}(Nt)$ is given
by the Wiener processes $W_{N,1,1}(Nt)$ and $W_{N,M+1,1}(Nt)$; the desired
result now follows from \citet{horvathmiller}.\qed

\textit{Proof of Theorem \ref{thrca5}.} Repeating the arguments used in the
proofs of Theorems \ref{thrca1} and \ref{thrca111} on the intervals $%
(m_{\ell -1},m_{\ell }]$, one can show that 
\begin{equation}
\sup_{0<t<1}|\widehat{\mathfrak{c}}_{N,1}(t)-\mathfrak{c}_{1}(t)|=O_{P}%
\left( (\ln N)^{-\zeta }\right) ,  \label{sup-c}
\end{equation}%
with some $\zeta >0$, and 
\begin{equation}
\sup_{0<t<1}|\widehat{\mathfrak{b}}_{N}(t)-\mathfrak{b}(t)|=O_{P}\left( (\ln
N)^{-\zeta _{1}}\right) ,  \label{rcasim1}
\end{equation}%
with some $\zeta _{1}>0$. Thus it holds that

\begin{equation}
\sup_{0<t,s<1}|\hat{\mathfrak{g}}_{N}(t,s)-{\mathfrak{g}}(t,s)|=o_{P}(1),
\label{rcaest1}
\end{equation}%
and we can use the results in the proof of Theorem \ref{thrca4} to establish
Theorem \ref{thrca5}. \qed

\begin{proof}[Proof of Theorem \protect\ref{amoc}]
The proof follows from standard arguments (see e.g. \citealp{csorgo1997}),
and therefore we only report the most important passages. We begin by
considering (\ref{part3}), and show that, under (\ref{restr-1})%
\begin{equation*}
N^{1/2}\left( \frac{k^{\ast }}{N}\left( \frac{N-k^{\ast }}{N}\right) \right)
\left\vert \widehat{\beta }_{k^{\ast },1}-\widehat{\beta }_{k^{\ast
},2}\right\vert \overset{P}{\rightarrow }\infty .
\end{equation*}%
Under our assumptions, the same arguments as above yield%
\begin{align}
&\widehat{\beta }_{k^{\ast },1}\overset{P}{\rightarrow }\beta _{0},
\label{b1alt} \\
&\widehat{\beta }_{k^{\ast },2}\overset{P}{\rightarrow }\beta _{A}.
\label{b2alt}
\end{align}%
Hence (\ref{part3}) follows immediately. As far as (\ref{part33}) is
concerned, we begin by noting that, under the conditions of Theorem \ref%
{thrca5}, one can show following the same passages as in the proof of that
theorem that, if $|\beta _{0}-\beta _{A}|$ is bounded, as $N\rightarrow
\infty $, then there are functions $\mathfrak{c}_{1}^{\ast }(t)$ and $%
\mathfrak{b}^{\ast }(t)$ such that 
\begin{equation}
\sup_{0<t<1}|\widehat{\mathfrak{c}}_{N,1}(t)-\mathfrak{c}_{1}^{\ast
}(t)|=o_{P}\left( 1\right) ,  \label{supc2}
\end{equation}%
\begin{equation}
\sup_{0<t<1}\left\vert \hat{\mathfrak{b}}_{N}(t)-\mathfrak{b}^{\ast
}(t)\right\vert =o_{P}(1).  \label{rcasim2}
\end{equation}%
(\ref{supc2}) and (\ref{rcasim2}) entail that there exists a function ${%
\mathfrak{g}}^{\ast }\left( t,s\right) $ 
\begin{equation}
\sup_{0<t,s<1}\left\vert \hat{\mathfrak{g}}_{N}\left( t,s\right) -{\mathfrak{%
g}}^{\ast }\left( t,s\right) \right\vert =o_{P}(1),  \label{rcaest2}
\end{equation}%
which entails that (\ref{part3}) follows as long as 
\begin{equation*}
N^{1/2}\left( \frac{k^{\ast }}{N}\left( \frac{N-k^{\ast }}{N}\right) \right) 
\frac{r_{N}^{\kappa -1/2}}{\mathfrak{g}^{1/2}(t^{\ast },t^{\ast })}\left( 
\frac{k^{\ast }}{N}\left( \frac{N-k^{\ast }}{N}\right) \right) ^{-\kappa
+1/2}\left\vert \widehat{\beta }_{k^{\ast },1}-\widehat{\beta }_{k^{\ast
},2}\right\vert \overset{P}{\rightarrow }\infty
\end{equation*}%
having let $\left\lfloor Nt^{\ast }\right\rfloor =k^{\ast }$. But this
follows immediately from (\ref{b1alt}) and (\ref{b2alt}), noting that, by
definition $\mathfrak{g}(t^{\ast },t^{\ast })=O\left( \left( \frac{k^{\ast }%
}{N}\left( \frac{N-k^{\ast }}{N}\right) \right) ^{1/2}\right) $. Finally, (%
\ref{part2}) follows from the same logic.
\end{proof}

\newpage

\begin{figure}[t]
\caption{Empirical rejection frequencies under alternatives -
homoskedasticity and mid-sample break}
\label{fig:FigHo}\centering
\hspace{-2.5cm} 
\begin{minipage}{0.4\textwidth}
\centering
    \includegraphics[scale=0.4]{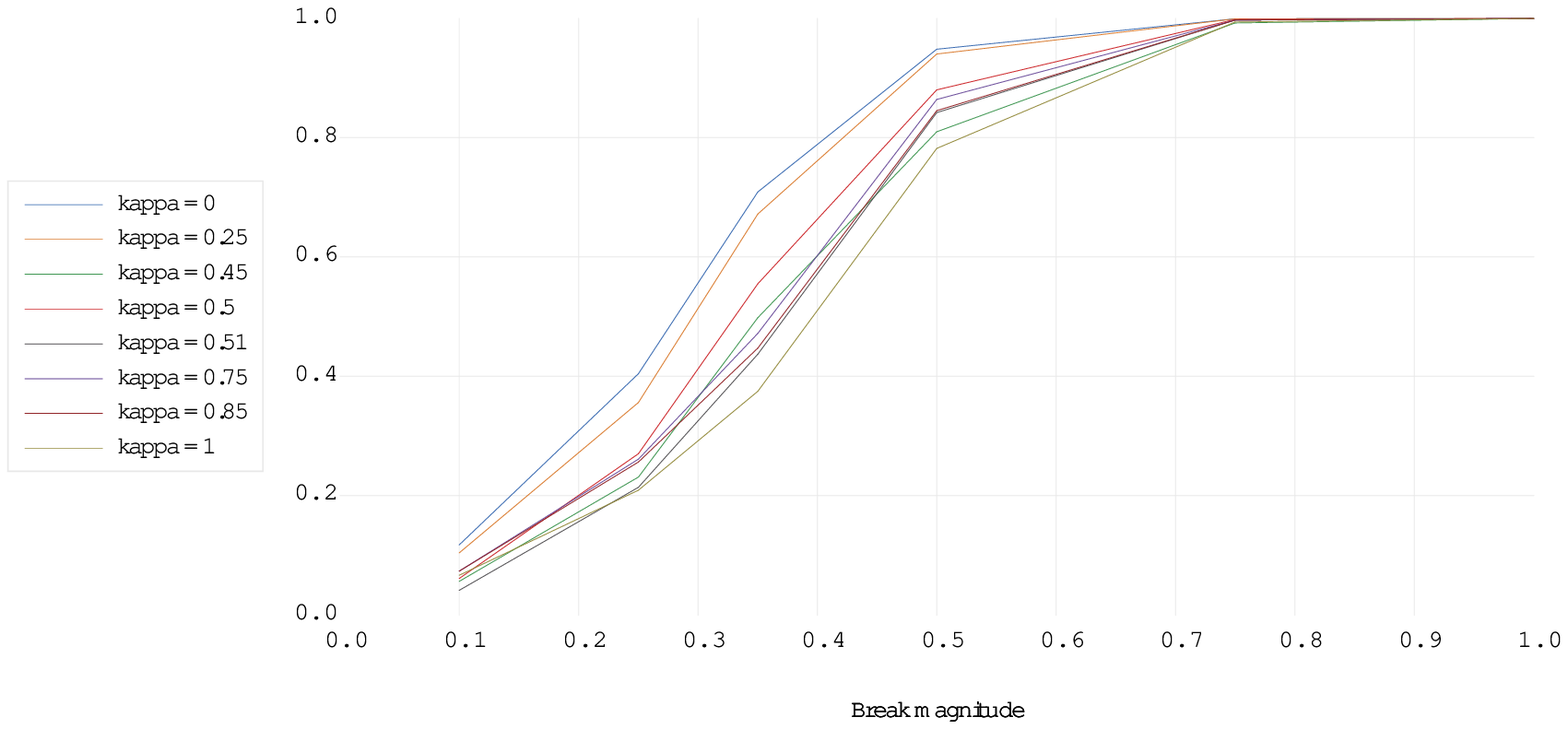}
      \captionof{subfigure}{$\beta_0=0.5$}
    \label{fig:t11}
\end{minipage}%
\begin{minipage}{0.4\textwidth}
\centering
   \includegraphics[scale=0.4]{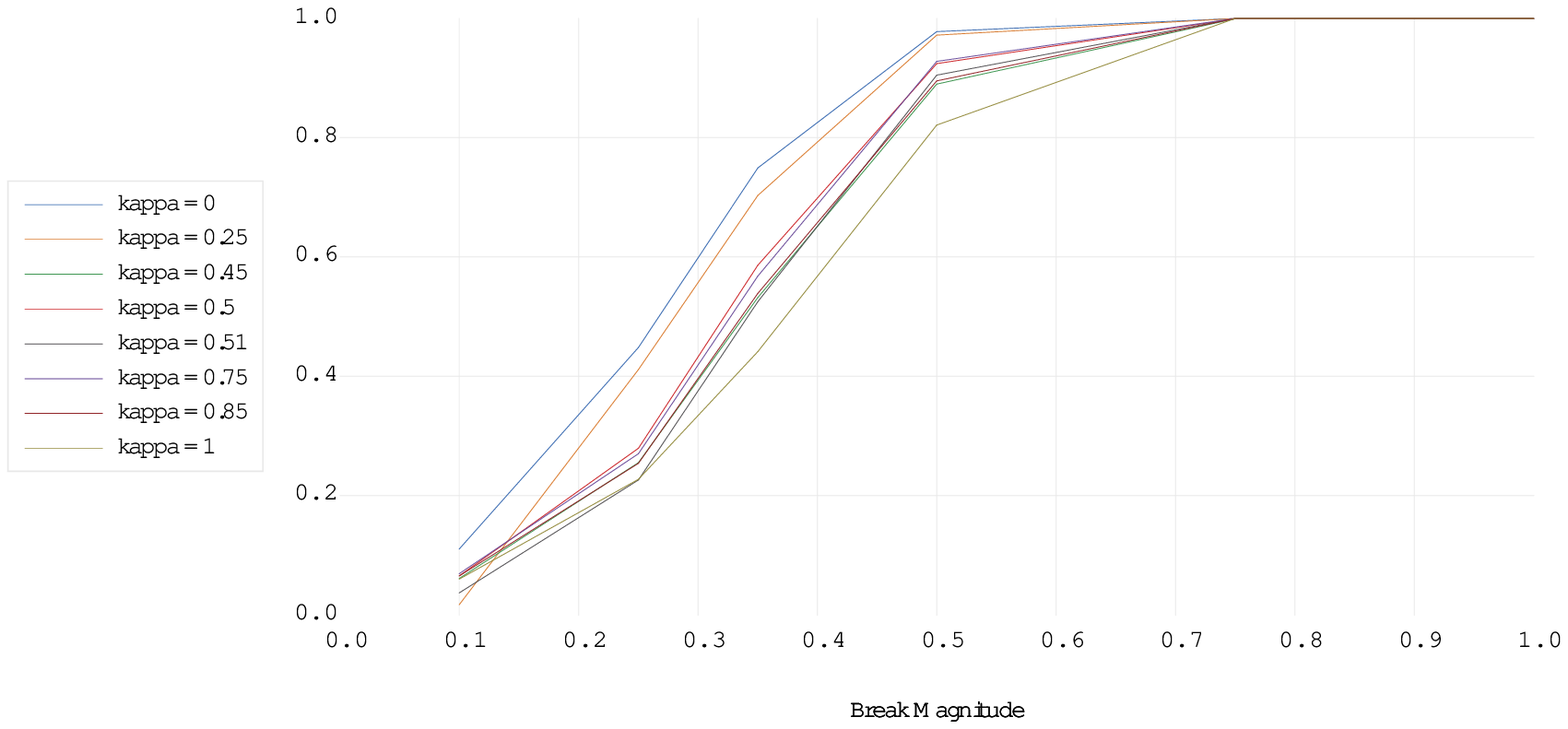}
    \captionof{subfigure}{$\beta_0=0.75$}
    \label{fig:t12}
\end{minipage} \\[0.25cm]
\par
\hspace{-2.5cm} 
\begin{minipage}{0.4\textwidth}
\centering
    \includegraphics[scale=0.4]{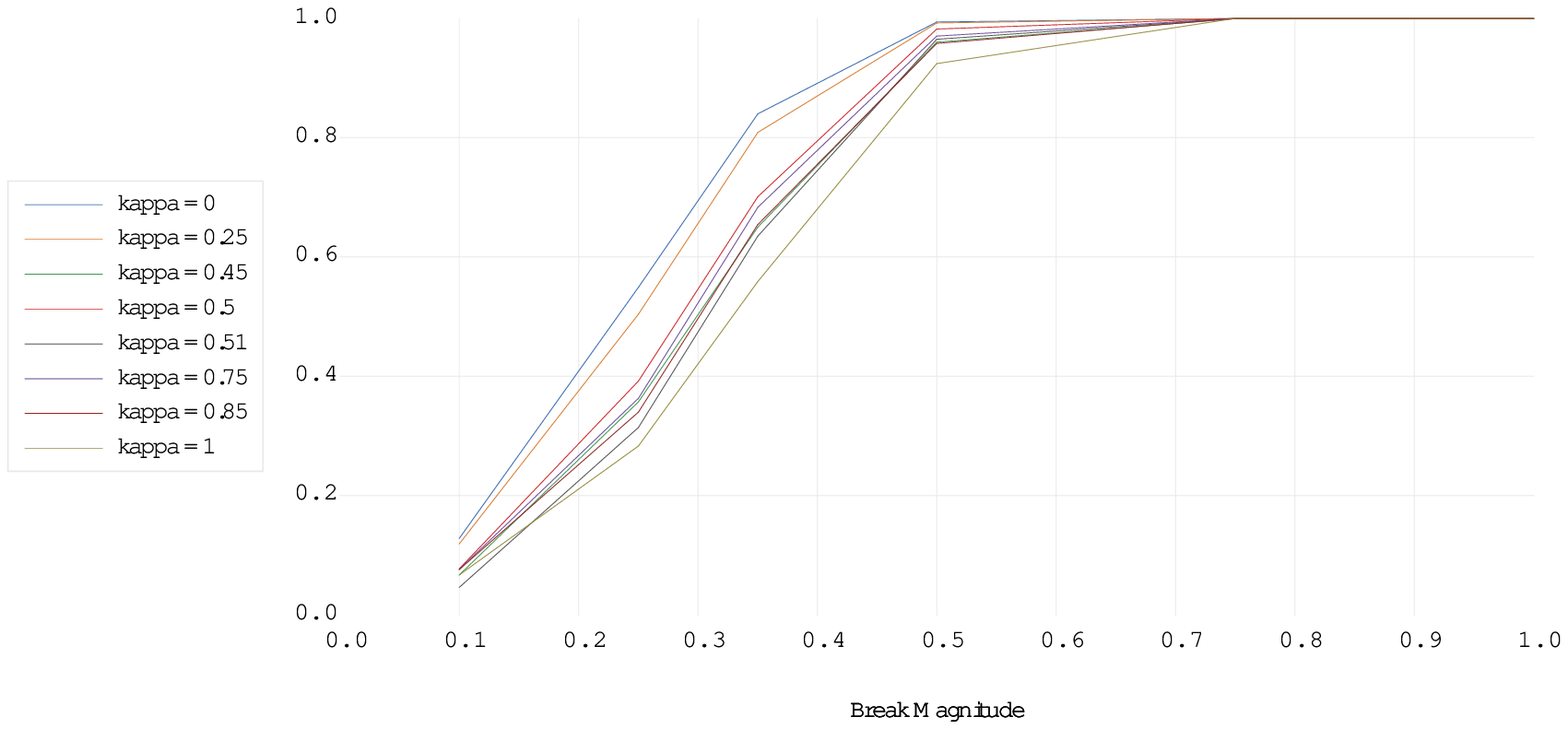}
    \captionof{subfigure}{$\beta_0=1$}
    \label{fig:t13}
\end{minipage}%
\begin{minipage}{0.4\textwidth}
\centering
   \includegraphics[scale=0.4]{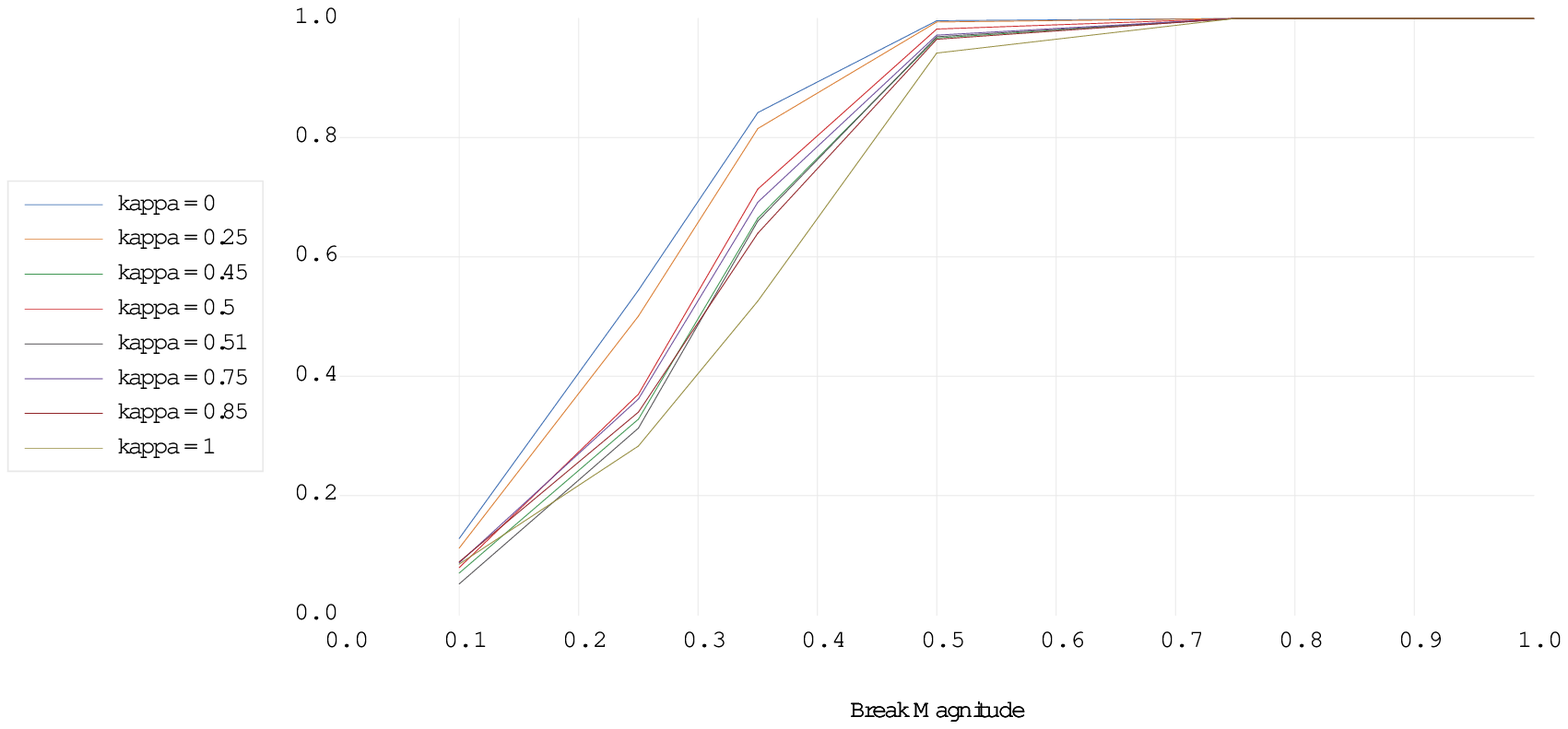}
    \captionof{subfigure}{$\beta_0=1.05$}
    \label{fig:t14}
\end{minipage} \\[0.25cm]
\par
%\caption*{.}
\end{figure}

\begin{figure}[!b]
\caption{Empirical rejection frequencies under alternatives -
heteroskedasticity in $\protect\epsilon_{i,2}$ and mid-sample break}
\label{fig:FigHeE}\centering
\hspace{-2.5cm} 
\begin{minipage}{0.4\textwidth}
\centering
    \includegraphics[scale=0.4]{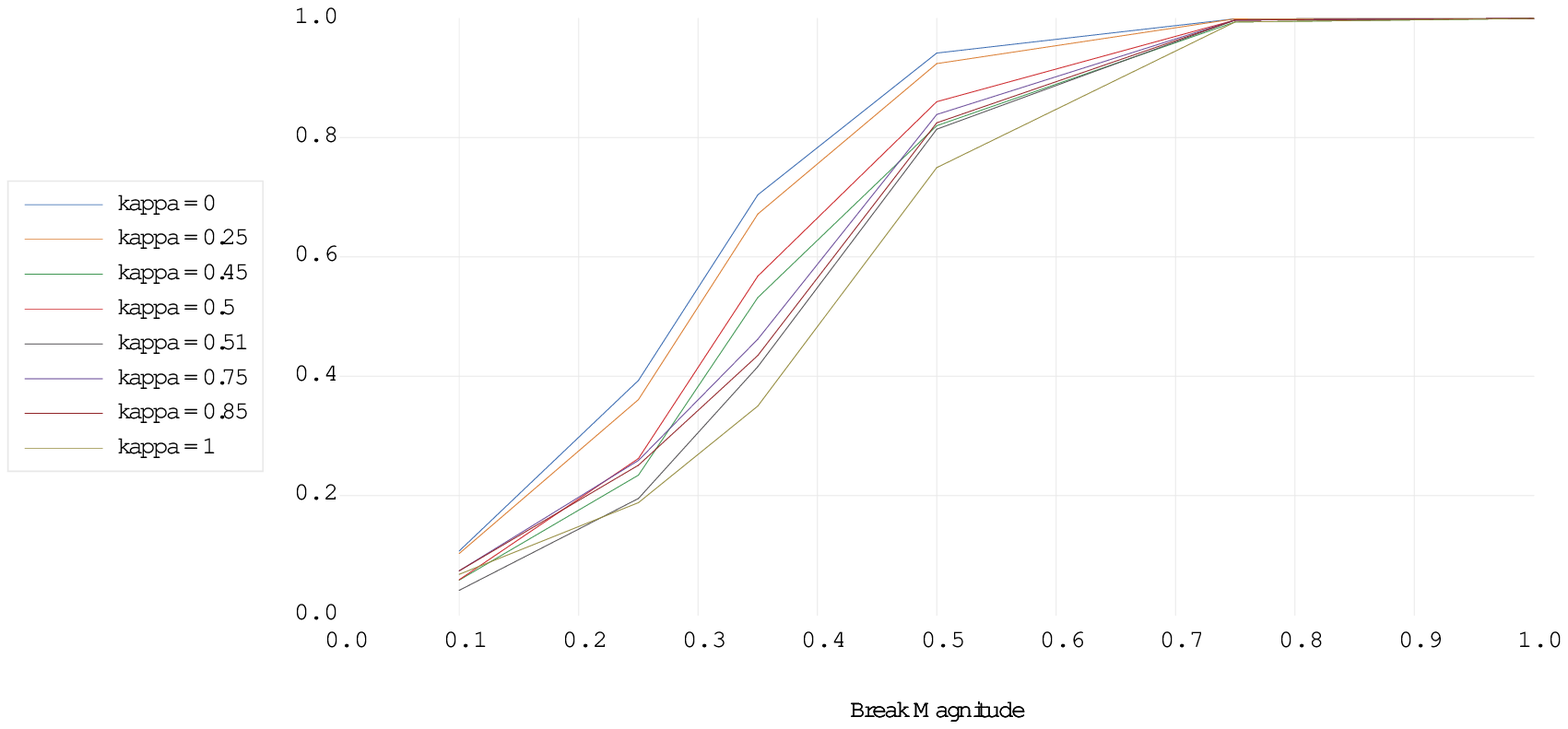}
    \captionof{subfigure}{$\beta_0=0.5$}
    \label{fig:t21}
\end{minipage}%
\begin{minipage}{0.4\textwidth}
\centering
   \includegraphics[scale=0.4]{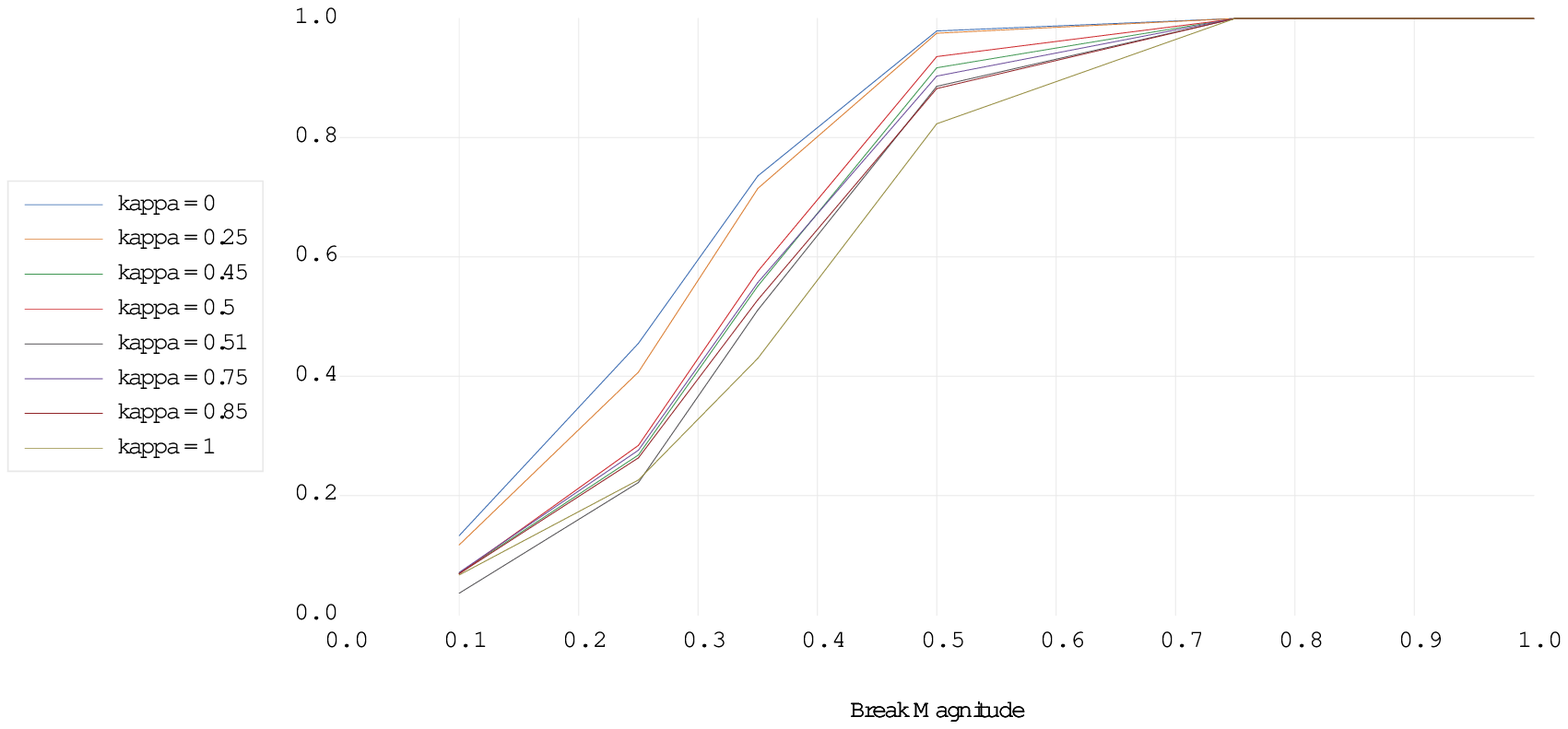}
    \captionof{subfigure}{$\beta_0=0.75$}
    \label{fig:t22}
\end{minipage} \\[0.25cm]
\par
\hspace{-2.5cm} 
\begin{minipage}{0.4\textwidth}
\centering
    \includegraphics[scale=0.4]{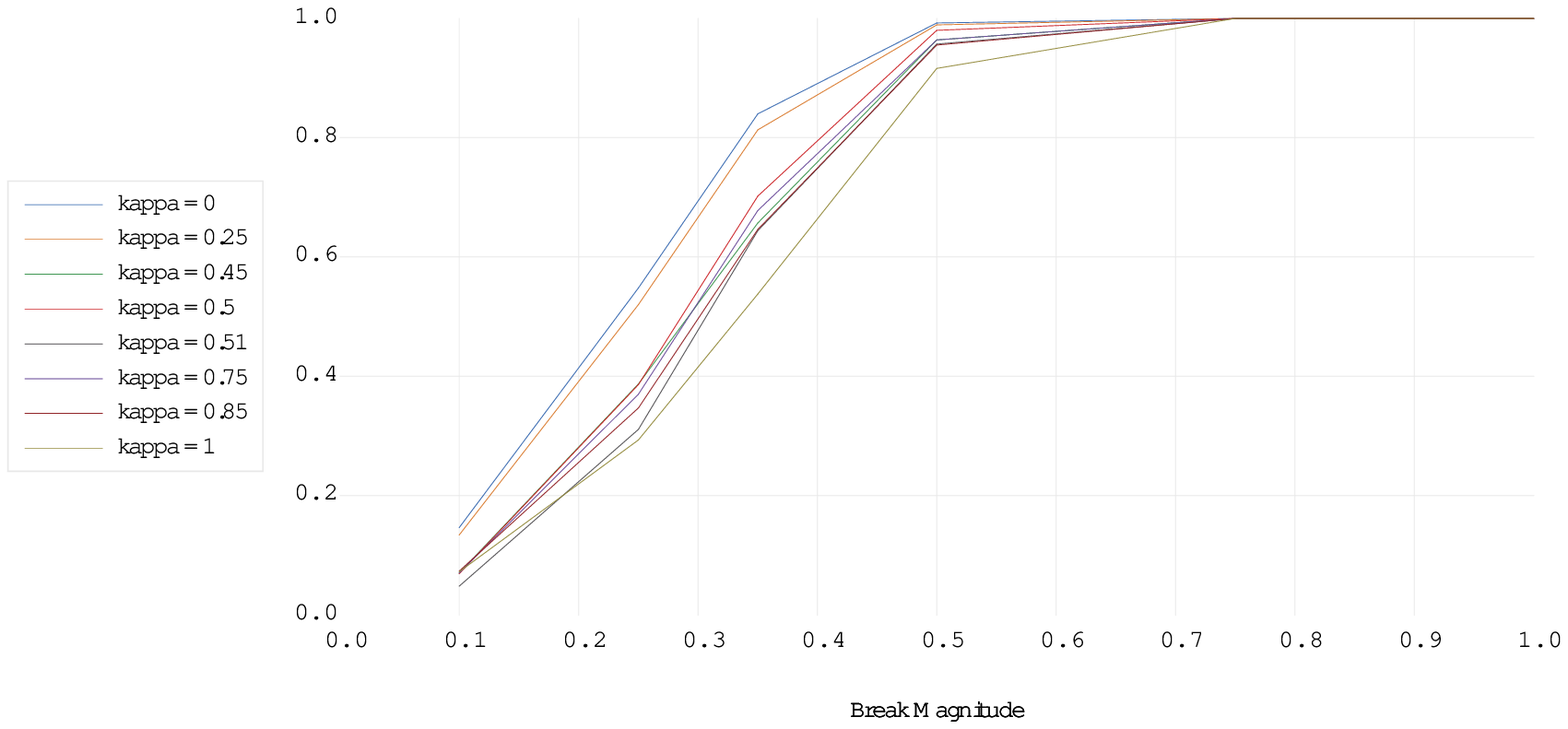}
   \captionof{subfigure}{$\beta_0=1$}
    \label{fig:t23}
\end{minipage}%
\begin{minipage}{0.4\textwidth}
\centering
   \includegraphics[scale=0.4]{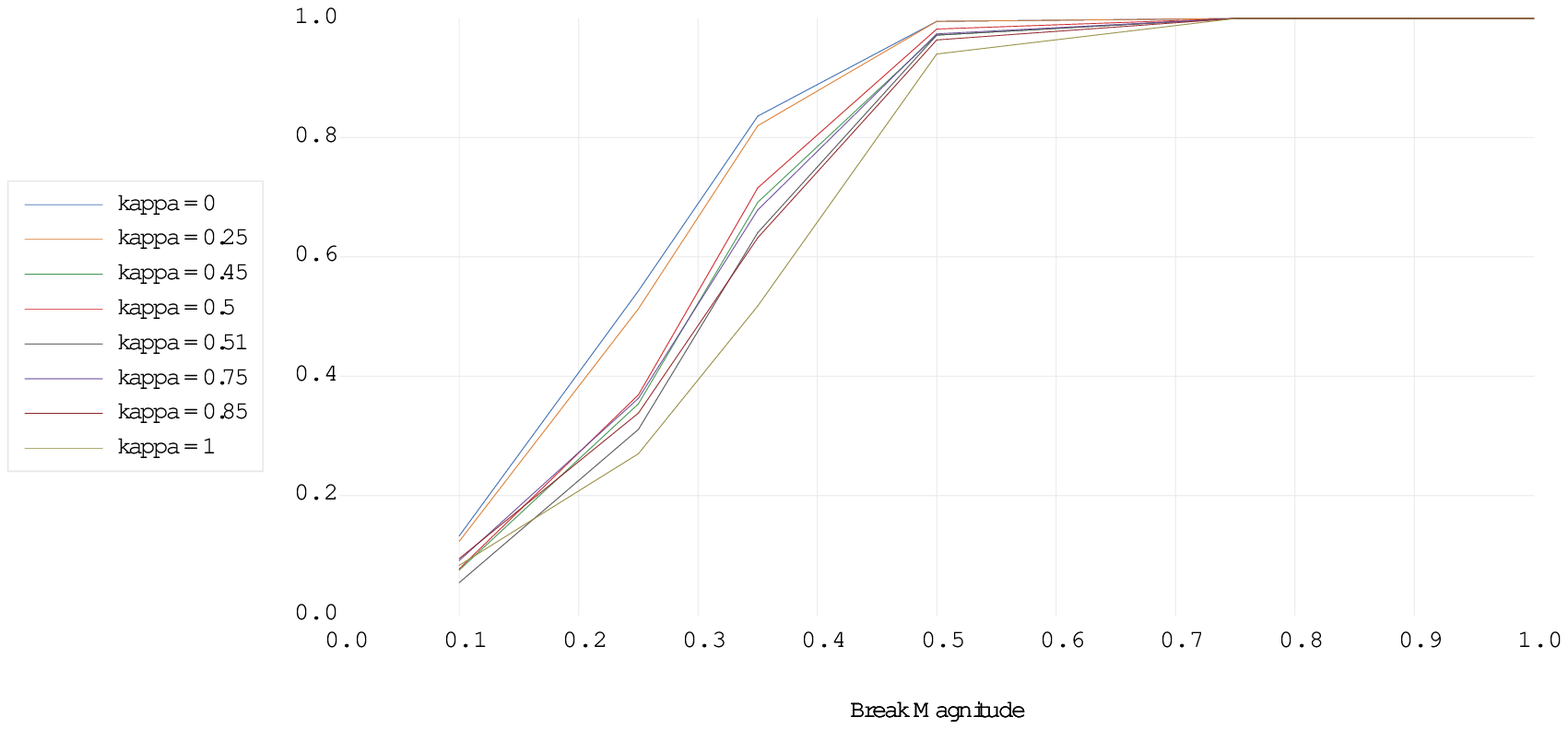}
    \captionof{subfigure}{$\beta_0=1.05$}
    \label{fig:t24}
\end{minipage} \\[0.25cm]
\par
%\caption*{.}
\end{figure}

\begin{figure}[!t]
\caption{Empirical rejection frequencies under alternatives -
heteroskedasticity in $\protect\epsilon_{i,1}$ and mid-sample break}
\label{fig:FigHeB}\centering
\hspace{-2.5cm} 
\begin{minipage}{0.4\textwidth}
\centering
    \includegraphics[scale=0.4]{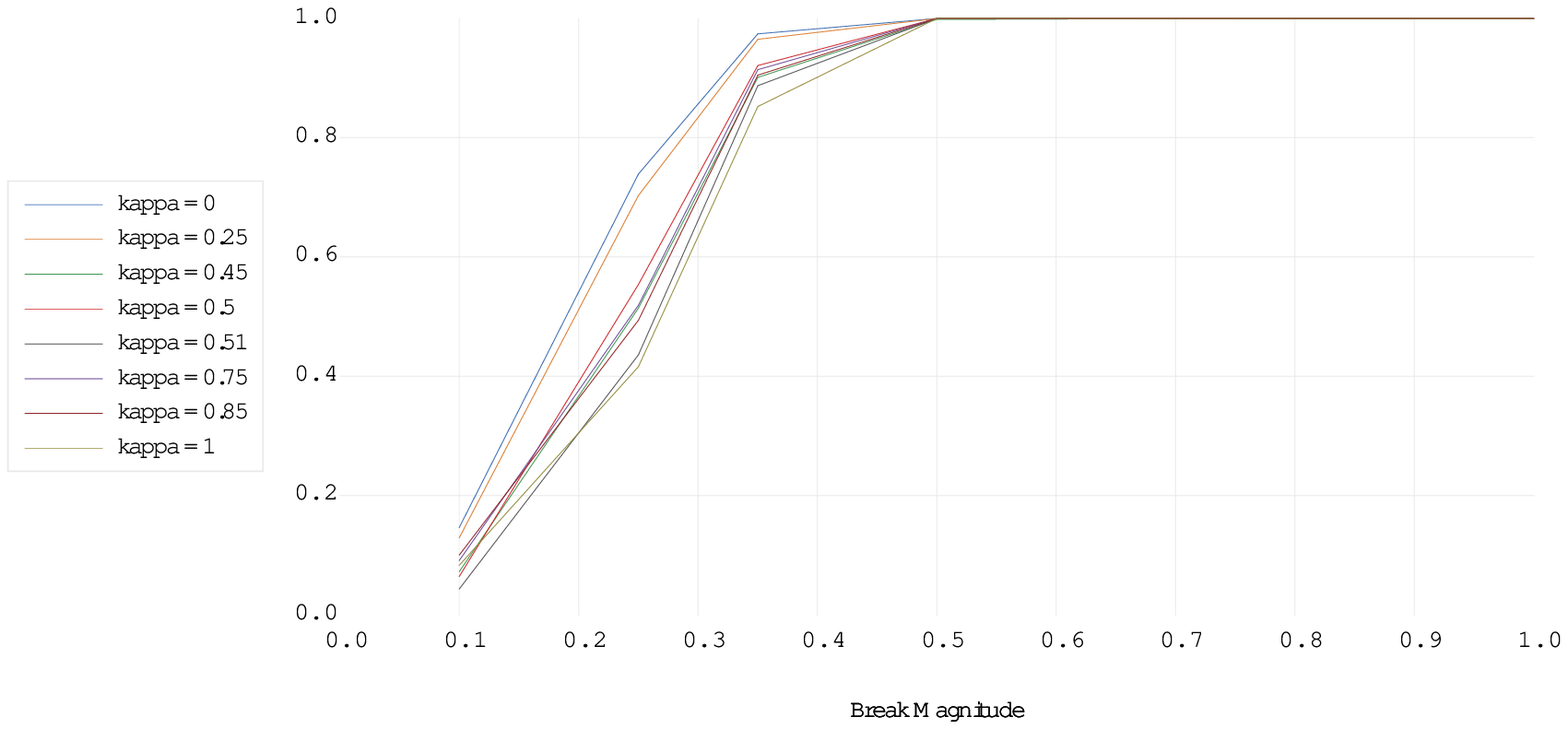}
    \captionof{subfigure}{$\beta_0=0.5$}
    \label{fig:t31}
\end{minipage}%
\begin{minipage}{0.4\textwidth}
\centering
   \includegraphics[scale=0.4]{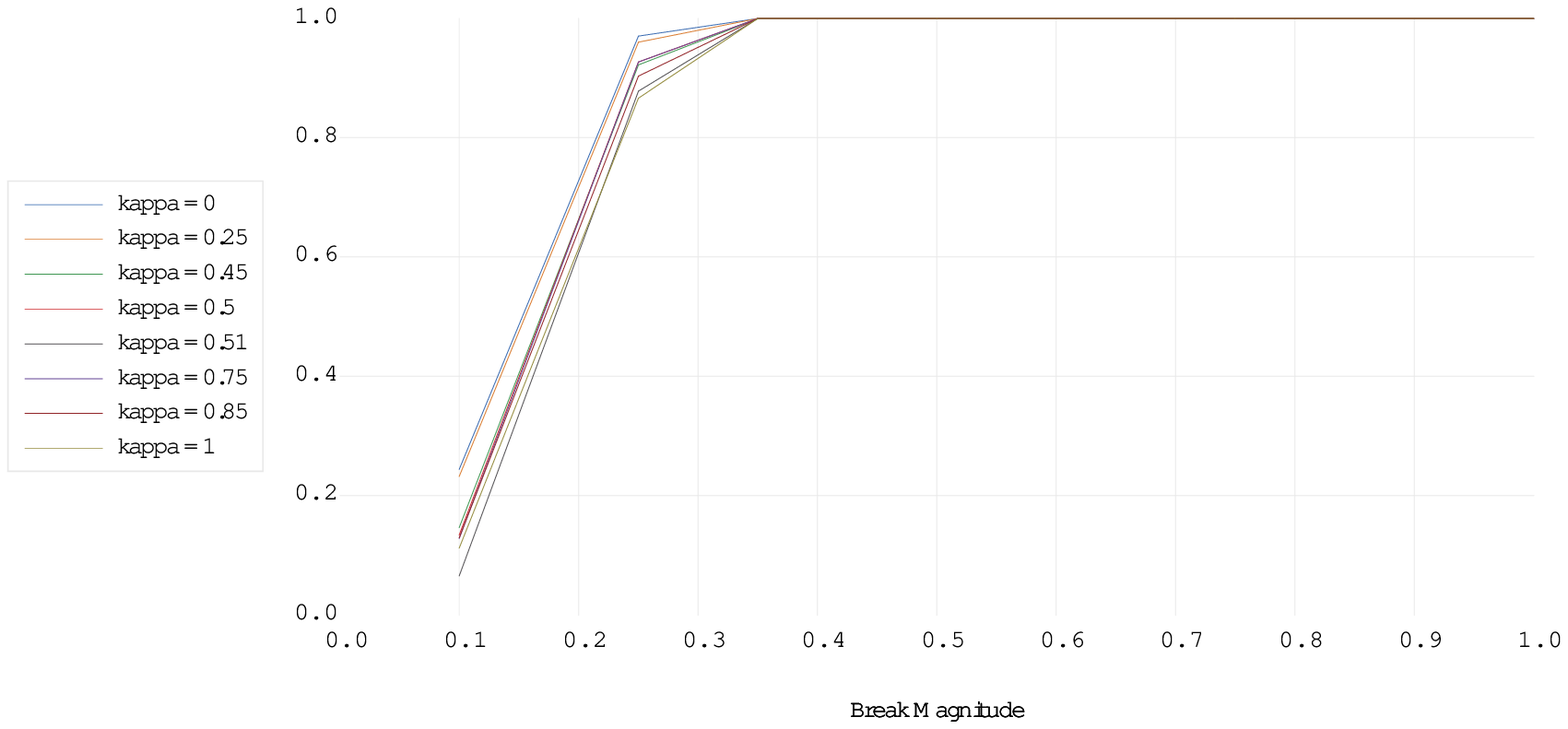}
    \captionof{subfigure}{$\beta_0=0.75$}
    \label{fig:t32}
\end{minipage} \\[0.25cm]
\par
\hspace{-2.5cm} 
\begin{minipage}{0.4\textwidth}
\centering
    \includegraphics[scale=0.4]{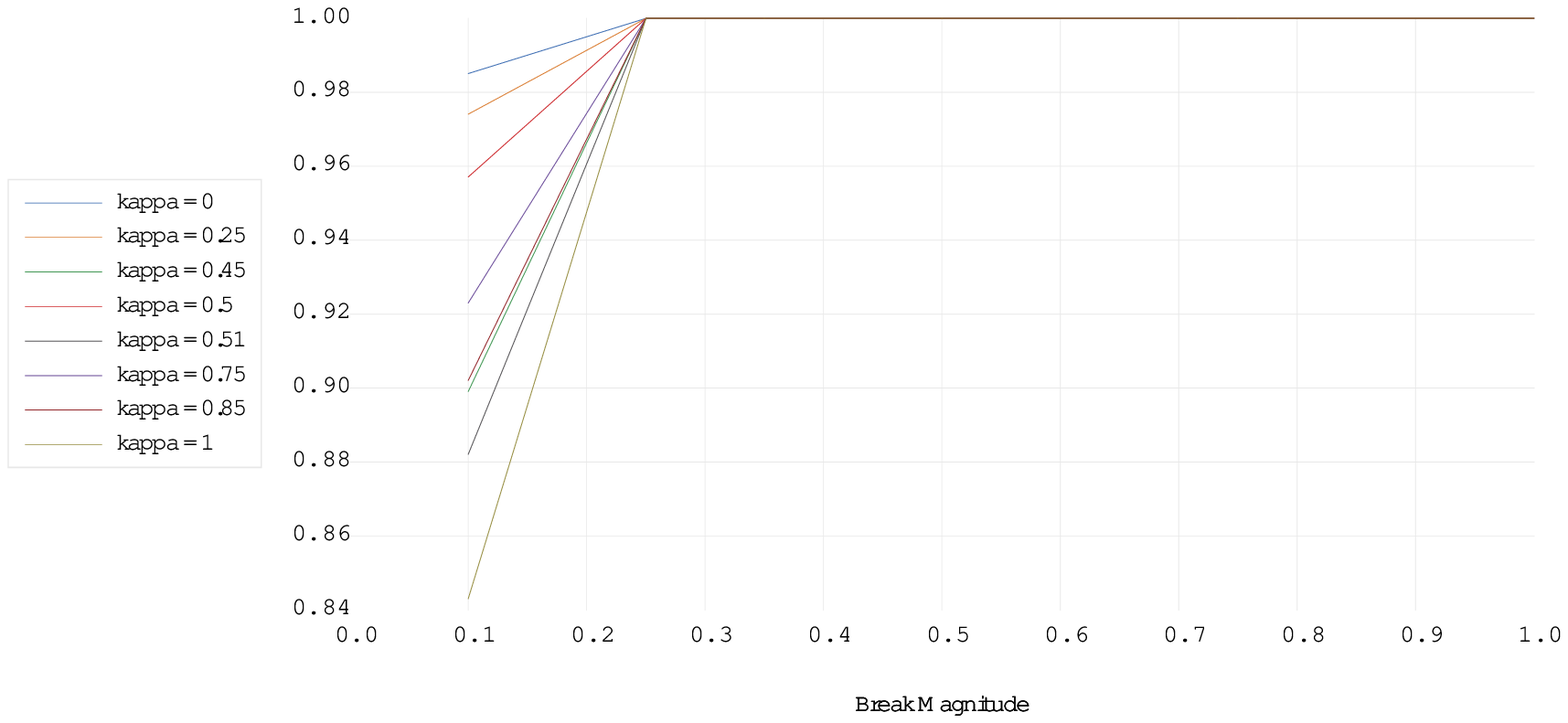}
    \captionof{subfigure}{$\beta_0=1$}
    \label{fig:t33}
\end{minipage}%
\begin{minipage}{0.4\textwidth}
\centering
   \includegraphics[scale=0.4]{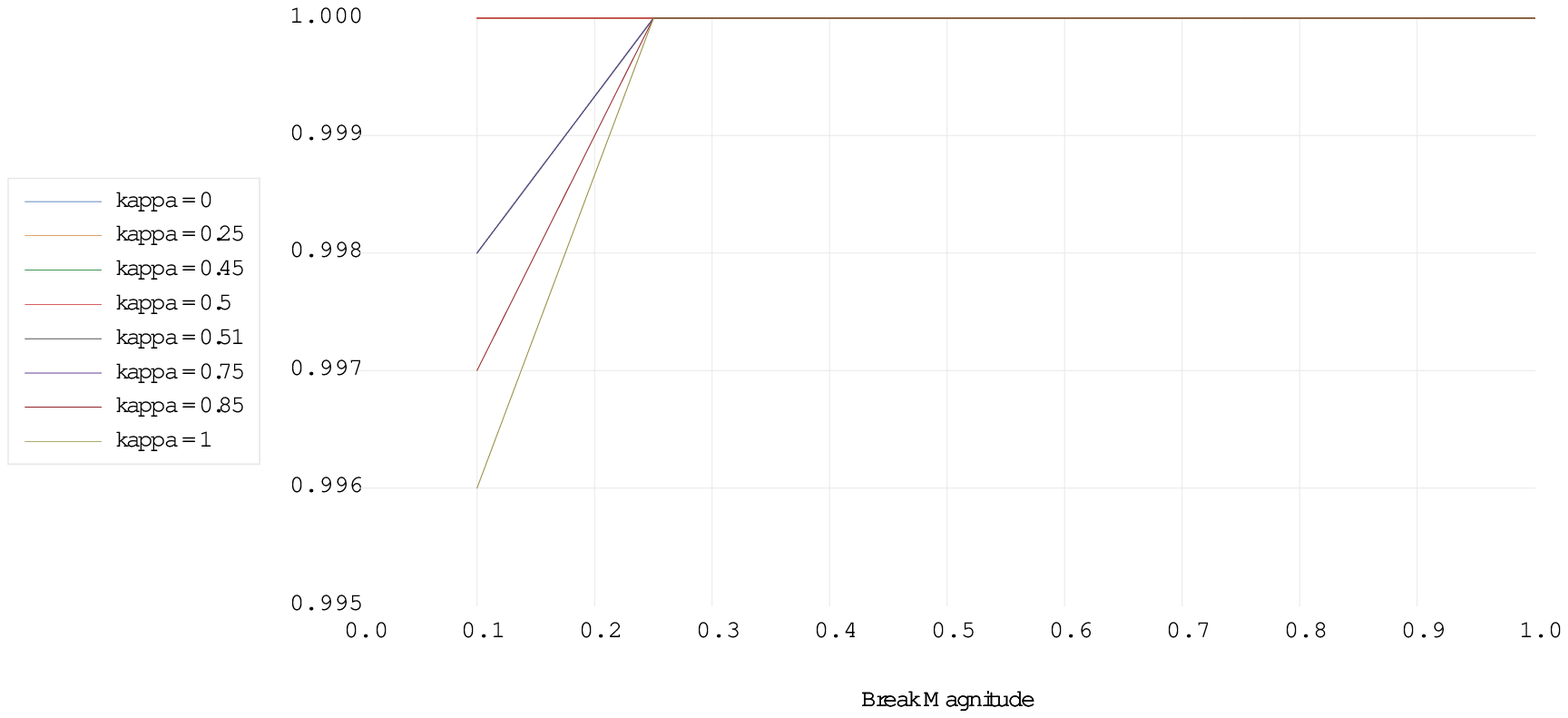}
   \captionof{subfigure}{$\beta_0=1.05$}
    \label{fig:t34}
\end{minipage} \\[0.25cm]
\par
%\caption*{.}
\end{figure}

\begin{figure}[!b]
\caption{Empirical rejection frequencies under alternatives -
heteroskedasticity in $\protect\epsilon_{i,1}$ and $\protect\epsilon_{i,2}$
and mid-sample break}
\label{fig:FigHeEB}\centering
\hspace{-2.5cm} 
\begin{minipage}{0.4\textwidth}
\centering
    \includegraphics[scale=0.4]{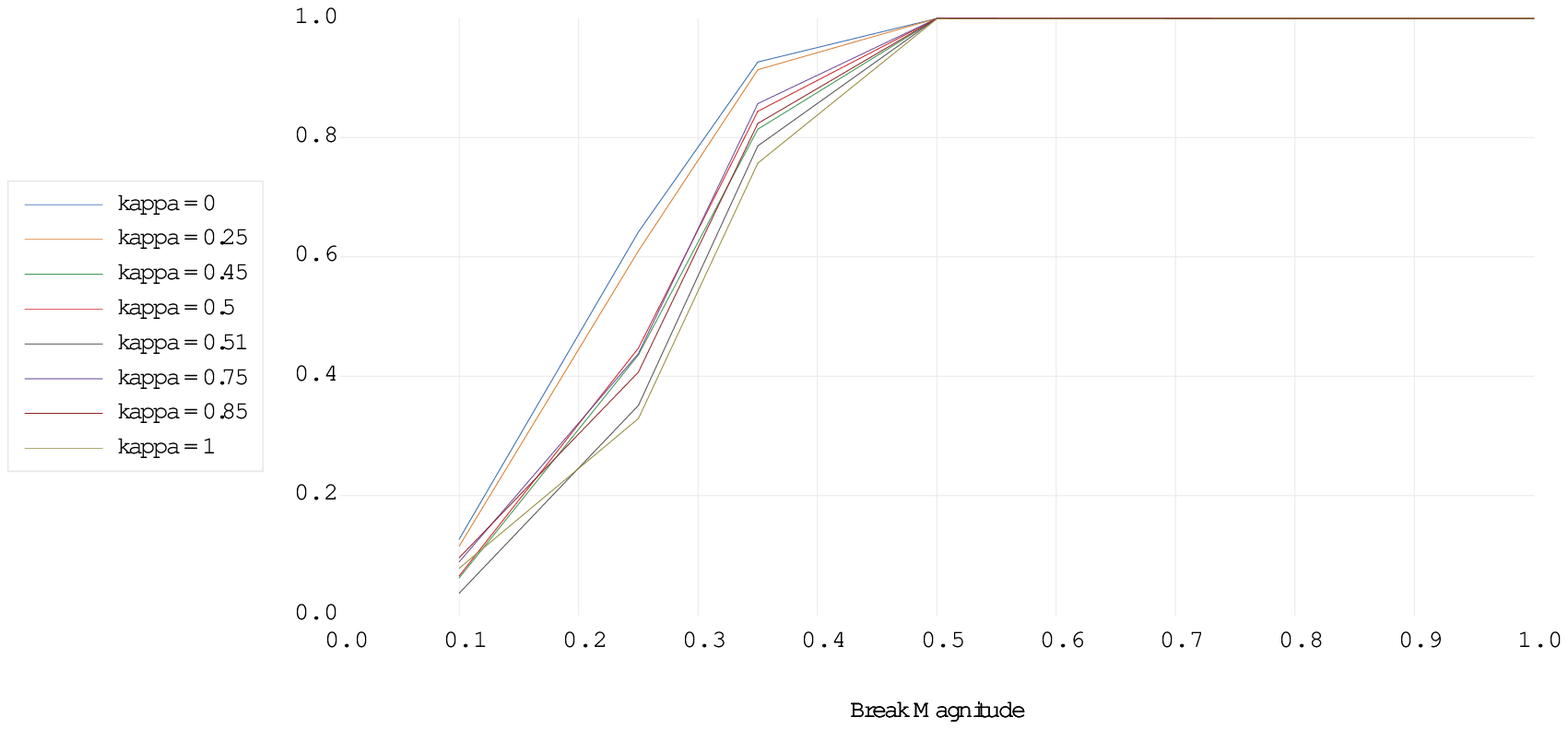}
    \captionof{subfigure}{$\beta_0=0.5$}
    \label{fig:t41}
\end{minipage}%
\begin{minipage}{0.4\textwidth}
\centering
   \includegraphics[scale=0.4]{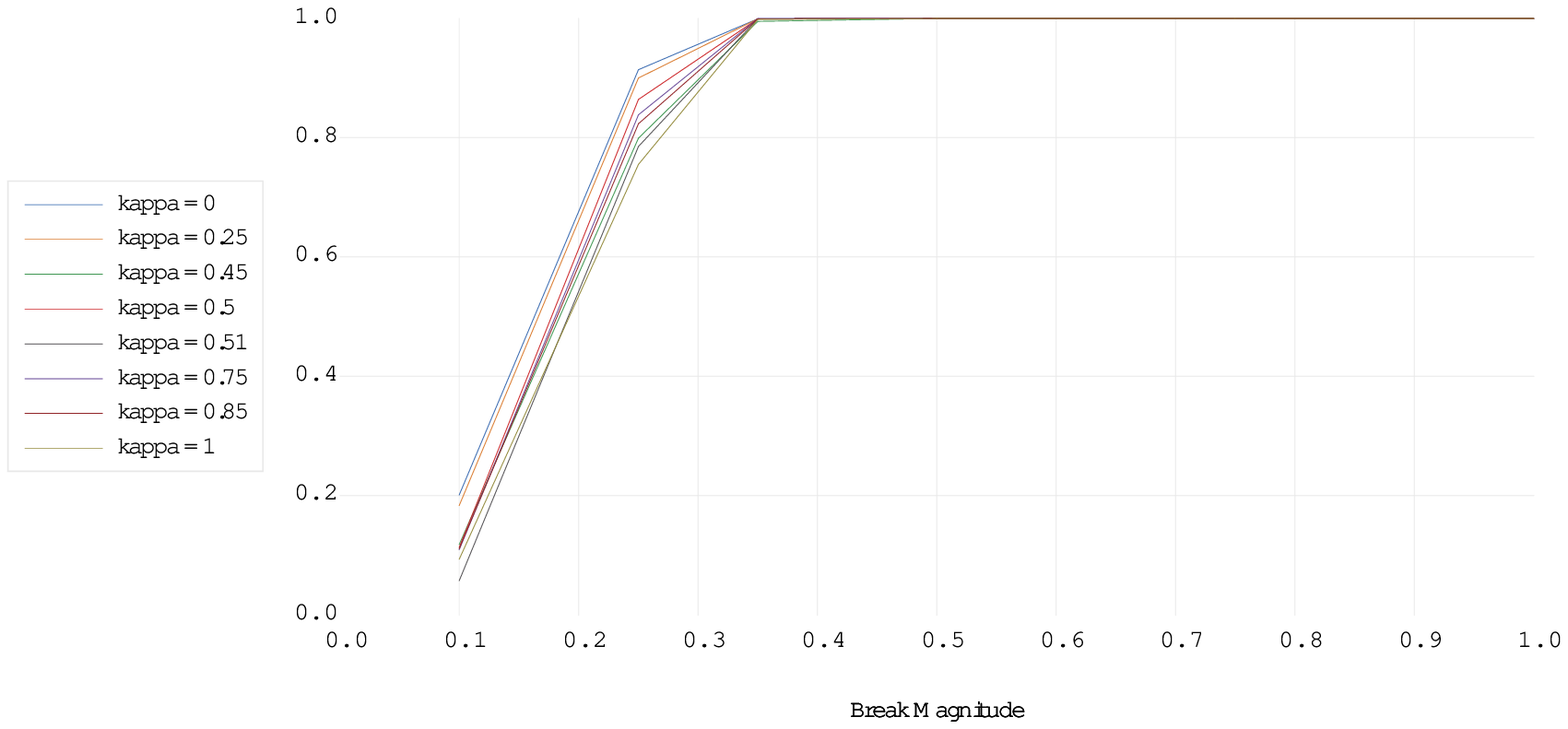}
    \captionof{subfigure}{$\beta_0=0.75$}
    \label{fig:t42}
\end{minipage} \\[0.25cm]
\par
\hspace{-2.5cm} 
\begin{minipage}{0.4\textwidth}
\centering
    \includegraphics[scale=0.4]{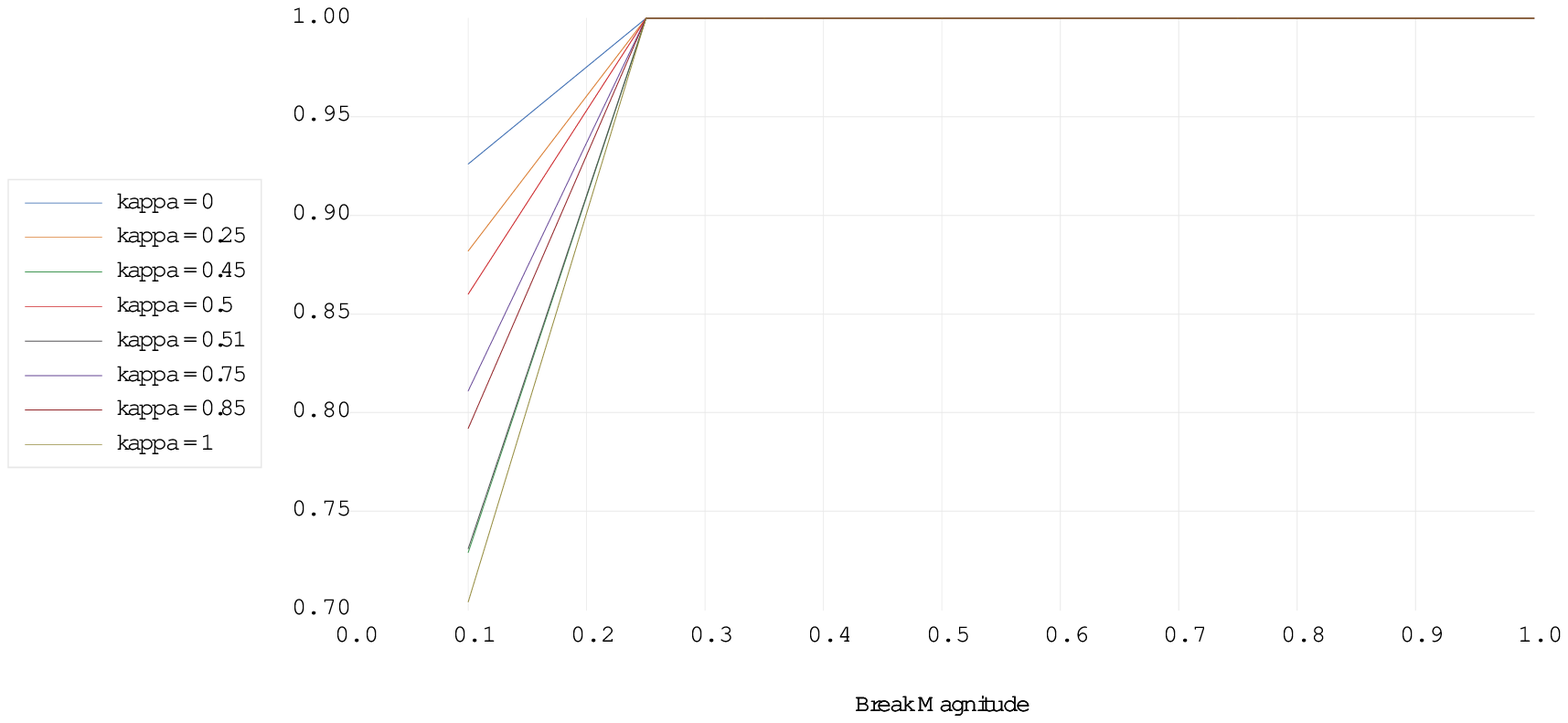}
    \captionof{subfigure}{$\beta_0=1$}
    \label{fig:t43}
\end{minipage}%
\begin{minipage}{0.4\textwidth}
\centering
   \includegraphics[scale=0.4]{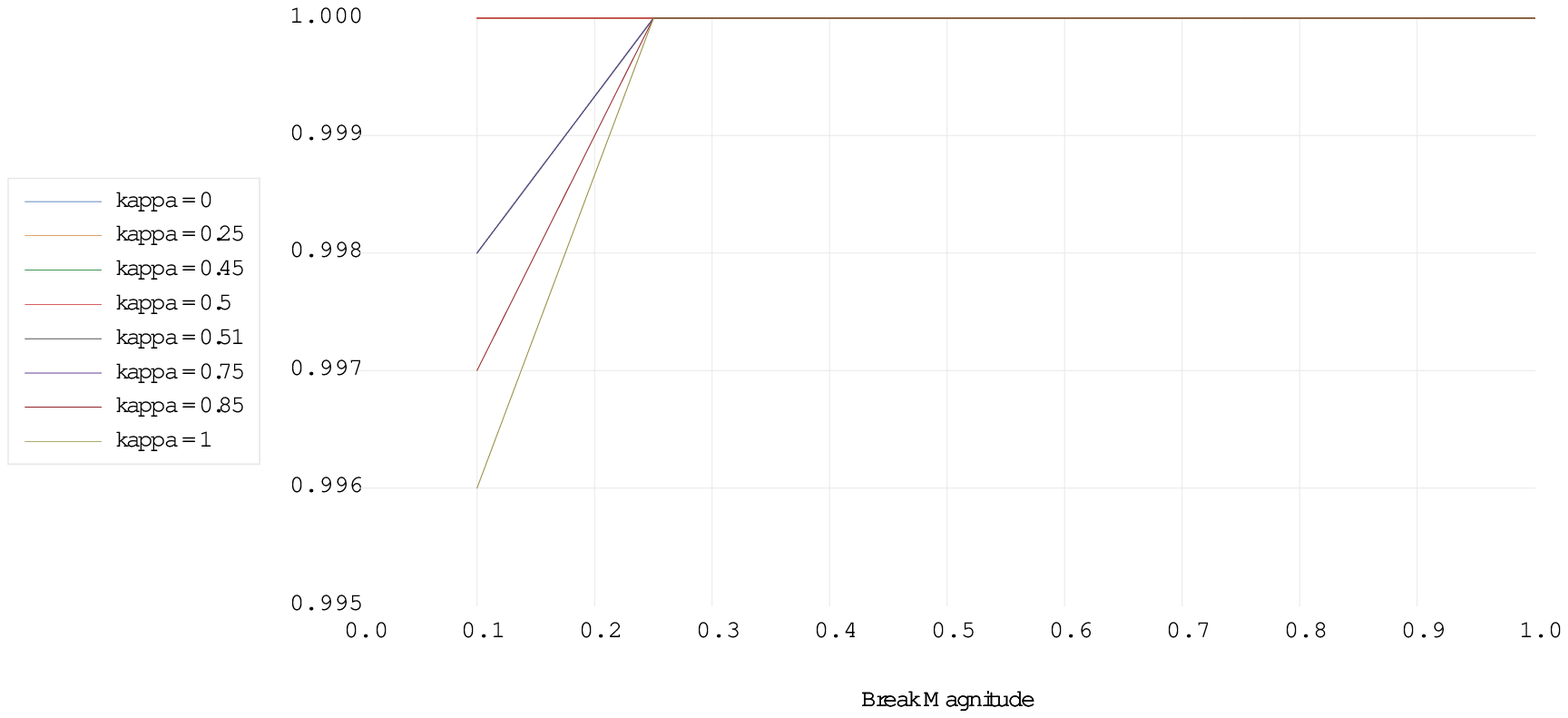}
    \captionof{subfigure}{$\beta_0=1.05$}
    \label{fig:t44}
\end{minipage} \\[0.25cm]
\par
%\caption*{.}
\end{figure}

\begin{figure}[!t]
\caption{Empirical rejection frequencies under alternatives -
homoskedasticity and end-of-sample break}
\label{fig:FigHo2}\centering
\hspace{-2.5cm} 
\begin{minipage}{0.4\textwidth}
\centering
    \includegraphics[scale=0.4]{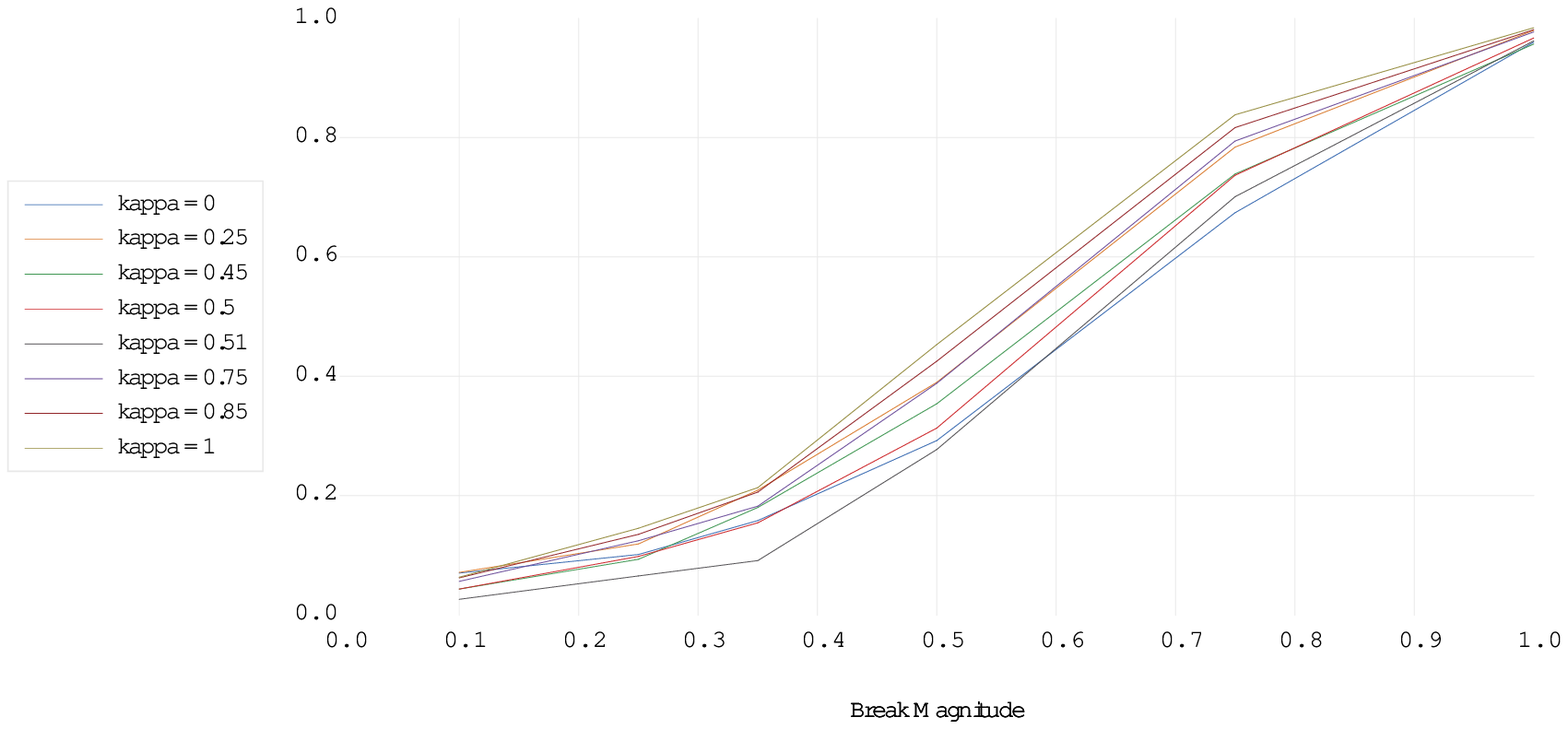}
    \captionof{subfigure}{$\beta_0=0.5$}
    \label{fig:t51}
\end{minipage}%
\begin{minipage}{0.4\textwidth}
\centering
   \includegraphics[scale=0.4]{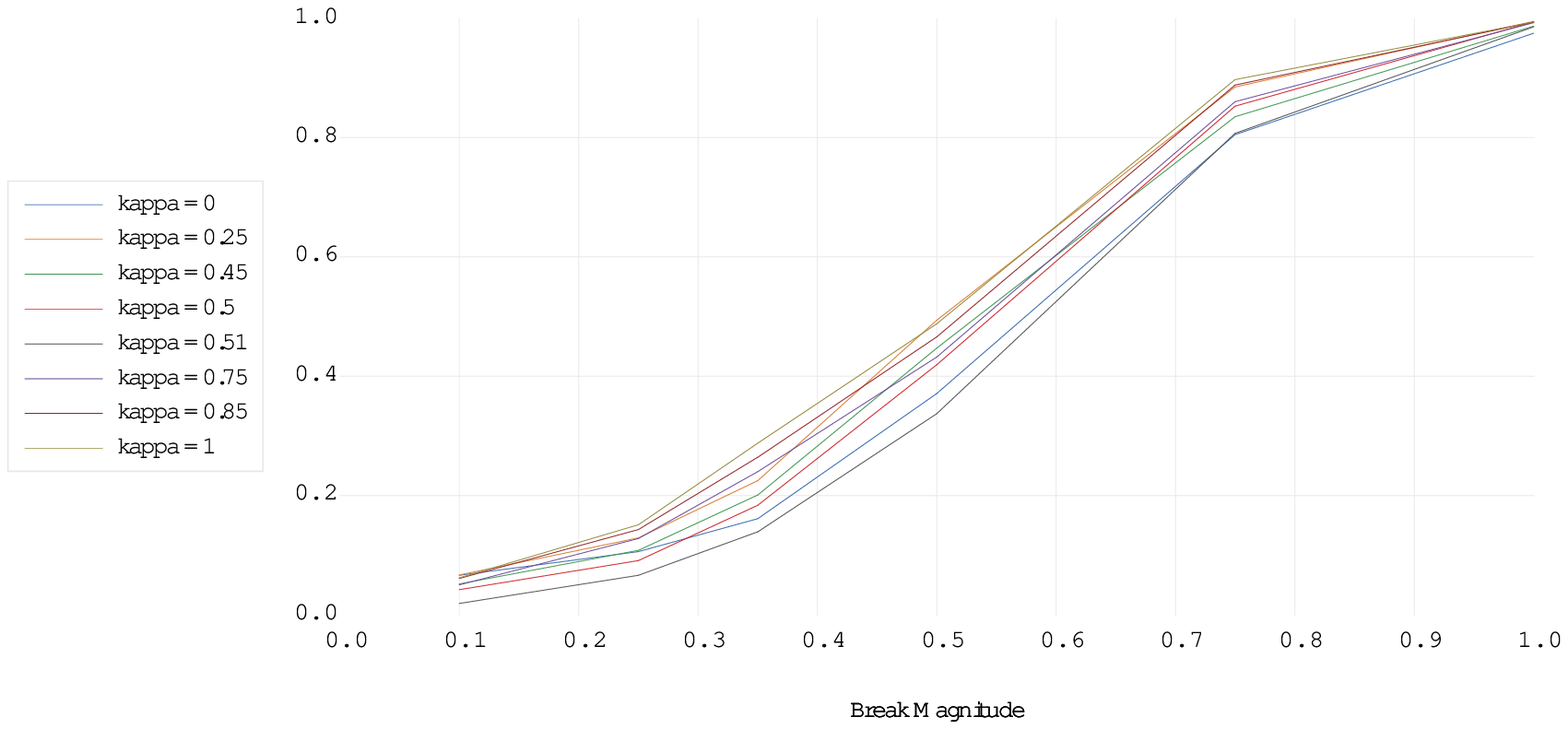}
    \captionof{subfigure}{$\beta_0=0.75$}
    \label{fig:t52}
\end{minipage} \\[0.25cm]
\par
\hspace{-2.5cm} 
\begin{minipage}{0.4\textwidth}
\centering
    \includegraphics[scale=0.4]{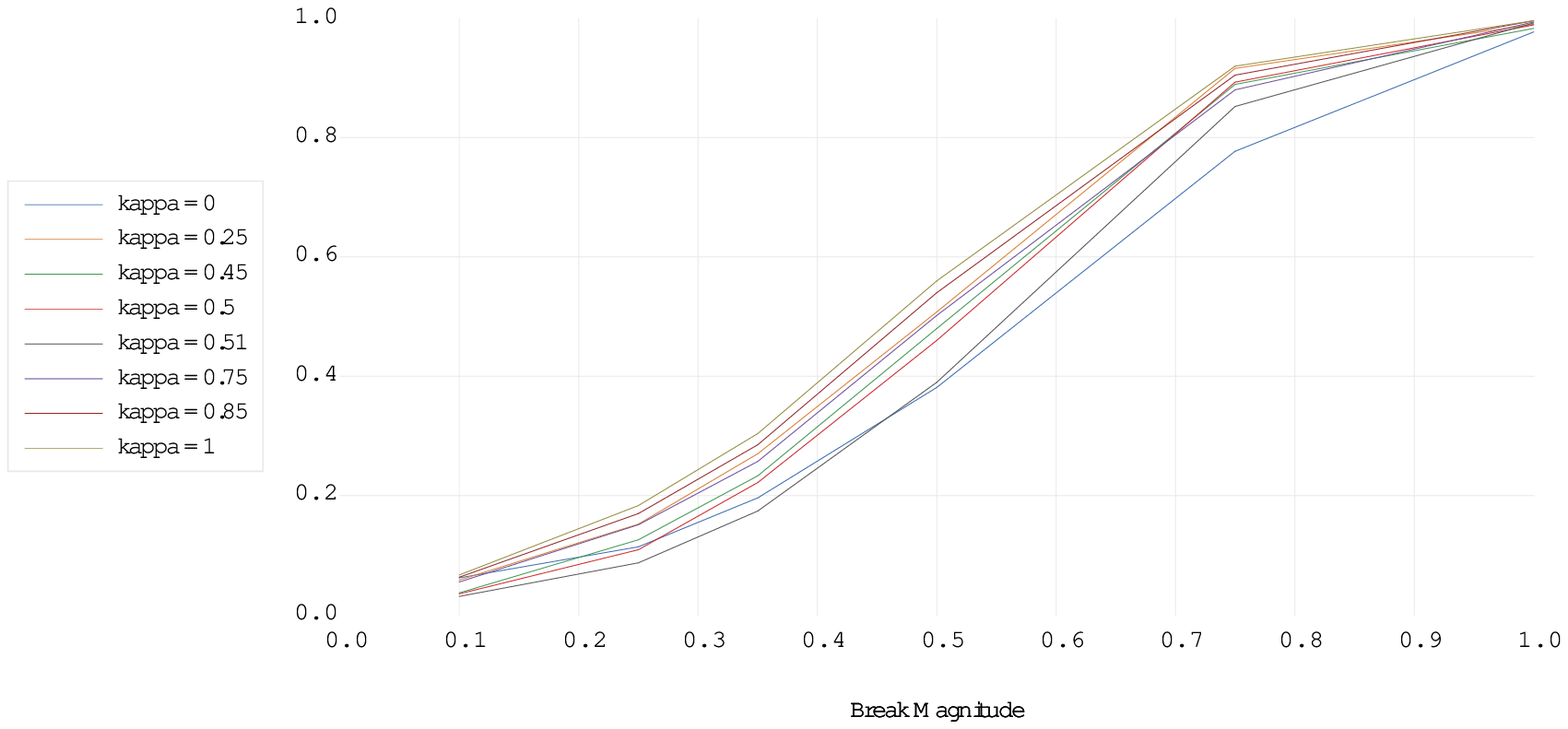}
    \captionof{subfigure}{$\beta_0=1$}
    \label{fig:t53}
\end{minipage}%
\begin{minipage}{0.4\textwidth}
\centering
   \includegraphics[scale=0.4]{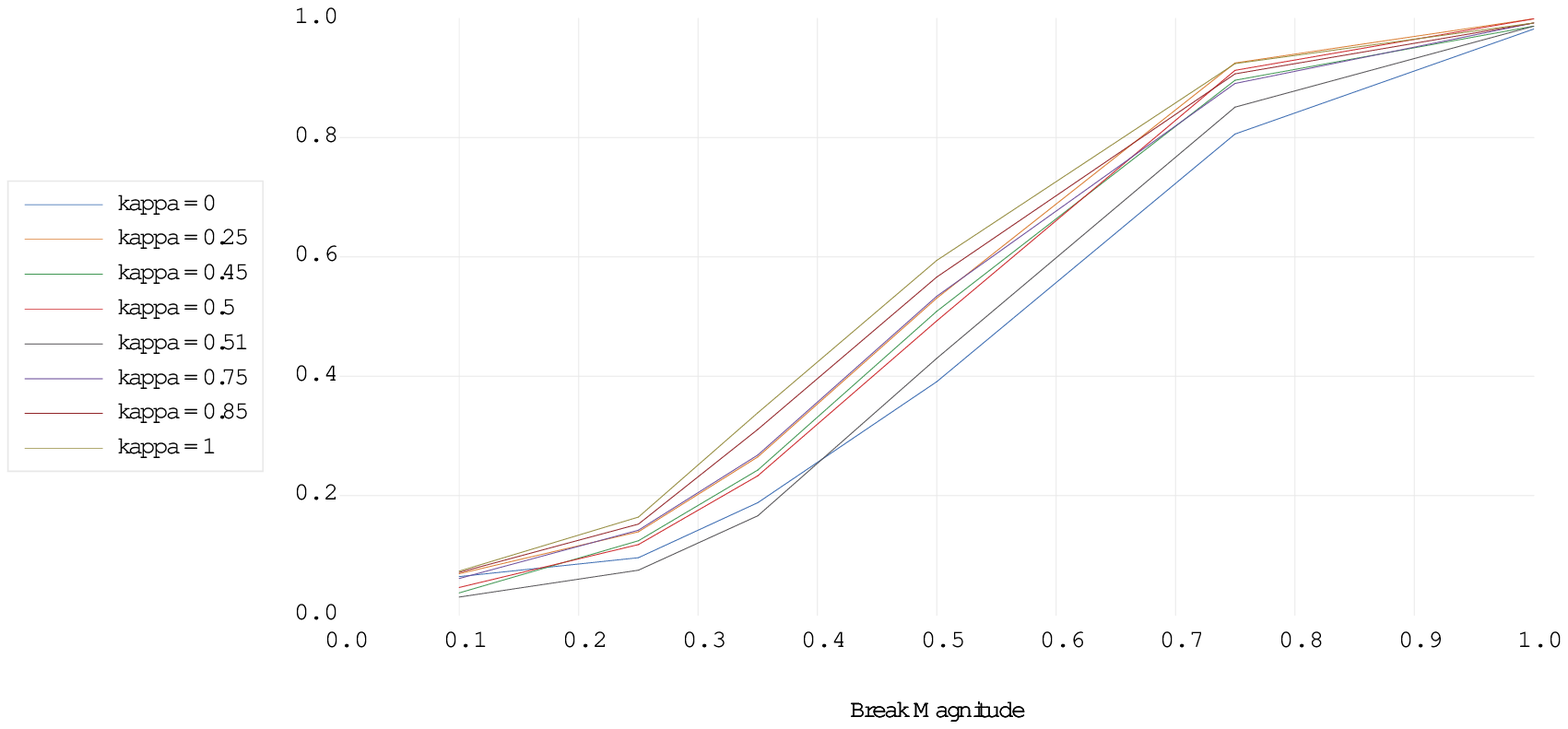}
    \captionof{subfigure}{$\beta_0=1.05$}
    \label{fig:t54}
\end{minipage} \\[0.25cm]
\par
%\caption*{.}
\end{figure}

\begin{figure}[!b]
\caption{Empirical rejection frequencies under alternatives -
heteroskedasticity in $\protect\epsilon_{i,2}$ and end-of-sample break}
\label{fig:FigHeE2}\centering
\hspace{-2.5cm} 
\begin{minipage}{0.4\textwidth}
\centering
    \includegraphics[scale=0.4]{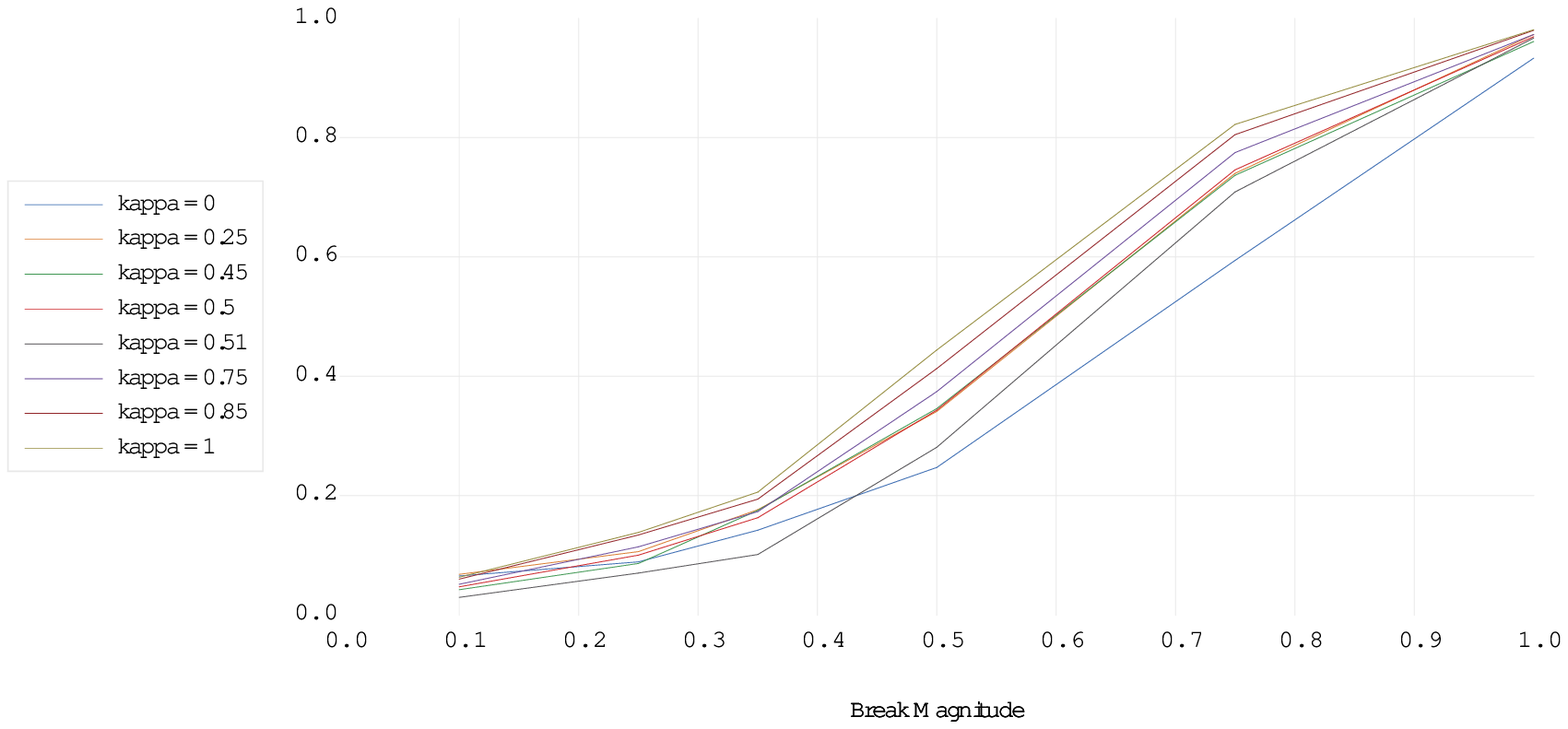}
    \captionof{subfigure}{$\beta_0=0.5$}
    \label{fig:t61}
\end{minipage}%
\begin{minipage}{0.4\textwidth}
\centering
   \includegraphics[scale=0.4]{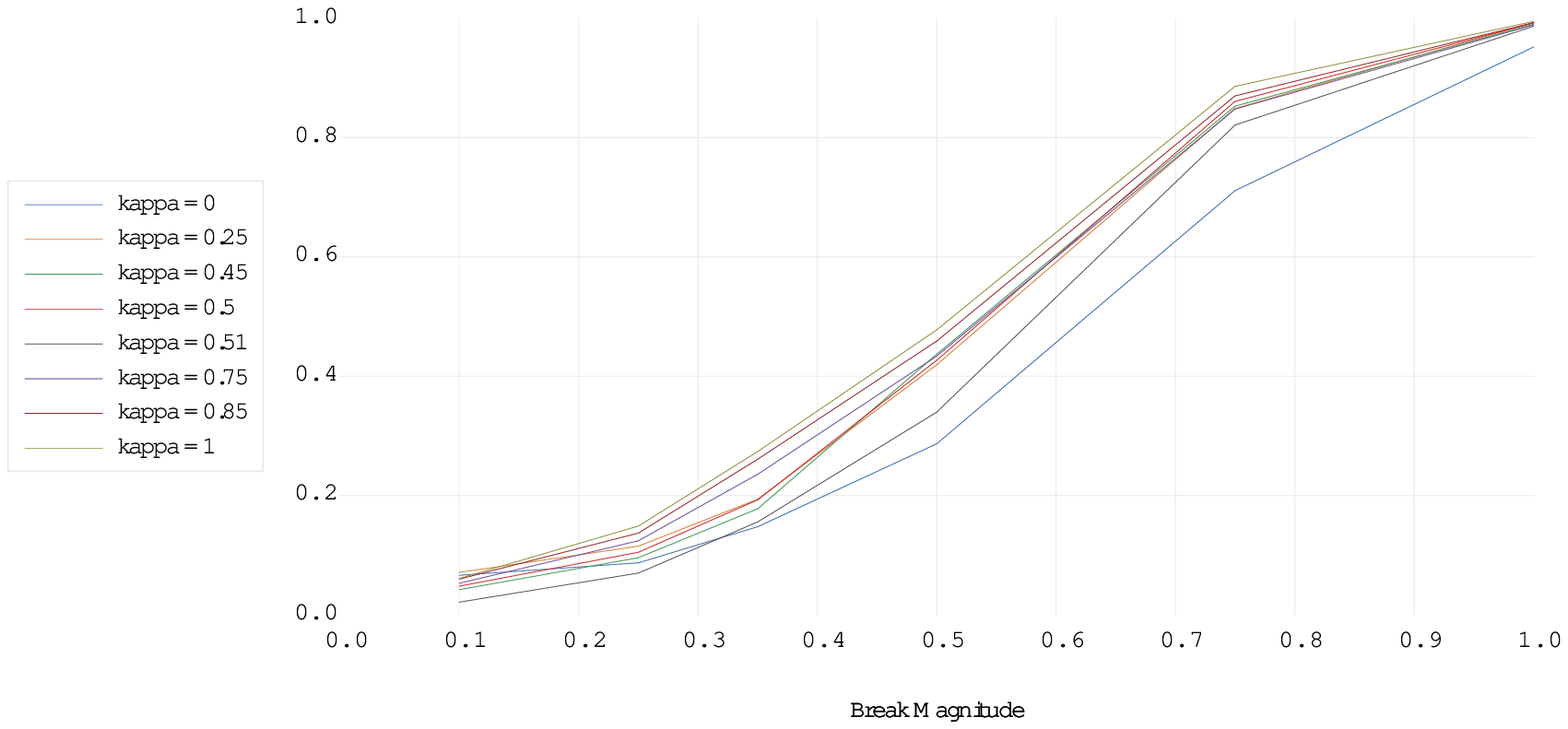}
   \captionof{subfigure}{$\beta_0=0.75$}
    \label{fig:t62}
\end{minipage} \\[0.25cm]
\par
\hspace{-2.5cm} 
\begin{minipage}{0.4\textwidth}
\centering
    \includegraphics[scale=0.4]{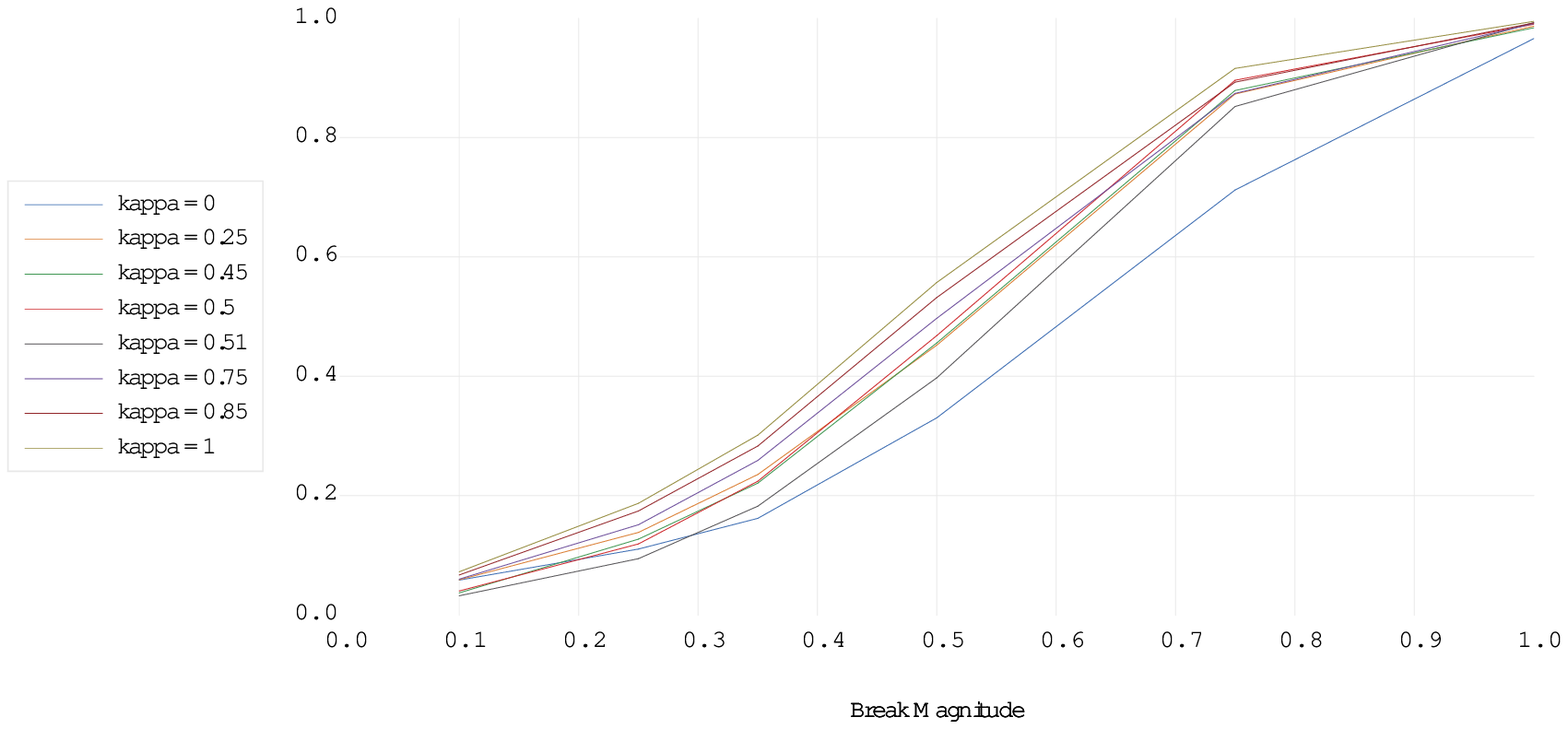}
    \captionof{subfigure}{$\beta_0=1$}
    \label{fig:t63}
\end{minipage}%
\begin{minipage}{0.4\textwidth}
\centering
   \includegraphics[scale=0.4]{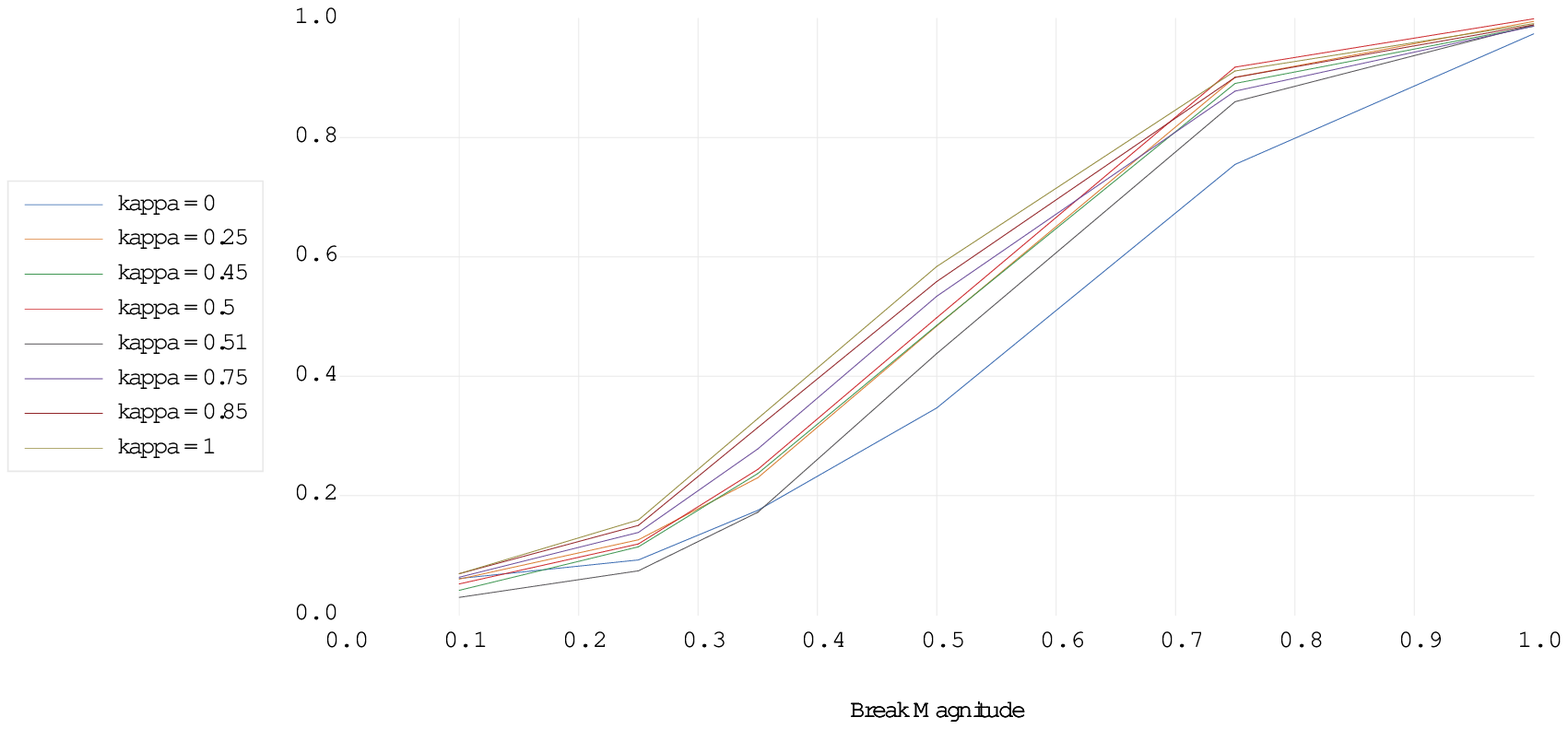}
    \captionof{subfigure}{$\beta_0=1.05$}
    \label{fig:t64}
\end{minipage} \\[0.25cm]
\par
%\caption*{.}
\end{figure}

\begin{figure}[!t]
\caption{Empirical rejection frequencies under alternatives -
heteroskedasticity in $\protect\epsilon_{i,1}$ and end-of-sample break}
\label{fig:FigHeB2}\centering
\hspace{-2.5cm} 
\begin{minipage}{0.4\textwidth}
\centering
    \includegraphics[scale=0.4]{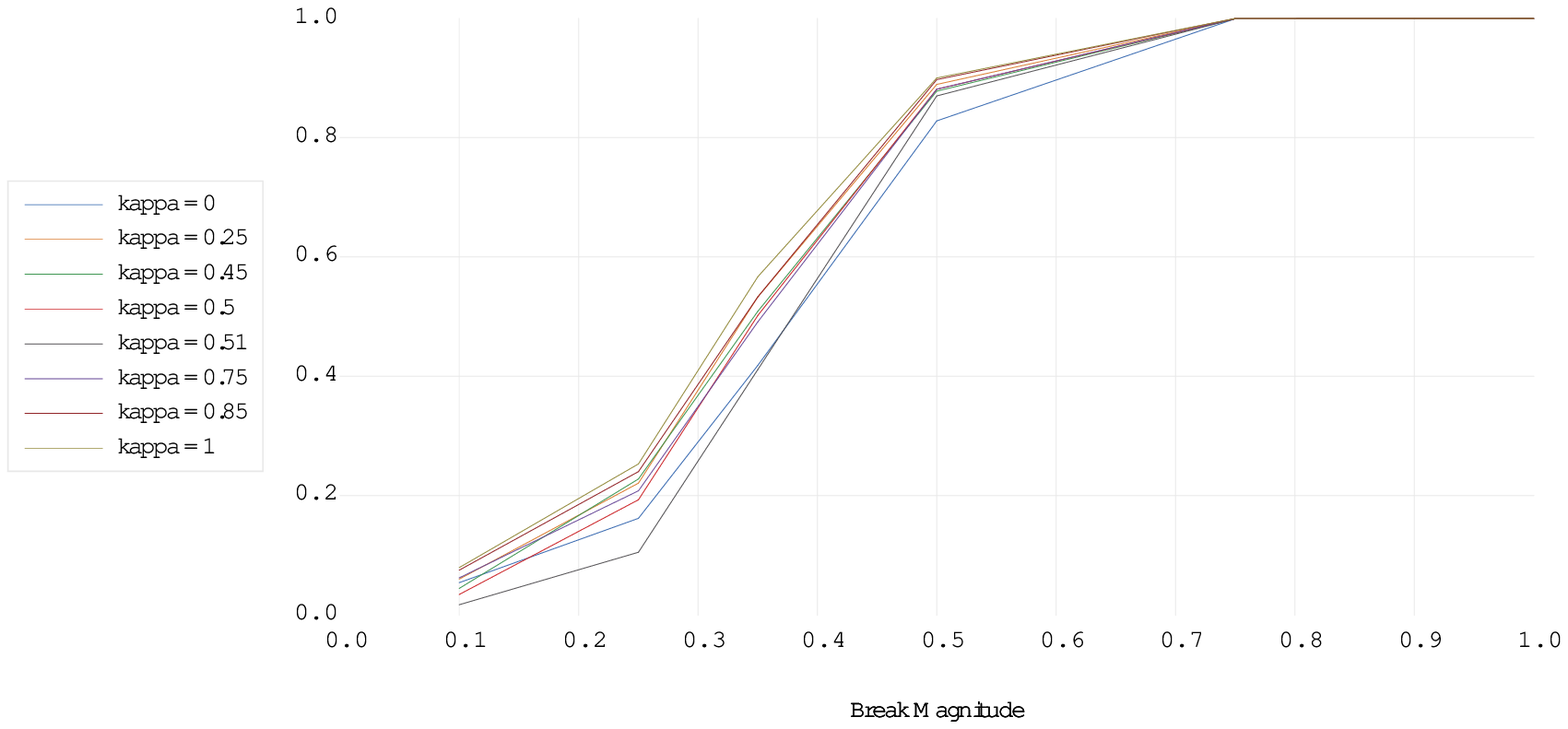}
    \captionof{subfigure}{$\beta_0=0.5$}
    \label{fig:t71}
\end{minipage}%
\begin{minipage}{0.4\textwidth}
\centering
   \includegraphics[scale=0.4]{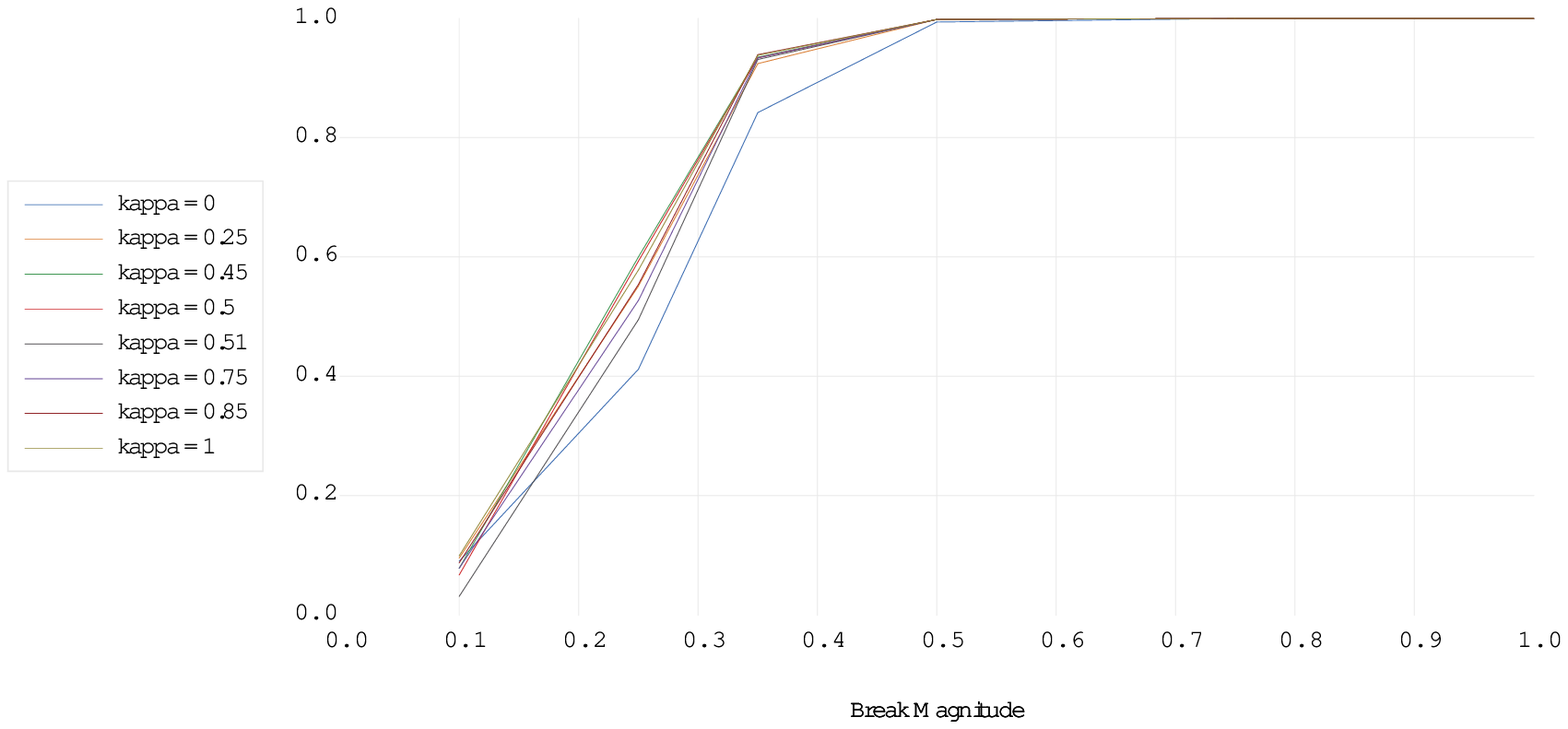}
    \captionof{subfigure}{$\beta_0=0.75$}
    \label{fig:t72}
\end{minipage} \\[0.25cm]
\par
\hspace{-2.5cm} 
\begin{minipage}{0.4\textwidth}
\centering
    \includegraphics[scale=0.4]{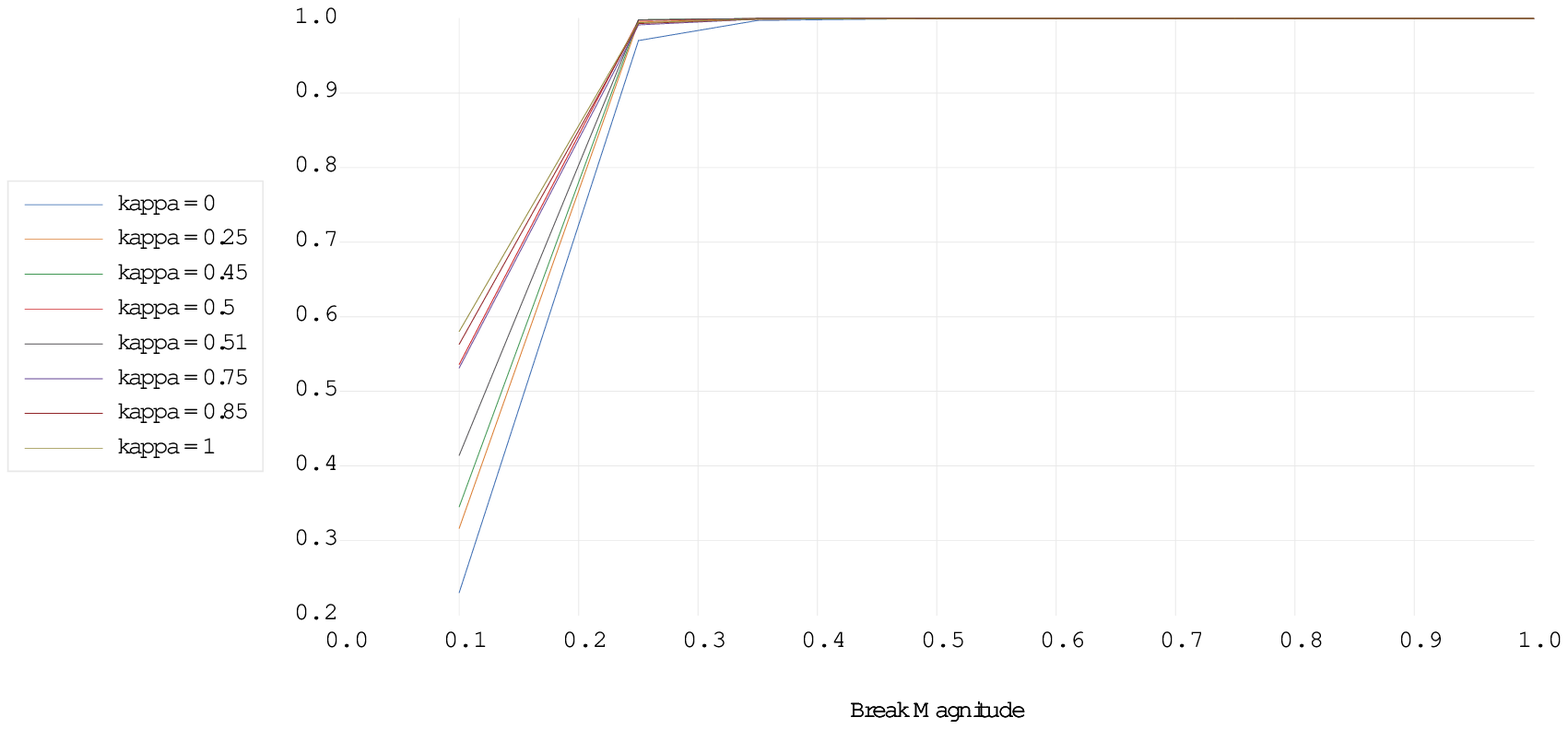}
    \captionof{subfigure}{$\beta_0=1$}
    \label{fig:t73}
\end{minipage}%
\begin{minipage}{0.4\textwidth}
\centering
   \includegraphics[scale=0.4]{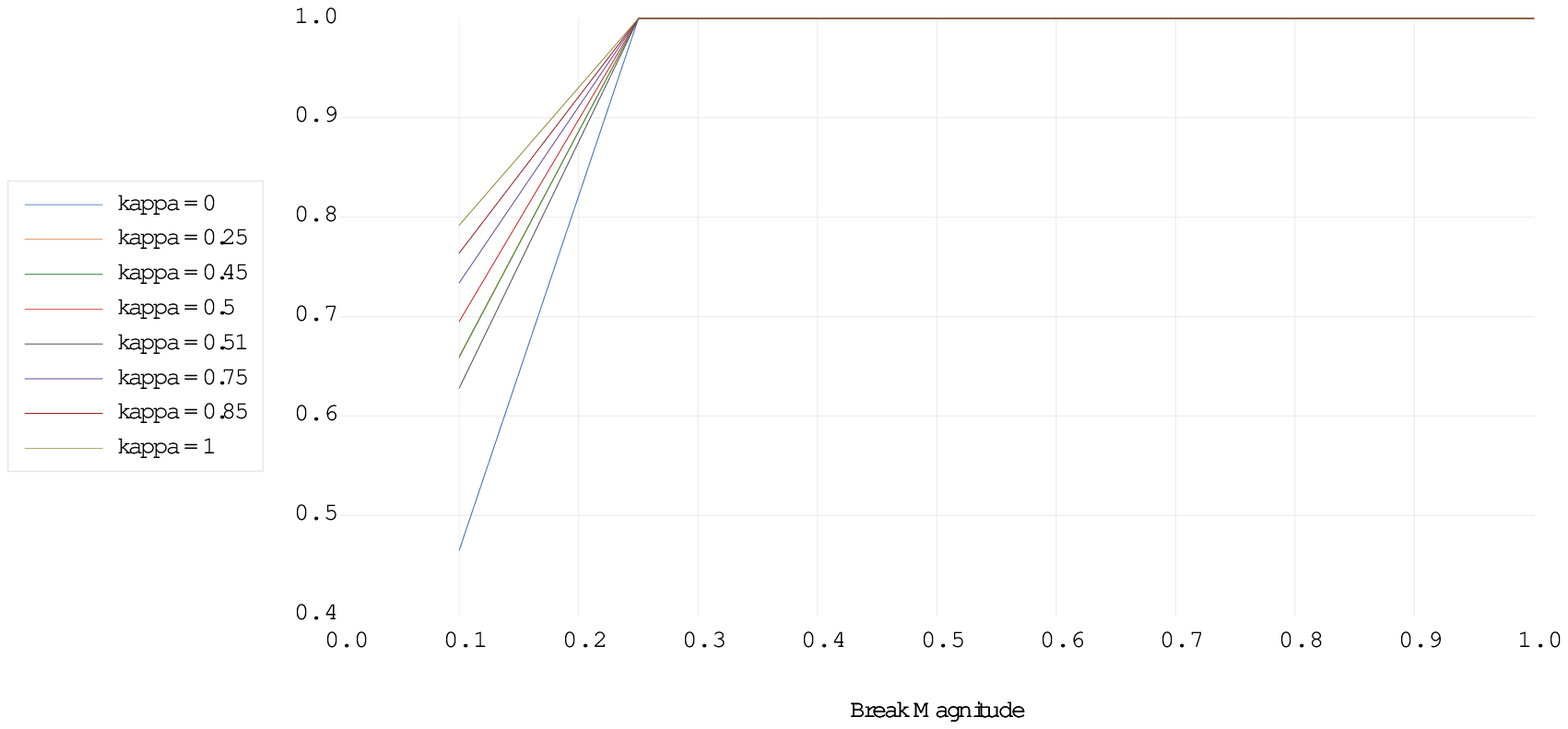}
   \captionof{subfigure}{$\beta_0=1.05$}
    \label{fig:t74}
\end{minipage} \\[0.25cm]
\par
%\caption*{.}
\end{figure}

\begin{figure}[!b]
\caption{Empirical rejection frequencies under alternatives -
heteroskedasticity in $\protect\epsilon_{i,1}$ and $\protect\epsilon_{i,2}$
and end-of-sample break}
\label{fig:FigHeEB2}\centering
\hspace{-2.5cm} 
\begin{minipage}{0.4\textwidth}
\centering
    \includegraphics[scale=0.4]{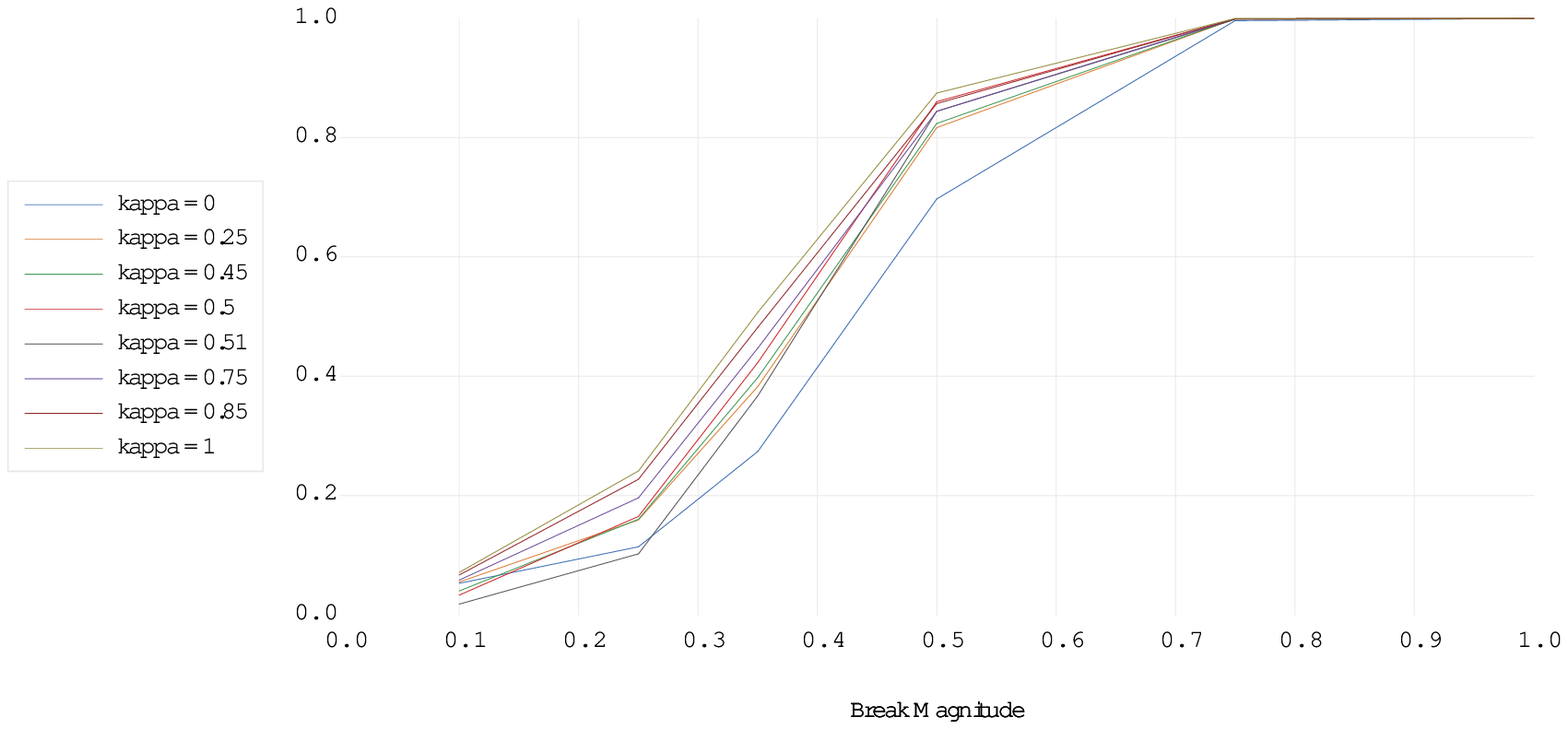}
    \captionof{subfigure}{$\beta_0=0.5$}
    \label{fig:t81}
\end{minipage}%
\begin{minipage}{0.4\textwidth}
\centering
   \includegraphics[scale=0.4]{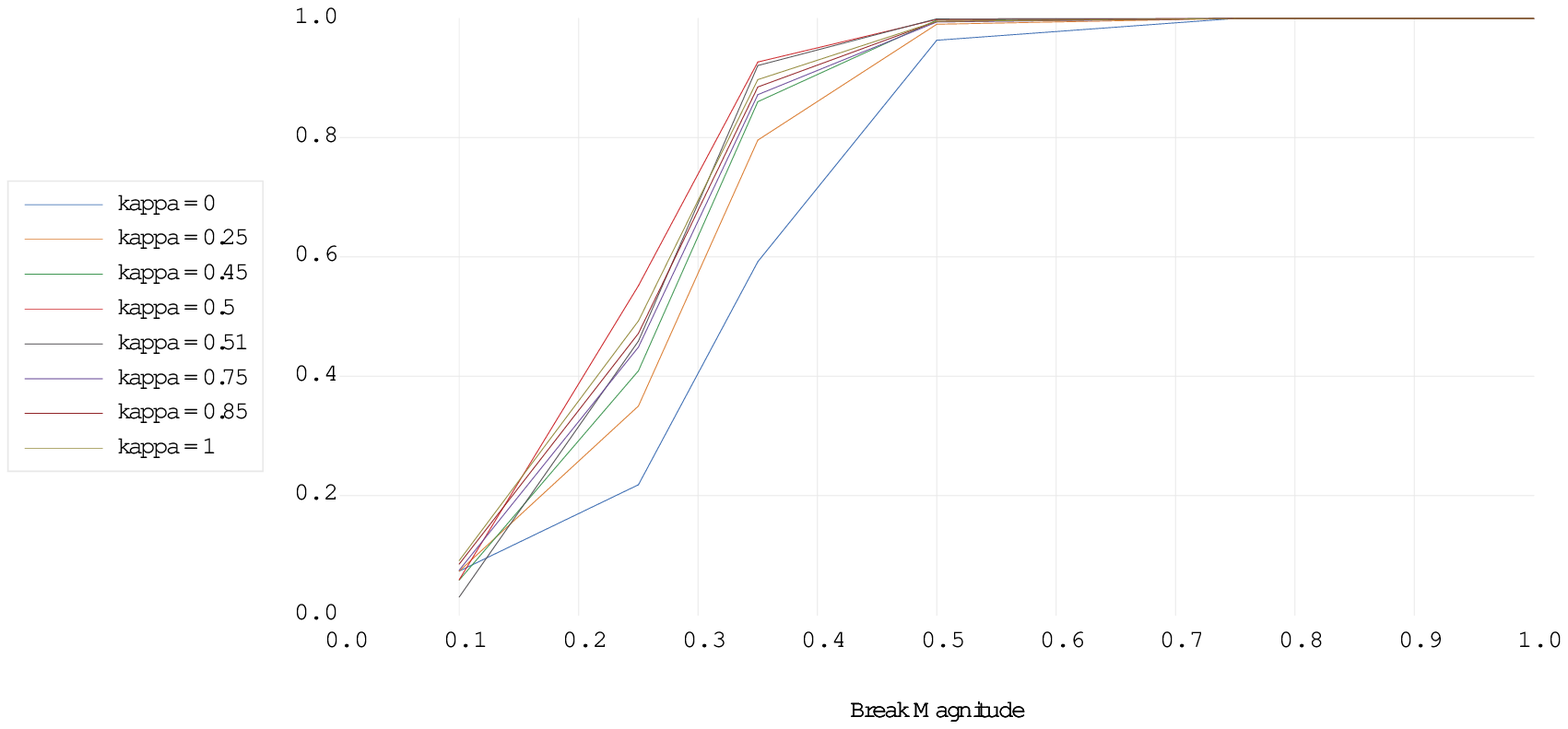}
    \captionof{subfigure}{$\beta_0=0.75$}
    \label{fig:t82}
\end{minipage} \\[0.25cm]
\par
\hspace{-2.5cm} 
\begin{minipage}{0.4\textwidth}
\centering
    \includegraphics[scale=0.4]{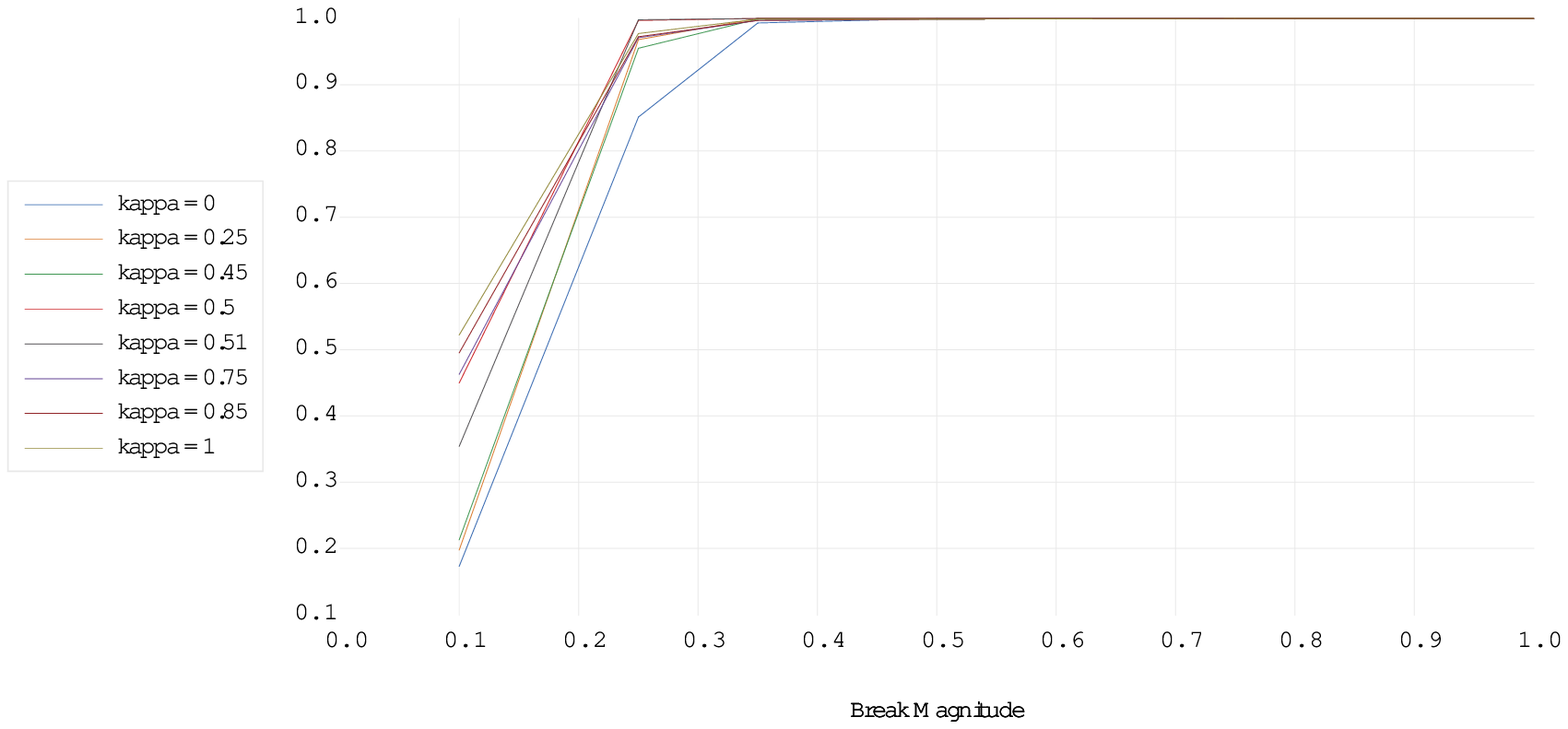}
    \captionof{subfigure}{$\beta_0=1$}
    \label{fig:t83}
\end{minipage}%
\begin{minipage}{0.4\textwidth}
\centering
   \includegraphics[scale=0.4]{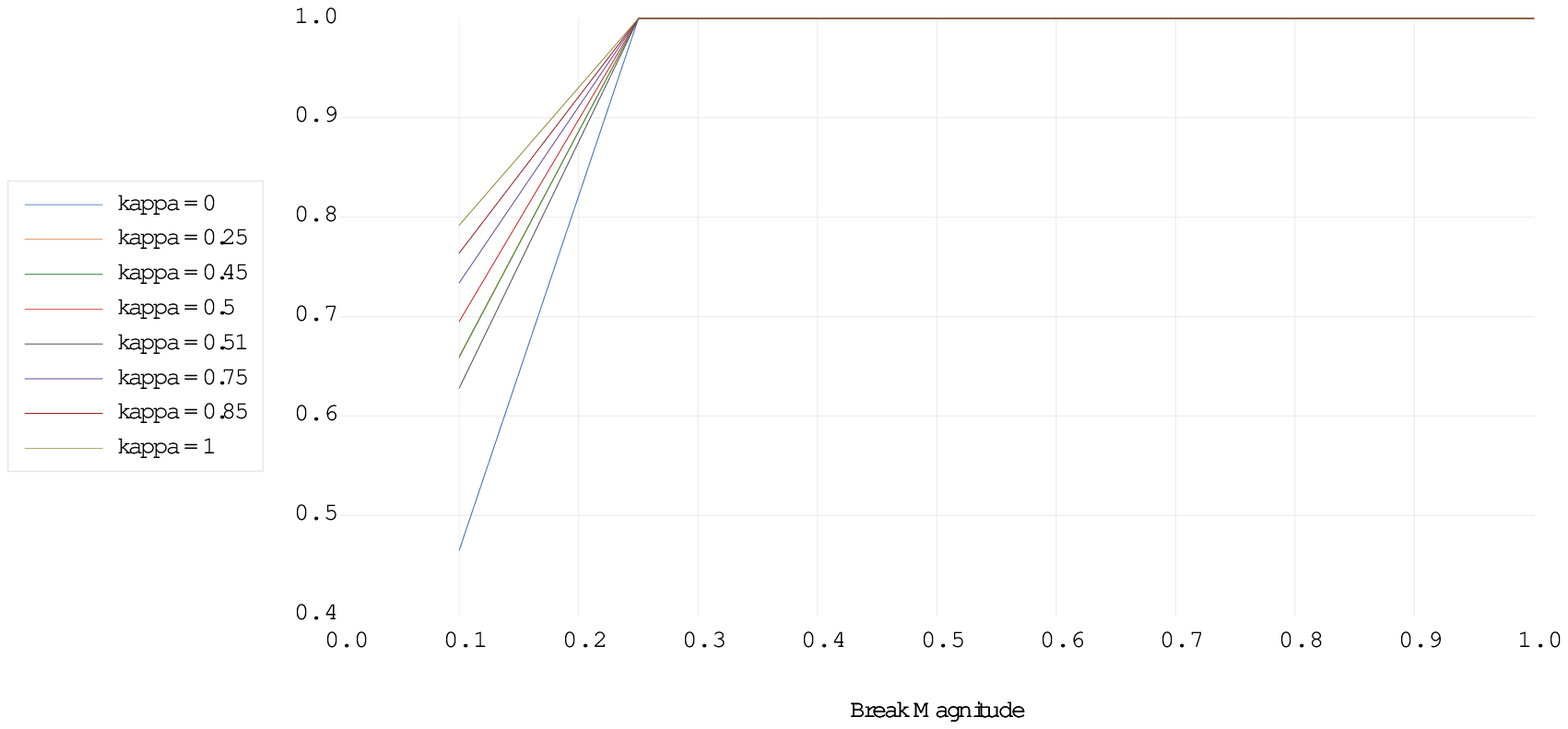}
    \captionof{subfigure}{$\beta_0=1.05$}
    \label{fig:t84}
\end{minipage} \\[0.25cm]
\par
%\caption*{.}
\end{figure}

\begin{figure}[!b]
\caption{Empirical rejection frequencies - heteroskedasticity in $\protect%
\epsilon_{i,2}$ and mid-sample break}
\label{fig:bubb1}\centering
\hspace{-2.5cm} 
\begin{minipage}{0.4\textwidth}
\centering
    \includegraphics[scale=0.4]{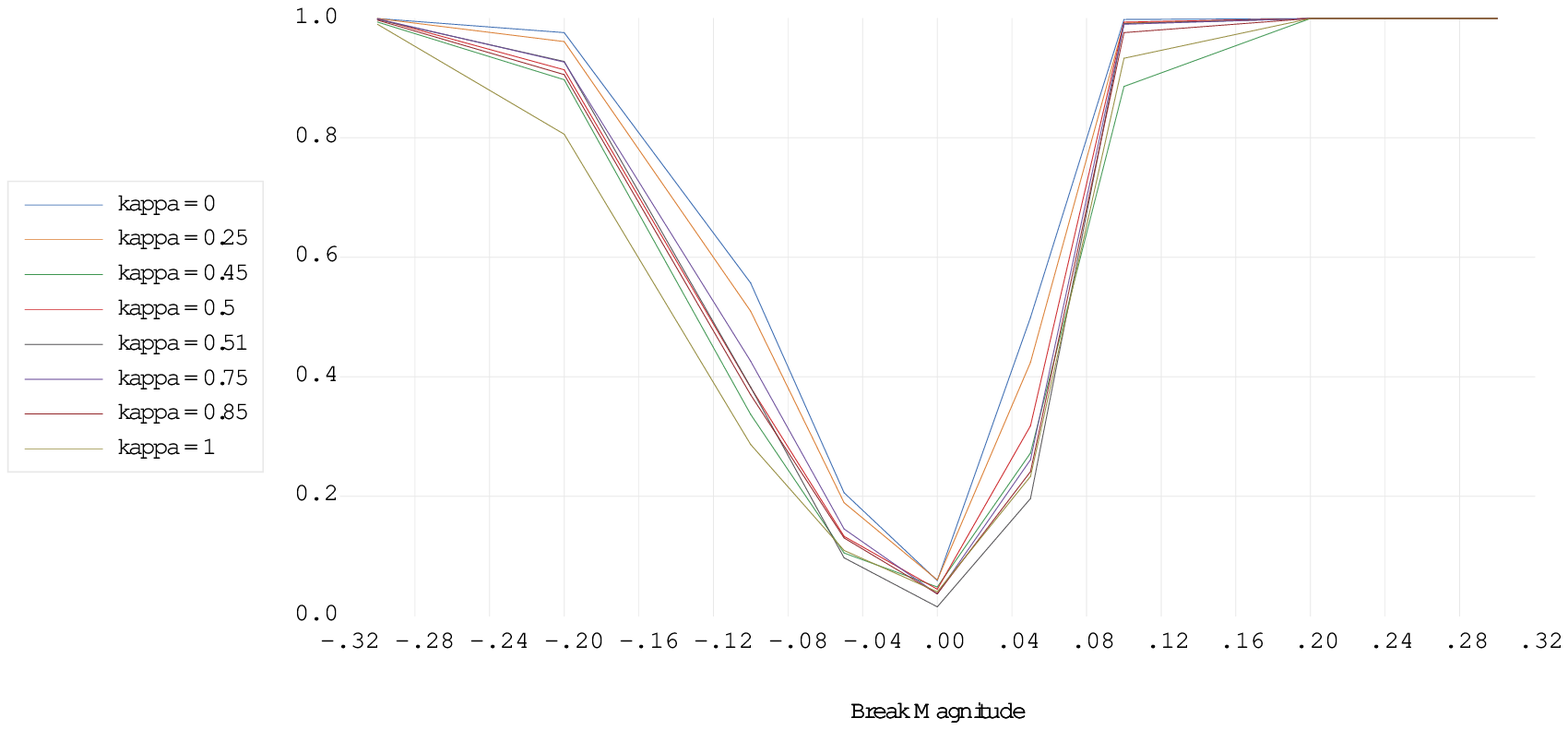}
    \captionof{subfigure}{$\beta_0=0.98$}
    \label{fig:phi98m}
\end{minipage}%
\begin{minipage}{0.4\textwidth}
\centering
   \includegraphics[scale=0.4]{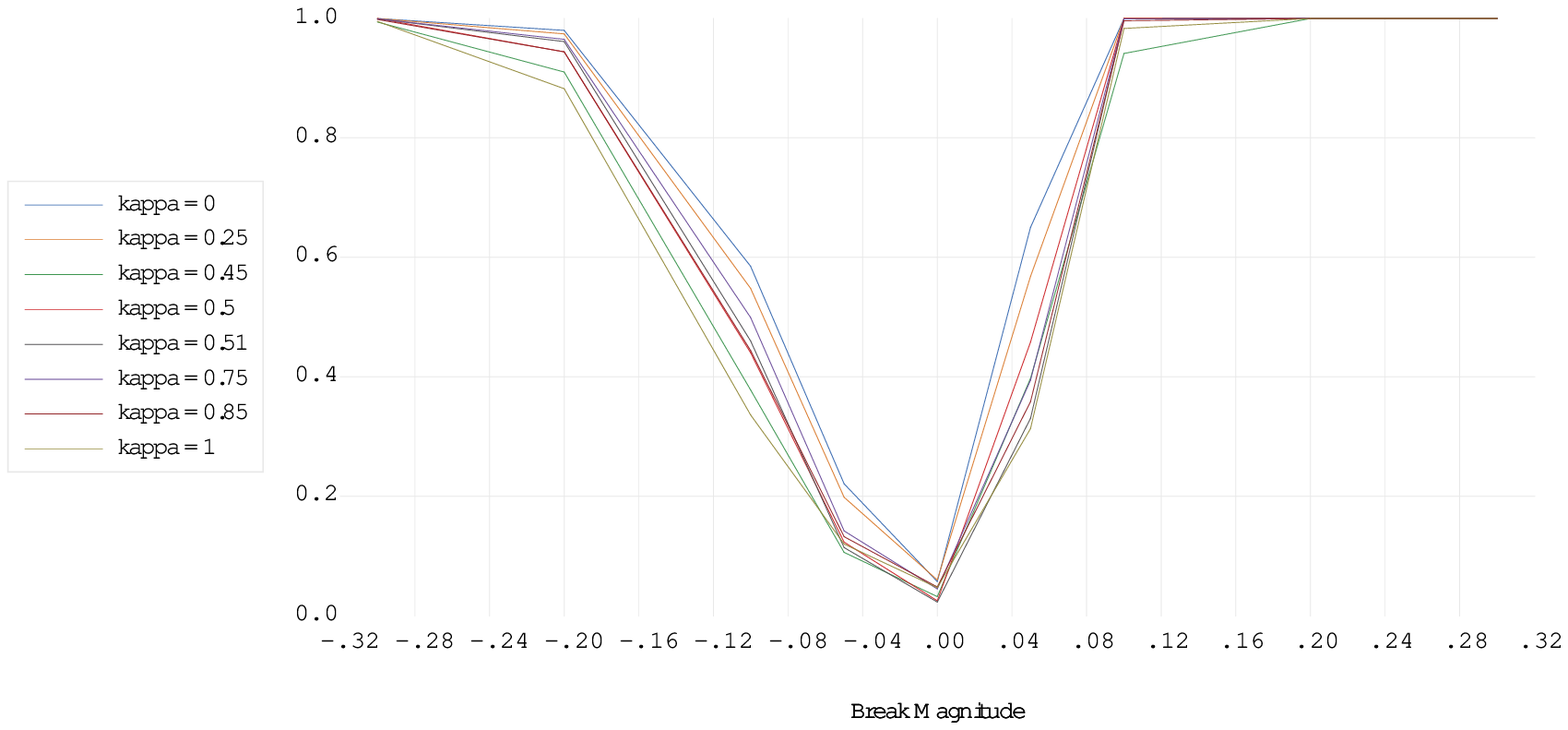}
    \captionof{subfigure}{$\beta_0=0.99$}
    \label{fig:phi99m}
\end{minipage} \\[0.25cm]
\par
\hspace{-2.5cm} 
\begin{minipage}{0.4\textwidth}
\centering
    \includegraphics[scale=0.4]{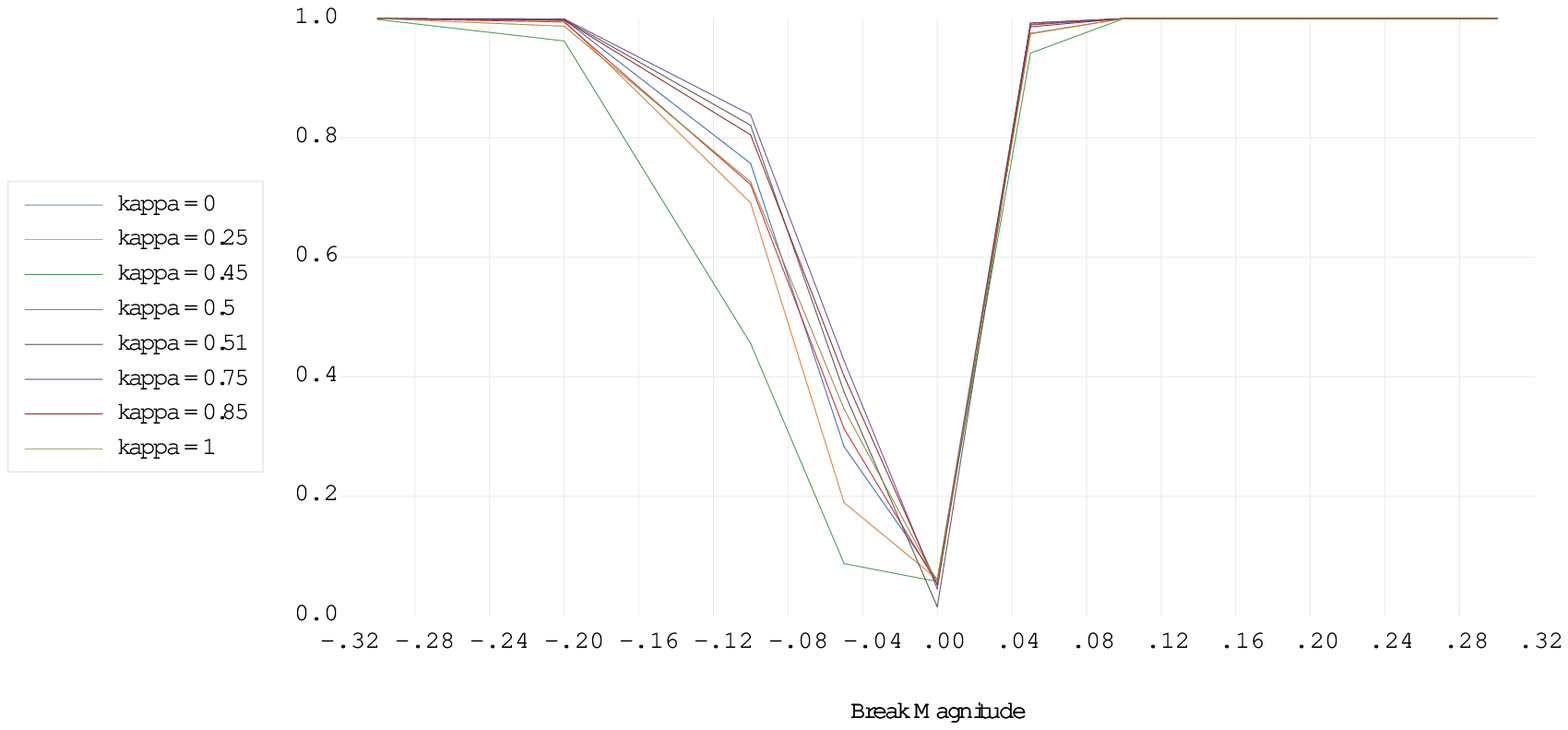}
    \captionof{subfigure}{$\beta_0=1.01$}
    \label{fig:phi01m}
\end{minipage}%
\begin{minipage}{0.4\textwidth}
\centering
   \includegraphics[scale=0.4]{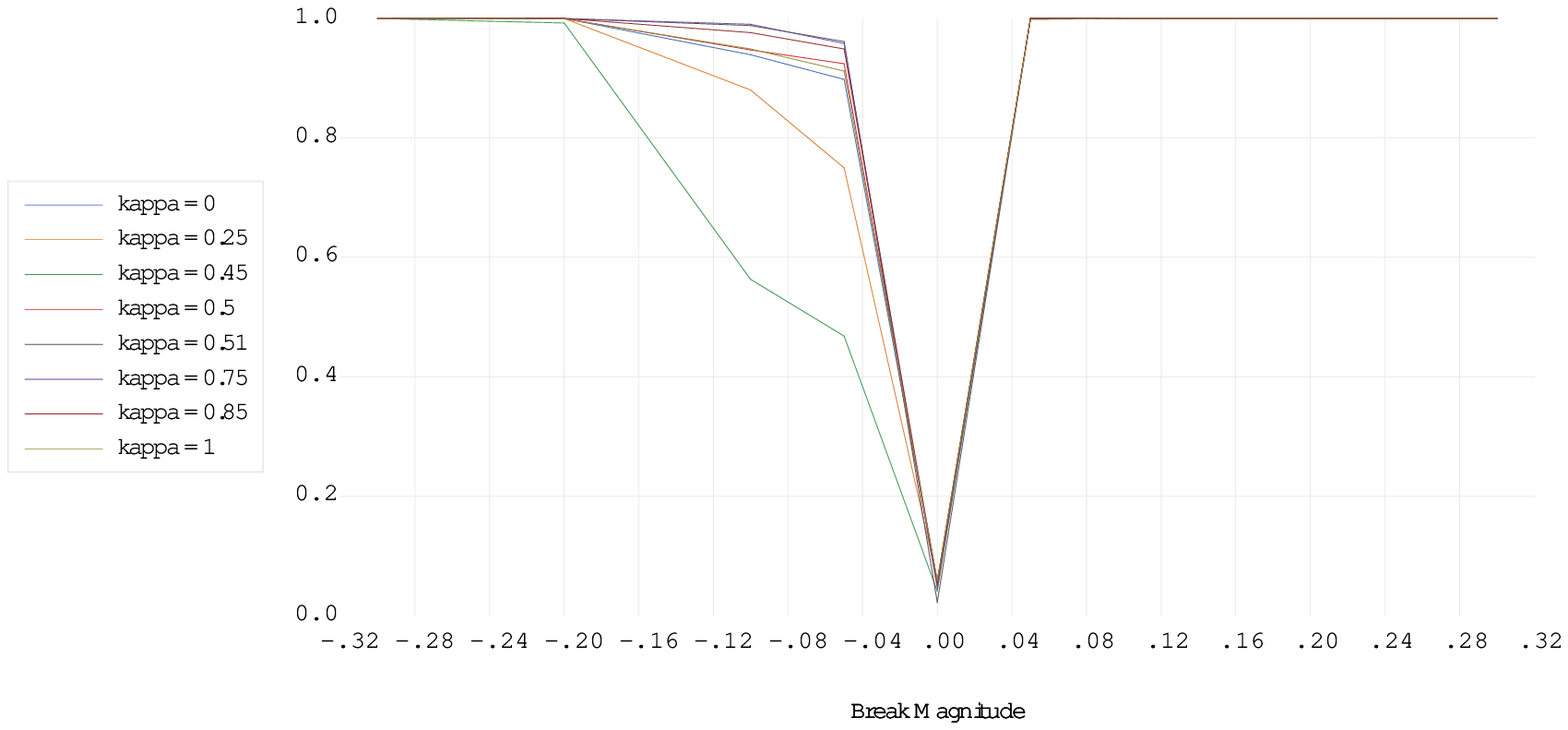}
    \captionof{subfigure}{$\beta_0=1.02$}
    \label{fig:phi02m}
\end{minipage} \\[0.25cm]
\par
%\caption*{.}
\end{figure}

\begin{figure}[!b]
\caption{Empirical rejection frequencies - heteroskedasticity in $\protect%
\epsilon_{i,2}$ and end-of-sample break}
\label{fig:bubb2}\centering
\hspace{-2.5cm} 
\begin{minipage}{0.4\textwidth}
\centering
    \includegraphics[scale=0.4]{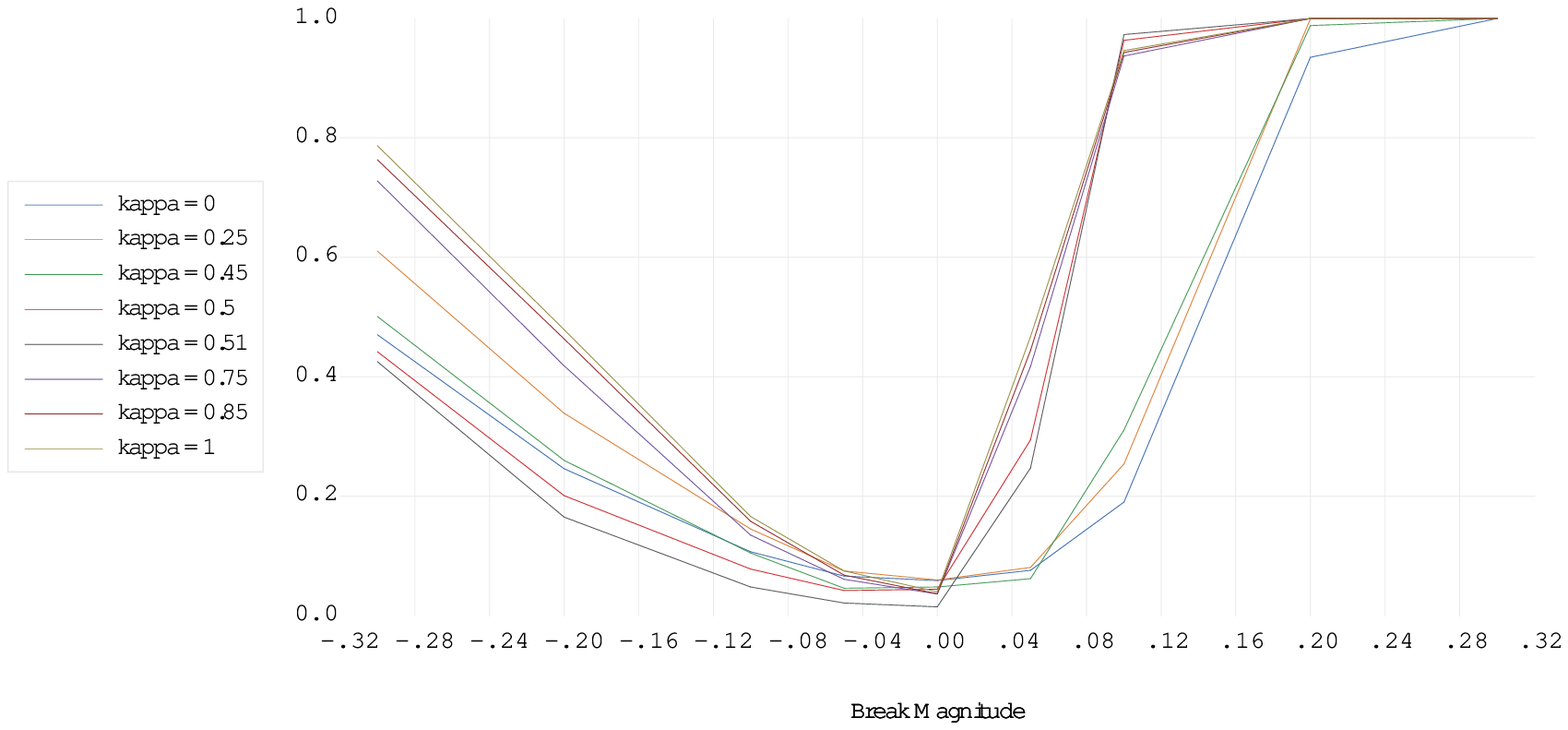}
    \captionof{subfigure}{$\beta_0=0.98$}
    \label{fig:phi98e}
\end{minipage}%
\begin{minipage}{0.4\textwidth}
\centering
   \includegraphics[scale=0.4]{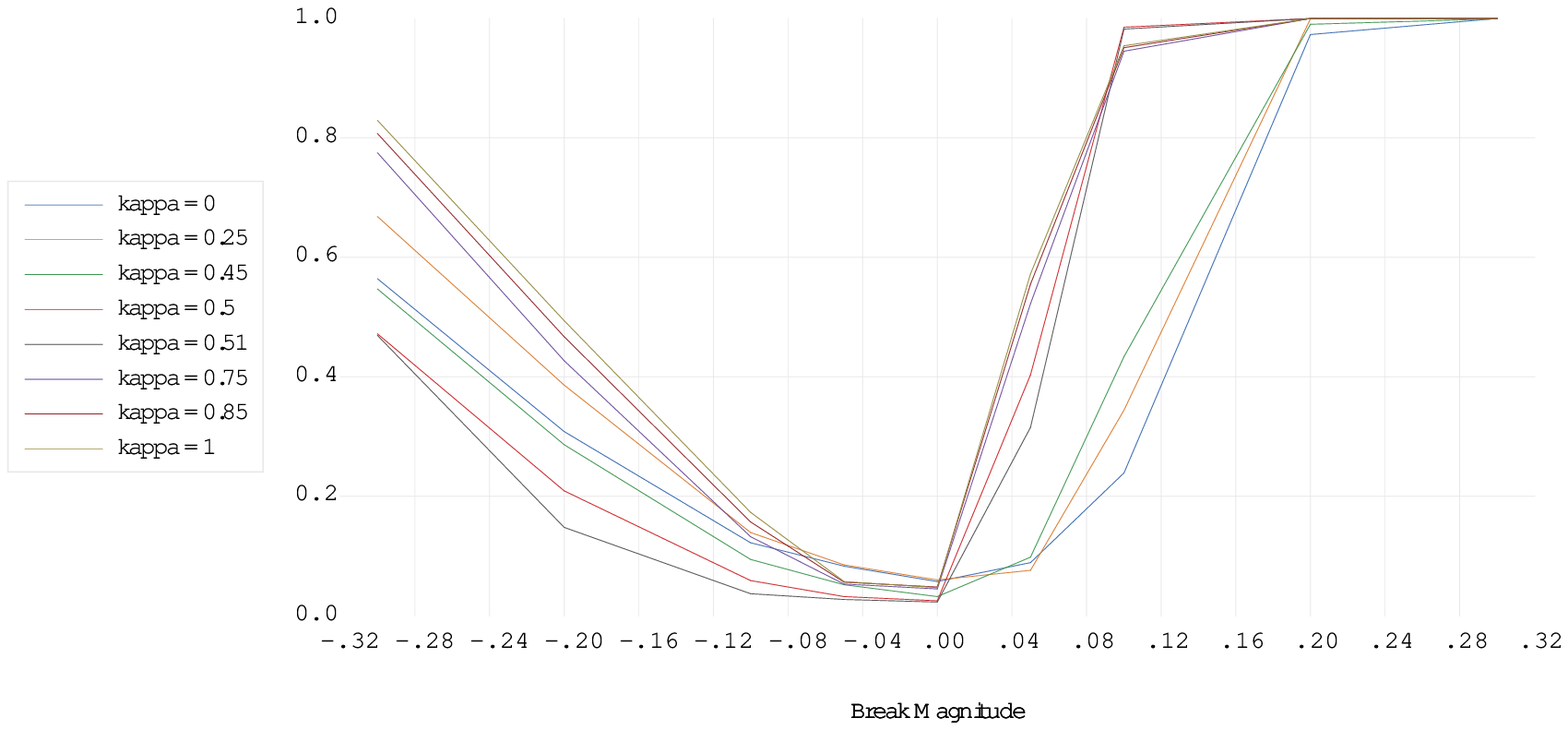}
    \captionof{subfigure}{$\beta_0=0.99$}
    \label{fig:phi99e}
\end{minipage} \\[0.25cm]
\par
\hspace{-2.5cm} 
\begin{minipage}{0.4\textwidth}
\centering
    \includegraphics[scale=0.4]{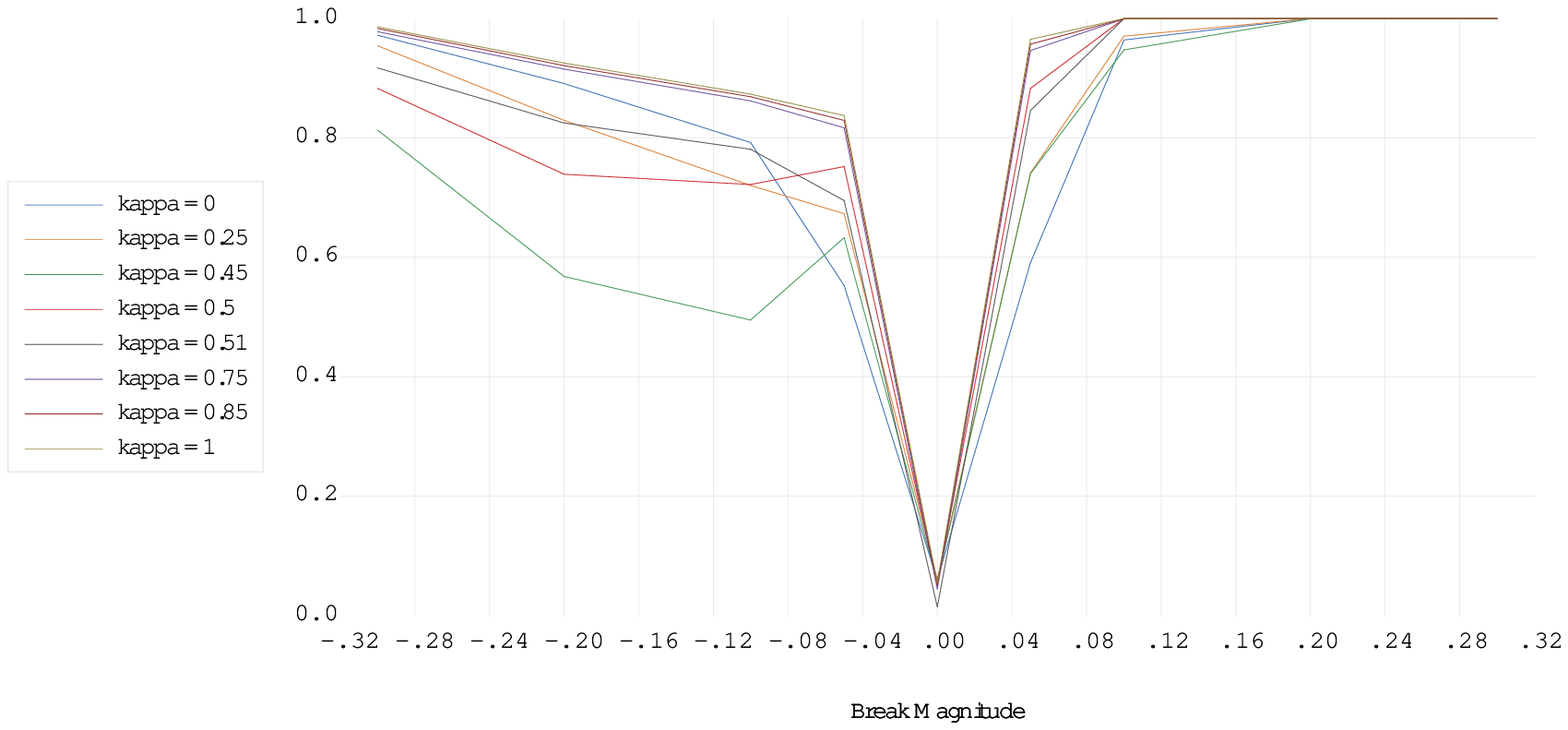}
    \captionof{subfigure}{$\beta_0=1.01$}
    \label{fig:phi01e}
\end{minipage}%
\begin{minipage}{0.4\textwidth}
\centering
   \includegraphics[scale=0.4]{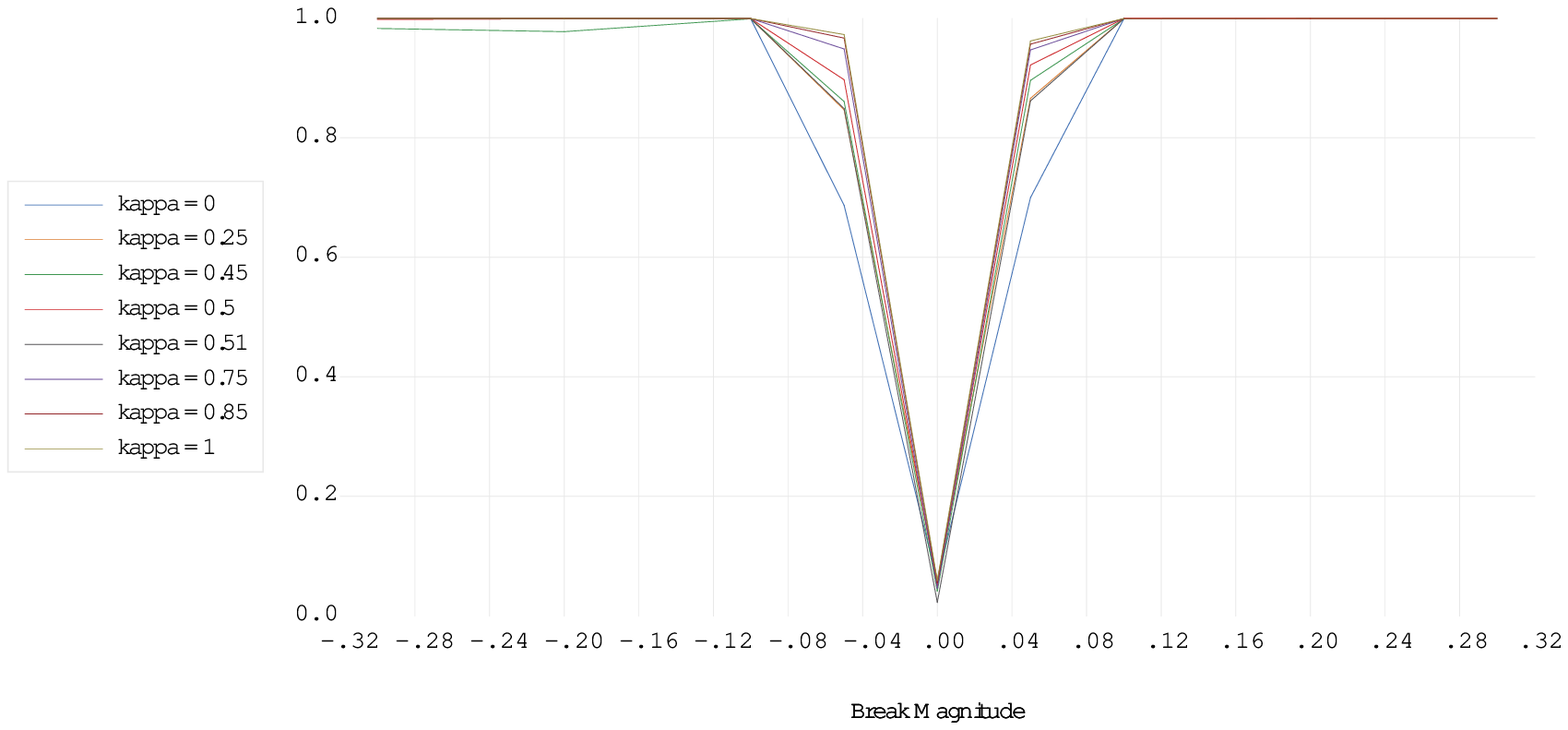}
    \captionof{subfigure}{$\beta_0=1.02$}
    \label{fig:phi02e}
\end{minipage} \\[0.25cm]
\par
%\caption*{.}
\end{figure}

\begin{figure}[!t]
\caption{Monthly log differences of US CPI}
\label{fig:FigCPI}\centering
%\centering
\includegraphics[scale=0.75]{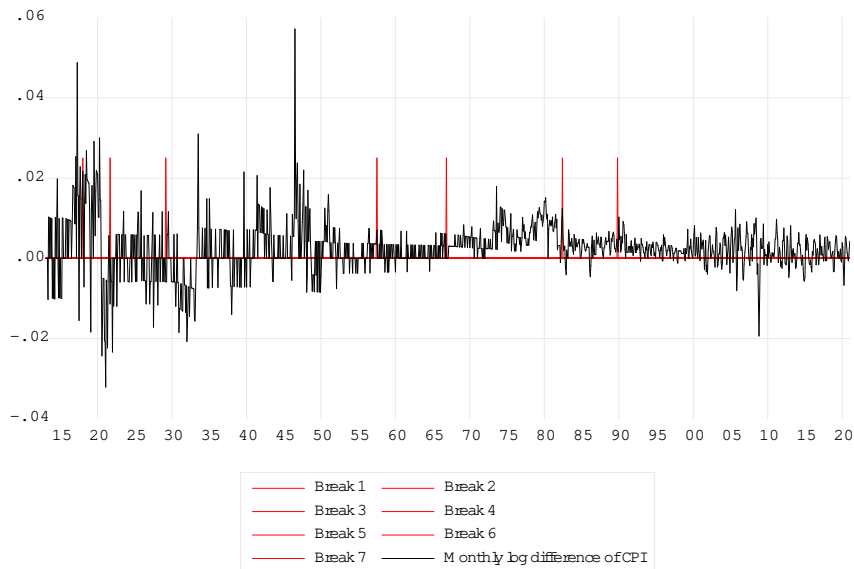}
\end{figure}

\newpage

\end{document}